\documentclass[reqno]{amsart}

\usepackage[all]{xy}
\usepackage{lscape}
\usepackage{amsmath, amssymb, amsfonts, comment}
\usepackage{mathrsfs}
\usepackage{stmaryrd}
\usepackage{color}
\usepackage{mathtools}
\usepackage[normalem]{ulem}
\usepackage{framed}
\usepackage{tikz-cd}
\usepackage{hyperref}
\usepackage{silence}
\usepackage{bbm}

\newcommand{\fdg}[1]{\widetilde{#1}}

\DeclareMathOperator*{\colim}{colim}
\DeclareMathOperator{\coker}{coker}
\DeclareFontFamily{U}{wncy}{}
    \DeclareFontShape{U}{wncy}{m}{n}{<->wncyr10}{}
    \DeclareSymbolFont{mcy}{U}{wncy}{m}{n}
    \DeclareMathSymbol{\Sha}{\mathord}{mcy}{"58}

\DeclareMathOperator{\image}{im}
\DeclareMathOperator{\coim}{coim}

\DeclareMathOperator{\Filt}{Filt}

\newcommand{\ul}{\underline}
\newcommand{\Aa}{\mathcal{A}}
\newcommand{\Cc}{\mathcal{C}}
\newcommand{\CR}{\mathfrak{cr}}
\newcommand{\dbydx}{i^\dagger}

\newcommand{\Bb}{\mathcal{B}}

\newcommand{\Dd}{\mathcal{D}}

\newcommand{\Ii}{\mathcal{I}}

\newcommand{\Mm}{\mathcal{M}}
\newcommand{\Nn}{\mathcal{N}}
\newcommand{\Oo}{\mathcal{O}}
\newcommand{\Pp}{\mathcal{P}}
\newcommand{\Qq}{\mathcal{Q}}
\newcommand{\Ss}{\mathcal{S}}
\newcommand{\Tt}{\mathcal{T}}

\newcommand{\Ww}{\mathcal{W}}
\newcommand{\Zz}{\mathcal{Z}}

\newcommand{\lL}{\mathbb{L}}
\newcommand{\nN}{\mathbb{N}}

\newcommand{\qQ}{\mathbb{Q}}
\newcommand{\rR}{\mathbb{R}}
\newcommand{\sS}{\mathbb{S}}
\newcommand{\zZ}{\mathbb{Z}}

\newcommand{\I}{\mathrm{I}}
\newcommand{\im}{\image}
\newcommand{\id}{\mathrm{id}}
\newcommand{\Hom}{\mathrm{Hom}}

\renewcommand{\H}{\mathrm{H}}
\newcommand{\End}{\mathrm{End}}
\newcommand{\Tw}{\mathrm{Tw}}
\newcommand{\cAs}{\mathrm{cAs}}
\newcommand{\cAsa}{\mathrm{cAs}^{\text{!`}}}
\newcommand{\cAi}{\mathrm{cA}_\infty}
\newcommand{\cLie}{\mathrm{cLie}}
\newcommand{\Li}{\mathrm{L}_\infty}
\newcommand{\cLi}{\mathrm{cL}_\infty}
\newcommand{\As}{\mathrm{As}}
\newcommand{\Asa}{\mathrm{As}^{\text{!`}}}
\newcommand{\Ai}{\mathrm{A}_\infty}
\newcommand{\B}{\mathrm{B}}
\newcommand{\hB}{\hat{\mathrm{B}}}
\newcommand{\hOm}{\hat{\mathrm{\Omega}}}

\newcommand{\Gr}{\mathrm{Gr}\,}

\newcommand{\op}{\mathrm{op}}
\newcommand{\Prim}{\mathrm{Prim}}

\newcommand{\graded}{\mathsf{g}}
\newcommand{\dg}{\mathsf{dg}}
\newcommand{\pg}{\mathsf{pg}}

\newcommand{\pgA}{\mathsf{pg\Aa}}
\newcommand{\complete}{\widehat{\Filt}}
\newcommand{\Comp}{\complete}
\newcommand{\compa}{\complete^{\mathsf{gr}}\! \!}

\DeclareMathOperator{\ob}{ob}

\newcommand{\cdgA}{\widehat{\mathsf{dg \Aa}}}

\newcommand{\capgA}{\widehat{\mathsf{\pgA}}^{\mathsf{gr}}}
\newcommand{\cdgSA}{\mathsf{\sS \text{-}Mod}(\cdgA)}

\newcommand{\capgSA}{\mathsf{\sS \text{-} Mod}(\widehat{\Aa}^{\mathsf{gr}})}

\newcommand{\fapgA}{\Filt^{\mathsf{gr}}(\Aa)}
\newcommand{\fapg}{\Filt^{\mathsf{gr}}}

\newcommand{\Mod}{\mathsf{Mod}}
\newcommand{\Coop}{\mathsf{Coop}}

\newcommand{\cOp}{\mathsf{cOp}}

\newcommand{\Alg}{\mathsf{Alg}}
\newcommand{\Hoalg}{\mathsf{Hoalg}}

\newcommand{\ringK}{\mathbf{R}}

\newcommand{\del}{\partial}

\newcommand{\free}{\mathcal{T}}
\newcommand{\cfree}{\mathsf{c}\mathcal{T}}
\newcommand{\cofree}{\free^c}
\newcommand{\antishriek}{\text{!`}}
\newcommand{\Ooa}{\Oo^\text{!`}}
\newcommand{\Ppa}{\Pp^\text{!`}}

\newcommand{\oD}{\overline{D}}

\newcommand{\dt}{\delta_\theta}

\newcommand{\coideal}[1]{\widetilde{#1}}

\theoremstyle{plain}
\newtheorem{thm}[subsection]{Theorem}
\newtheorem{prop}[subsection]{Proposition}
\newtheorem{lemma}[subsection]{Lemma}

\newtheorem{cor}[subsection]{Corollary}

\theoremstyle{definition}

\newtheorem{defi}[subsection]{Definition}
\newtheorem{fact}[subsection]{Fact}
\newtheorem{construction}[subsection]{Construction}
\newtheorem{notation}[subsection]{Notation}

\theoremstyle{remark}
\newtheorem{remark}[subsection]{Remark}

\newtheorem{examples}[subsection]{Examples}

\newcommand{\curvA}{\vcenter{
\xymatrix@M=0pt@R=3pt@C=3pt{&&\\
& *{} \ar@{{*}}[u] \ar@{-}[d]\\
&&}}}

\newcommand{\as}{\vcenter{
\xymatrix@M=0pt@R=4pt@C=4pt{&&\\
& *{} \ar@{-}[lu] \ar@{-}[ur] \ar@{-}[d] &\\
&&}}}

\newcommand{\ass}{\vcenter{
    \xymatrix@M=0pt@R=4pt@C=4pt{&&&&\\
      \ar@{-}[drdr] && \ar@{-}[dr] && \ar@{-}[dldl]\\
      &&&& \\
      && *{}\ar@{-}[d] && \\
      &&&&}}
- 
\vcenter{ \xymatrix@M=0pt@R=4pt@C=4pt{
      &&&&\\ \ar@{-}[ddrr] && \ar@{-}[dl] && \ar@{-}[ddll]\\
      &&&& \\
      && *{} \ar@{-}[d] &&\\
      &&}}}

\newcommand{\curvAs}{\vcenter{\xymatrix@M=0pt@R=4pt@C=4pt{
      &&&&\\
      &&&&\\
      & \ar@{-}[dr] && \ar@{{*}}[u] \ar@{-}[dl] &\\
      && *{} \ar@{-}[d] &&\\
      &&}}
-
\vcenter{ \xymatrix@M=0pt@R=4pt@C=4pt{
      &&&&\\
      &&&&\\
      & \ar@{{*}}[u] \ar@{-}[dr] && \ar@{-}[dl] & \\
      && *{} \ar@{-}[d] &&\\
      &&}}}

\newcommand{\thetamun}{\vcenter{
\xymatrix@M=1pt@R=3pt@C=3pt{&&\\ & \mu_n^S \ar@{-}[ul] \ar@{-}[u] \ar@{-}[ur] \ar@{-}[d] &\\ & \theta \ar@{-}[d] &\\
&&}}}

\newcommand{\muntheta}{\vcenter{
\xymatrix@M=1pt@R=3pt@C=3pt{&&\\ & \theta \ar@{-}[u] &\\ & \mu_n^S \ar@{-}[ul] \ar@{-}[u] \ar@{-}[ur] \ar@{-}[d] &\\
&&}}}

\newcommand{\munpucurvA}{\vcenter{
\xymatrix@M=1pt@R=3pt@C=3pt{&&&\\ & *{} \ar@{-}[u] \ar@{-}[u] &&\\ & \mu_{n+1}^{S_j} \ar@{-}[ul] \ar@{-}[u] \ar@{-}[ur] \ar@{-}[d] &&\\
&&}}}

\newcommand{\munpucurvAp}{\vcenter{
\xymatrix@M=1pt@R=3pt@C=3pt{&&&\\ & *{} \ar@{-}[u] \ar@{-}[u] &&\\ & \mu_{n+1}^{S_j'} \ar@{-}[ul] \ar@{-}[u] \ar@{-}[ur] \ar@{-}[d] &&\\
&&}}}

\newcommand{\gentheta}{\vcenter{
\xymatrix@M=0pt@R=3pt@C=3pt{&&\\ &&\\ & *{} \ar@{-}[uu] &\\
& *{} \ar@{{*}}[u] \ar@{-}[d]\\
&&}}}

\newcommand{\unit}{\vcenter{
\xymatrix@M=0pt@R=3pt@C=3pt{&&\\ & *{} \ar@{-}[u] &\\
& *{} \ar@{-}[u] \ar@{-}[d]\\
&&}}}

\newcommand{\curvCork}{\tiny \vcenter{\xymatrix@R=2pt@C=2pt{
      &&&&\\
      & \ar@{-}[dr] && \theta_2 \ar@{-}[dl] &\\
      && \theta_1 \ar@{-}[d] &&\\
      &&}}}

\newcommand{\curvCog}{\tiny \vcenter{\xymatrix@R=2pt@C=2pt{
      &&\\
      & x_k \ar@{-}[d] \ar@{-}[u] &\\
      & x_k \ar@{-}[d] &\\
      &&}}}

\newcommand{\term}{\tiny \vcenter{\xymatrix@R=2pt@C=2pt{
      &&&&\\
      & t_1 \ar@{-}[ul] \ar@{-}[u] \ar@{-}[dr] & \dots & t_k \ar@{-}[u] \ar@{-}[ur] \ar@{-}[dl] &\\
      && \mu \ar@{-}[d] &&\\
      &&}}}
      
\newcommand{\dtbottom}{\tiny \vcenter{\xymatrix@R=2pt@C=2pt{
      &&&&\\
      & t_1 \ar@{-}[ul] \ar@{-}[u] \ar@{-}[dr] & \dots & t_k \ar@{-}[u] \ar@{-}[ur] \ar@{-}[dl] &\\
      && \dt(\mu) \ar@{-}[d] &&\\
      &&}}}

\newcommand{\dttop}{\tiny \vcenter{\xymatrix@R=2pt@C=2pt{
      &&&&&&\\
      & t_1 \ar@{-}[ul] \ar@{-}[u] \ar@{-}[drr] & \dots & \dt(t_j) \ar@{-}[ul] \ar@{-}[u] \ar@{-}[ur] \ar@{-}[d] & \dots & t_k \ar@{-}[u] \ar@{-}[ur] \ar@{-}[dll] &\\
      &&& \mu \ar@{-}[d] &&&\\
      &&&&}}}

\newcommand{\tin}{\tiny \vcenter{\xymatrix@R=2pt@C=2pt{
      &&&&&&\\
      &&& t_j \ar@{-}[ul] \ar@{-}[u] \ar@{-}[ur] &&&\\
      & t_1 \ar@{-}[ul] \ar@{-}[u] \ar@{-}[drr] & \dots & \theta \ar@{-}[u] \ar@{-}[d] & \dots & t_k \ar@{-}[u] \ar@{-}[ur] \ar@{-}[dll] &\\
      &&& \mu \ar@{-}[d] &&&\\
      &&&&}}}

\newcommand{\dtmu}{\tiny \vcenter{\xymatrix@R=2pt@C=2pt{
      &&\\
      & \mu \ar@{-}[d] \ar@{-}[ul] \ar@{-}[u] \ar@{-}[ur] &\\
      & \theta \ar@{-}[d] &\\
      &&}}
      -
      \sum_{E(\mu)} \vcenter{\xymatrix@R=2pt@C=2pt{
      &&\\
      & \theta \ar@{-}[u] &\\
      & \mu \ar@{-}[ul] \ar@{-}[u] \ar@{-}[ur] \ar@{-}[d] &\\
      &&}}}

\newcommand{\tij}{\tiny \vcenter{\xymatrix@R=2pt@C=2pt{
      &&&\\
      && t_j \ar@{-}[u] \ar@{-}[ul] \ar@{-}[ur] \ar@{-}[dl] &\\
      & t_i \ar@{-}[d] \ar@{-}[ul] &&\\
      &&}}}

\newcommand{\tdtij}{\tiny \vcenter{\xymatrix@R=2pt@C=2pt{
      &&&\\
      && t_j \ar@{-}[u] \ar@{-}[ul] \ar@{-}[ur] \ar@{-}[dl] &\\
      & \dt(t_i) \ar@{-}[d] \ar@{-}[ul] &&\\
      &&}}}

\newcommand{\tidtj}{\tiny \vcenter{\xymatrix@R=2pt@C=2pt{
      &&&\\
      && \dt(t_j) \ar@{-}[u] \ar@{-}[ul] \ar@{-}[ur] \ar@{-}[dl] &\\
      & t_i \ar@{-}[d] \ar@{-}[ul] &&\\
      &&}}}

\newcommand{\titj}{\tiny \vcenter{\xymatrix@R=2pt@C=2pt{
      &&&\\
      && t_j \ar@{-}[u] \ar@{-}[ul] \ar@{-}[ur] \ar@{-}[d] &\\
      && \theta \ar@{-}[u] \ar@{-}[dl] &\\
      & t_i \ar@{-}[d] \ar@{-}[ul] &&\\
      &&}}}

\newcommand{\toDdtij}{\tiny \vcenter{\xymatrix@R=2pt@C=2pt{
      &&&\\
      && t_j \ar@{-}[u] \ar@{-}[ul] \ar@{-}[ur] \ar@{-}[dl] &\\
      & \oD \left(\dt(t_i)\right) \ar@{-}[d] \ar@{-}[ul] &&\\
      &&}}}

\newcommand{\tioDdtj}{\tiny \vcenter{\xymatrix@R=2pt@C=2pt{
      &&&\\
      && \oD \left(\dt(t_j) \right) \ar@{-}[u] \ar@{-}[ul] \ar@{-}[ur] \ar@{-}[dl] &\\
      & t_i \ar@{-}[d] \ar@{-}[ul] &&\\
      &&}}}

\newcommand{\tittj}{\tiny \vcenter{\xymatrix@R=1pt@C=1pt{
&&&&\\
&&& t_j \ar@{-}[dd] \ar@{-}[ul] \ar@{-}[u] \ar@{-}[ur] &\\
\ar@{.}[rrrr] &&&& \\
&&& \theta & \\
\ar@{.}[rrrr] &&&& \\
&& t_i \ar@{-}[d] \ar@{-}[ul] \ar@{-}[uur] && \\
&&&}}}

\newcommand{\dottjdt}{\tiny \vcenter{\xymatrix@R=1pt@C=1pt{
&&&&\\
& \delta_i^3 \ar@{-}[dddrr] &&&&\\
\ar@{.}[rrrrrr] &&&&&& \\
&&&& t_j \ar@{-}[dl] \ar@{-}[uul] \ar@{-}[uu] \ar@{-}[uur] &&\\
&&& \delta_i^2 \ar@{-}[ddl] &&& \\
\ar@{.}[rrrrrr] &&&&&& \\
&& \delta_i^1 \ar@{-}[d] \ar@{-}[ul] && \\
&&&}}}

\newcommand{\tktj}{\tiny \vcenter{\xymatrix@R=1pt@C=1pt{
&&&&\\
& \theta_l \ar@{-}[dddrr] &&&&\\
\ar@{.}[rrrrrr] &&&&&& \\
&&&& t_j \ar@{-}[dl] \ar@{-}[uul] \ar@{-}[uu] \ar@{-}[uur] &&\\
&&& \theta_k \ar@{-}[ddl] &&& \\
\ar@{.}[rrrrrr] &&&&&& \\
&& t_i \ar@{-}[d] \ar@{-}[ul] && \\
&&&}}}

\newcommand{\tott}{\tiny \vcenter{\xymatrix@R=1pt@C=1pt{
&&&&\\
& \theta_2 \ar@{-}[dddrr] &&&&\\
\ar@{.}[rrrrrr] &&&&&& \\
&&&& t_j \ar@{-}[dl] \ar@{-}[uul] \ar@{-}[uu] \ar@{-}[uur] &&\\
&&& \theta_1 \ar@{-}[ddl] &&& \\
\ar@{.}[rrrrrr] &&&&&& \\
&& t_i \ar@{-}[d] \ar@{-}[ul] && \\
&&&}}}

\newcommand{\dotttj}{\tiny \vcenter{\xymatrix@R=1pt@C=1pt{
&&&&&\\
\ar@{.}[rrrrrr] &&&&&& \\
&&&& t_j \ar@{-}[dl] \ar@{-}[uul] \ar@{-}[uu] \ar@{-}[uur] &&\\
&&& \delta_i^3 \ar@{-}[ddl] \ar@{-}[uuull] &&& \\
\ar@{.}[rrrrrr] &&&&&& \\
&& \delta_i^{12} \ar@{-}[d] && \\
&&&}}}

\newcommand{\xltj}{\tiny \vcenter{\xymatrix@R=1pt@C=1pt{
&&&&&\\
\ar@{.}[rrrrr] &&&&& \\
&&& t_j \ar@{-}[d] \ar@{-}[uul] \ar@{-}[uu] \ar@{-}[uur] &&\\
&&& x_l \ar@{-}[dd] &&\\
\ar@{.}[rrrrr] &&&&& \\
& \ar@{-}[dr] && x_k \ar@{-}[dl] && \\
&& t_i \ar@{-}[d] &&& \\
&&&&}}}

\newcommand{\tidoo}{\tiny \vcenter{\xymatrix@R=1pt@C=1pt{
&&&&\\
&&& \delta_j^{23} \ar@{-}[dd] \ar@{-}[ul] \ar@{-}[u] \ar@{-}[ur] &&\\
\ar@{.}[rrrr] &&&& \\
&&& \delta_j^1 \ar@{-}[dl] & \\
&& t_i \ar@{-}[dd] \ar@{-}[uuull] && \\
\ar@{.}[rrrr] &&&& \\
&&&&}}}

\newcommand{\tidot}{\tiny \vcenter{\xymatrix@R=1pt@C=1pt{
&&&&&\\
&& \delta_j^{2} \ar@{-}[u] \ar@{-}[ddr] && \delta_j^{3} \ar@{-}[ddl] \ar@{-}[ul] \ar@{-}[ur] &\\
\ar@{.}[rrrr] &&&& \\
&&& \delta_j^1 \ar@{-}[dl] & \\
&& t_i \ar@{-}[dd] \ar@{-}[uuull] && \\
\ar@{.}[rrrr] &&&& \\
&&&&}}}

\newcommand{\titko}{\tiny \vcenter{\xymatrix@R=1pt@C=1pt{
&&&&\\
&&& t_j \ar@{-}[d] \ar@{-}[ul] \ar@{-}[u] \ar@{-}[ur] &\\
&&& x_{l} \ar@{-}[dd] \ar@{-}[u] &&\\
\ar@{.}[rrrr] &&&& \\
&&& x_k \ar@{-}[dl] & \\
&& t_i \ar@{-}[dd] \ar@{-}[uuull] && \\
\ar@{.}[rrrr] &&&& \\
&&&&}}}

\newcommand{\titkt}{\tiny \vcenter{\xymatrix@R=1pt@C=1pt{
&&&&&\\
&& \theta_{l} \ar@{-}[ddr] && t_j \ar@{-}[ddl] \ar@{-}[ul] \ar@{-}[u] \ar@{-}[ur] &\\
\ar@{.}[rrrr] &&&& \\
&&& \theta_k \ar@{-}[dl] & \\
&& t_i \ar@{-}[dd] \ar@{-}[uuull] && \\
\ar@{.}[rrrr] &&&& \\
&&&&}}}

\newcommand{\toto}{\tiny \sum_{\overset{\circ}{E}(t_i)} \vcenter{\xymatrix@R=1pt@C=1pt{
&&&&&&\\
& \theta_2 \ar@{-}[ddr] && \ar@{-}[ddl] && t_j \ar@{-}[dddll] \ar@{-}[ul] \ar@{-}[u] \ar@{-}[ur] &\\
\ar@{.}[rrrrrr] &&&&&& \\
&& \theta_1 \ar@{-}[dr] &&&& \\
&&& t_i \ar@{-}[dd] &&& \\
\ar@{.}[rrrrrr] &&&&&& \\
&&&&&}}
+
\vcenter{\xymatrix@R=1pt@C=1pt{
&&&&&&\\
& \ar@{-}[dddrr] && \theta_2 \ar@{-}[ddr] && t_j \ar@{-}[ddl] \ar@{-}[ul] \ar@{-}[u] \ar@{-}[ur] &\\
\ar@{.}[rrrrrr] &&&&&& \\
&&&& \theta_1 \ar@{-}[dl] && \\
&&& t_i \ar@{-}[dd] &&& \\
\ar@{.}[rrrrrr] &&&&&& \\
&&&&&}}
+
\vcenter{\xymatrix@R=1pt@C=1pt{
&&&&&&\\
& \theta_2 \ar@{-}[dddrr] && \ar@{-}[ddr] && t_j \ar@{-}[ddl] \ar@{-}[ul] \ar@{-}[u] \ar@{-}[ur] &\\
\ar@{.}[rrrrrr] &&&&&& \\
&&&& t_i \ar@{-}[dl] && \\
&&& \theta_1 \ar@{-}[dd] &&& \\
\ar@{.}[rrrrrr] &&&&&& \\
&&&&&}}
}

\newcommand{\toxoa}{\tiny \vcenter{\xymatrix@R=1pt@C=1pt{
&&&&&\\
& x_k \ar@{-}[dd] &&& t_j \ar@{-}[dddll] \ar@{-}[ul] \ar@{-}[u] \ar@{-}[ur] &\\
\ar@{.}[rrrrr] &&&&&& \\
& x_k \ar@{-}[dr] &&&& \\
&& t_i \ar@{-}[dd] &&& \\
\ar@{.}[rrrrr] &&&&&& \\
&&&&}}}
 
\newcommand{\toxob}{\tiny \vcenter{\xymatrix@R=1pt@C=1pt{
&&&&\\
&&& t_j \ar@{-}[d] \ar@{-}[ul] \ar@{-}[u] \ar@{-}[ur] &&\\
&&& x_k \ar@{-}[dd] &\\
\ar@{.}[rrrr] &&&& \\
& \ar@{-}[dr] && x_k \ar@{-}[dl] & \\
&& t_i \ar@{-}[dd] && \\
\ar@{.}[rrrr] &&&& \\
&&&&}}}

\newcommand{\xttt}{\tiny \vcenter{\xymatrix@R=1pt@C=1pt{
&&&&&\\
\ar@{.}[rrrrr] &&&&& \\
&&& t_j \ar@{-}[d] \ar@{-}[uul] \ar@{-}[uu] \ar@{-}[uur] &&\\
&&& x_k \ar@{-}[dd] &&\\
\ar@{.}[rrrrr] &&&&& \\
& \ar@{-}[dr] && x_k \ar@{-}[dl] && \\
&& t_i \ar@{-}[d] &&& \\
&&&&}}}

\newcommand{\xtto}{\tiny \vcenter{\xymatrix@R=1pt@C=1pt{
&&&&\\
& \ar@{-}[ddr] && t_j \ar@{-}[ddl] \ar@{-}[ul] \ar@{-}[u] \ar@{-}[ur] &\\
\ar@{.}[rrrr] &&&& \\
&& t_i \ar@{-}[d] &&&\\
&& x_k \ar@{-}[dd] &&&& \\
\ar@{.}[rrrr] &&&& \\
&& x_k \ar@{-}[d] && \\
&&&}}}

\newcommand{\ttto}{\tiny \vcenter{\xymatrix@R=1pt@C=1pt{
&&&&\\
&& \theta_2 \ar@{-}[ddr] && \ar@{-}[ddl] &\\
\ar@{.}[rrrrrr] &&&&&& \\
&&& \theta_1 &&&\\
&&& t_j \ar@{-}[ddl] \ar@{-}[uuull] \ar@{-}[u] \ar@{-}[uuurr] &&& \\
\ar@{.}[rrrrrr] &&&&&& \\
&& t_i \ar@{-}[d] \ar@{-}[ul] && \\
&&&}}}

\newcommand{\ttxo}{\tiny \vcenter{\xymatrix@R=1pt@C=1pt{
&&&&\\
&& x_k \ar@{-}[dd] \ar@{-}[u] &&&\\
\ar@{.}[rrrrr] &&&&& \\
&& x_k &&&\\
&&& t_j \ar@{-}[ddl] \ar@{-}[ul] \ar@{-}[uuu] \ar@{-}[uuur] && \\
\ar@{.}[rrrrr] &&&&& \\
&& t_i \ar@{-}[d] \ar@{-}[ul] && \\
&&&}}}

\urldef\joanhomepage\url{math.univ-toulouse.fr/~jmilles}
\urldef\gabrielhomepage\url{drummondcole.com/gabriel/academic/}
\newcommand{\mailurl}[1]{\email{\href{mailto:#1}{#1}}}

\author{Joan Bellier-Mill\`{e}s}
\address{B.-M.: IMT ; UMR5219\\ Université de Toulouse ; CNRS ; UPS, F-31400 Toulouse, France}
\mailurl{joan.milles@math.univ-toulouse.fr}
\urladdr{\joanhomepage}

\title{Homotopy theory of curved operads and curved algebras}
\author{Gabriel C. Drummond-Cole}
\address{D.-C.:}
\mailurl{gabriel.c.drummond.cole@gmail.com}
\urladdr{\gabrielhomepage}

\thanks{Bellier-Mill\`{e}s was supported by ANR-13-BS02-0005-02 CATHRE, ANR-14-CE25-0008-0 SAT and ANR-17-CE40-0014 CatAG, funded by Agence Nationale pour la Recherche.
Drummond-Cole was supported by IBS-R003-D1.}


\begin{document}

\begin{abstract}
Curved algebras are algebras endowed with a predifferential, which is an endomorphism of degree $-1$ whose square is not necessarily $0$. This makes the usual definition of quasi-isomorphism meaningless and therefore the homotopical study of curved algebras cannot follow the same path as differential graded algebras.

In this article, we propose to study curved algebras by means of curved operads. We develop the theory of bar and cobar constructions adapted to this new notion as well as Koszul duality theory. To be able to provide meaningful definitions, we work in the context of objects which are filtered and complete and become differential graded after applying the associated graded functor.

This setting brings its own difficulties but it nevertheless permits us to define a combinatorial model category structure that we can transfer to the category of curved operads and to the category of algebras over a curved operad using free-forgetful adjunctions.

We address the case of curved associative algebras. We recover the notion of curved $\Ai$-algebras, and we show that the homotopy categories of curved associative algebras and of curved $\Ai$-algebras are Quillen equivalent.
\end{abstract}

\maketitle

\section*{Introduction}

\subsection*{Motivation}

The primary goal of this paper is to give a framework to deal with the homological and homotopical theory of curved algebras. 

The most elementary definition of an associative algebra is an underlying vector space or module endowed with algebraic structures (multiplication, unit). The study of extensions or of deformations of an associative algebra leads to the definition of the Hochschild (co)homology, which is defined by means of the Hochschild (co)chain complex. Revisiting this definition in terms of derived functors leads up to the notion of resolution and in particular of quasi-isomorphism. Going further, the theory of model categories provides powerful tools to extend the previous ideas to many other contexts including the study of other kinds of algebras (commutative algebras, Lie algebras, \dots). The category of dg modules over a ring is endowed with a model category structure whose weak equivalences are quasi-isomorphisms. Hochschild (co)homology has a meaningful interpretation in this context and it opens the door for (co)homology theories of other types of algebras. 

Now suppose that we want to follow this path for \emph{curved} associative algebras. 
Such algebras are equipped not with a differential but rather with a predifferential. Instead of squaring to zero, the square of the predifferential is equal to the bracket with a closed element called the curvature. 
This difference means that it is not reasonable to expect the category of curved algebras to have an underlying category of dg modules.
There is therefore the need to define a new category (possibly containing the category of dg modules as a subcategory) endowed with a notion of weak equivalences to replace quasi-isomorphisms. 
Only then we will be able to achieve our goal.

\subsection*{Approach and antecedents}

At the heart of the homological and homotopical study of algebras are the notions of an operad, used to encode algebras, and Koszul duality theory for such operads. 
Koszul duality theory is a homological theory in its definition and in its range of applications. 
The approach taken in this paper is:
\begin{enumerate} 
      \item to define curved algebras as representations of curved operads,
      \item to find an appropriate base category to study curved algebras, and 
      \item to extend Koszul duality theory to the curved context.
\end{enumerate}

The constructions given in this paper are in some sense dual to the constructions given in \cite{HirshMilles}. 
However, the homotopy theory of curved algebras is complicated by the need to work in a kind of filtered context. 
In an unfiltered context, a homotopy algebra equipped with a non-zero curvature is \emph{isomorphic} to an ``algebra'' which has a curvature but otherwise the zero algebra structure~\cite[7.3]{Positselski:TKDCKDCCC} (see also Proposition 3.3 in \cite{DSV20} for a filtered version with the curvature in filtration degree $0$). 

Moving to Koszul duality poses its own problems and this is the reason why we cannot use the existing literature \cite{vL13, vL14}. 
The formulas that one is led naturally to write are infinite sums whose terms eventually go to smaller and smaller submodules of the filtration. 
These sums then only make sense if one further refines from a filtered to a complete context. 
Depending on the side of the duality, there are further minor refinements to make in order to fully capture the phenomena at play; see below for the gory details. 
The filtered context already appears in \cite{DDL18}, and previously, in a more restricted setting in \cite{lP12}. 

In this last article, Positselski has similar motivations to ours.  
He proposes a framework to study the derived category of a curved associative or $\Ai$-algebra. He is also able to develop a Koszul duality theory in this context. We highly recommend the reading of this article's engaging and fruitful introduction. 
We cannot however use his framework since our main example doesn't fit into it. We therefore propose a filtered framework which is different in two ways: we consider filtrations on the base ring which aren't induced by a maximal ideal and we consider filtrations on our objects (algebras, operads, \dots ) which aren't induced by the filtration on the base ring. Another example we have in mind is the curved $\Ai$-algebras appearing in Floer theory in \cite{Fukaya}, which fit into the framework introduced here.

\subsection*{On filtered objects}

Dealing with filtered and complete filtered objects poses multiple technical challenges. 
If the ground category is Abelian, the categories of filtered objects and complete filtered objects in the ground category are only quasi-Abelian~\cite{Schneiders:QACS}.
Colimits are more difficult to compute and the monoidal product must be redefined in order to inherit various desirable qualities.

These details can be addressed more or less by hand; for example, the reference~\cite[7.3]{bF17} contains a good deal of the setup work for the case where the ground ring is a field. It is also possible to perform an $\infty$-category treatment as in \cite{GP18}. 
But for us, a useful way to organize the necessary changes was by recognizing that the various categories of objects of interest form a lattice of \emph{normal reflective embeddings}. 
A subcategory inclusion is reflective when it admits a left adjoint (called a \emph{reflector}) and a reflective inclusion of closed symmetric monoidal categories is a normal reflective embedding when the reflector is extended into a strong symmetric monoidal functor. 

Most of these categorical details are siloed off in Appendix~\ref{appendix: categorical stuff}.

\subsection*{Our main algebraic categories}

Let us give a little more precise focus here to guide the development of the exposition.
The reader is encouraged to think of most of the symmetric monoidal categories as making up a scaffolding for the two cases of actual interest.
These are \emph{curved augmented operads} in the category of \emph{gr-dg complete filtered objects} (treated in Section~\ref{section: operads}) and \emph{altipotent cooperads} in the category of \emph{dg complete filtered objects} (treated in Section~\ref{section:cooperads}).

\subsection*{Our operads}

Again, a curved associative algebra is an associative algebra $A$ endowed with a predifferential $d$ satisfying
\[ d^2 = [\theta,\, -] \]
where $\theta$ is a closed element in $A$ of degree $-2$. In the filtered complete context $A = F_0A \supset F_1A \supset \cdots$, a filtered object endowed with a predifferential is gr-dg when its associated graded is differentially graded, and we will assume that $\theta \in F_1A$ so that our curved algebras are gr-dg. 
The curved algebras in this article can therefore be considered as infinitesimal deformations of flat algebras. We extend this definition to define the notion of a curved operad. A first example is given by the endomorphism operad of a gr-dg object and there is a curved operad $\cAs$ whose algebras are precisely curved associative algebras.

\subsection*{Our cooperads}

To tell a story about the bar-cobar adjunction, we need a notion dual to the notion of (complete) curved augmented operads. 
We use (complete) altipotent cooperads for this dual notion. 
Altipotence is a variation of the notion of conilpotence adapted to the complete setting. These cooperads are only coaugmented after applying the graded functor.

\subsection*{Outline of the remaining contents}

We introduce several (co)free constructions. In the operadic context, we define a free pointed gr-dg operad and a free curved operad. The first one is used to build the cobar construction of an altipotent cooperad, the latter is used to endow the category of curved operads with a model structure. In the cooperadic context, we provide a cofree construction and we use it to define the bar construction of a curved operad.

In the definition of the cobar construction, the notion of an \emph{infinitesimal coideal} of a cooperad appears. This is a little different than an ordinary coideal; it has the same underlying unfiltered object as the cooperad. The only difference is that the counit is constrained to lie in degree $1$ of the filtration. 
As in the article \cite{HirshMilles}, this can be seen as an incarnation of the fact that curvature and counit play a dual role and that asking for the curvature to be in filtration degree $1$ corresponds dually to the requirement that the counit also must be in filtration degree $1$.

The bar and the cobar constructions, presented in Section \ref{section: bar cobar constructions}, fit as usual in an adjunction and are represented by a notion of curved twisting morphisms. These are different from the constructions in \cite{vL14} because of the filtered complete framework and by the fact that we consider a curvature only on one side of the adjunction. In the context of $\ringK$-modules (for a field $\ringK$ of characteristic $0$), the counit of the adjunction provides a graded quasi-isomorphism between curved operads. This resolution is functorial and we can hope to obtain a ``smaller" resolution when dealing with specific examples. This is the objective of the Koszul duality theory for curved operads that we develop in Section \ref{section: koszul duality for curved operads}. The constructions in this section are a little bit more subtle than the classical constructions because of the fact that infinite sums appear. In particular, it is difficult to describe the Koszul dual cooperad associated with a quadratic curved operad. Nevertheless, we define under certain conditions the Koszul dual operad which is easier to compute. Moreover, in the situation where the curved operad is Koszul, a Poincaré--Birkhoff--Witt type isomorphism provides a description of the underlying $\sS$-module by means of the (classical) Koszul dual cooperad of a quadratic operad which is the associated graded of the curved operad. Under the Koszul condition, we again obtain a resolution of the quadratic curved operad, smaller in a precise sense: the generators embed in the bar construction as the 0-homology group for the ``syzygy degree".

We make explicit the case of the curved operad $\cAs$ encoding curved associative algebras in Section \ref{section: assocase}. The curved operad $\cAs$ is Koszul and we can compute its Koszul dual cooperad (and operad), as well as the Koszul resolution. The algebras over the Koszul resolution are the curved $\Ai$-algebras which appear in the literature (to give only a few examples see \cite{GJ90, CD01, bK06, Fukaya, Nicolas, lP19, DSV20}). We finally show that the homotopy categories of curved associative algebras and of curved $\Ai$-algebras are Quillen equivalent.

\subsection*{Model category structure}

Speaking of resolutions and Quillen equivalence are indications that we have a model category structure in mind. 
We establish the existence and properties of this structure in Appendix \ref{appendix: MCS on gr-dg objects}. 
More precisely, there we describe a model structure on the base category of complete gr-dg $\ringK$-modules. This model category enjoys several nice properties: it is a proper cofibrantly generated model structure, it is combinatorial and it is a monoidal model category structure. 
Classical theorems allow us to transfer this cofibrantly generated model structure along a free-forgetful adjunction. 
This enables us to endow the category of complete curved operads with a cofibrantly generated model structure. The bar-cobar resolution and the Koszul resolution are cofibrant only in the underlying category of complete gr-dg $\sS$-modules. 
Similarly, we describe a free functor in the context of algebras over a curved operad and endow the category of algebras over a curved operad with a cofibrantly generated model structure. We provide base change results to compare the homotopy categories of algebras over some curved operads.

In \cite{LegrignouLejay:HTLC}, Le Grignou and Lejay endow the category of algebras over a curved cooperad with a model structure. 
For a curved cooperad, we could try to compare our curved algebras over the dual operad with their algebras and their model structure. 
Because of the freedom we have for our filtrations, it is not clear how to compare the two categories of algebras and the two model structures in general.

\subsection*{Conventions}

The ground category is a Grothendieck category $\Aa$ equipped with a closed (symmetric monoidal) tensor product. We assume moreover that the tensor product preserve colimits in each variable.
For various parts of the exposition, weaker hypotheses suffice but this seems a reasonable place to cut things off.
For example, $\Aa$ can be 
\begin{itemize}
\item the category of $\ringK$-modules for $\ringK$ a commutative ring,
\item the category of graded $\ringK$-modules or complexes of $\ringK$-modules,
\item the category of sheaves of $\ringK$-modules on a topological space $X$,
\item the category of graded sheaves of $\ringK$-modules or complexes of sheaves of $\ringK$-modules on $X$, or
\item for a (graded) ringed space $(X,\, \Oo_X)$, the category of sheaves of (graded) $\Oo_X$-modules.
\end{itemize} 
When we deal with symmetric operads, we want to assume that $\Aa$ is $\mathbb{Q}$-linear.
In the examples above, this is the assumption that $\ringK$ is a $\mathbb{Q}$-algebra.

We use the notation $(\Mm,\, \amalg,\, \otimes,\, \mathbbm{1})$ to denote an additive closed monoidal category with small colimits and limits. 
We assume moreover that the monoidal structure preserves colimits in each variable. 
Examples of such categories $\Mm$ are given by the categories $\Aa$, $\Filt(\Aa)$, $\complete(\Aa)$, and $\compa(\Aa)$ when $\Aa$ satisfies the conditions above.

We want to work with \emph{$\sS$-modules} or \emph{collections} in $\Mm$, which are functors from the groupoid of finite sets to $\Mm$.
We also want to (simultaneously) work with \emph{$\nN$-modules} in $\Mm$, which are functors from the groupoid of ordered finite sets to $\Mm$. 
We sometimes implicitly pass to a skeleton of either of these categories with objects $[n]=\{1,\ldots, n\}$.

The categories of $\sS$-modules and $\nN$-modules support a number of monoidal products built using the monoidal product of $\Mm$. 
The primary one we will want to use is the \emph{composition product} that we denote by $\circ$. 
The forgetful functor from $\sS$-modules to $\nN$-modules does not intertwine the composition products on each of these categories, but nevertheless we will use the same symbol for both cases. 
In the few cases where this might cause confusion, we will be explicit about the distinctions.

In Sections \ref{section: koszul duality for curved operads} and \ref{section: assocase}, we restrict ourselves to categories $\Aa$ of (unbounded) $\zZ$-graded $\ringK$-modules or complexes of $R$-modules, for $R$ a commutative ring.

When $M$ is a $\mathbb{Z}$-graded object, we denote by $s M$ the suspension of $M$, that is the graded object such that $(sM)_n \coloneqq M_{n-1}$.

\tableofcontents

\section*{Acknowledgements}
We would like to thank Mathieu Anel, Joey Hirsh, Brice Le Grignou, Damien Lejay, Victor Roca Lucio, Massimo Pippi, Joseph Tapia, Bertrand To\"en, Bruno Vallette and Sinan Yalin for useful discussions and for sharing early drafts of related work with us.
The first author also thanks Damien Calaque, Ricardo Campos and Joost Nuiten for a discussion about their recent preprint \cite{CCN21} that helped improve a revision of this article.

\section{The filtered framework}

In this preliminary section we establish definitions, terminology, and notation for the base world of filtered and complete objects in which we will work.

\subsection{Predifferential graded objects}

We first fix our convention for filtered objects and predifferential graded objects.

\begin{defi}
\leavevmode
\begin{itemize}
\item A \emph{filtered object} $(\vec{X},\, F)$ (often expressed just as $X$) in $\Aa$ is a $\mathbb{N}^{\op}$-indexed diagram 
\[
F_0 X\gets F_1 X\gets F_2 X\gets \cdots
\]
where each map is a monomorphism.
We often think of the object $(\vec{X},\, F)$ as the object $X\coloneqq F_0 X$ equipped with the extra data of the family of subobjects $\{F_p X\}$.

Every non filtered object $X$ gives rise to a trivially filtered object with 
\[F_p X = 
\begin{cases}
X & p=0\\
\varnothing & p>0.
\end{cases}\]
Morphisms of filtered modules $f : (\vec{X},\, F) \rightarrow (\vec{Y},\, F')$ are morphisms of diagrams. 
In other words, they are $\Aa$-maps $X\to Y$ which are \emph{filtration preserving} in that $f(F_p X) \subset F'_p Y$, for all $p \in \nN$.

\item To the filtered object $(\vec{X},\, F)$, we associate the graded object (i.e., object indexed by the elements of $\mathbb{N}^{\op}$) $\Gr X$ defined by $(\Gr X)_p \coloneqq F_p X/ F_{p+1} X$.
\end{itemize}
We denote by $\Filt(\Aa)$ the category of filtered objects in the category $\Aa$.
\end{defi}

\begin{remark}
\begin{itemize}
\item
Since we are only interested in \emph{decreasing} filtrations, we omit the adjective ``decreasing'' in this article.
\item
The indexing category $\nN$ is not the most general possibility. 
The case of $\zZ$-filtered objects is also interesting but it's technically useful for us to work with a partially ordered monoid with the identity as its lowest element. 
Of course, there are many examples of such monoids other than $\nN$, such as $\qQ_+$ and $\rR_+$. 
But in most examples which arise in practice, it seems that coarsening the filtration to a discrete countable filtration by a monoid isomorphic to $\nN$ does no harm in terms of the algebra that interests us. 
On the other hand, working with non-discrete filtrations poses a number of technical problems, most importantly the fact that the associated graded functor fails to be conservative in general. 
So we restrict our attention to $\nN$-filtrations.
\end{itemize}
\end{remark}

\begin{remark}
It is well-known that the category $\Filt(\Aa)$ is not Abelian in general.
However, this category is a reflective subcategory of the category of $\nN^{\op}$-indexed diagrams in $\Aa$ and thus is bicomplete.
See Appendix~\ref{appendix: categorical stuff} for details.
\end{remark}

\begin{notation}
Denote by $\graded \Aa$ the closed symmetric monoidal category of $\zZ$-graded $\Aa$-objects and by $\dg\Aa$ the closed symmetric monoidal category of $\zZ$-graded complexes in $\Aa$ (dg objects in $\Aa$ for short). A \emph{predifferential} on a graded object $X$ in $\graded \Aa$ is a degree $-1$ map of graded objects $d : X \rightarrow X$. We denote by $\pg\Aa$ the closed symmetric monoidal category of predifferential graded objects (pg objects in $\Aa$ for short). 
\end{notation}

\begin{remark}
When $\Aa$ is a Grothendieck category, the category $\pg\Aa$ of predifferential graded objects is again a Grothendieck category. Colimits are taken degreewise and filtered colimits are degreewise exact hence exact. Finally, if $\{ U_i\}_{i \in I}$ is a family of generators of $\Aa$, then $\{ D^{n,\infty}U_i := (\amalg_{k\leq n} U_i,\, d) \}_{n \in \zZ,\, i \in I}$, where the $k$-th copy of $U_i$ is in degree $k$, $d$ sends $u$ in the $k$-th copy to $u$ in the $(k-1)$-th copy and where the filtration is given by $F_p (D^{n,\infty}U_i) = \amalg_{k\leq n-2p} U$, is a family of generators of $\pg\Aa$.
\end{remark}

\subsection{``Gr'' and associated graded}
For any property $p$ of $\graded \Aa$, we say that a filtered object $X$ is \emph{gr-$p$} if the associated graded object $\Gr X$ is $p$.
We extend this terminology in the obvious way to all other contexts where we have a variant of the associated graded functor (and we make the definitions precise when the terminology is not obvious). This convention already appears in \cite{gS73}.

The following example illustrates this terminological choice.

\begin{defi}
Let $(\vec{X},\, F,\, d)$ be an object of $\Filt(\pg\Aa)$. 
When the predifferential $d$ induces a differential on $\Gr X$, that is to say, 
\[d^2 : F_p X\to F_p X\] 
factors through $F_{p+1}X$ for all $p$, we call $(\vec{X},\, F,\, d)$ \emph{gr-dg}. 
A natural way to associate a dg object to a gr-dg object $(\vec{X},\, F,\, d)$ is to consider the associated graded object
\[ (X,\, d)^{\text{gr}} \coloneqq (\Gr X,\, \Gr d). \]
Accordingly, we say that the \emph{gr-homology} of the gr-dg object $(\vec{X},\, F,\, d)$ is the graded object
\[ \H_\bullet^{\text{gr}} (\vec{X},\, F,\, d) \coloneqq \H_\bullet((X,\, d)^{\text{gr}}) = \H_\bullet(\Gr X,\, \Gr d).\]
We define the corresponding notion of ``quasi-isomorphism" between gr-dg objects. We say that a map $f : (\vec{X},\, F,\, d) \to (\vec{Y},\, F',\, d')$ is a \emph{graded quasi-isomorphism} if it induces a quasi-isomorphism $f^{\text{gr}} : (X,\, d)^{\text{gr}} \to (Y,\, d')^{\text{gr}}$, that is, when $\H_\bullet^{\text{gr}}(f) : \H_\bullet^{\text{gr}} (\vec{X},\, F,\, d) \to \H_\bullet^{\text{gr}} (\vec{Y},\, F',\, d')$ is an isomorphism.
\end{defi}

\begin{remark}
The graded quasi-isomorphisms are the weak equivalences in the model category structure on gr-dg $\ringK$-modules given in Appendix \ref{appendix: MCS on gr-dg objects}.
\end{remark}

From now on we use this kind of ``gr'' terminology without comment.

\subsection{Complete objects}
Given a filtered object $X=(\vec{X},\, F)$, the filtration structure induces maps $F_p X\xrightarrow{i_p} F_0 X$. 

\begin{defi}
We say that $X$ is \emph{complete} if the natural morphism
\[
X\to \lim_p \coker i_p
\]
is an isomorphism.
\end{defi}
We use the notation $\complete(\Aa)$ for the category of complete filtered objects in $\Aa$.
\begin{remark}
Complete objects are a reflective subcategory of filtered objects, with a \emph{completion functor} as a reflector which we will write
\[V\mapsto \hat{V},\quad V\mapsto \widehat{V},\quad V\mapsto V^\wedge,\]
whichever seems typographically most appropriate in context.
See Appendix~\ref{appendix: categorical stuff} for details.
\end{remark}

\subsection{Monoidal products}
The closed symmetric monoidal product of $\Aa$ extends to the filtered (see Corollary~\ref{cor: product on filtered objects}) and complete (Corollary~\ref{cor: product on complete filtered objects}) settings.
The product in the filtered setting has components:
\[
F_p(V\bar\otimes W) \coloneqq \image\left(\left(\colim_{a+b\ge p} F_a V\otimes F_b W\right)\to V\otimes W\right),
\]
and the product in the complete setting is the completion:
\[
V\hat\otimes W \coloneqq (V\bar\otimes W)^\wedge.
\]
In both cases the internal hom objects can be calculated in $\mathbb{N}^{\op}$-indexed diagrams in $\Aa$.

\begin{defi}
We denote by $\compa(\Aa)$ the category of complete gr-dg objects $(\vec{X},\, F,\, d)$.
It is a full subcategory of $\Comp(\pgA)$. Moreover, when $\Aa$ is assumed to be a Grothendieck category, it is a reflexive subcategory of complete pg modules (see Corollary \ref{cor: complete gr-dg reflective}) and the closed symmetric monoidal product of $\Aa$ extends to the complete gr-dg setting (see Corollary \ref{cor: product on fdg objects}). The monoidal product is again $\hat{\otimes}$.
\end{defi}

\section{Operads in the complete and filtered setting}
\label{section: operads}

Let $(\Mm,\, \amalg,\, \otimes,\, \mathbbm{1})$ be an additive closed monoidal category with small colimits and limits.
We assume moreover that the monoidal structure preserves colimits in each variable. In this situation, $\Mm$ is enriched over the category of sets (see \cite[1.1.7]{Fresse3}) and for a set $K$ and an object $M \in \Mm$, we have a tensor product $K \otimes M$ in $\Mm$ given by
\[ K \otimes M \coloneqq \amalg_{k \in K} M.\]
As said earlier, examples of such categories $\Mm$ are given by the categories $\Aa$, $\Filt(\Aa)$, $\complete(\Aa)$, and $\compa(\Aa)$ when $\Aa$ satisfies our constraints.

\subsection{\texorpdfstring{Complete $\sS$-objects}{Complete S-objects}}

We present the monoidal category of symmetric objects (or $\sS$-modules) in $\Mm$ denoted by ($\sS$-$\mathsf{Mod}(\Mm),\, \circ$, $I$). We refer for instance to~\cite[5.1]{LodayVallette} for the case of modules and to \cite[2.2]{Fresse3} for more general situations.\\

An \emph{$\sS$-module in $\Mm$} is a collection $M = \{M(0),\, M(1),\, \ldots,\, M(n),\, \ldots \}$ of right-$\sS_n$-objects $M(n)$ in $\Mm$. An action of $\sS_n$ on an object $M(n)$ is defined as a morphism of monoids $\sS_n \to \End_{\Mm}(M(n))$. 
A morphism $f : M \rightarrow N$ in the category $\sS$-$\mathsf{Mod}(\Mm)$ is a componentwise morphism.

We define the monoidal product $\circ$ of $\sS$-objects in $\Mm$ by
\[
M \circ N (n) \coloneqq \coprod_{k \geq 0} \left( M(k) \otimes_{\sS_k} \left( \coprod_{i_1 + \cdots + i_k = n} \mathrm{Ind}_{\sS_{i_1} \times \cdots \times \sS_{i_k}}^{\sS_n} \left( N(i_1) \otimes \cdots \otimes N(i_k) \right) \right)\right),
\]
where $\mathrm{Ind}_{H}^G M \coloneqq G \otimes_{H} M$ is the \emph{induced representation}. 
We emphasize that, when for example $\Mm = \Comp(\Aa)$, the sum $\amalg$ stands for the completion of the sum in $\Aa$ with respect to the filtration.

To simplify the notation, we denote by
\[ \capgSA \text{ the category } \sS\text{-}\mathsf{Mod}(\compa(\Aa)) \]
and by
\[ \cdgSA \text{ the full subcategory of } \capgSA \]
given by complete pg-$\sS$-modules $M = \{M(n) \}$ such that the $M(n)$ are dg-objects.\\
We remark that the category of $\sS$-modules defined in \cite{LodayVallette} can be seen as a full subcategory of the category $\cdgSA$ by defining on the $\sS$-module $M$ the trivial filtration $F_0 M = M$ and $F_p M = \{ 0\}$ for all $p > 0$.
For example, $I \coloneqq \{0,\, \ringK,\, 0,\, \ldots\}$ endowed with the trivial filtration and the trivial differential is an object of $\cdgSA$.

This gives a monoidal product on the categories $\capgSA$ and $\cdgSA$. We are interested in the category $(\capgSA,\, \circ,\, I)$ in Section \ref{section: curved operads} and in the category $(\cdgSA,\, \circ,\, I)$ in Section \ref{section:cooperads}.

\begin{remark}
Using the convention that an empty tensor product is equal to $\ringK$, we get that one of the components of $(M \circ N)(0)$ is $M(0)$.
\end{remark}

\begin{defi}
An \emph{operad} $(\Oo,\, \gamma,\, \eta)$ in the category $\Mm$ is a monoid in the monoidal category ($\sS \text{-} \textsf{Mod}(\Mm)$, $\circ$, $I$).
The map $\gamma : \Oo \circ \Oo \rightarrow \Oo$ is the \emph{composition product} and the map $\eta : I \rightarrow \Oo$ is the \emph{unit}. We denote by $\mathsf{Op}(\Mm)$ the category of operads in $\Mm$.
\end{defi}

Thus we have, for instance:
\begin{itemize}
\item a \emph{graded operad} is an operad in the closed symmetric monoidal category $(\graded\Aa,\otimes)$,
\item a \emph{filtered} operad is an operad in the closed symmetric monoidal category $(\Filt(\Aa),\bar\otimes)$, and
\item a \emph{complete} operad is an operad in the closed symmetric monoidal category $(\complete(\Aa),\hat\otimes)$.
\item a \emph{complete gr-dg} operad is an operad in the closed symmetric monoidal category $(\compa(\Aa),\hat\otimes)$. It is in particular endowed with a gr-differential $d : \Oo \to \Oo$ which is a derivation for the composition on $\Oo$.
\end{itemize}

\begin{remark}
The phrase ``complete filtered operad'' is a priori ambiguous. 
We will always use this phrase to mean an operad in a symmetric monoidal category of complete filtered objects in some ground category.
We will never use it to mean an operad equipped with a complete filtration each stage of which is itself an operad.
\end{remark}

Let $\Oo$ be an operad in $\Mm$.
The total object obtained by taking the coproduct (in $\Mm$)
\[\coprod_{n \geq 0} \Oo(n)\] 
supports a pre-Lie bracket $\{-,\, -\}$, and then a Lie bracket $[-,\, -]$, which is given on $\Oo(p) \otimes \Oo(q)$ by
\[
\{-,\, -\} \coloneqq \sum_{i=1}^p \sum_{\textrm{P}} (- \circ_i -)^{\sigma_{\textrm{P}}},
\]
The notation $- \circ_i -$ stands for the partial composition product on $\Oo$ where we plug the element in $\Oo(q)$ in the $i$th entry of the element in $\Oo(p)$.
The sum runs over some ordered partitions $\mathrm{P}$ by means of which we can define the permutation $\sigma_{\mathrm{P}}$. (We refer to \cite{LodayVallette}, Section 5.4.3 for more details.) 
If we fix a map $\mu : I \to \Oo(1)$ (or an element $\mu \in \Oo(1)$ when it is meaningful), we therefore obtain, by a slight abuse of notation, an endomorphism $[\mu,\, -]$ of $\Oo$ given by the formula
\[
[\mu,\, -] \coloneqq (\mu \circ_1 -) - (-1)^{|\mu||-|} \sum_{j=1}^q (- \circ_j \mu).
\]
(Or
\[
[\mu,\, \nu] \coloneqq \mu \circ_1 \nu - (-1)^{|\mu||\nu|} \sum_{j=1}^q \nu \circ_j \mu
\]
when this is meaningful.)
Moreover, the associativity of the map $\gamma$ shows that $[\mu,\, -]$ is a derivation.

\subsection{Curved operads}
\label{section: curved operads}

We give the definition of a curved operad in the category $\compa(\Aa)$ and we present the curved endomorphism operad in this context. 
We emphasize that our operads are allowed to have 0-ary elements, that is $\Oo(0) \neq \varnothing$ a priori. 
We use the notation of the book~\cite{LodayVallette}.

\begin{defi}
A \emph{curved operad} $(\Oo,\, \gamma,\, \eta,\, d,\, \theta)$ is a complete gr-dg operad $(\Oo,\, \gamma,\, \eta,\, d)$ equipped with a map $\theta : I \to (F_1\Oo(1))_{-2}$ such that
\[
\left\{ \begin{array}{lcll}
d^2 & = & [\theta,\, -] & (\textrm{or } [-,\, \theta] + d^2 = 0),\\
d(\theta) & = & 0 & (\theta \textrm{ is \emph{closed}}).
\end{array} \right.
\]
The map (or element) $\theta$ is called the \emph{curvature}.		
\end{defi}

\begin{remark}
When dealing with curvature, we often think about the map $\theta$ as an element $\theta \in F_1\Oo(1)$ of degree $-2$. If $\Aa$ isn't a concrete category, it is possible to replace ``the element $\theta$" by ``the map $\theta$" everywhere in this article.
\end{remark}

\begin{defi}
We define two categories with objects the curved operads.
\begin{itemize}
\item
A {\em morphism} $f : (\Oo,\, d,\, \theta) \rightarrow (\Pp,\, d',\, \theta')$ of curved operads is a morphism of operads such that
\[
\left\{ \begin{array}{lcl}
f \cdot d & = & (-1)^{|f|} d' \cdot f,\\
f(\theta) & = & \theta'.
\end{array} \right.
\]
We denote by $\cOp(\Aa)$ the category of curved operads in $\compa(\Aa)$.
\item
A {\em lax morphism} $(\Oo,\, d,\, \theta) \rightarrow (\Pp,\, d',\, \theta')$ of curved operads is a pair $(f, a)$ where $f : \Oo \to \Pp$ is a morphism of operads and $a : I \to F_1\Pp$ is an $\sS$-module map of degree $-1$ such that
\begin{align}
d' \cdot f + [a, f] & = f \cdot d \quad \textrm{ and} \label{eq: lax morphism diff}\\
f(\theta) & = \theta' + d' \cdot a + \frac{1}{2}[a, a]. \label{eq: lax morphism curvature}
\end{align}

The composition of lax morphisms is given by $(g, b) \cdot (f, a) := (g\cdot f, g(a)+b)$ and the identity is $(\id, 0)$. 
We denote by $\cOp^{\textrm{lax}}(\Aa)$ the category of curved operads in $\compa(\Aa)$ and by $\Hom^{\textrm{lax}}$ its morphisms.
\end{itemize}

When the category of $\sS$-modules has a notion of weak equivalence, a \emph{weak equivalence between curved operads} is a (lax) morphism of curved operads whose underlying $\sS$-module map is a weak equivalence.
\end{defi}

\begin{remark}
\begin{itemize}
\item
There is an inclusion $\cOp(\Aa) \to \cOp^{\textrm{lax}}(\Aa)$ given by the identity on objects and by $f \mapsto (f, 0)$ on morphisms. It is not full.
\item
The notion of lax morphism is used only in Section \ref{section: bar cobar constructions} in order to define bar and cobar constructions and an adjunction between these functors.
\end{itemize}
\end{remark}

An important example of a curved operad is given by the endomorphism operad of a complete gr-dg object.

\begin{defi}
\label{defi: endomorphism operad}
The \emph{endomorphism operad} of a complete gr-dg object $(X,\, d)$ is the curved operad
\[
\End_X \coloneqq (\{ \Hom (X^{\otimes n},\, X) \}_{n \geq 0},\, \gamma,\, \del,\, \theta),
\]
where the composition map $\gamma$ is given by the composition of functions and
\[
\left\{ \begin{array}{lcll}
\del(f) & \coloneqq & [d,\, f], & \textrm{for } f\in \Hom (X^{\otimes n},\, X),\\
\theta & \coloneqq & d^2.
\end{array} \right.
\]
It is straightforward to check that this really defines a curved operad.
\end{defi}

\begin{defi}
A {\em representation} of a curved operad $\Oo$ on the complete gr-dg object $X$ is a morphism of curved operads $\Oo \to \End_X$. 
We also say that this defines an \emph{$\Oo$-algebra structure} on $X$.
\end{defi}

\begin{remark}
An example of such an algebra is given by curved associative algebras (see Section \ref{section: assocase}) and curved Lie algebras in a complete setting.
\end{remark}

An \emph{augmentation} of a curved operad $(\Oo,\, \gamma,\, \eta,\, d,\, \theta)$ is a map $\Oo \to I$ of curved operads such that $\varepsilon \cdot \eta = \id_I$. 

An augmentation realizes the underlying $\sS$-object of $\Oo$ as the biproduct (i.e., both the product and coproduct) of $I$ and $\ker\varepsilon$, which we denote $\overline{\Oo}$ as usual.

\subsection{Free complete operad}
\label{section: free operad}

The construction of the free operad given in \cite{BauesJibladzeTonks, Rezk} and in \cite{LodayVallette}, Section 5.5.1 and 5.8.6 applies in the filtered and in the complete setting by replacing the biproduct and the composition product $\circ$ by their filtered or complete analog. 

We recall the definition of the tree operad. For $M$ an $\sS$-module in $\Mm$, we define 
\begin{align*}
\Tt_0 M &= I
\\
\intertext{and recursively define}
\Tt_n M&=I\amalg (M\circ \Tt_{n-1} M).
\end{align*}
There is an inclusion $\Tt_0 M\to \Tt_1 M$.
Then given a map $\iota_{n-1}:\Tt_{n-1} M\to \Tt_n M$, there is a map $\iota_n:\Tt_n M\to \Tt_{n+1} M$ which takes the $I$ factor to the $I$ factor and takes $M\circ \Tt_{n-1} M\to M\circ \Tt_n M$ using $\id_M$ on the $M$ factor and $\iota_{n-1}$ on the other factor.
We write $\Tt M$ the colimit (in filtered $\sS$-modules) over $n$ of $\Tt_n M$ and $\widehat{\Tt} M$ the colimit (in complete $\sS$-modules) over $n$ of $\Tt_n M$.

There is an injection from $I = \Tt_0 M \to \Tt_n M$ which passes to a map $\eta: I \to \Tt M$ (resp. $\eta: I \to \widehat{\Tt} M$). Similarly, the map $M \to \Tt_1 M$ induces a map $j: M \to \Tt M$ (resp. $j: M \to \widehat{\Tt} M$).

The two constructions $\Tt M$ and $\widehat{\Tt} M$ are the free operads in the filtered and in the complete setting with the same arguments as in Theorem 5.5.1 in \cite{LodayVallette}. 

\begin{thm}
\label{thm: free operad}
There is an operad structure $\gamma$ (resp. $\hat \gamma$) on $\Tt M$ (resp. $\widehat{\Tt} M$) such that $\Tt (M) \coloneqq (\Tt M,\, \gamma,\, j)$ (resp. $\widehat{\Tt} (M) \coloneqq (\widehat{\Tt} M,\, \hat\gamma,\, j)$) is the free operad on $M$ in the category of filtered operads (resp. of complete operads).
\end{thm}

\begin{proof}
See \cite[Theorem 5.5.1]{LodayVallette}.
\end{proof}

Next we verify that this standard construction commutes with completion, so that the complete free operad is the aritywise completion with respect to the filtration of the usual free operad.

A lax symmetric monoidal functor induces a functor on operads, acting aritywise (see, e.g.,~\cite[Proposition 3.1.1]{bF17} or~\cite[Theorem 12.11(1) applied to Corollary 11.16]{YauJohnson:FPAM}).
Therefore, given an operad $\Oo$ in filtered objects, the aritywise completion $\widehat{O}$ is (functorially) a complete operad. 
Moreover, the completion functor and the inclusion of complete operads into filtered operads are still adjoint~\cite[Corollary 12.13]{YauJohnson:FPAM}
(see also~\cite[Proposition 3.1.5]{bF17}, where the hypotheses require $G$ to be strong symmetric monoidal but the argument only uses the lax structure).
We record the conclusion as follows.

\begin{prop}
Aritywise completion and inclusion form an adjunction between complete filtered operads and filtered operads.
\end{prop}

This construction is then compatible with the free operad in the following sense.

\begin{prop}
\label{prop:freeOperad}
Let $M$ be a filtered graded $\sS$-object.
The completion of the free filtered (graded) operad on $M$ and the map of complete $\sS$-objects $\widehat{M}\to \widehat{\Tt M}$ induced by completion exhibits $\widehat{\Tt M}$ as the free complete operad $\widehat{\Tt} \widehat{M}$ on the complete $\sS$-object $\widehat{M}$.
\end{prop}

\begin{proof}
By adjunction, maps of complete operads from $\widehat{\Tt M}$ to a complete graded operad $\mathcal{X}$ are in bijection first with operad maps from $\Tt M$ to $\mathcal{X}$, then to $\sS$-object maps from $M$ to $\mathcal{X}$, and finally to complete $\sS$-object maps from $\widehat{M}$ to $\mathcal{X}$. 

A priori there are two recipes for a map from $\widehat{M}$ to $\widehat{\Tt M}$.
To see that they agree, it suffices to compare them on $M$, where both are the composition $M\to \Tt M\to \widehat{\Tt M}$.
\end{proof}

\begin{remark}
\label{remark: on the free operad}
\begin{enumerate}
\item
When $(M,\, d_M)$ is a predifferential graded $\sS$-object, the free (complete) operad on $M$ is naturally endowed with a predifferential $\tilde d_M$ induced by $d_M$ defined as follows: it is the unique derivation which extends the map
\[ M \xrightarrow{d_M} M \hookrightarrow \free M. \]
We denote by $\Tt(M,\, d_M)$ the free pg operad $(\free M,\, \tilde d_M)$. Moreover, when $(M,\, d_M)$ is a gr-dg $\sS$-object, so is $\Tt(M,\, d_M)$ and it is the free gr-dg operad on $(M,\, d_M)$. This is true in the filtered and in the complete settings.
\item
To lighten the notation, when the setting is explicit, we use the notation $\Tt$ for the free operad both in the filtered and in the complete setting.
\item
By the previous proposition, it is possible to think of an element in the free complete operad on the complete graded $\sS$-object $M$ as a (possibly infinite) sum of trees, whose vertices of arity $k$ are indexed by element in $M(k)$.
\end{enumerate}
\end{remark}

\subsection{Free curved operad on a gr-dg $\sS$-module}
\label{section: free curved operad}

We now provide a curved version of the free-forgetful adjunction.

\begin{defi}
We say that a tuple $(\Oo,\, d_\Oo,\, \theta)$ is a \emph{pointed complete gr-dg operad} when $(\Oo,\, d_\Oo)$ is a complete gr-dg operad and $\theta$ is a \emph{closed} element in $F_1 \Oo_{-2}$ (that is $d_\Oo(\theta) = 0$).
\end{defi}

In the next proposition, $\free$ denotes either the free filtered gr-dg operad or the free complete gr-dg operad.

\begin{prop}
\label{prop: free pointed operad}
We define the functor
\[ \begin{array}{lcccc}
\free_+ & : & \capgSA & \to & \mathsf{Pointed\ filtered\ (compl.)\ gr}\textsf{-}\mathsf{dg\ operads},\\
&& (M,\, d_M) & \mapsto & \left(\free(M \amalg \vartheta I),\, \tilde d_M,\, \vartheta\right),
\end{array} \]
where $\vartheta$ is a formal parameter in homological degree $-2$ and in filtration degree $1$ and $\tilde d_M$ is the unique derivation which extends the map
\[ M \amalg \vartheta I \to M \xrightarrow{d_M} M \to \free(M \amalg \vartheta I). \]
There is an adjunction
\[
\xymatrix{\free_+ : \capgSA \ar@<.5ex>@^{->}[r] & \mathsf{Pointed\ filt.\ (compl.)\ gr}\textsf{-}\mathsf{dg\ operads} : U. \ar@<.5ex>@^{->}[l]}
\]
Moreover, the pointed filtered, resp. complete, gr-dg operad $(\free_+ (M,\, d_M),\, \gamma,\, j_+)$, where $\gamma$ is the free composition product and $j_+$ is the map $M \xrightarrow{j} \free M \to \free (\vartheta I \amalg M)$ is the free filtered (complete) gr-dg operad.
\end{prop}

\begin{proof}
We denote by $\mathsf{pOp}$ the category of pointed operads under consideration and we consider a pointed operad $(\Oo,\, d_\Oo,\, \theta)$ in it. By Theorem \ref{thm: free operad}, we have
\begin{align*}
\Hom_{\mathsf{Op}}(\free (M \amalg \vartheta I),\, \Oo) & \cong \Hom_{\capgSA}(M \amalg \vartheta I,\, U(\Oo))\\
& \cong \Hom_{\capgSA}(M,\, U(\Oo)) \times \Hom_{\capgSA}(\vartheta I,\, U(\Oo)),
\end{align*}
where $U(\Oo)$ is the underlying filtered (complete) graded $\sS$-modules associated with $\Oo$ (that is more explicitly $U((\Oo,\, \gamma,\, \eta)) = \Oo$).
The fiber over the map $\vartheta \mapsto \theta$ on the right-hand side is naturally isomorphic to $\Hom_{\capgSA}(M,\, U(\Oo))$. On the other hand, the fiber over the map $\vartheta \mapsto \theta$ on the left-hand side consists of those operad maps $\free_+ M \to \Oo$ sending $\vartheta$ to $\theta$. The condition on the predifferentials follows directly from the construction of $\tilde d_M$. We therefore obtain the adjunction
\[ \Hom_{\mathsf{pOp}}(\free_+ (M,\, d_M),\, (\Oo,\, d_\Oo,\, \theta)) \cong \Hom_{\capgSA}((M,\, d_M),\, U(\Oo,\, d_\Oo,\, \theta)), \]
where $U(\Oo,\, d_\Oo,\, \theta)$ is the underlying filtered (complete) gr-dg $\sS$-modules associated with $(\Oo,\, d_\Oo,\, \theta)$ (that is more explicitly $U((\Oo,\, \gamma,\, \eta,\, d_\Oo,\, \theta)) = (\Oo,\, d_\Oo)$). 
It is clear from the proof (and Theorem \ref{thm: free operad}) that the tuple $(\free_+ (M,\, d_M),\, \gamma,\, j_+)$ is the free filtered (complete) gr-dg operad.
\end{proof}

We construct the free curved operad by means of the functor $\free_+$.

\begin{thm}
\label{thm: free curved operad}
We define the functor
\[ \begin{array}{lcccc}
\cfree & : & \capgSA & \to & \mathsf{Curved\ operads}\\
&& (M,\, d_M) & \mapsto & \left(\free(M \amalg \vartheta I)/\left(\im({d_M}^2 - [\vartheta,\, -])\right),\, \bar d_M,\, \bar \vartheta\right),
\end{array} \]
where $\vartheta$ is a formal parameter in homological degree $-2$ and in filtration degree $1$, the map $\bar d_M$ is the derivation induced by $\tilde d_M$ and $\left(\im({d_M}^2 - [\vartheta,\, -])\right)$ is the ideal generated by the image of the map ${d_M}^2 - [\vartheta,\, -] : M \to \free(M \amalg \vartheta I)$. 

There is an adjunction
\[
\xymatrix{\cfree : \capgSA \ar@<.5ex>@^{->}[r] & \mathsf{Curved\ operads} : U. \ar@<.5ex>@^{->}[l]}
\]

Moreover, let $\bar \jmath : (M,\, d_M) \to \cfree(M,\, d_M)$ denote the composition $(M,\, d_M) \xrightarrow{j} \free(M,\, d_M) \twoheadrightarrow \cfree(M,\, d_M)$. The curved operad $(\cfree(M,\, d_M),\, \bar \gamma,\, \bar \jmath)$ is the free curved operad on $(M,\, d_M)$ in the category of curved operads.
\end{thm}

\begin{proof}
The map $\tilde d_M$ induces a well-defined map $\bar d_M$ on the quotient of the free operad $\free(M \amalg \vartheta I)/\left(\im({d_M}^2 - [\vartheta,\, -])\right)$ since $\tilde d_M (\vartheta) = 0$. From the fact that
\[ \im\left({\tilde{d_M}}^2 - [\vartheta,\, -]\right) \subset \left( \im({\tilde{d_M}}^2 - [\bar \vartheta,\, -]) \right) = \left(\im\left({d_M}^2 - [\vartheta,\, -]\right)\right)\]
we obtain that $\cfree (M,\, d_M)$ is a curved operad. 

We denote by $U_p$ the forgetful functor from curved operads to pointed gr-dg operads. Morphisms of curved operads are morphisms of pointed gr-dg operads between curved operads. This means that a morphism of curved operads $\bar f : \cfree(M,\, d_M) \to (\Oo,\, d_\Oo,\, \theta)$ is the same thing as a morphism of pointed gr-dg operads $\bar f : U_p\cfree(M,\, d_M) \to U_p(\Oo,\, d_\Oo,\, \theta)$. Morphisms $U_p\cfree(M,\, d_M) \to U_p(\Oo,\, d_\Oo,\, \theta)$ in pointed gr-dg operads coincide with morphisms $\free_+ (M,\, d_M) \to U_p(\Oo,\, d_\Oo,\, \theta)$ in pointed gr-dg operads with the condition that the ideal $\left(\im\left({d_M}^2 - [\vartheta,\, -]\right)\right)$ is sent to $0$. Since $\left( \im({\tilde{d_M}}^2 - [\bar \vartheta,\, -]) \right) = \left(\im\left({d_M}^2 - [\vartheta,\, -]\right)\right)$ and morphisms of pointed gr-dg operads commute with the predifferentials and send the marked point to the marked point, we get that this last condition is automatically satisfied when $U_p(\Oo,\, d_\Oo,\, \theta)$ is the underlying pointed gr-dg operad of a curved operad. 
Hence, we have
\[ \Hom_{\cOp}(\cfree(M,\, d_M),\, (\Oo,\, d_\Oo,\, \theta)) \cong \Hom_{\mathsf{pOp}}(\free_+ (M,\, d_M),\, U_p(\Oo,\, d_\Oo,\, \theta)). \]
The Theorem follows from Proposition \ref{prop: free pointed operad}.
\end{proof}

\begin{remark}
This adjunction is useful to define a model category structure on complete curved operads so that the bar-cobar resolution and the Koszul resolution provide $\sS$-cofibrant resolutions (see Appendix \ref{appendix: MCS on gr-dg objects} for the details on the model structure).
\end{remark}

We now give the notion of quasi-free complete curved operad. The cobar construction (see Section \ref{section: bar cobar constructions}) is a quasi-free operad and cofibrant complete curved operad (see Appendix \ref{appendix: MCS on gr-dg objects}) are retracts of quasi-free complete curved operads.

\begin{defi}
\label{defi: quasi-free}
We call a complete curved operad $\Oo$ \emph{quasi-free} if there exists a complete $\sS$-module $M$ and a predifferential $d : \free_+(M) \to \free_+(M)$ such that
\[ \Oo \cong (\free_+(M),\, d)/\left(\im\left(d^2 - [\vartheta,\, -]\right)\right). \]
\end{defi}

This notion is different from that of a free complete curved operad because in the quasi-free case the predifferential $d$ is not a priori induced by a map $M \to M$.

\section{Complete cooperads}
\label{section:cooperads}

We present in this section the notion of a complete cooperad and associated notions, most importantly \emph{altipotence}, a variation of conilpotence where instead of being eventually zero, iterated powers of the decomposition map eventually arrive in arbitrarily high filtration degree. 
This material is fairly technical, but the raison d'\^{e}tre and upshot are straightforward: the goal throughout is to establish a context in which the usual formulas for the cofree conilpotent cooperad in terms of decorated trees extend to the complete context.
The tension is to find a full subcategory of cooperads that is big enough to contain the (complete) tree cooperad but small enough so that the tree cooperad is the cofree object there. 
It is possible (in fact likely) that there are more natural alternatives to working with altipotence.
Since for us the altipotent cooperads are not of intrinsic interest but merely a tool to study complete operads, in that capacity they are perfectly adequate.

\begin{defi}\label{cooperad}
We define a \emph{cooperad} $(\Cc,\, \Delta,\, \varepsilon)$ in the category $\Mm$ to be a comonoid in the monoidal category ($\sS$-$\mathsf{Mod}(\Mm)$,\, $\circ$,\, $I$).
The map $\Delta : \Cc \to \Cc \circ \Cc$ is the \emph{decomposition map} and the map $\varepsilon : \Cc \to I$ is the \emph{counit map}.
Morphisms of cooperads are morphisms of comonoids in $\sS$-modules.
We denote by $\Coop(\Mm)$ the category of cooperads in $\Mm$.
\end{defi}
Thus:
\begin{itemize}
\item a graded cooperad (in $\Aa$) is a cooperad in the category $(\graded\Aa,\, \otimes)$,
\item a filtered cooperad is a cooperad in the category $(\Filt(\Aa),\, \bar\otimes)$, and 
\item a \emph{complete cooperad} is a cooperad in the category $(\complete(\Aa),\, \hat \otimes)$,
\item a \emph{dg complete cooperad} is a cooperad in the category $(\complete(\dg\Aa),\, \hat \otimes)$.
\end{itemize}

Because generically the associated graded functor is not colax, it does not take filtered cooperads to graded cooperads.
For that we need a flatness condition.

\begin{prop}
\label{prop: graded of gr-flat coop is coop}
There is a functor $\Gr$ from the category of gr-flat filtered (or complete) cooperads to graded cooperads covering the associated graded functor.
\end{prop}
See Appendix~\ref{appendix: gr-flat} for the proof.

\begin{remark}[Warning]
\label{remark: associated graded is only a cooperad if we have gr-flat}
Because the associated graded functor does not extend to arbitrary filtered or complete cooperads, any ``gr'' definition which uses a cooperadic structure on the associated graded is implicitly assumed to only be defined for gr-flat filtered or gr-flat complete cooperads.
\end{remark}

\subsection{The tree cooperad}
\label{section: tree cooperad}

We recall the construction of the \emph{tree cooperad} in an additive category.
This construction is taken directly from~\cite[5.8.6]{LodayVallette} with very mild adaptations (using coinvariants instead of invariants and working in greater generality).
We recall it explicitly because we will need to work with it intimately.\\

The following construction works in a cocomplete symmetric monoidal additive category equipped with the composition product $\circ$ and unit $I$. We will work in the more restricted category $\Mm$ described in the beginning of Section \ref{section: operads} for the results arriving after. 

\begin{construction}[Tree cooperad]
\label{construction: tree cooperad}
For $M$ an $\sS$-module in $\Mm$, we write $\Tt^c M$ as the colimit (in $\sS$-modules) over $n$ of $\Tt_n M$, where $\Tt_n M$ is defined in Section \ref{section: free operad}.

There is a projection using the zero map from $\Tt_n M\to I$ which passes to a map $\varepsilon:\Tt^c M\to I$.
There is an inclusion $\eta:I=\Tt_0 M\to \Tt^c M$.

Finally we define a decomposition map.
We proceed inductively.
First define $\Delta_0: \Tt_0 M\to \Tt_0 M\circ \Tt_0 M$ as the canonical map $I\to I\circ I$.
We use the same definition for $I\subset \Tt_n M$. 
For the $M\circ \Tt_{n-1} M$ summand of $\Tt_n M$, we define a map to the biproduct
\[
M\circ \Tt_{n-1}M\to (I\circ (M\circ \Tt_{n-1}M)) \amalg ((M\circ \Tt_{n-1}M) \circ \Tt_n M)
\]
by $I\circ \id$ on the first factor and 
\begin{multline*}
M\circ \Tt_{n-1}M\xrightarrow{\id_M\circ \Delta_{n-1}} M\circ (\Tt_{n-1} M\circ \Tt_{n-1}M)\\
\cong (M\circ \Tt_{n-1}M)\circ \Tt_{n-1}M \xrightarrow{\id\circ \iota_{n-1}}
(M\circ \Tt_{n-1}M)\circ \Tt_n M
\end{multline*}
on the second factor.

This passes to the colimit to give a map $\Delta:\Tt^c M\to \Tt^c M\circ \Tt^c M$.
Verifying left unitality is direct, right unitality is by induction, and associativity is by a combination of induction and combinatorics involving the $I$ factor.

This defines a cooperad, called the \emph{tree cooperad}.
\end{construction}
The tree cooperad is often constructed (e.g., in~\cite{LodayVallette}) as an explicit model for the cofree conilpotent cooperad.
We care about this explicit model more or less because it lets us perform calculations. 
However, in the complete case which interest us, the tree cooperad is not conilpotent in general. 
Therefore our goal is to weaken the conilpotence condition and find a more relaxed setting in which the tree cooperad will still be cofree.

\begin{lemma}
\label{lemma: natural retract}
There is a natural retract of collections $r_n$ from $\Tt_n M$ to $\Tt_{n-1}M$.
\end{lemma}
\begin{proof}
We can define $r_n$ inductively.
For $n=1$ the retract $r_1: I\amalg M\circ \Tt_0 M\to I$ is projection to the first factor.
Suppose given $r_j$ for $1\le j<n$.
Define the $n$ component at $M$, 
\[
r_n: I\amalg (M\circ \Tt_n M)\to I\amalg (M\circ \Tt_{n-1} M),
\] 
as the identity on the first factor and $\id\circ r_{n-1}$ on the second factor.
By induction $r_n$ is a retract of the desired inclusion.
\end{proof}

The composite of $r_2 \cdots r_n : \Tt_n M \to \Tt_1 M$ with the projection $\Tt_1 M \to M \circ I \cong M$ provides a map $\epsilon_n(M) : \Tt_n M \to M$, which extends to a map $\epsilon(M) : \Tt M \to M$ (functorial in $M$).

We will use the notation $Q_n M$ for the ``graded'' component $\Tt_n M/\Tt_{n-1} M$.
\begin{cor}
\label{cor: cofree as a coproduct}
The underlying collection of the tree cooperad $\Tt^c M$ is the coproduct of $Q_n M$
over all $n$.
Similarly, the finite stage $\Tt_n M$ is the coproduct of these quotients through stage $n$.
\end{cor}

\begin{proof}
The retract of Lemma~\ref{lemma: natural retract} realizes the object $\Tt_n M$ as the coproduct of $Q_n M$ and $\Tt_{n-1} M$, since $\sS$-$\mathsf{Mod}(\Mm)$ is additive with kernels. 
Then by induction $\Tt^c_n M$ is isomorphic to the coproduct
\[
\coprod_{j\le n}Q_j M,
\]
and the colimit of these is the infinite coproduct.
\end{proof}

\begin{remark}
Intuitively we can think of the $n$th graded stage $Q_n M$ as consisting of rooted trees decorated by elements of $M$ with maximal number of vertices encountered in a simple path from the root to a leaf precisely equal to $n$.
\end{remark}

\begin{cor}
\label{cor: projections determine}
Let $M$ be a complete collection in $\Comp(\Aa)$. 
A map of complete collections with codomain $\Tt^c M$ is determined uniquely by its projections to $Q_n M$ (but need not exist).
\end{cor}
\begin{proof}
This follows from Lemma~\ref{lemma: comparison in complete} which says the map from the coproduct to the product is monic.
\end{proof}

\begin{construction}
\label{construction, Phi extensions}
Let $\mathcal{C}$ be a complete cooperad and $M$ a complete collection equipped with a map of collections $\varphi:\mathcal{C}\to M$.

We inductively construct a map $\Phi_n:\Cc\to Q_n M$ as follows.
We define $\Phi_0$ as $\Cc\to I\cong \Tt^c_0 M$ as the counit $\varepsilon$.
We define $\Phi_1$ as $\varphi$.
Next assuming $\Phi_0,\ldots,\Phi_{n-1}$ are defined, we define a map to the biproduct
\[\sum_{i=0}^n\Phi_i : \Cc \to \Tt^c_n M = I \amalg (M \circ \Tt^c_{n-1} M)\] as $\Phi_0$ on the first summand $I$ and 
\[\left(\varphi \circ \sum_{i=1}^{n}\Phi_{i-1}\right)\cdot \Delta\]
on the second summand.
Inductively this agrees with $\sum_{i=0}^{n-1}\Phi_i$ except on the factor $Q_n M$ of the biproduct $\Tt^c_n M$, and so defines $\Phi_n$.
\end{construction}
This gives a map into the product of quotients $\prod Q_n M$.

\begin{lemma}
\label{lemma: projections are determined by varphi with the tree cooperad}
Let $\mathcal{C}$ be a complete cooperad and $M$ a complete collection. 
Let $\Phi:\mathcal{C}\to \Tt^c M$ be a map of collections and let $\varphi:\mathcal{C}\to M$ be the projection
\[
\mathcal{C}\xrightarrow{\Phi} \Tt^c M\to \prod Q_n M\to  M.
\]
Then $\Phi$ is a map of cooperads if and only if the projection
\[
\mathcal{C}\xrightarrow{\Phi} \Tt^c M\to \prod_n Q_n M\to  Q_N M
\]
coincides with the map $\Phi_N$ obtained from $\varphi$ via Construction~\ref{construction, Phi extensions} for all $N$.
\end{lemma}
\begin{proof}
The map $\Phi$ is compatible with counits if and only if $\Cc\xrightarrow{\Phi}\Tt^c M\to Q_0 M\cong I$ is the counit $\epsilon$ of $\Cc$, which coincides with $\Phi_0$.
Compatibility with $\Phi_1$ is already a hypothesis of the lemma.

Let us first show necessity.
Suppose shown that the projection of $\Phi$ to $\Tt_{n-1} M$ must be $\sum_{j=0}^{n-1}\Phi_{j-1}$ in order that $\Phi$ be a map of cooperads.
Let $N$ be arbitrary (a priori unrelated to $n$) and consider the composition of the ``comultiplication'' $\Delta_N:\Tt_N M\to \Tt_N M\circ \Tt_N M$ used in Construction~\ref{construction: tree cooperad} and the map $\Tt_N M\circ \Tt_N M\xrightarrow{\epsilon_N(M)\circ\id}M\circ \Tt_N M$. 
More or less by unitality, this composition is the inclusion of the non-$I$ summand of $\Tt_N M$ into the non-$I$ summand of $\Tt_{N+1} M$.
Moreover, the restriction of this composition along $Q_N M\to \Tt_N M$ lifts as follows 
(again essentially by definition of $\Delta_N$):
\[
\begin{tikzcd}
&&&
(M\circ \Tt_{N-1} M)/ (M\circ \Tt_{N-2} M)
\ar[d, hookrightarrow]
\\
Q_N M
\ar[urrr,dashed, hookrightarrow]\rar
&
\Tt_N M\rar[swap]{\Delta_N}
&
\Tt_N M\circ \Tt_N M\rar[swap]{\epsilon_N(M)\circ\id}
&
M\circ \Tt_N M.
\end{tikzcd}
\]
In particular for $N>n$ the projection of the horizontal composition to $M\circ \Tt_{n-1}M$ vanishes.
Then the following diagram commutes:
\[
\begin{tikzcd}
\Cc\ar{d}[swap]{\Delta}\ar{r}{\Phi}& \Tt^c M\rar{\text{projection}}\ar{d}[swap]{\Delta}&\Tt_n M\ar{d}{}\\
\Cc\circ\Cc\ar{r}[swap]{\Phi\circ \Phi}&\Tt^c M\circ \Tt^c M\ar[r]& M\circ \Tt_{n-1}M.
\end{tikzcd}
\]
The composition along the bottom of the diagram is $\varphi\circ \sum_{j=0}^{n-1}\Phi_{j}$ by induction, which suffices to show that the projection of $\Phi$ to $\Tt_n M$ is $\sum_{j=0}^n \Phi_j$, extending the induction.

Now we turn to sufficiency.
Suppose that $\Phi$ has projections as in Construction~\ref{construction, Phi extensions}.
We would like to argue that $\Phi$ is a cooperad map.
We reduce the question to an inductive question about the finite stages $\Tt_r M\circ \Tt_s M$ by the following categorical argument.
First, $\circ$ preserves filtered colimits in each variable, since $\otimes$ does, so maps to $\Tt^c M\circ \Tt^c M$ are the same as maps to $\colim_{r,s}\Tt_r M\circ \Tt_s M$. 
All of the connecting maps in the diagram for this colimit split, which means that $\colim_{r,s}\Tt_r M\circ \Tt_s M$ splits as a coproduct.
Writing the coproduct explicitly is somewhat inconvenient since $\circ$ doesn't naively distribute over coproducts on the right, but in any event every term in this coproduct appears at the stage indexed by some finite pair $(r,s)$.
Then by Lemma~\ref{lemma: comparison in complete}, there is a monomorphism from $\colim_{r,s}\Tt_r M\circ \Tt_s M$ to $\prod_{r,s}\Tt_r M\circ \Tt_s M$.
Moreover, the projection of $\Tt^c M\xrightarrow{\Delta}\Tt^c M\circ \Tt^c M\to \Tt_r M\circ \Tt_s M$ factors through $\Tt_{r+s}M$.
Therefore to check commutativity of the square
\[
\begin{tikzcd}
\Cc\rar{\Phi}\dar
&\Tt^c M\dar{\Delta}
\\
\Cc\circ\Cc\rar{\Phi\circ \Phi}
&
\Tt^c M\circ \Tt^c M
\end{tikzcd}
\]
it suffices to check commutativity of the square
\[
\begin{tikzcd}
\Cc\rar{\Phi}\dar
\ar[bend left,"\sum^{r+s}_j\Phi_{j}"]{rr}
&\Tt^c M\rar & \Tt_{r+s}M\dar{\text{projection of }\Delta}\\
\Cc\circ\Cc\rar{\Phi\circ \Phi}
\ar[bend right,"\left(\sum_j^{r} \Phi_{j}\right)\circ\left(\sum_j^{s}\Phi_{j}\right)",swap]{rr}
&
\Tt^c M\circ \Tt^c M\rar& \Tt_{r} M\circ \Tt_{s} M
\end{tikzcd}
\]
for all $r$ and $s$.

We proceed by induction on $r$. 
The base cases with $r=0$ follow from unitality considerations, so assume that $r$ is strictly positive.
Consider the following diagram:
\[
\begin{tikzcd}
\Cc\ar{rrrr}{\sum_j^{r+s}\Phi_{j}}
\dar[swap]{\Delta}
&&&&
\Tt_{r+s}M\dar{\text{projection}}
\\
\Cc\circ\Cc\dar[swap]{\id\circ\Delta}
\ar{rrrr}{\varphi\circ \sum_j^{r+s-1}\Phi_j}
&&&&
M\circ \Tt_{r+s-1}M
\dar{\text{projection of }\Delta}
\\
\Cc\circ\Cc\circ\Cc
\ar{rrrr}[swap]{\varphi\circ \sum_j^{r-1}\Phi_j\circ \sum_j^{s}\Phi_j}
&&&&
M\circ \Tt_{r-1}M \circ \Tt_s M.
\end{tikzcd}
\]
The upper square commutes by construction of $\Phi_j$ and the lower square commutes by the inductive premise. 
The outer rectangle constitutes the correct commutativity for all of the summands of the next step in the induction except for the $I\circ \Tt_s M$ term, which again follows by unitality.
Note that comparing this outside cell to the next inductive step uses the fact that $(\id\circ\Delta)\cdot \Delta=(\Delta\circ\id)\cdot \Delta$ in the cooperad $\Cc$.
This completes the induction. 
\end{proof}
This lemma constitutes a uniqueness result for extending $\varphi$ to a cooperad map $\Phi$. 
However, it is only a partial existence result because Construction~\ref{construction, Phi extensions} a priori lands in the product of $Q_n M$, not in the subobject $\Tt^c M$ of the product.
In the next section we will provide conditions for a refined existence result, conditions under which the projections $\Phi_n$ determine a map with codomain $\Tt^c M$.

\subsection{Gr-coaugmentation and gr-conilpotence}

We now move on to the definition of altipotence. To manage that, we generalize the notions of coradical filtration and of primitives for a coaugmented cooperad to the setting of complete filtered cooperads. 

\begin{defi}
When $(M,\, F)$ is a (complete) filtered object in $\Mm$, we define the \emph{$k$-shifted filtrations of $F$} by $(M,\, F)[k] = (M,\, F[k])$ where
\[F[k]_p M = \left\{
\begin{array}{ll}
F_0 M & \text{when } k+p \leq 0,\\
F_{k+p}M & \text{when } k+p \geq 1.
\end{array} \right.\]
In a slight abuse of notation, we often simply write $M[k]$ instead of $(M,\, F)[k]$ and $F_p M[k]$ instead of $F[k]_p M$.
\end{defi}

We denote by $I[-1]$ the (complete) filtered $\sS$-module endowed with the following filtration
\[ F_0 I[-1] = I,\ F_1 I[-1] = I \text{ and } F_p I[-1] = \varnothing \text{ for all } p \geq 2.\]

\begin{defi}\label{gr-coaug}
Let $(\Cc,\, \Delta,\, \varepsilon)$ be a complete cooperad.
\begin{itemize}
\item
A \emph{coaugmentation} for this cooperad is a map $\eta : I \to \Cc$ of cooperads (on $I$, $\Delta(1) = 1\otimes 1$) such that $\varepsilon \cdot \eta = \id_I$.
\item
When $\Cc$ is gr-flat, a \emph{gr-coaugmentation} $\eta$ is a map of filtered complete $\sS$-modules, such that applying the graded functor $\Gr$ to $(\Cc,\, \Delta,\, \varepsilon,\, \eta)$ provides a flat coaugmented cooperad.
\item
In the context of complete gr-coaugmented cooperads, we call \emph{morphism of complete gr-coaugmented cooperads} a morphism of cooperads which commutes with the gr-coaugmentations after applying the graded functor $\Gr$.
\end{itemize}
Let $(\Cc,\, \Delta,\, \varepsilon,\, \eta)$ be a gr-coaugmented gr-flat cooperad.
\begin{itemize}
\item
The \emph{infinitesimal coideal} of $\eta$, notated $\coideal{\Cc}$, is the complete filtered $\sS$-module which is the pushout (in complete filtered objects) of the diagram $I[-1]\gets I\xrightarrow{\eta} \Cc$, where $I\to I[-1]$ is the identity in filtration degree zero. We denote by $\coideal{\eta}$ the map $I[-1] \to \coideal{\Cc}$ in the pushout.
\end{itemize}
\end{defi}

\begin{remark}
\label{remark: gr-coaugmentation}
In keeping with our conventions, a gr-coaugmentation only has to be compatible with the decomposition map, the differential, should one exist, and the counit on the graded level.
However, a gr-coaugmentation still satisfies $\varepsilon\cdot \eta=\id_I$ at the filtered level. 
This follows from the fact that since $I$ has nothing in positive filtration degree, the counit $\varepsilon$ factors through its $0$th graded component. 
\end{remark}

\begin{remark}
The infinitesimal coideal $\coideal{\Cc}$ has the same underlying object as $\Cc$, but the image of $\eta$ is forced to live in filtration degree $1$.
There is a map of complete objects $\Cc\to\coideal{\Cc}$. 
Then given a map of complete collections $\coideal{\Cc}\xrightarrow{\varphi} M$, we can use the recipe of Construction~\ref{construction, Phi extensions} to extend $\Cc\to \coideal{\Cc}\to M$ to a map $\Cc\to \prod Q_n M$. 
We use the same notation in this case (i.e., we do not record whether the map $\varphi$ started with domain $\Cc$ or $\coideal{\Cc}$). 
\end{remark}

The tree cooperad is coaugmented by the inclusion of $I$ into $\Tt^c M$.

Now we make some definitions related to conilpotence in the complete context. 
We cannot follow either~\cite{LodayVallette} or~\cite{LegrignouLejay:HTLC} without modification.

\begin{defi}
\label{conil}
Let $(\Cc,\, \Delta,\, \varepsilon,\, d)$ be a complete (gr-flat) cooperad endowed with a gr-coaugmentation $\eta$. 
We define the \emph{reduced decomposition map} $\bar \Delta : \Cc \to \Cc \circ \Cc$ by
\[
\bar \Delta \coloneqq \Delta - ((\eta \cdot \varepsilon) \circ \id) \cdot \Delta - (\id \circ (\eta \cdot \varepsilon))\cdot \Delta + ((\eta \cdot \varepsilon) \circ (\eta \cdot \varepsilon)) \cdot \Delta.
\]
\end{defi}
\begin{defi}[Coradical filtration]
We define a bigraded collection $\CR^{p,n}\Cc$ of subcollections of $\Cc$ recursively as follows.
We define $\CR^{p,0}\Cc$ as $F_p\Cc$ and define $\CR^{0,n}\Cc$ as $F_0\Cc$.
Next, supposing that $\CR^{p',n'}\Cc$ is defined for $n'<n$ and $p'\le p$, we first define a subobject $\widetilde{\CR}^{p,n}(\Cc\circ \Cc)$ of $\Cc\circ \Cc$. We define it as the sub-object in $\Cc\circ \Cc$
\[
\coprod_{m \geq 0}\bigcup_{p_0+\cdots+p_m + \delta = p}F_{p_0}\Cc(m) \otimes_{\sS_m} \bigotimes_{i=1}^m F_{p_i}\CR^{p_i+\delta,n-1}\Cc.
\]
Then define $\CR^{p,n}\Cc$ as the following pullback:
\[
\begin{tikzcd}
\CR^{p,n}\Cc\ar[tail,dashed,r]\ar[dashed,d] & \Cc\ar[d,"\bar\Delta"]
\\
\widetilde{\CR}^{p,n}(\Cc\circ\Cc)\rar[tail] & \Cc\circ\Cc.
\end{tikzcd}
\]
Recursively, there are natural inclusions $\CR^{p,n}\Cc\to \CR^{p-1,n}\Cc$ and $\CR^{p,n}\Cc\to \CR^{p,n+1}\Cc$.
\end{defi}

\begin{defi}
Let $\Cc$ be a complete (gr-flat) cooperad which is equipped with a gr-coaugmentation $\eta$. 
For $k \in \zZ$, we define the \emph{$k$-primitives} of $\Cc$ to be the $\sS$-module:
\[ \Prim_k \Cc \coloneqq \bigcap_{p\geq 0} \CR^{p+k,p}\Cc, \]
where we fix $\CR^{p+k,p}\Cc = \CR^{0,p}\Cc = F_0 \Cc$ when $p+k \leq 0$.
\end{defi}

\begin{remark}
By means of the fact that cooperad morphisms are filtered and using an induction, we prove that the coradical filtration is functorial, that is: if $f : \Cc \to \Cc'$ is a morphism of complete gr-coaugmented cooperads, then it induces maps $\CR^{p,n}\Cc \to \CR^{p,n}\Cc'$ and $\Prim_k\Cc \to \Prim_k\Cc'$.
\end{remark}

\begin{defi}
\label{defi: altipotence}
Let $(\mathcal{C},\eta)$ be a complete gr-coaugmented (gr-flat) cooperad.
If 
\begin{enumerate}
\item 
\label{item: altipotence condition about conilpotence}
the map \[
\colim_n \CR^{p,n}\Cc\to \Cc
\]
is an isomorphism for all $p$,
\item
\label{item: altipotence condition about coaugmentation being strong} 
the coaugmentation $\eta(I)$ lies in $\Prim_0\Cc$, and 
\item 
\label{item: altipotence condition about primitives being right closed}
for all $k \in \zZ$ we have
\begin{align}
\tag{$\ast$}
\bar \Delta (\Prim_k \Cc) \subset \coprod_{m} \bigcup_{k_0+\cdots +k_m = k+1} F_{k_0}\Cc(m) \otimes \left( \Prim_{k_1} \Cc \otimes \cdots \otimes \Prim_{k_m} \Cc \right)
\end{align}
where we set $F_{k_0}\Cc = F_0 \Cc$ when $k_0 \leq 0$,
\end{enumerate}
then we call $\mathcal{C}$ \emph{altipotent}.
\end{defi}

\begin{remark}
\label{rem: eta and primitives}
Given a gr-coaugmentation $\eta$ for a complete gr-flat cooperad $\Cc$, we have $\eta(I) \subset F_0\Cc \backslash F_1 \Cc$. Assuming that $(\Cc, \eta)$ is altipotent, we therefore get $\eta(I) \subset \Prim_0 \Cc \backslash \Prim_1 \Cc$ since $\Prim_1 \Cc \subset F_1 \Cc$.
\end{remark}

\begin{remark}
The Latin prefix ``alti-'' denotes height. 
This condition is a priori stronger than gr-conilpotence and weaker than conilpotence.
In a gr-conilpotent cooperad, repeatedly applying the reduced decomposition map would eventually increase the filtration degree by $1$. 
In a conilpotent cooperad, repeatedly applying the reduced decomposition map would eventually reach zero.
In an altipotent cooperad, repeatedly applying the reduced decomposition map eventually increases the filtration degree past any fixed number.
\end{remark}

\subsection{Existence of extensions in the altipotent setting}
In this section we show that the tree cooperad satisfies the lifting property of the cofree altipotent cooperad. 
We continue to use the conventions $F_k \Cc = F_0\Cc$ and $\CR^{k,n}\Cc = \CR^{0,n}\Cc = F_0 \Cc$ when $k \leq 0$.

\begin{lemma}
\label{lemma: cr pieces eta}
Let $\mathcal{C}$ be a complete altipotent cooperad and $M$ a complete collection.
Let $\varphi:\mathcal{C}\to M$ be a map of collections which extends to a map $\coideal{\mathcal{C}}\to M$ that we still denote by $\varphi$. 
Then for $p \geq 0$, we have
\begin{alignat}{2}
\Phi_j(\eta(I)) \subset F_p Q_j M, & \quad \text{for all } j\geq 2p,\label{Eq1}\\
\Phi_j \left(\Prim_k \Cc \right) \subset F_{p+k} Q_j M, & \quad \text{for all } j\geq 2p+1 \text{ and all } k\in \zZ,\label{Eq2}
\end{alignat}
where the maps $\Phi_j$ are defined in Section \ref{section: tree cooperad}.
\end{lemma}

\begin{proof}
We prove the statement by induction on $p$. 
The base case $p=0$ is true because $F_0 Q_j M = Q_j M$ for Statement $(\ref{Eq1})$ and because $\Prim_k \Cc \subset F_k \Cc$ and the fact that the maps $\Phi_j$ are filtered for Statement $(\ref{Eq2})$. 

Let $j\ge 2p \ge 2$. We need to control the terms in
\[ \left(\varphi \circ \sum_{i=1}^{j}\Phi_{i-1}\right)\cdot \Delta(\eta(I)) \]
where at least one $\Phi_{j-1}$ appears. 
First, we recall that
\[ \Delta(\eta) = \bar \Delta (\eta) + (\eta \otimes \eta) \Delta_I \]
with $\bar \Delta(\eta(I))$ lying in 
\[ \coprod_{m} \coprod_{k_0+\cdots +k_m = 1} F_{k_0}\Cc(m) \otimes \left( \Prim_{k_1} \Cc \otimes \cdots \otimes \Prim_{k_m} \Cc \right) \]
since $\Cc$ is altipotent. 
By means of the fact that $\varphi$ extends to a map $\coideal{\Cc} \to M$, we have $\varphi(\eta(I)) \subset F_1\Cc$. Therefore, the induction hypothesis gives $\left(\varphi \otimes \Phi_{j-1}\right)(\eta \otimes \eta)\Delta_I(I) \subset F_{1+(p-1)} Q_j M = F_p Q_j M$ (using $2p-1 \geq 2(p-1)$). 
Then we consider the terms
\[ \left(\varphi \circ \sum_{i=1}^{j}\Phi_{i-1}\right)\cdot \bar \Delta(\eta(I)) \]
where at least one $\Phi_{j-1}$ appears. We need to control the piece
\[ \varphi(F_{k_0}\Cc(m)) \otimes \left( \Phi_{j_1}(\Prim_{k_1} \Cc) \otimes \cdots \otimes \Phi_{j_m}(\Prim_{k_m} \Cc) \right) \]
where at least one $j_a$ is $j-1$ (and $k_0+\cdots k_m=1$). 
For simplicity we assume that $j_m = j-1$ but the reasoning works for all $j_l$. 
We have
\begin{align*}
F_{k_0}\Cc(m) \otimes \left( \Prim_{k_1} \Cc \otimes \cdots \otimes \Prim_{k_{m-1}} \Cc \right) & \subset F_{k_0}\Cc(m) \otimes \left( F_{k_1} \Cc \otimes \cdots \otimes F_{k_{m-1}} \Cc \right)\\
& \subset F_{1-k_m} (\Cc(m) \otimes \Cc^{\otimes (m-1)}).
\end{align*}
Then the application $\Phi_{j-1}$ to $\Prim_{k_m} \Cc$ (we have $j-1 \geq 2p-1 = 2(p-1)+1$) lies in $F_{(p-1)+k_m} Q_{j-1}M$. The piece to control is therefore in $F_{(p-1+k_m)+(1-k_m)} Q_j M = F_p Q_j M$.

Let $j \geq 2p+1 \geq 3$. We emphasize the fact that we will use the induction hypothesis for $\eta(I)$ and $p$ that we just proved. 
We have on $(\Prim_k \Cc)(m)$
\[ \Delta = \bar \Delta + ((\eta \cdot \varepsilon) \otimes \id) \cdot \Delta + (\id \otimes (\eta \cdot \varepsilon)^{\otimes m}) \cdot \Delta. \]
By means of the fact that $\varphi$ extends to a map $\coideal \Cc \to M$, we have $\varphi(\eta(I)) \subset F_1\Cc$. Therefore, the induction hypothesis gives $\varphi(\eta(I)) \otimes \Phi_{j-1}((\Prim_k \Cc)(m)) \subset F_{1+(p-1)+k} Q_j M = F_{p+k} Q_j M$. 
By the induction hypothesis for $\eta(I)$ and $j-1 \geq 2p$ (Statement (\ref{Eq1})), we get that the term $(\id \otimes (\eta \cdot \varepsilon)^{\otimes m}) \cdot \Delta$ lies in $F_{k+p} Q_j M$ (we can assume that $m \ge 1$ otherwise this term doesn't appear in the computation of $\Phi_j((\Prim_k \Cc)(m))$ since $j \geq 3$). 
Finally, we have
\[ \bar \Delta (\Prim_k \Cc) \subset \coprod_{m} \coprod_{k_1+\cdots +k_m = k+1} F_{k_0}\Cc(m) \otimes \left( \Prim_{k_1} \Cc \otimes \cdots \otimes \Prim_{k_m} \Cc \right), \]
on which we apply the elements of the sum $\varphi \circ \left(\sum_{i=1}^{j}\Phi_{i-1}\right)$ where at least one $\Phi_{j-1}$ appears. 
We again assume for simplicity that $j_m = j-1$ and remark that $F_{k_0}\Cc(m) \otimes \left( \Prim_{k_1} \Cc \otimes \cdots \otimes \Prim_{k_{m-1}} \Cc \right) \subset F_{k+1-k_m} (\Cc(m) \otimes \Cc^{\otimes (m-1)})$. 
The application $\Phi_{j-1}$ to $\Prim_{k_m} \Cc$ lies in $F_{p-1+k_m} Q_{j-1}M$. The resulting term in $Q_j M$ is therefore in 
\[F_{(p-1+k_m)+(k+1-k_m)} Q_j M = F_{p+k}\Cc.\qedhere\]
\end{proof}

\begin{lemma}
\label{lemma: cr pieces bis}
Let $\mathcal{C}$ be a complete altipotent cooperad and let $M$ be a complete collection.
Let $\varphi:\coideal{\mathcal{C}}\to M$ be a map of collections.
Then for $p,\, n \geq 0$, we have
\begin{align}
\Phi_j \left(\CR^{p, n} \Cc \right) \subset F_p Q_j M, & \quad \text{for all } j\geq n+2p.\label{Eq3}
\end{align}
\end{lemma}

\begin{proof}
We will prove the statement by induction on $n$ and $p$. 
The base cases with $n=0$ (any $p$) and $p=0$ (any $n$) are true because $\CR^{p,0}\Cc$ is just $F_p\Cc$ and the maps $\Phi_j$ are filtered and because $F_0 Q_j M = Q_j M$.

Let $(p,\, n)$ be a pair of integers greater than or equal to $1$. Suppose that the statement $(\ref{Eq3})$ is true for all pairs $(p',\, n')$ with $n'<n$ and for all pairs $(p',\, n)$ with $p' < p$. 
Let us prove the statement for $(p,\, n)$. 

Let $j$ be an integer. The map $\Phi_j$ is defined recursively by means of the formula
\[ \sum_{i=0}^j \Phi_i = \varepsilon + \left( \varphi \circ \sum_{i=1}^j \Phi_{i-1} \right) \cdot \Delta. \]
Recall that $\Tt^c_j M = I \amalg (M\circ \Tt^c_{j-1} M)$ and $M\circ \Tt^c_{j-1} M$ has summands which are quotients of 
\[
M(m)\otimes (\Tt^c_{j-1} M)^{\otimes m}.
\]
It follows that $\Phi_j$ is given by the terms in the sum $\left(\varphi\circ \sum^{j-1} \Phi_*\right)\cdot {\Delta}$ which contain at least one application of $\Phi_{j-1}$ on some factor $\Tt^c_{j-1} M$. 

Assume $j\ge n+2p$. 
On $\CR^{p, n} \Cc(m)$, we have
\[
\Delta = \bar \Delta + ((\eta \cdot \varepsilon) \otimes \id) \cdot \Delta + (\id \otimes (\eta \cdot \varepsilon)^{\otimes m}) \cdot \Delta - ((\eta \cdot \varepsilon) \otimes (\eta \cdot \varepsilon)) \cdot \Delta.
\]
The term $\varphi (\eta (I)) \otimes \Phi_{j-1}(\CR^{p, n} \Cc(m))$ lies in $F_p Q_j M$ by the induction hypothesis for the pair $(p-1,\, n)$ since $\CR^{p, n} \Cc \subset \CR^{p-1, n} \Cc$ and by the fact that $\varphi (\eta(I)) \subset F_1 \Cc$. 
Because $j-1 \geq n+2p - 1 \geq 2p$, we have by Lemma \ref{lemma: cr pieces eta} that $\Phi_{j-1}(\eta(I))$ and therefore $\Phi_{j-1}(\eta\cdot \varepsilon (\CR^{p, n} \Cc(m)))$ lie in $F_p Q_{j-1}M$. It follows that the terms $(\id \otimes (\eta \cdot \varepsilon)^{\otimes m}) \cdot \Delta(\CR^{p, n} \Cc(m))$ and $- ((\eta \cdot \varepsilon) \otimes (\eta \cdot \varepsilon)) \cdot \Delta(\CR^{p, n} \Cc(m))$ lie in $F_pQ_j M$. 
Using the fact that $\bar \Delta (\CR^{p, n} \Cc(m)) \subset \widetilde{\CR}^{p, n} (\Cc \circ \Cc)$ and the induction hypothesis for $\Phi_{j-1}$ and couples $(p',\, n-1)$ and the fact that the $\Phi_i$s are filtered, we get that the remaining term also lies in $F_p Q_j M$. 
This finishes the proof.
\end{proof}

As a direct consequence of Lemma \ref{lemma: cr pieces bis}, we get the following corollary:

\begin{cor}
\label{cor: cr pieces}
Let $\mathcal{C}$ be a complete altipotent cooperad and let $M$ be a complete collection.
Let $\varphi:\coideal{\mathcal{C}}\to M$ be a map of collections.
Then the components of the composition of maps of collections
\[
\CR^{p,n}\Cc \to \Cc \xrightarrow{\prod \Phi_j} \prod Q_j M
\]
land in filtration degree at least $p$ in $Q_j M$ for $j\ge n+2p$.
\end{cor}

Finally, we obtain the following lemma.

\begin{lemma}
\label{lemma: extensions exist}
Let $\mathcal{C}$ be a complete altipotent cooperad and $M$ a complete collection.
Let $\varphi:\coideal{\mathcal{C}}\to M$ be a map of collections.
Then the map of collections
\[
\Cc \xrightarrow{\prod \Phi_n} \prod Q_n M
\]
factors as a map of collections through the tree cooperad $\Tt^c M$.
\end{lemma}

\begin{proof}
Corollary~\ref{cor: cr pieces} tells us that we have a (unique) dotted factorization as in the following diagram
\[
\begin{tikzcd}
& \prod_{j=0}^{n+2p-1} F_0(Q_j M)\times \prod_{j=n+2p}^\infty F_p(Q_j M) \ar[d,tail]\\
\CR^{p,n}\Cc\ar[ur, dashed]\ar[r]
&
\prod_{j=0}^\infty F_0(Q_jM).
\end{tikzcd}
\]
Taking componentwise quotients by filtration degree $p$ takes the vertical map to the left hand vertical map of the following commutative square:
\[
\begin{tikzcd}
\prod_{j=0}^{n+2p-1} F_0(Q_jM)/F_p(Q_jM)
\ar[d,tail]\rar{\cong}
&
\coprod_{j=0}^{n+2p-1} F_0(Q_jM)/F_p(Q_jM)
\ar[d]
\\
\prod_{j=0}^\infty F_0(Q_jM)/F_p(Q_jM)
&
\coprod_{j=0}^\infty F_0(Q_jM)/F_p(Q_jM).
\ar[l, tail]
\end{tikzcd}
\]
Consider the bottom right entry here. The coproduct passes through the quotient to give
\begin{align*}
\coprod_{j=0}^\infty F_0(Q_jM)/F_p(Q_jM)
&\cong 
\left(
\coprod_{j=0}^\infty F_0(Q_jM)
\right)
\Bigg/
\left(
\coprod_{j=0}^\infty F_p(Q_jM)
\right)\\
& \cong \Tt^c M/ F_p\Tt^c M.
\end{align*}
This gives a lift of the map from
\[\CR^{p,n}\Cc\to \prod Q_j M/F_p Q_j M\]
to a map
\[\CR^{p,n}\Cc \to \Tt^c M/F_p\Tt^c M\]
which maps $F_p \Cc\subset \CR^{p,n}\Cc$ to $0$.

Now (using condition~\eqref{item: altipotence condition about conilpotence} of Definition~\ref{defi: altipotence} for the isomorphism in the following line), we get a map from 
\[F_0\Cc/F_p\Cc \cong (\colim_n \CR^{p,n} \Cc)/F_p\Cc\cong \colim_n (\CR^{p,n}\Cc/F_p\Cc)\]
to $\Tt^c M/F_p\Tt^c M$.

By examining projections, we see that these maps are compatible for different choices of $p$ so that we can take the limit over $p$.
This then yields a map from the limit, i.e., the complete cooperad $\Cc$, to $\Tt^c M$ which by construction has projections $\Phi_n$.
\end{proof}

\subsection{Properties of the tree cooperad}

Now we show that the tree cooperad is cofree in the category of altipotent cooperads in the gr-flat setting. 
To do this, we prove a sequence of lemmas telling us more and more about the associated graded functor and this cooperad. 
First, we have the following.

\begin{lemma}
\label{lemma:monomorphisms and gr}
Let $D$ be a diagram. 
Suppose that the colimit functor from $D$-indexed diagrams in the ground category to the ground category preserves monomorphisms.
Then the associated graded functor $\Gr$ from (complete) filtered objects to graded objects preserves the colimits of $D$-indexed diagrams.
\end{lemma}

\begin{proof}
The graded functor commutes with completion, so it suffices to consider the filtered case.
Also, the graded functor from $\nN$-diagrams to graded objects is a left adjoint so commutes with all colimits. 
So preservation of $D$-indexed colimits by the associated graded can be checked in terms of preservation of those colimits by the inclusion of filtered objects into diagrams.
Colimits in filtered objects are obtained by applying the reflector to the same colimits in $\nN$-diagrams.
Then $D$-indexed colimits are preserved as soon as the reflector is an isomorphism. 
In $\nN$-diagrams colimits are computed objectwise, so the reflector being an isomorphism is implied by $D$-indexed colimits in the ground category preserving monomorphisms.
\end{proof}

\begin{cor}
\label{cor: preservation of colims}
The associated graded functor $\Gr$ from (complete) filtered objects to graded objects preserves filtered colimits, and in particular transfinite compositions and coproducts.
If the ground category is $\mathbb{Q}$-linear, then the associated graded functor also preserves the coinvariants of any finite group action.
\end{cor}
\begin{proof}
The first statement follows directly from Lemma~\ref{lemma:monomorphisms and gr} and the fact that the ground category satisfies AB5.
For the second statement, if the ground category is $\mathbb{Q}$-linear then the comparison map between coinvariants and invariants of a finite group action is a (natural) isomorphism. 
The $\Gr$ functor is then a composition of left and right adjoints so preserves such (co)invariants.
\end{proof}

\begin{lemma}
\label{lemma: trees are gr-flat}
The tree cooperad on a gr-flat $\sS$-module is gr-flat.

The inclusion of $I$ via $\eta$ is a gr-coaugmentation.
\end{lemma}
\begin{proof}
The general shape of the argument is to argue that $\Gr$ preserves all of the categorical building blocks of the tree cooperad, and then that these building blocks preserve flat objects.

The tree cooperad is built as a directed colimit of finite stages, each of which is built from the monoidal unit $I$ and previous stages by the $\circ$ product and a binary coproduct. 
So we will be done by induction as soon as we argue that each of these procedures preserves gr-flat objects.

The monoidal unit is gr-flat. 
For coproducts, directed colimits, by Corollary~\ref{cor: preservation of colims} we know that $\Gr$ preserves such colimits. 
Flat objects (and thus gr-flat objects) are closed under coproducts and directed colimits (the latter by AB5). 

The $\circ$ product is built as a coproduct of summands each of which either a monoidal product or a $\mathbb{S}_n$-quotient of a sum of monoidal products. 
We know by Lemma~\ref{lemma:gr-flat is monoidal subcat} that $\Gr$ preserves the monoidal product and the coproducts, and by Corollary~\ref{cor: preservation of colims} that it preserves the $\mathbb{S}_n$-quotients under the $\qQ$-linear assumption.
Flat objects are closed under the monoidal product by definition, and are (again) closed under arbitrary coproducts.
Moreover, the finite group coinvariants of an object $X$ are a retract of $X$.
Thus if the $X$ is flat then its coinvariants are.
Then the $\circ$ product preserves gr-flatness as well and we are done.

The given $\eta$ is already a coaugmentation before taking associated graded, which implies by functoriality that it remains one after doing so.
\end{proof}
\begin{lemma}
\label{lemma: trees are altipotent}
The tree cooperad on a gr-flat $\sS$-module, equipped with the inclusion of $I$ as $\eta$, is altipotent.
\end{lemma}

\begin{proof}
By definition $\CR^{p,0}\Tt^c M = F_p \Tt^c M$. 
We recall (from Corollary \ref{cor: cofree as a coproduct}) that $\Tt^c M = \colim_n \Tt^c_n M \cong \amalg_{j \geq 0} Q_j M$. 
On a term of the colimit, the decomposition map $\bar{\Delta}\Tt^c_n M$ is represented by summands
\begin{equation}
\label{eq: summand decomposition map}
(\Tt^c_{n_0} M)(m) \otimes_{\sS_m} \bigotimes_{i = 1}^m \Tt^c_{n_i} M
\end{equation}
with each $n_i$ strictly less than $n$, each of these maps provides a map
\begin{equation}
\label{eq: decomposition map on Q}
Q_n M \to (Q_{n_0} M)(m) \otimes_{\sS_m} \bigotimes_{i = 1}^m Q_{n_i} M. 
\end{equation}
(Moreover these maps characterize the decomposition map.) 
We claim by induction on $n \geq 0$ that for all $p\geq 0$ we have
\[
\CR^{p,n+1}\Tt^c M \cong \Tt^c_{n+1} M \amalg F_p \left(\amalg_{j = n+2}^\infty Q_j M\right).
\]

Let assume that $n=0$. We have $\Tt^c_1 M \subset \CR^{p,1}\Tt^c M$ since $\bar \Delta_{|\Tt^c_1 M} \equiv 0$ in the tree cooperad. Then we compute $\widetilde{\CR}^{p,1}(\Tt^c M \circ \Tt^c M) = F_p(\Tt^c M \circ \Tt^c M)$ (this shows that $F_p \Tt^c M \subset \CR^{p,1}\Tt^c M$). 
On $Q_j M$ for $j \geq 2$, we have that $\bar \Delta$ is non zero and that $\bar \Delta$ cannot decrease the filtration degree by the definition of $\Delta$. 
The desired isomorphism $\CR^{p,1}\Tt^c M \cong \Tt^c_{1} M \amalg F_p \left(\amalg_{j = 2}^\infty Q_j M\right)$ follows.

Assume the result to be true until the step $n$. 
From the inclusions $\CR^{p, r} \Cc \to \CR^{p, r+1} \Cc$ for all $r$ and the inclusion $F_p \Tt^c M \subset \CR^{p,1}\Tt^c M$, we get $F_p \Tt^c M \subset \CR^{p,n+1}\Tt^c M$. 
By means of the description of the summands of $\bar{\Delta}\Tt^c_{n+1} M$ given in \eqref{eq: summand decomposition map}, we get by induction that $\bar\Delta\Tt^c_{n+1} M$ is in $\widetilde{\CR}^{p,n+1}(\Tt^cM\circ\Tt^cM)$ for all $p$: take $p_i = 0$ and $\delta = p$ in
\begin{multline*}
\widetilde{\CR}^{p, n+1}(\Tt^c M \circ \Tt^c M) = \coprod_{m \geq 0}\bigcup_{p_0+\cdots+p_m + \delta = p}F_{p_0}\Tt^c M(m) \otimes_{\sS_m} \bigotimes_{i=1}^m F_{p_i}\CR^{p_i+\delta,n}\Tt^c M\\
\cong \coprod_{m \geq 0}\bigcup_{p_0+\cdots+p_m + \delta = p}F_{p_0}\Tt^c M(m) \otimes_{\sS_m} \bigotimes_{i=1}^m (F_{p_i} \Tt^c_n M) \amalg F_{p_i + \delta}(\amalg_{j= n+1}^\infty Q_j M)
\end{multline*}
(by the induction hypothesis) so $\Tt^c_{n+1} M \subset \CR^{p,n+1}\Tt^c M$. 
Finally, considering the decomposition map on $Q_j M$ for $j \geq n+2$ and the component in
\[
\coprod_{m \geq 0}\bigcup_{p_0+\cdots+p_m + \delta = p}F_{p_0} M(m) \otimes_{\sS_m} \bigotimes_{i=1}^m (F_{p_i} \Tt^c_n M) \amalg F_{p_i + \delta}(\amalg_{j= n+1}^\infty Q_j M),
\]
we see that it has to come from $F_p Q_j M$ to arrive in $\widetilde{\CR}^{p, n+1}(\Tt^c M \circ \Tt^c M)$. 
This completes the induction.

Then for fixed $p$ the inclusions $\Tt^c_{n+1}M\to \CR^{p,n+1}\Tt^c M$ pass to the colimit over $n$ so that 
\[
\Tt^c M\cong \colim_n \Tt^c_n M\to \colim_n \CR^{p,n+1}\Tt^c M\to \Tt^c M
\]
is the identity. 
Since our ground category satisfies AB5, the colimit of monomorphisms is monic so this suffices to show that the final map here is an isomorphism. 
This is Condition~\eqref{item: altipotence condition about conilpotence} of Definition~\ref{defi: altipotence}.

Condition~\eqref{item: altipotence condition about coaugmentation being strong} is easy because $\bar\Delta\eta=0$.

In order to prove Condition~\eqref{item: altipotence condition about primitives being right closed}, we compute
\[
\Prim_k \Tt^c M  = \bigcap_{p \geq 0} \CR^{p+k, p}\Tt^c M \cong \left\{ \begin{array}{rl}
I \amalg \coprod_{p \geq 0} F_{p+k} Q_{p+1} M & \text{for $k \leq 0$,}\\
\coprod_{p \geq 0} F_{p+k} Q_{p+1} M & \text{for $k > 0$.}
\end{array}\right.
\]
The map \eqref{eq: decomposition map on Q} (for $n = p+1$ and $n_i = p_i+1$) restricts on $F_{p+k} Q_{p+1}M$ to give a map
\begin{multline*}
F_{p+k} Q_{p+1} M \to \\
\bigcup_{k_0 + (p_1 + k_1)+ \cdots =(p_m+k_m) = p+k}(F_{k_0} Q_{p_0+1} M)(m) \otimes_{\sS_m} \bigotimes_{i = 1}^m F_{p_i+k_i}Q_{p_i+1} M,
\end{multline*}
where $\sum p_i+1 = p+1$. 
This proves Condition~\eqref{item: altipotence condition about primitives being right closed} and concludes the proof.
\end{proof}

\begin{cor}
\label{cor: the tree cooperad is cofree}
The tree cooperad, viewed as a functor from gr-flat complete $\sS$-modules to altipotent cooperads, is right adjoint to the infinitesimal coideal.
That is, the tree cooperad is cofree in the category of altipotent complete cooperads.
\end{cor}

\begin{proof}
By Lemma~\ref{lemma: trees are altipotent}, the tree cooperad (equipped with $\eta$) lies in the full subcategory of altipotent cooperads.
By Corollary~\ref{cor: projections determine}, an $\mathbb{S}$-module map from $\Cc$ to the tree cooperad is uniquely determined by its projections to $Q_n M$.
By Lemma~\ref{lemma: projections are determined by varphi with the tree cooperad}, such a map is a cooperad map from $\Cc$ to the tree cooperad if and only if it is further determined by the projection $\Cc\to \Tt^c M\to Q_1 M\cong M$ via the recipe of Construction~\ref{construction, Phi extensions}.

Being compatible with the gr-coaugmentation means that the triangle
\[
\begin{tikzcd}
I\ar{rr}\ar[dr]&& \Tt^c M\\
& \Cc\ar[ur]
\end{tikzcd}
\]
commutes after taking the associated graded. This then implies that $I\to \Cc\to M$ should factor through strictly positive filtration degree.
This means that $\varphi:\Cc\to M$ must factor through the infinitesimal coideal $\coideal{\Cc}\to M$.

The procedure just outlined gives an injective map from the set of altipotent cooperad maps $\Cc\to \Tt^c M$ to the set of complete $\sS$-module maps $\coideal{\Cc}\to M$. 
In the other direction, given a map of collections $\varphi:\Cc\to \coideal{\Cc}\to M$, Lemma~\ref{lemma: extensions exist} tells us that the map of collections $\Phi:\Cc\to \prod Q_n M$ obtained via Construction~\ref{construction, Phi extensions} factors through $\Tt^c M$.
Then the other direction of Lemma~\ref{lemma: projections are determined by varphi with the tree cooperad} ensures that this extension is in fact a map of cooperads. 
This establishes surjectivity.
\end{proof}

\begin{remark}
It is possible to see an element in the cofree altipotent complete cooperad on $M$ as a (possibly infinite) sum of trees, whose vertices of arity $k$ are indexed by element in $M(k)$.
\end{remark}

\subsection{Coderivations}
\label{section: coderivations}

To have a complete treatment, we briefly review the standard fact that coderivations valued in a cofree cooperad are determined by their projections to the cogenerators.

Let $f$ and $g$ be maps of $\mathbb{S}$-modules from $P$ to $Q$ and let $h$ be a map of $\mathbb{S}$-modules from  $R$ to $S$. 
We define a map $h\circ'(f;g)$ from $R\circ P$ to $S\circ Q$.
The components of $h\circ'(f;g)$ are induced by 
\[
h(n)\otimes \bigotimes_{i=1}^{n-1}f(k_i) \otimes g(k_n): R(n)\otimes \bigotimes_{i=1}^n P(k_i) 
\to
S(n)\otimes \bigotimes_{i=1}^n Q(k_i).
\]
This is closely related to the \emph{infinitesimal composite} of $f$ and $g$ as defined in \cite[Section 6.1.3]{LodayVallette}.
Neither is quite a generalization of the other as defined, but the family resemblance should be clear.

Now suppose $f:\Cc\to\Dd$ is a map of complete cooperads then a \emph{coderivation of $f$} is a morphism $d$ of complete $\sS$-modules $\Cc\to\Dd$ such that
\[
\Delta_\Dd\cdot d = (d\circ f)\cdot \Delta + (f\circ' d)\cdot \Delta.
\]
A coderivation of $\id_\Cc$ is also called just a coderivation of $\Cc$.

\begin{prop}
\label{prop: characterization coderivations}
Let $M$ be a complete $\sS$-module and $\Cc$ a gr-coaugmented altipotent cooperad equipped with a map $f:\Cc\to \Tt^c M$.
Projection to cogenerators $\Tt^c M\to M$ induces a bijection between coderivations of $f$ and $\mathbb{S}$-module maps from $\Cc$ to $M$.
\end{prop}

\begin{proof}
Equip the $\mathbb{S}$-module $\Cc 1\amalg \Cc x$ with a decomposition map induced by that of $\Cc$:
\[\Cc 1\amalg \Cc x\to (\Cc 1\amalg \Cc x)\circ (\Cc 1\amalg \Cc x)\]
where, 
\begin{itemize} 
\item writing $i$ for the canonical inclusions of $\Cc 1$ and $\Cc x$ into $\Cc 1 \amalg \Cc x$ and
\item writing $\dbydx$ for the inclusion of $\Cc x$ into $\Cc 1 \amalg \Cc x$ in the $\Cc 1$ factor,
\end{itemize}
we set
\[
\Cc 1\xrightarrow{\Delta} (\Cc \circ \Cc)1 \cong (\Cc 1\circ \Cc 1) \xrightarrow{i\circ i}(\Cc 1\amalg \Cc x)\circ (\Cc 1\amalg \Cc x)
\]
and
\[
\Cc x\xrightarrow{\Delta}(\Cc \circ \Cc)x \cong (\Cc x\circ \Cc x) 
\xrightarrow{\left(i \circ \dbydx\right) + \left(\dbydx\circ \left(\dbydx;i\right)\right)}
(\Cc 1\amalg \Cc x)\circ (\Cc 1\amalg \Cc x).
\]
Coassociativity of $\Cc$ implies that $\Cc 1\amalg \Cc x$ is a gr-coaugmented altipotent cooperad with this structure decomposition and structure data
\begin{gather*}
\text{counit: }\Cc 1 \amalg \Cc x\xrightarrow{\text{projection}} \Cc 1\cong \Cc\xrightarrow{\epsilon} I;
\\
\text{gr-coaugmentation: }I\xrightarrow{\eta}\Cc\cong \Cc 1\xrightarrow{i}\Cc 1\amalg \Cc x.
\end{gather*}
It is also clear that this construction is functorial.
Then
\begin{align*}
\Hom_{\mathsf{Coop}}(\Cc 1 \amalg \Cc x,\Tt^c M)&\cong \Hom_{\mathbb{S}}(\coideal{\Cc 1 \amalg \Cc x}, M)
\\
&\cong \Hom_{\mathbb{S}}(\coideal{\Cc} 1,M)\times \Hom_{\mathbb{S}}(\Cc x,M)
\\
&\cong \Hom_{\mathsf{Coop}}(\Cc,\Tt^c M)\times \Hom_{\mathbb{S}}(\Cc, M).
\end{align*}

The final entry has a natural projection to $\Hom_{\mathsf{Coop}}(\Cc,\Tt^c M)$. 
The fiber over the morphism $f$ is naturally isomorphic to $\Hom_{\mathbb{S}}(\Cc,M)$.
On the other hand, the fiber over the morphism $f$ in
$\Hom_{\mathsf{Coop}}(\Cc 1 \amalg \Cc x,\Tt^c M)$
consists of those cooperad maps of the form 
\[
\Cc 1\amalg \Cc x\cong \Cc\amalg \Cc\xrightarrow{f\amalg d}\Tt^c M
\] for some $d$. 

The condition to be a cooperad map is automatically satisfied on the $\Cc 1$ factor since $f$ is a map of cooperads.
On the $\Cc x$ factor, for $f\amalg d$ to be a map of cooperads we get an equation.
One side of the equation is
\[
\Delta_{\Tt^c M}\cdot d.
\]
The other side is a sum of two terms coming from the two terms defining the decomposition map on $\Cc x$. The first term is
\[
\begin{tikzcd}
\Cc x \dar{\cong}\rar{\Delta} & 
(\Cc\circ \Cc)x\rar{\cong}\ar{dr}{\cong}
& \Cc x\circ \Cc x\dar{\cong}
\rar{i\circ \dbydx}\dar
&
(\Cc 1 \amalg \Cc x)\circ (\Cc 1 \amalg \Cc x)\dar{(f\amalg d)\circ (f\amalg d)}
\\
\Cc\ar{rr}{\Delta}&&
\Cc\circ\Cc\rar{d\circ f}&
\Tt^c M\circ \Tt^c M
\end{tikzcd}
\]
and the other is
\[
\begin{tikzcd}
\Cc x \dar{\cong}\rar{\Delta}
&
(\Cc \circ \Cc)x\rar{\cong}\ar{dr}{\cong}
& \Cc x\circ \Cc x
\dar{\cong}
\ar{rr}{\dbydx \circ\left(\dbydx,i\right)}
&&
(\Cc 1 \amalg \Cc x) \circ (\Cc 1\amalg \Cc x)\dar{ f \circ (f\amalg d)} \\
\Cc\ar{rr}{\Delta}&&
\Cc\circ\Cc
\ar{rr}{f \circ\left(\dbydx,f\right)}
&&
\Tt^c M\circ \Tt^c M.
\end{tikzcd}
\]
The equation of the first of these three terms with the sum of the latter two is precisely the condition for $d$ to be a coderivation of $f$.
\end{proof}

\section{Bar and cobar constructions in the curved setting}
\label{section: bar cobar constructions}

In this section, we define a bar-cobar adjunction between the categories of complete augmented curved operads with lax morphisms and complete altipotent cooperads with morphisms of complete gr-coaugmented cooperads. 
More precisely, we define a functor \emph{bar} $\hB$ which associates a complete altipotent cooperad to an augmented curved complete operad and we define a functor \emph{cobar} $\hOm$ in the opposite direction. We obtain an adjunction between the two functors
\[
\xymatrix{\hB : \mathsf{Compl.\ aug.\ curved\ operads }^{\mathrm{lax}} \ar@<.5ex>@^{->}[r] & \mathsf{Compl.\ altip.\ cooperads} : \hOm. \ar@<.5ex>@^{->}[l]}
\]
We consider \emph{lax} morphisms between complete augmented curved operads for our cobar construction to be a functor. 
From this adjunction we will obtain an $\sS$-cofibrant resolution of a curved operad in the model category structure given in Appendix \ref{appendix: MCS on gr-dg objects}. \\

We consider curved operads in the category $\Mm$ defined as the full subcategory of gr-flat objects in $\compa(\Aa)$ and cooperads in the category $\Mm'$ defined as the full subcategory of gr-flat objects in $\Comp(\dg\Aa)$. By Lemma \ref{lemma:gr-flat is monoidal subcat}, this ensures that the functor $\Gr$ is strong monoidal.

\subsection{Bar construction}

Let $(\Oo,\, \gamma,\, d,\, \varepsilon,\, \theta,\, \eta)$ be an augmented curved complete operad. 
The augmentation ideal $\overline{\Oo} \coloneqq \ker (\varepsilon : \Oo \to I)$ of $\Oo$ is a complete gr-dg $\sS$-module.

\begin{remark}
The predifferential $d$ on $\Oo$ induces a predifferential $\bar{d}$ on the complete $\sS$-module $\overline{\Oo}$ and the composition product $\gamma$ induces by means of $\eta$ a map $\bar \gamma_{(1)} : \overline{\Oo} \circ_{(1)} \overline{\Oo} \to \overline{\Oo}$.
\end{remark}

The \emph{bar construction} of the augmented curved complete operad $(\Oo,\, \gamma,\, d,\, \varepsilon,\, \theta,\, \eta)$ is given by the altipotent complete cooperad
\[
\hB \Oo \coloneqq \left( \cofree(s \overline{\Oo}),\, \Delta_\beta,\, \varepsilon_\beta,\, d_{\beta} \coloneqq d_0+d_1+d_2,\, \eta_\beta \right),
\]
where the map $d_2$ is the unique coderivation of degree $-1$ which extends the map
\[
{\cofree}(s \overline{\Oo}) \twoheadrightarrow s^2 \left( \overline{\Oo} \circ_{(1)} \overline{\Oo} \right) \xrightarrow{\gamma_s \otimes \bar{\gamma}_{(1)}} s \overline{\Oo},
\]
where $\gamma_s (s\otimes s) = s$, the map $d_1$ is the unique coderivation of degree $-1$ which extends the map
$$
{\cofree}(s \overline{\Oo}) \twoheadrightarrow s \overline{\Oo} \xrightarrow{\id_s \otimes \bar{d}} s \overline{\Oo},
$$
and the map $d_0$ is the unique coderivation of degree $-1$ which extends the map
$$
{\cofree}(s \overline{\Oo}) \twoheadrightarrow \I \xrightarrow{-s \theta} s \overline{\Oo}.
$$
For instance, we have pictorially
$$
d_0 (\as) = -\vcenter{\tiny{
\xymatrix@M=1pt@R=4pt@C=4pt{&&\\
& *{} \ar@{-}[lu] \ar@{-}[ur] \ar@{-}[d] &\\ & s\theta \ar@{-}[d] &\\ &&}}} + \vcenter{\tiny{
\xymatrix@M=1pt@R=4pt@C=4pt{&&\\ s\theta \ar@{-}[u] &&\\
& *{} \ar@{-}[lu] \ar@{-}[ur] \ar@{-}[d] &\\
&&}}} + \vcenter{\tiny{
\xymatrix@M=1pt@R=4pt@C=4pt{&&\\ && s\theta \ar@{-}[u] \\
& *{} \ar@{-}[lu] \ar@{-}[ur] \ar@{-}[d] &\\
&&}}} \in s^2 \left( \overline{\Oo} \circ_{(1)} \overline{\Oo} \right),
$$
for $\as \in \Oo(2)$ of degree $-1$ and $\theta$ is identified with $\theta (\vcenter{\tiny{
\xymatrix@M=1pt@R=6pt@C=1pt{& *{} \ar@{-}[d] &\\ & *{} &}}}) \in \Oo(1)$.
The counit $\varepsilon_\beta$ is the usual projection onto the trivial tree and the gr-coaugmentation $\eta_\beta$ is the inclusion of the trivial tree.

\begin{remark}
Since $d_{\beta} (\vcenter{\tiny{
\xymatrix@M=1pt@R=6pt@C=1pt{& *{} \ar@{-}[d] &\\ & *{} &}}}) = d_{0} (\vcenter{\tiny{
\xymatrix@M=1pt@R=6pt@C=1pt{& *{} \ar@{-}[d] &\\ & *{} &}}}) = -\vcenter{\tiny{
\xymatrix@M=1pt@R=2pt@C=1pt{& *{} \ar@{-}[d] &\\ & s\theta \ar@{-}[d] &\\ &&}}}$, it follows that the bar construction is not coaugmented as a filtered cooperad whenever the curvature is non zero.
However, when $\Oo$ is gr-flat, it is gr-coaugmented as a complete cooperad since $\theta$ is in $F_1 \Oo$ (see Definition \ref{gr-coaug}).
\end{remark}

\begin{lemma}
The coderivation $d_{\beta}$ is a differential and the bar construction induces a functor $\hB : \mathsf{Comp.\ aug.\ curved\ op.}^{\mathrm{lax}} \rightarrow \mathsf{Compl.\ altip.\ coop.}$.
\end{lemma}

\begin{proof}
We can split the square of the coderivation $d_{\beta}$ as follows
$$
(d_0 + d_1 + d_2)^2 = \underbrace{{d_0}^2} + \underbrace{d_0 d_1 + d_1 d_0} + \underbrace{{d_1}^2 + d_0 d_2 + d_2 d_0} + \underbrace{d_1 d_2 + d_2 d_1} + \underbrace{{d_2}^2}.
$$
We have ${d_0}^2 = 0$ because of sign considerations.
For the same reason and due to the fact that $\theta$ is closed, we get $d_0 d_1 + d_1 d_0 = 0$.
The bracket of two coderivations is a coderivation so $[d_0,\, d_2] = d_0 d_2 + d_2 d_0$ and ${d_1}^2$ are coderivations.
Therefore, the corestriction of the equality ${d_1}^2 + d_0 d_2 + d_2 d_0 = 0$ to $s \overline{\Oo}$ is enough to prove the equality.
The corestriction of ${d_1}^2 + d_0 d_2 + d_2 d_0$ to $s \overline{\Oo}$ is equal to ${\bar d\ }^2 + [-\bar \theta,\, -]$, which is zero since $d^2 - [\theta,\, -]$ is and $\Oo$ is augmented.
The map $d$ is a derivation with respect to $\gamma$ for the augmented curved complete operad $\Oo$, so by sign considerations, we obtain that $d_1 d_2 + d_2 d_1 = 0$.
The associativity of the composition product $\gamma$ of the augmented curved complete operad $\Oo$ gives the last equality ${d_2}^2 = 0$.

Let $(f, a) : (\Oo, d, \theta) \to (\Pp, d', \theta')$ be a lax morphism of complete augmented curved operads. 
By Corollary \ref{cor: the tree cooperad is cofree}, we define the map $\hB (f, a) : \cofree(s\overline{\Oo}) \to \cofree(s\overline{\Pp})$ as the altipotent cooperad morphism characterized by
\[
\widetilde{\cofree}(s\overline{\Oo}) \twoheadrightarrow \widetilde{I} \amalg s\overline{\Oo} \xrightarrow{sa + f} s\overline{\Pp}.
\]
The cooperad $\cofree(s\overline{\Oo})$ can be seen as an infinitesimal $\cofree(s\overline{\Pp})$-comodule by means of the cooperad morphism $\hB (f, a)$. 
Noting $d_\beta$, resp. $d'_\beta$, the coderivation on $\cofree(s^{-1}\overline{\Oo})$, resp. on $\cofree(s^{-1}\overline{\Pp})$, the map $d_\beta' \cdot \hB(f, a) - \hB(f, a) \cdot d_\beta$ is a coderivation $\cofree(s\overline{\Oo}) \to \cofree(s\overline{\Pp})$. (This is a direct computation for two coderivations and a cooperad morphism.) 
To prove that it is zero, it is therefore enough to prove that when composed with the projection to the cogenerators $s\overline{\Pp}$, it is zero (\cite[Lemma 15]{MerkulovVallette}). 
We compute the contributory terms. 
Construction \ref{construction, Phi extensions} says
\begin{align*}
& \sum_{k=0}^2\Phi_k(1) = 1 + sa + sa \otimes sa,\\
& \sum_{k=0}^2\Phi_k(s\mu) = sf(\mu) + \sum_i sf(\mu) \circ_i sa + sa \otimes sf(\mu)+ \dots,\\
& \sum_{k=0}^2\Phi_k(s\mu \circ_j s\nu) = sf(\mu) \circ_j sf(\nu).
\end{align*}
We therefore obtain (using that the curved operad $\Oo$ and $\Pp$ are augmented)
\begin{align*}
& \left(d_\beta' \cdot \hB(f, a) - \hB(f, a) \cdot d_\beta\right)_{|I}^{|s\overline{\Pp}} = s\left( f(\theta) - \theta' - d' a - \frac{1}{2}[a, a]\right),\\
& \left(d_\beta' \cdot \hB(f, a) - \hB(f, a) \cdot d_\beta\right)_{|s\overline{\Pp}}^{|s\overline{\Pp}} = -s\left( d'f(-) - fd(-) + [a, f(-)]\right),\\
& \left(d_\beta' \cdot \hB(f, a) - \hB(f, a) \cdot d_\beta\right)_{|\cofree(s\overline{\Pp})^{(2)}}^{|s\overline{\Pp}} = 0.
\end{align*}
So by the fact that $(f, a)$ is a lax morphism, we conclude that $\hB (f, a)$ commutes with the differentials and $\hB$ is a well-defined functor.
\end{proof}

\subsection{Cobar construction}

We define here a cobar construction $\hOm$ which associates to complete curved operads (resp. lax morphisms) complete altipotent (dg) cooperads (resp. gr-coaugmented morphisms). 
 
When $M$ is a filtered complete graded $\sS$-module, we denote by $s^{-1} M$ the desuspension of $M$, that is the filtered complete graded $\sS$-module such that $(s^{-1}M)_n \coloneqq M_{n+1}$ and $F_p(s^{-1}M) \coloneqq s^{-1} F_p M$.

\subsubsection{Cobar construction of an altipotent cooperad}

Let $(\Cc,\, \Delta,\, d,\, \varepsilon,\, \eta)$ be a complete altipotent cooperad. 
Every altipotent cooperad admits a decomposition $\Cc \cong I \amalg \overline{\Cc}$ as $\sS$-modules (see Remark \ref{remark: gr-coaugmentation}) and we use this decomposition in order to get an as small as possible construction. 
The \emph{cobar construction} $\hOm \Cc$ of $\Cc$ is defined as the (quasi-free) complete augmented curved operad
\[
\hOm \Cc \coloneqq \left( \free(s^{-1} \overline{\Cc}),\, \gamma_\omega,\, d_\omega \coloneqq d_1+d_2,\, \theta_\omega \coloneqq -s^{-1}d 1 - s^{-2} \overline{\Delta}_{(1)}1\right)
\]
where $1$ stands for $\eta (1) \in \Cc$, 
the curvature $\theta_\omega$ lies in $F_1 (s^{-1} \overline{\Cc} \amalg s^{-1} \overline{\Cc} \circ_{(1)} s^{-1} \overline{\Cc})$ since $\Cc$ is gr-coaugmented, and where
\begin{enumerate}
\item the map $d_1$ is the unique derivation of degree $-1$ which extends the map
\[
s^{-1} \overline{\Cc} \xrightarrow{\id_{s^{-1}} \otimes \bar{d}} s^{-1} \overline{\Cc} \rightarrowtail \free(s^{-1} \overline{\Cc}),
\]
\item 
and the map $d_2$ is the unique derivation of degree $-1$ which extends the map
\[
s^{-1} \overline{\Cc} \xrightarrow{\Delta_{s^{-1}} \otimes \overline{\Delta}_{(1)}} s^{-2}(\overline{\Cc} \circ_{(1)} \overline{\Cc}) \rightarrowtail \free(s^{-1} \overline{\Cc}),
\]
with $\Delta_{s^{-1}}(s^{-1}) = - s^{-1} \otimes s^{-1}$.
\end{enumerate}

\begin{remark}
The curvature of the cobar construction vanishes precisely when the cooperad $\Cc$ is coaugmented.
\end{remark}

\begin{lemma}
\label{lem: truncated cobar derivation squares to}
We have the equality ${d_\omega}^2 = [\theta_\omega,\, -]$, the curvature $\theta_\omega$ is closed and the cobar construction induces a functor
\[
\hOm : \mathsf{Compl.\ altip.\ coop.} \rightarrow \mathsf{Comp.\ aug.\ curved\ op.}^{\mathrm{lax}}.
\]
\end{lemma}

\begin{proof}
Because of the fact that ${d_\omega}^2 = \frac{1}{2}[d_\omega,\, d_\omega] = {d_1}^2 + (d_1 d_2 + d_2 d_1) + {d_2}^2$ and $[\theta_\omega,\, -]$ are derivations, it is enough to prove the equality ${d_\omega}^2 = [\theta_\omega,\, -]$ on the generators $s^{-1} \overline{\Cc}$. 
As $d$ commutes with $\varepsilon$, the map $\bar d$ is equal to the restriction of $d$. 
On $s^{-1} \overline{\Cc}$, we have ${d_1}^2 = {\bar d}^2 = \overline{{d}^2} = 0$ since $d$ is a differential. 
The computation
\begin{align*}
(d \otimes \id + \id \otimes d) \cdot (\overline{\Delta}_{(1)} c) & = (d \otimes \id + \id \otimes d) \cdot (\Delta_{(1)} c - 1 \otimes c - \sum c \otimes 1)\\
& = \Delta_{(1)} \cdot d c - d1 \otimes c - 1 \otimes dc - \sum (dc \otimes 1 + (-1)^{|c|} c \otimes d1)\\
& = \overline{\Delta}_{(1)} \cdot d c - d1 \otimes c - (-1)^{|c|} \sum c \otimes d1
\end{align*}
shows that $d_1 d_2 + d_2 d_1 = [-s^{-1} d 1,\, -]$. 
A similar computation of $(\overline{\Delta}_{(1)} \otimes \id) \cdot \overline{\Delta}_{(1)} - (\id \otimes \overline{\Delta}_{(1)}) \cdot \overline{\Delta}_{(1)}$ shows that ${d_2}^2 = [-s^{-2} \overline{\Delta}_{(1)} 1,\, -]$. 
Finally, ${d_\omega}^2 = [\theta_\omega,\, -]$. 
These computations are also meaningful on $s^{-1}1$ and therefore $\theta_\omega$ is closed.

It remains to describe the image by the cobar construction of a morphism of complete altipotent cooperads. 
Let $f : (\Cc,\, d,\, \eta) \to (\Cc',\, d',\, \eta')$ be a morphism of complete altipotent cooperads. 
The map of $\sS$-modules underlying $f$ is characterized by the morphisms of $\sS$-modules $\bar f : \overline{\Cc} \to \overline{\Cc}'$ and $s^{-1}\overline{ f \cdot \eta} : I \to F_1(s^{-1} \overline{\Cc}')$ ($f$ is gr-coaugmented). 
The map $\free(\bar f) : \free(s^{-1}\overline{\Cc}) \to \free(s^{-1}\overline{\Cc}')$ is a morphism of operads and $s^{-1}\overline{f \cdot \eta}$ induces an element $a : I \to F_1(s^{-1} \overline{\Cc}') \rightarrowtail F_1\free(s^{-1} \overline{\Cc}')$ of degree $-1$. 
We now show that the couple $(\free(\bar f), a)$ is a lax morphism of complete augmented curved operads. 

First, the morphism $\bar f$ commutes with $d$ and $d'$, thus $\free(\bar f)$ commutes with the derivations $d_1$ and $d_1'$. 
Then, for an operad morphism $\free(\bar f)$ and two derivations $d_2$ and $d_2'$, we have that $d_2' \cdot \free(\bar f) - \free(\bar f) \cdot d_2$ is a derivation $\free(s^{-1}\overline{\Cc}) \to \free(s^{-1}\overline{\Cc}')$. 
Moreover, $[a,\, f(-)] = \gamma_\omega (a\otimes \bar f - \bar f\otimes a)$ is also a derivation (by a direct computation). 
It follows (by \cite[Lemma 14]{MerkulovVallette}) that it is enough to check on $s^{-1}\overline{\Cc}$ the equality
\begin{equation}
\label{eq: differentials and lax morphisms}
d_2' \cdot \free(\bar f) + [a, \free(\bar f)] = \free(\bar f) \cdot d_2
\end{equation}
corresponding to Equation \eqref{eq: lax morphism diff} in the definition of lax morphisms. 
It is a consequence of the following equality resulting from the fact that $f$ is a cooperad morphism
\[
\bar \Delta_{(1)} \bar f = \overline{\Delta_{(1)} f} = \overline{(f \otimes f) \Delta_{(1)}} = (\bar f \otimes \bar f)\bar \Delta_{(1)} + \bar f(1) \otimes \bar f + \sum_i \bar f \circ_i \bar f(1).
\]
Secondly and similarly, the equality $\free(\bar f)(\theta_\omega) = \theta'_\omega + d'a + \frac{1}{2}[a, a]$ corresponding to Equation \eqref{eq: lax morphism curvature} in the definition of lax morphisms results from the fact that $f$ is a dg cooperad morphism and also a morphism of gr-coaugmented cooperads.
\end{proof}

Because of the previous lemma, the cobar construction is a quasi-free complete curved operad
\[ \hOm\Cc = \left( \free_+(s^{-1} \overline{\Cc}),\, d_\omega \right)/\left( \im\left({d_\omega}^2 - [\theta_\omega,\, -] \right)\right), \]
where the functor $\free_+$ applied to a complete gr-dg $\sS$-module $M$ is realized as the pointed complete gr-dg operad $\left( \free(M \amalg \vartheta I),\, \vartheta \right)$, with the generator $\vartheta$ which lives in arity 1, filtration degree 1 and degree $-2$ (the correct arity, filtration degree 1 and degree for a curvature).

\subsection{Convolution algebras}

The bar $\hB$ and cobar $\hOm$ constructions represent a bifunctor of curved twisting morphisms in a curved convolution Lie algebra.

\subsubsection{Convolution curved Lie algebra}

Let $\Oo$ be a curved complete operad and $\Cc$ be a complete altipotent cooperad.
We use the notation $\Hom_{\sS} (\Cc,\, \Oo)$ for
\[
\prod_n\Hom_{\sS}(\Cc(n),\, \Oo(n)),
\] 
where $\Hom_{\sS}$ stands for the morphisms of complete filtered $\sS$-modules, and we fix the element
\begin{multline*}
\Theta \coloneqq (\theta_\Oo \cdot \varepsilon_\Cc : \Cc \to \Oo) \in F_1\Hom_{\sS} (\Cc,\, \Oo)_{-2} =\\
\{ f \in \Hom_{\sS} (\Cc,\, \Oo)_{-2} \textrm{ such that } \forall p \geq 0,\ f(F_p\Cc)\subset F_{p+1}\Oo\}
\end{multline*}
which is well-defined since $\theta_\Oo$ is in filtration level $1$ of $\Oo$ and $F_1I = \varnothing$.
We define on $\Hom_{\sS} (\Cc,\, \Oo)$ the pre-Lie product $\star$ given by
\[
f \star g : \Cc \xrightarrow{\Delta_{(1)}} \Cc \circ_{(1)} \Cc \xrightarrow{f \circ g} \Oo \circ_{(1)} \Oo \xrightarrow{\gamma_{(1)}} \Oo.
\]

The product $\star$ induces a Lie bracket $\{f,\, g\} \coloneqq f \star g -(-1)^{|f||g|} g \star f$ on $\Hom_{\sS} (\Cc,\, \Oo)$.
We denote by $\partial$ the derivation of $\star$ given by $\partial (f) \coloneqq d_\Oo \cdot f - (-1)^{|f|} f \cdot d_\Cc$.

\begin{lemma}
The tuple $(\Hom_{\sS} (\Cc,\, \Oo),\, \{-,\, -\},\, \partial,\, \Theta)$ forms a curved Lie algebra, called the \emph{convolution curved Lie algebra}.
\end{lemma}

\begin{proof}
We do the computations for the curvature:
\begin{align*}
\partial^2(f) &= d_{\Oo} \cdot \partial(f) - (-1)^{|\partial(f)|} \partial(f) \cdot d_{\Cc}\\
& = {d_{\Oo}}^{2} \cdot f - (-1)^{|f|} d_{\Oo} \cdot f \cdot d_{\Cc} + (-1)^{|f|} \left(d_{\Oo} \cdot f \cdot d_{\Cc} - (-1)^{|f|} f \cdot {d_{\Cc}}^2\right)\\
&= {d_{\Oo}}^{2} \cdot f = [\theta,\, -] \cdot f = \Theta \star f - f \star \Theta = \{\Theta,\, f\}
\end{align*}
and $\partial(\Theta) = d_{\Oo} \cdot \theta \cdot \varepsilon -(-1)^{|\Theta|} \theta \cdot \varepsilon \cdot d_{\Cc} = 0$ since $d_{\Oo} \cdot \theta = 0$ and $\varepsilon \cdot d_{\Cc} = 0$.
\end{proof}

Fix an augmented curved complete operad $(\Oo,\, \gamma,\, d_\Oo,\, \varepsilon_\Oo,\, \theta,\, \eta_\Oo)$ and a complete altipotent cooperad $(\Cc,\, \Delta,\, d_\Cc,\, \varepsilon_\Cc,\, \eta_\Cc)$.
An element $\alpha : \Cc \to \Oo$ of degree $-1$ in the curved Lie algebra $\Hom_\sS(\Cc,\, \Oo)$ such that $\alpha \cdot \eta_\Cc : I \to F_1\Oo$ and $\varepsilon_{\Oo} \cdot \alpha = 0$ is called a \emph{curved twisting morphism} if it is a solution of the \emph{curved Maurer-Cartan equation}
\[
\Theta + \partial (\alpha) + \frac{1}{2}\{\alpha,\, \alpha\} = 0.
\]
We denote by $\Tw(\Cc,\, \Oo)$ the set of curved twisting morphisms. 

\begin{lemma}
Curved twisting morphisms forms a bifunctor
\[
\Tw(-,-) : \mathsf{altipotent\ coop.}^{\mathrm{op}} \times \mathsf{curved\ op.}^{\mathrm{lax}} \to \mathsf{Set}.
\]
\end{lemma}

\begin{proof}
The functoriality in the left variable is given by pre-composition and doesn't cause any trouble. 
We only prove the functoriality in the right variable. Let $(f, a) : \Oo \to \Pp$ be a lax morphism between curved operads. To a curved twisting morphism $\alpha : \Cc \to \Oo$, we associate the map $G(f, a)(\alpha) \coloneqq f \cdot \alpha + a \cdot \varepsilon_\Cc : \Cc \to \Pp$. 
We first compute
\[
\partial (f \cdot \alpha + a \cdot \varepsilon_\Cc) = - [a, f] \cdot \alpha + f \cdot (\partial \alpha) + d_\Pp \cdot a \cdot \varepsilon_\Cc.
\]
Then
\begin{align*}
(f \cdot \alpha + a \cdot \varepsilon_\Cc) \star (f \cdot \alpha + a \cdot \varepsilon_\Cc) = f \cdot (\alpha \star \alpha) + [a, f] \cdot \alpha + a^2 \cdot \varepsilon_\Cc
\end{align*}
It follows that
\begin{align*}
\partial (G(f, a)) + G(f, a) \star G(f, a) & = f\cdot (\partial \alpha + \alpha \star \alpha) + (d_\Pp a + a^2) \cdot \varepsilon_\Cc\\
& = -f \cdot \theta_\Oo \cdot \varepsilon_\Cc + \left( f(\theta_\Oo) - \theta_\Pp \right) \cdot \varepsilon_\Cc = -\theta_\Pp \cdot \varepsilon_\Cc.
\end{align*}
Therefore $G(f, a) : \Cc \to \Pp$ is a curved twisting morphism. 
It is direct to check that $G(\id, 0) = \id$ and $G$ preserves the composition. This concludes the proof.
\end{proof}

We show that it is representable on the left and on the right by the bar and the cobar constructions.

\begin{prop}
\label{prop: bar-cobar adjunction}
For any complete altipotent cooperad $\Cc$ and any complete augmented curved operad $\Oo$, there are natural bijections
\[\Hom_{\mathsf{comp.\, aug.\, curv.\, op.}^{\mathrm{lax}}}(\hat \Omega \Cc,\Oo)
\cong
\Tw(\Cc,\, \Oo)
\cong
\Hom_{\mathsf{compl.\, alti.\, coop.}}(\Cc,\hat \B\Oo).
\]
Therefore the functors $\hat \B$ and $\hat \Omega$ form a pair of adjoint functors between the categories of complete augmented curved operads with lax morphisms and the category of complete altipotent cooperads with gr-coaugmented morphisms.
\end{prop}

\begin{proof}
We make the first bijection explicit. 
A lax morphism of augmented complete operads $(f_\alpha, a_\alpha) : \Tt (s^{-1}\overline{\Cc}) \to \Oo$ is uniquely determined by a map $-s \bar\alpha : s^{-1}\overline{\Cc} \to \Oo$ of degree $0$ such that, since $f_\alpha$ is augmented, $\varepsilon_\Oo \cdot (s\bar\alpha) = 0$ and an $\sS$-module map $a_\alpha : I \to F_1\Oo$ of degree $-1$. 
This is equivalent to the data of a map $\alpha \coloneqq a_\alpha \cdot \varepsilon_\Cc + \bar \alpha : \Cc \to \Oo$ of degree $-1$ such that $\alpha \cdot \eta_\Cc : I \to F_1\Oo$ and $\varepsilon_{\Oo} \cdot \alpha = 0$. 
Moreover, $f_\alpha$ commutes with the predifferential up to the term $[a_\alpha, f_\alpha]$ if and only if the following diagram commutes
\[
\xymatrix@C=50pt{s^{-1}\overline{\Cc} \ar[r]^{-s\bar\alpha} \ar[d]_{d_{1} + d_{2}} & \Oo \ar[r]^{d_{\Oo}} & \Oo\\
s^{-1}\overline{\Cc} \amalg (s^{-1}\overline{\Cc} \circ_{(1)} s^{-1}\overline{\Cc}) \ar[rr]_{\hspace{.5cm} (-s\bar\alpha) + (-s\bar\alpha) \circ_{(1)} (-s\bar\alpha)} && \Oo \amalg \Oo \circ_{(1)} \Oo \ar[u]_{\id + \gamma_{(1)}}}
\]
up to the term $[a_\alpha, -s\bar\alpha]$. 
This corresponds to the equality $\partial(\bar\alpha) + [a_\alpha, \bar\alpha] + \bar\alpha \star \bar\alpha = 0$ on $\overline{\Cc}$, or equivalently to the restriction of the equality $\Theta + \partial(\alpha) + \alpha \star \alpha = 0$ on $\overline{\Cc}$. 
The fact that the morphism $f_\alpha$ sends the curvature $\theta_\omega$ to $\theta_\Oo + d_\Oo a_\alpha + \frac{1}{2}[a_\alpha, a_\alpha]$ coincides with the restriction of the equation $\Theta + \partial(\alpha) + \alpha \star \alpha = 0$ to $\im \eta_\Cc$ since $\alpha \cdot \eta_\Cc = a_\alpha$. 

We now make the second bijection explicit. 
A morphism of complete altipotent cooperads $g_{\alpha} : \Cc \rightarrow {\Tt^{c}}(s\overline{\Oo})$ is uniquely determined by a map $s\alpha : \coideal{\Cc} \rightarrow s\overline{\Oo}$ (see Corollary \ref{cor: the tree cooperad is cofree}), that is by a map $\alpha : \Cc \rightarrow \Oo$ of degree $-1$ satisfying $\alpha \cdot \eta_\Cc : I \to F_1\Oo$ and $\varepsilon_\Oo \cdot \alpha = 0$.
Moreover, $g_{\alpha}$ commutes with the differential if and only if the following diagram commutes
\[
\xymatrix@C=120pt@R=10pt{\Cc \ar[r]^{\varepsilon_\Cc + s\alpha + ((s\alpha) \circ_{(1)} (s\alpha)) \cdot \Delta_{(1)} \hspace{1cm}} \ar[dd]_{d_{\Cc}} & I \amalg s\overline{\Oo} \amalg s\overline{\Oo} \circ_{(1)} s\overline{\Oo} \ar[dd]^{d_{\beta}^{|s\overline{\Oo}} = (d_0 + d_{1} + d_{2})^{|s\overline{\Oo}}}\\
&\\
\Cc \ar[r]_{s\alpha} & s\overline{\Oo}}
\]
Similarly as before, this is equivalent to the equality $\Theta + \partial(\alpha) + \alpha \star \alpha = 0$.
This proves the proposition.
\end{proof}

\begin{examples}
\begin{itemize}
\item
To the identity (lax) morphism $\id_{\hOm \Cc} = (\id_{\hOm \Cc}, 0) : \hOm \Cc \rightarrow \hOm \Cc$ of complete augmented curved operads corresponds the curved twisting morphism $\iota : \Cc \rightarrow \hOm \Cc$ defined by $\Cc \twoheadrightarrow \overline{\Cc} \cong s^{-1}\overline{\Cc} \rightarrowtail \free(s^{-1}\overline{\Cc})$.
\item
To the identity morphism $\id_{\hat\B \Oo} : \hat\B \Oo \rightarrow \hat\B \Oo$ of complete altipotent cooperads corresponds the curved twisting morphism $\pi \in \Hom_\sS( \hat\B \Oo,\, \Oo)$ defined by $\cofree(s\overline{\Oo}) \twoheadrightarrow s\overline{\Oo} \cong \overline{\Oo} \rightarrowtail \Oo$.
\end{itemize}
\end{examples}

\begin{remark}
Throughout the bar-cobar adjunction, a curved operad morphism $(f_\alpha, 0)$ corresponds to a twisting morphism $\alpha$ such that the composite $I \xrightarrow{\eta_\Cc} \Cc \xrightarrow{\alpha} \Oo$ cancels and to a coaugmented cooperad morphism $g_\alpha$. 
Examples of such morphisms are given by the two listed above.
\end{remark}

\subsection{Bar-cobar resolution}
\label{sec: bar-cobar resolution}

In this section, we assume that the category $\Aa$ is the category of $\ringK$-modules, for $\ringK$ a field of characteristic 0, since we use results of Sections 6.6 and 6.7 in \cite{LodayVallette} which are proved in this setting. 
There is however a slight difference since we consider \emph{unbounded} chain complexes instead of bounded below chain complexes. 
The proofs given in \cite{LodayVallette} require modification to extend to the unbounded situation. 
For this reason, we restrict our results to the case of an operad $\Oo$ such that $\Gr \Oo$ is endowed with an auxiliary weight gradation such that that the underlying chain complex is bounded below for each fixed weight. 
This weight gradation on $\Gr \Oo$ should not be confused with the filtration $F_\bullet$ on $\Oo$ that appears throughout this text. (The two filtrations differ in general.) To prove the ($\sS$-)cofibrancy of the bar-cobar resolutions, we require moreover that the filtration $F_\bullet$ on $\Oo$ comes from some gradation on $\Oo$. 

We provide two \emph{$\sS$-cofibrant resolutions} of a graded curved operad $\Oo$ whose associated graded $\Gr \Oo$ is also weight-graded (for an independent gradation), that is (by definition) resolutions $\mathcal{R} \mathrel{\ooalign{$\xrightarrow{\sim\mkern4mu}$\cr%
  \hidewidth$\rightarrow\mkern4mu$}} \Oo$ (graded quasi-isomorphisms and strict surjections) such that the gr-dg $\sS$-module map $\free (\vartheta I) \to \mathcal{R}$ is cofibrant.

A complete operad $\Pp$ in graded $\sS$-modules is \emph{weight graded} if it is endowed with an extra grading $\Pp^{(w)}$ for $w \geq 0$, compatible with the composition product. It is connected for this weight grading if we have
\[ \Pp = I \amalg \Oo^{(1)} \amalg \Pp^{(2)} \amalg \cdots \amalg \Pp^{(w)} \amalg \cdots, \]
that is $\Pp^{(0)} = I$. 
Moreover, we call $\Pp$ \emph{bounded below} when for each fixed $w \geq 0$, the chain complex underlying $\Pp^{(w)}$ is bounded below (for the homological degree).\\

In the following results, we make use of Theorem 6.6.3 in \cite{LodayVallette}. 
This theorem applies to connected weight graded operads $\Pp$ that satisfy $\Pp(0) = \{ 0\}$ and whose underlying chain complex is positively graded. 
In order to extend this result to operads with no condition on $\Pp(0)$, we have to adapt the proof of Lemma 6.5.9 in \cite{LodayVallette}. 
More specifically in the proof that $\B \Pp \circ_\pi \Pp$ is acyclic, the description of the contracting homotopy makes use of an order on the elements in $\Pp$ given by the labels of their entries. 
This contracting homotopy can be replaced by the following formula (with the notations of \cite{LodayVallette} and for $p_l$ either in $I$ or in $\overline{\Pp}$)
\begin{multline*}
h(t ;\, p_1,\, \ldots,\, p_k) \coloneqq\\
\frac{1}{\mathrm{card} \{p_l \in \overline{\Pp} \}} \sum_{\{p_l \in \overline{\Pp}\}} (-1)^{|t|}(t_l ;\, p_1,\, \ldots,\, p_{l-1},\, 1,\, \ldots,\, 1,\, p_{l+1},\, \ldots,\, p_k),
\end{multline*}
where $t_l$ is the element of $\B \Pp$ obtained by grafting the tree $t$ with $sp_l$ above. 
There are no other modifications in the proof.

Then since the filtrations used in the proof of (comparison) Lemma 6.7.1 in \cite{LodayVallette} are defined on a fixed weight of the right twisted composite product, we can check that the proof applies also to cooperads $\Cc$ which are bounded below weight graded. Assuming that $\Pp$ is bounded below weight graded, we obtain that the cooperad $\B \Pp$ is bounded below for the induced weight grading. 
This shows that Theorem 6.6.3 in \cite{LodayVallette} is valid for a connected bounded below weight graded operad $\Pp$.

\begin{thm}
\label{thm: bar-cobar resolution}
Let $\Oo$ be a complete curved operad such that $\Gr \Oo$ is a connected bounded below weight graded operad. The counit of the adjunction of Proposition \ref{prop: bar-cobar adjunction}
\[
f_\pi = (f_\pi, 0) :\hat \Omega \hat \B\Oo\to \Oo
\]
is a graded quasi-isomorphism and a surjection. 

Assuming moreover that the filtration $F_\bullet$ on $\Oo$ comes from a gradation, this map is an $\sS$-cofibrant resolution in the model category structure on complete curved operads given in Appendix \ref{appendix: MCS on gr-dg objects}.
\end{thm}

\begin{proof}
Since $\Gr$ is strong monoidal, we have $\Gr \hOm \hat \B\Oo \cong \Omega \B \Gr \Oo$, where we have denoted by $\Omega$ and $\B$ the classical cobar and bar constructions for an augmented graded (uncurved) operad $\Gr \Oo$. 
Then by the generalization of Theorem 6.6.3 in \cite{LodayVallette} described before the present theorem, it follows from the fact that $\Gr \Oo$ is connected bounded below weight graded that we have a quasi-isomorphism
\[
\Gr \hOm \hat \B\Oo \cong \Omega \B \Gr \Oo \xrightarrow{\sim} \Gr \Oo.
\]
This shows that $\hOm \hat \B\Oo \to \Oo$ is a graded quasi-isomorphism which is easily seen to be a strict surjection. 
Indeed, the counit of the adjunction is given by the composition
\[ \hOm \hat \B \Oo \twoheadrightarrow \free(\overline{\Oo}) \xrightarrow{\gamma_\Oo} \Oo \]
where the two maps are strict surjections, so the composition is as well. 

Assuming that the curvature on $\Oo$ is non-zero guarantees that the curvature on $\hOm \hat \B \Oo$ is non-zero and that $\ker ([\theta_\omega,\, -]) \cong \free(\theta_\Oo)$. 
We assume now that the filtration $F_\bullet$ on $\Oo$ comes from a gradation $\Oo^{(p)}$ (that is $F_p \Oo = \amalg_{q \geq p} \Oo^{(q)}$). 
It remains to show that the map $\free(\vartheta I) \to \hOm \hat\B \Oo$, sending $\vartheta$ to $\theta_\omega = \theta_\Oo$, is an $\sS$-cofibration. Let $t \in \hOm \hat\B \Oo$. We have ${d_\omega}^2(t) = [\theta_\Oo,\, t]$ where the bracket is computed by means of the free product. The identity $[\theta_\Oo,\, t] = 0$ implies $t \in \free(\theta_\Oo)$ (induction on the roots of the terms in $t$). 
Also we have $(\im d_\omega) \cap \free(\theta_\Oo) = \{0\}$ since $t \in (\im d_\omega) \cap \free(\theta_\Oo)$ implies $t = d_\omega (t')$ and $0 = d_\omega t = {d_\omega}^2(t') = [\theta_\Oo,\, t']$ so $t' \in \free (\theta_\Oo)$ and $t = d_\omega (t') = 0$ (by a direct inspection). 
Then assuming for some integers $k \leq l$ that ${d_\omega}^k(t) = {d_\omega}^l(t')$, we can show that $t - {d_\omega}^{l-k}(t') \in \free(\theta_\Oo)$ (using the previous arguments). 
Because $\Oo$ is assumed to be graded, $\hOm \hat \B \Oo$ is also graded by the sum of the gradation of the elements in $\overline{\Oo}$. 
Let us write $\hOm \hat \B \Oo = \amalg_{p \geq 0} R^{(p)}$ the direct sum decomposition coming from this gradation. 
Let us define for $k \geq 0$
\[
\left\{\begin{array}{l}
S_{2k} \coloneqq \free(\theta_\Oo) + \langle R^{(0)} \amalg \cdots \amalg R^{(k-1)} \amalg (R^{(k)} \cap \ker (\Gr d_\omega)) \rangle,\\
S_{2k+1} \coloneqq \free(\theta_\Oo) + \langle R^{(0)} \amalg \cdots \amalg R^{(k)} \rangle.
\end{array}\right.
\]
where the notation $\langle V \rangle$ stands for the gr-dg $\sS$-module generated by $V$. 
It is direct to see that $(S_l)_l$ is an increasing filtration such that $\free( \theta_\Oo) \subset S_0$. 
Moreover, we have $S_0 / \free (\theta_\Oo) \cong \amalg \hat \Zz_{q, n}^{1, \infty}\{m\}$, where $\hat \Zz^{1, \infty}_{q, n}\{m\}$ is the notation for the tensor product of $\hat \Zz^{1, \infty}_{q, n}$ with an $\sS_{m}$-module of rank 1, since ${d_\omega}^2 = [\theta_\Oo,\, -]$ and $(\im d_\omega) \cap \free(\theta_\Oo) = \{0\}$ (so iterated powers of $d_\omega$ never cancel outside $\free(\theta_\Oo)$) and because ${d_\omega}^k(t) = {d_\omega}^l(t')$ for $k \leq l$ implies that $t - {d_\omega}^{l-k}(t') \in \free(\theta_\Oo)$. 
Similarly
\[
S_{2k}/S_{2k-1} \cong (\langle C_{2k} \rangle,\, d_\omega) \cong \amalg \hat \Zz_{q, n}^{1, \infty}\{m\}
\]
where $C_{2k}$ is a complement subspace of
\begin{multline*}
 \ker (\Gr d_\omega) \cap \left(\amalg_{l < k}{d_\omega}^{2(k-l)}(R^{l})\right) + R^{(k)} \cap \left(\amalg_{l < k} {d_\omega}^{2(k-l)+1}(R^{l}) \right)\\
\text{ in }  R^{(k)} \cap \ker (\Gr d_\omega),
\end{multline*}
where we use the fact that for any integer $l$ and $r \in R^{(l)}$, we have $d_\omega(r) \in R^{(l)} \amalg R^{(l+1)}$ since ${d_\omega}^2(r) = [\theta_\Oo,\, r] \in R^{(l+1)}$. 
On the other side
\[
S_{2k+1}/S_{2k} \cong (\langle C_{2k+1},\, \bar d_\omega) \cong \amalg \hat \Zz_{q, n}^{1, \infty} \{m\} \coprod \amalg \ringK\{m'\}
\]
where $C_{2k+1}$ is a complement subspace of
\[
R^{(k)} \cap \left(\ker (\Gr d_\omega) + \amalg_{l < k}{d_\omega}^{2(k-l)}(R^{(l)}) \right) \text{ in } R^{(k)},
\]
and where the copies of $\ringK\{m'\}$ correspond to copies $\hat \Zz_{q', n'}^{0, \infty} \{m'\}$ in $S_{2k+1}$ for which the part $\hat \Zz_{q', n'-1}^{1, \infty} \{m'\}$ already appears in $S_{2k}$. 
Indeed $(\Gr d_\omega)^2 = 0$ implies that $(\im (d_\omega)_{|R^{(k)}}) \cap R^{(k)} \subset \ker(\Gr d_\omega)$.
Finally by Proposition \ref{prop: cofibration are retract of good cofibration} and Remark \ref{rem: example of cofibrations}, we obtain that $\free (\theta_\Oo) \to \hOm \hat \B \Oo$ is an $\sS$-cofibration. 

This concludes the proof.
\end{proof}

\begin{remark}
In the previous theorem, when the curvature is zero, the resolution $\hOm \hat \B \Oo \to \Oo$ coincides with the usual bar-cobar resolution (in a complete setting) and is cofibrant in the corresponding model category.
\end{remark}

\section{Koszul duality for curved operads}
\label{section: koszul duality for curved operads}

In this section, we describe the Koszul dual cooperad of a homogeneous quadratic curved operad $\Oo$ and we relate it to the bar construction $\hB \Oo$. We propose a definition for a curved operad to be Koszul and then we show that in this case, the Koszul dual cooperad is quasi-isomorphic to the bar construction $\hB \Oo$. It follows from the proof a sort of Poincaré--Birkhoff--Witt theorem which provides a way to make the Koszul dual cooperad explicit.

Under the Koszul property, we finally obtain an $\sS$-cofibrant resolution of curved complete operads
\[ \hOm \Ooa \to \Oo, \]
which has the advantage of having a domain with a smaller space of generators in comparison with the bar-cobar resolution. 

The two generic examples are the following:
\begin{itemize}
\item
a curved associative algebra whose curvature is a sum of squares (with conditions so that it is Koszul);
\item
the operads encoding curved associative algebras or curved Lie algebras.
\end{itemize}

In this section, we work with the category $\Aa$ of $\ringK$-modules where $\ringK$ is a field. We consider the closed monoidal category $\Mm = \compa(\Aa)$ on the operadic side and the closed monoidal category $\Mm' = \Comp(\dg\Aa)$ on the cooperadic side. 
Section \ref{section: subcomonoid generated} applies to more general categories $\Aa$ (except for Proposition \ref{prop: cooperad generated by, revisited} and after Section \ref{sect: curved Koszul operad}).

\subsection{Subcomonoid generated by a symmetric module}
\label{section: subcomonoid generated}

We denote by $\Nn$ the category $\sS$-$\mathsf{Mod}(\Mm)$ and by $\Nn'$ the category $\sS$-$\mathsf{Mod}(\Mm')$. 
Let $\Cc$ be a comonoid in $\Nn'$ and $S$ be an object in $\Nn'$. In Appendix B of \cite{bV08}, Vallette defined the notion of the subcomonoid of $\Cc$ generated by $S$ in an abelian setting. We extend the definition to quasi-abelian categories. The only difference is the necessity to consider strict morphisms in order to identify certain coimages with cokernels (and in the dual case certain images with kernels). We only emphasize the differences in the unfamiliar case of the subcomonoid generated by an object. The case of an ideal and of a monoid quotient generated by an object is dual.

\begin{defi}
\begin{itemize}
\item
Let $\Ii \rightarrowtail \Cc \twoheadrightarrow \Qq$ be an exact sequence in $\Nn'$, where $\Cc$ is a comonoid. The epimorphism $\Cc \twoheadrightarrow \Qq$ in $\Nn'$ is a \emph{coideal epimorphism} if $\Ii \rightarrowtail \Cc$ is a monomorphism of comonoids in $\Nn'$.
\item
Let $\xi : \Cc \twoheadrightarrow S$ be an epimorphism in $\Nn'$, where $\Cc$ is a comonoid. We consider the category $\Ss_\xi$ of sequences $(\textbf{S}) : \Ii \rightarrowtail \Cc \twoheadrightarrow \Qq$ as in the previous item and such that the composite $\Ii \rightarrowtail \Cc \twoheadrightarrow S$ is equal to $0$. A morphism between $(\textbf{S})$ and $(\textbf{S}') : \Ii' \rightarrowtail \Cc \twoheadrightarrow \Qq'$ is given by a pair $(i : \Ii \to \Ii',\, p : \Qq \to \Qq')$ such that $i$ is a morphism of comonoids and $p$ is a morphism in $\Nn'$, and such that the following diagram commutes:
\[
\begin{tikzcd}
\Ii \dar[swap,"i"]\ar[dr] &&\\
\Ii' \rar & \Cc \rar\ar[dr] & \Qq \dar["p"]\\
&& \Qq'.
\end{tikzcd}
\]
\item
We aim to consider the largest subcomonoid of $\Cc$ vanishing on $S$. This notion is given by the terminal object $(\overline{\textbf{S}}) : \Cc(S) \rightarrowtail \Cc \twoheadrightarrow (S)$ in $\Ss_\xi$, when the latter admits one.
\end{itemize}
\end{defi}

We now make the coideal quotient $(S)$ and the subcomonoid $\Cc(S)$ explicit.

\begin{defi}
We denote by $i_\Cc : \Cc \hookrightarrow S \amalg \Cc$ the inclusion and by $\pi_\Cc : S \amalg \Cc \twoheadrightarrow \Cc$ the projection.
\begin{itemize}
\item
The \emph{multilinear part} in $S$ of $(\Cc \circ (S \amalg \Cc)) \circ \Cc$ is given either by the cokernel
\[ \coker \left(\Cc \circ \Cc \circ \Cc \xrightarrow{\id \circ i_\Cc \circ \id} (\Cc \circ (S \amalg \Cc)) \circ \Cc\right), \]
or equivalently, by the kernel
\[ \ker \left((\Cc \circ (S \amalg \Cc)) \circ \Cc \xrightarrow{\id \circ \pi_\Cc \circ \id} \Cc \circ \Cc \circ \Cc\right), \]
by means of the fact that the monoidal structure commutes with colimits and $i_\Cc$ is a section of $\pi_\Cc$. 
It is denoted by $\Cc \circ (\underline{S} \amalg \Cc) \circ \Cc$. 
\item
The \emph{coideal quotient $(S)$} of $\Cc$ generated by $\xi : \Cc \twoheadrightarrow S$ is given by the coimage
\[ (S) := \coim \left(\Cc \xrightarrow{\Delta^{\circ 2}} \Cc^{\circ 3} \xrightarrow{\mathrm{proj} \cdot (\id_\Cc \circ (\xi \amalg \id_\Cc) \circ \id_\Cc)} \Cc \circ (\underline{S} \amalg \Cc) \circ \Cc\right), \]
where the map $\mathrm{proj}$ is the projection $\Cc \circ (S \amalg \Cc) \circ \Cc \to \Cc \circ (\underline{S} \amalg \Cc) \circ \Cc$.
\end{itemize}
\end{defi}

\begin{lemma}
The coideal quotient of $\Cc$ generated by $\xi : \Cc \twoheadrightarrow S$ is also given by the coimage
\[ \coim \left(\Cc \xrightarrow{(\id \otimes \Delta) \cdot \Delta_{(1)}} \Cc \circ_{(1)} \Cc^{\circ 2} \xrightarrow{\id_\Cc \circ_{(1)} \xi \circ \id_\Cc} \Cc \circ_{(1)} (S \circ \Cc)\right). \]
\end{lemma}

\begin{proof}
It is clear that
\begin{multline*}
\ker \left(\Cc \xrightarrow{\Delta^{\circ 2}} \Cc^{\circ 3} \xrightarrow{\mathrm{proj} \cdot (\id_\Cc \circ (\xi \amalg \id_\Cc) \circ \id_\Cc)} \Cc \circ (\underline{S} \amalg \Cc) \circ \Cc\right) \subseteq\\
\ker \left(\Cc \xrightarrow{(\id \otimes \Delta) \cdot \Delta_{(1)}} \Cc \circ_{(1)} \Cc^{\circ 2} \xrightarrow{\id_\Cc \circ_{(1)} \xi \circ \id_\Cc} \Cc \circ_{(1)} (S \circ \Cc)\right)
\end{multline*}
since $\Cc \circ_{(1)} \Cc^{\circ 2}$ is a quotient of $\Cc^{\circ 3}$. 
The coassociativity of $\Delta$ gives the converse inclusion.
\end{proof}

The epimorphism $\Cc \twoheadrightarrow (S)$ is strict (it is given by a $\coim$). We therefore have $(S) \cong \coker (\Cc(S) \to \Cc)$. 
By the universal property of the cokernel, the equivalence between the fact that the composition $\Cc(S) \rightarrowtail \Cc \to I$ is non zero and the fact that the counit does not factor through the coideal quotient $(S)$. 
This ensures that the cooperad $\Cc(S)$ is non zero.

\begin{prop}
\label{prop: cooperad generated by, revisited}
We suppose that the counit $\Cc \to I$ does not factor through the coideal quotient $(S)$. Then, the \emph{subcomonoid of $\Cc$ generated by $S$} is
\[ \Cc(S) := \ker(\Cc \to (S)), \]
that is
\[ \Cc(S) = \ker\left(\Cc \xrightarrow{(\id \otimes \Delta) \cdot \Delta_{(1)}} \Cc \circ_{(1)} \Cc^{\circ 2} \xrightarrow{\id_\Cc \circ_{(1)} \xi \circ \id_\Cc} \Cc \circ_{(1)} (S \circ \Cc)\right). \]
\end{prop}

\begin{proof}
The counit is given by the composite $\Cc(S) \rightarrowtail \Cc \to I$. See Proposition 60 in \cite[Appendix B]{bV08} for the fact that the comultiplication restricts to $\Cc(S)$.
\end{proof}

\subsection{Koszul dual cooperad}

We define the notion of a homogeneous quadratic curved operad and the Koszul dual cooperad associated with it. 

Following \cite[5.5.4]{LodayVallette}, we describe a weight grading on the free operad. It is given by the number of generating operations needed in the construction of a given element in the tree operad. Let $M$ be an $\sS$-module in $\Mm$ and let $\free M$ be the tree operad on $M$. We define the weight $w(\mu)$ of an element $\mu$ recursively as follows: $w(\id) = 0$, $w(\mu) = 1$ when $\mu \in M(n)$ for some $n$; and more generally $w(\nu;\, \nu_1,\, \dots ,\, \nu_m) = w(\nu) + w(\nu_1) + \dots + w(\nu_m)$ for some $m$ and $\nu \in M(m)$.
We denote by $\free M^{(r)}$ the $\sS$-module of elements of weight $r$.

\begin{defi}
\label{defi: homogeneous quadratic curved operad}
A curved complete operad $(\Oo,\, \theta)$ is called \emph{homogeneous quadratic} if all of the following hold:
\begin{enumerate}
\item
the operad $\Oo$ admits a homogeneous quadratic presentation
\[ (\Oo,\, \theta) \cong (\free E/(R),\, \theta) \cong \cfree(E)/\left( R \amalg (\vartheta -\tilde\theta(1))\right), \]
where $E$ is an $\sS$-module in $\Mm$ and $(R)$ is the ideal generated by a strict $\sS$-submodule $R$ of $\free E^{(2)}$ in $\Mm$ and the curvature $\theta$ is induced by a map $\tilde\theta : \I \to F_1 \free E^{(2)}$,
\item
the sub-$\sS$-module $R$ is a direct sum of homological degree and filtration degree homogeneous subspaces,
\item
the predifferential of $\Oo$ vanishes, and
\item the counit $\cofree(sE) \to I$ does not factor through the coideal quotient $(S)$, where we fix
\[ \Cc \coloneqq \cofree(sE) \text{ and } S \coloneqq \left(I \amalg \cofree(sE)^{(2)}\right) / \left(s^2 R \amalg (1 + s^2\tilde\theta(1))\right). \]
\end{enumerate}
\end{defi}

\begin{remark}
\begin{itemize}
\item
The condition that the differential of $\Oo$ is $0$ implies that the bracket with the curvature is $0$.
\item
The filtration on $E$ induces a filtration on $\free E$. The fact that the inclusion $R \hookrightarrow \free E^{(2)}$ is strict means that the filtration on $R$ is the one induced by the filtration on $E$.
\end{itemize}
\end{remark}

\begin{defi}
Let $\Oo \coloneqq \cfree(E)/\left( R \amalg (\vartheta -\theta)\right)$ be a curved complete operad equipped with a homogeneous quadratic presentation. 
We define the \emph{Koszul dual cooperad} $\Ooa$ of $\Oo$ to be the sub-cooperad ${\Cc(S)}$ of ${\cofree}(sE)$ generated by the strict epimorphism $\Cc \twoheadrightarrow S$ (with the notations of Definition \ref{defi: homogeneous quadratic curved operad}). It is an altipotent cooperad since it is a (strict) sub-cooperad of an altipotent cooperad.
\end{defi}

We sometimes denote it by $\Ooa = {\Cc}(sE,\, s^2 R \amalg (1+s^2\tilde\theta(1)))$. 
Its counit is given by the composite $\varepsilon_{\Ooa} : \Ooa \rightarrowtail {\cofree}(sE) \twoheadrightarrow I$.

\subsection{Koszul dual operad}

It is hard to explicitly describe the elements of $\Ooa$ by means of its definition. 
Under certain assumptions and forgetting the filtration, we can however get a better understanding by means of the Koszul dual operad $\Oo^!$. 

Considering a specific case of Section 4 in \cite{HirshMilles}, we propose the following definition.

\begin{defi}
We say that a complete operad $\Oo$ is \emph{constant-quadratic} if it admits a presentation of the form $\Oo = \free E/(R)$, where $E$ is a complete dg $\sS$-module and $(R)$ is the ideal generated by a complete dg $\sS$-module $R \subset I \amalg \free E^{(2)}$. (We assume that the filtration on $R$ is induced by the filtration on $E$.) We require that $R$ is a coproduct of homological degree and filtration degree homogeneous subobjects. Thus the complete operad $\Oo$ is homological degree graded and filtration degree graded and has a weight filtration induced by the $\sS$-module of generators $E$. We assume further that the following conditions hold:
\begin{itemize}
\item[(I)] The space of generators is minimal, that is $R \cap I = \{0\}$.
\item[(II)] The space of relations is maximal, that is $(R)\cap \{I \amalg E \amalg \free E^{(2)}\} = R$.
\end{itemize}
\end{defi}

In the context of the definition above, we denote by $q : \free E \to \free E^{(2)}$ the canonical projection and by $q R \subset \free E^{(2)}$ the image under $q$ of $R$. 
Since $R \cap I = \{0\}$, there exists a filtered map $\varphi : q R \rightarrow I$ such that $R$ is the graph of $\varphi$:
\[
R = \{X - \varphi(X),\, X \in q R\}.
\]

We propose the following generalization of Section 2.1.9 in \cite{GinzburgKapranov}. 
In our context, the situation isn't symmetric since we define a constant-quadratic (resp. curved) operad associated with a curved (resp. constant-quadratic) operad. 
We make it however as symmetric as possible so that it is easy to extend these definitions to a curved-constant-quadratic setting. 
Let $E$ be a complete $\sS$-module of finite dimension in each arity. 

To define a pairing between $\free(E)^{(2)}$ and $\free(E^*)^{(2)}$, we need to choose a tree basis of the trees made of two vertices. When the trees are reduced (that is each vertex has at least one leaf), the tree basis can be provided by shuffle trees (as explained in \cite[Section 8.2.5]{LodayVallette}). When one of the tree has no leaves (necessarily the one above), we choose as a tree basis the one such that the tree of arity $0$ is put further left. This defines a natural pairing $\free(E)^{(2)} \otimes \free(E^*)^{(2)}$ which can be extended in a unique way to a natural pairing $\langle-,\, -\rangle : \left(\vartheta I \amalg \free(E)^{(2)}\right) \otimes \left(I \amalg \free(E^*)^{(2)}\right)$ such that $\langle \vartheta I,\, \free(E^*)^{(2)}\rangle = \{ 0\}$, $\langle \free(E)^{(2)},\, I\rangle = \{ 0\}$ and $\langle \vartheta 1,\, 1\rangle = 1$ where $1$ is the generator of $I$. 

We propose the two following definitions of the Koszul dual operad in the curved / constant-quadratic context. In the definition below and in the proposition below, we forget the possible filtration on $E$. Moreover, curved operads are not filtered (unlike the definitions in Section \ref{section: curved operads}).

\begin{defi}
\label{defi: Koszul dual operad}
In this definition, $E$ is an $\sS$-module of finite dimension in each arity.
\begin{enumerate}
\item
Given a homogeneous quadratic curved operad
\[ \Oo = \cfree(E)/\left( R \amalg (\vartheta -\theta)\right), \]
we define its \emph{Koszul dual complete constant-quadratic operad} by
\[ \Oo^! \coloneqq \free(s^{-1} \mathcal{S}^{-1} \underset{\H}{\otimes} E^*)/\left( (R \amalg (\vartheta - \theta))^{\perp}\right), \]
where $\mathcal{S}^{-1} \coloneqq \End_{s^{-1}\ringK}$ is an endomorphism operad and $\underset{\H}{\otimes}$ is the Hadamard product defined by $(\Oo \underset{\H}{\otimes} \Pp)(n) \coloneqq \Oo(n) \otimes \Pp(n)$. 
It is a constant-quadratic operad.
\item
Given a constant-quadratic complete operad
\[ \Oo = \free(E)/(R), \]
we define its \emph{Koszul dual curved operad} by
\[ \Oo^! \coloneqq \cfree(s^{-1} \mathcal{S}^{-1} \underset{\H}{\otimes} E^*) / \left( R^{\perp} \right). \]
\end{enumerate}
\end{defi}

\begin{prop}
Let $\Oo$ be a homogeneous quadratic curved operad (resp. a constant-quadratic operad) generated by an $\sS$-module $E$ of finite dimension in each arity. 
We have
\[\left(\Oo^!\right)^! \cong \Oo. \]
\end{prop}

\begin{proof}
The natural map $E \to (E^*)^*$ is an isomorphism since $E$ is finite-dimensional.  
Then it is direct to check the two computations.
\end{proof}

\begin{remark}
We emphasize the fact that the complete filtration is not required to define the Koszul dual operad of a curved operad. The complete filtration is needed for the following two reasons:
\begin{enumerate}
\item
to define the Koszul dual cooperad,
\item
to define a homotopical context.
\end{enumerate}
\end{remark}

We go back to the complete filtered setting. Under the condition that the operad $\Oo$ is Koszul, we will see in Theorem \ref{thm: qi for Koszul operad and PBW} a description of the elements in $\Ooa$ by means of the elements in the Koszul dual cooperad of the (uncurved) operad $\Gr \Oo$. 
This will give a way to precisely describe the elements in $\Ooa$.

\subsection{Bar construction of a curved quadratic operad}
\label{subsection: bar construction}

We fix a homogeneous quadratic curved operad $\Oo = \cfree(E)/\left( R \amalg (\vartheta -\theta)\right)$. We consider on $\hat \B \Oo$ a second homological degree called \emph{syzygy degree}. 
The operad $\Oo \cong (\free E/(R),\, \theta)$ is weight-graded by the weight grading of $\free E$ (that is the number of generators in $E$) since $R$ is homogeneous. We define the syzygy degree of an element in $\hat \B \Oo$ recursively as follows: the syzygy degree of $\id$ is $0$, the syzygy degree of an element in $s\overline{\Oo}$ is $1$ minus its weight in $\Oo$ and the syzygy degree of an element $(\nu;\, \nu_1,\, \dots ,\, \nu_m)$ is the sum of the syzygy degrees of the elements $\nu$, $\nu_1$, \dots , $\nu_m$. 
The differentials $d_0$ and $d_2$ lower the syzygy degree by $1$.

\begin{prop}\label{prop: Koszul dual cooperad and syzygy degree}
Let $\Oo = \cfree(E)/\left( R \amalg (\vartheta -\theta)\right)$ be a complete curved operad equipped with a homogeneous quadratic presentation
. Let $\Ooa = {\Cc}(sE,\, s^2 R \amalg (1+s^2\tilde\theta(1)))$ be its Koszul dual cooperad. 
The natural inclusion
\[ i : \Ooa \hookrightarrow {\cofree}(sE) \hookrightarrow \hat \B \Oo \]
induces an isomorphism of complete cooperads
\[ i : \Ooa \xrightarrow{\cong} \H_0(\hat \B \Oo), \]
where the homology degree is taken to be the syzygy degree.
\end{prop}

\begin{proof}
Syzygy degree $0$ elements in $\hB \Oo$ are given by ${\cofree}(sE)$. 
Syzygy degree $-1$ elements coincide with
\[
{\cofree}(sE) \circ_{(1)} \left( \left({\cofree}(sE)^{(2)} / s^2 R\right) \circ {\cofree}(sE)\right)
\]
seen in $\hB \Oo$. The differential $d$ provides therefore a map
\[ {\cofree}(sE) \to {\cofree}(sE) \circ_{(1)} \left( \left({\cofree}(sE)^{(2)} / s^2 R\right) \circ {\cofree}(sE)\right).\]
Through the isomorphism
\[ {\cofree}(sE)^{(2)} / s^2 R \cong \left(I \amalg {\cofree}(sE)^{(2)}\right) / \left(s^2 R \amalg (1 + s^2\tilde\theta(1))\right) = S, \]
this map coincides with the map
\[
\left({\Cc} \xrightarrow{(\id \otimes \Delta) \cdot \Delta_{(1)}} {\Cc} \circ_{(1)} {{\Cc}}^{\circ 2} \xrightarrow{\id_{{\Cc}} \circ_{(1)} \xi \circ \id_{{\Cc}}} {\Cc} \circ_{(1)} (S \circ {\Cc})\right).
\]
We conclude by Proposition \ref{prop: cooperad generated by, revisited}.
\end{proof}

\subsection{Curved Koszul operad}
\label{sect: curved Koszul operad}

Since we are working with $\Aa = \dg$-$\ringK$-$\Mod$ (with $\ringK$ a field), the cooperad splits as $\sS$-modules as $\Ooa \cong I \amalg \overline{\Ooa}$. We denote by $\eta : I \to \Ooa$ the $\sS$-module map coming from this isomorphism.

Using the weight grading on $\cofree(sE)$, we can write
\[ \cofree(sE) \cong \cofree(sE)_{\text{even}} \coprod \cofree(sE)_{\text{odd}}, \]
where $\cofree(sE)_{\text{even}} \coloneqq \coprod_{k \geq 0} \cofree(sE)^{(2k)}$ and $\cofree(sE)_{\text{odd}} \coloneqq \coprod_{k \geq 0} \cofree(sE)^{(2k+1)}$.

\begin{lemma}
\label{lem: even odd}
The cooperad $\Ooa$ splits as follows:
\[ \Ooa \cong \Ooa_{\text{even}} \coprod \Ooa_{\text{odd}},\]
where $\Ooa_{\text{even}} = \Ooa \cap {\cofree}(sE)_{\text{even}}$ and $\Ooa_{\text{odd}} = \Ooa \cap {\cofree}(sE)_{\text{odd}}$.
\end{lemma}

\begin{proof}
We have
\begin{multline*}
\Ooa = \ker\left({\cofree}(sE) \xrightarrow{(\id \otimes \Delta) \cdot \Delta_{(1)}} {\cofree}(sE) \circ_{(1)} {\cofree}(sE)^{\circ 2}\right.\\
\left. \xrightarrow{\id_{{\cofree}(sE)} \circ_{(1)} \xi \circ \id_{{\cofree}(sE)}} {\cofree}(sE) \circ_{(1)} (S \circ {\cofree}(sE))\right)
\end{multline*}
where $S = \left(I \amalg \cofree(sE)^{(2)}\right) / \left(s^2 R \amalg (1 + s^2\tilde\theta(1))\right)$. 
The map $(\id \otimes \Delta) \cdot \Delta_{(1)}$ and $\id_{{\cofree}(sE)} \circ_{(1)} \xi \circ \id_{{\cofree}(sE)}$ stabilize the odd part, resp. the even part. It follows that the kernel splits as desired.
\end{proof}

As a consequence and using the definition of $\Ooa$ and the fact that $\eta$ is a gr-coaugmentation, we can write
\begin{align}
\label{eq: eta}
\eta(1) = 1 + \left(s^2\tilde\theta + \sum s^2 r\right) + \dots \in I \amalg sE \amalg \cofree(sE)^{(2)} \amalg \cdots,
\end{align}
with $\sum s^2 r \in s^2 R$. 
There is a natural morphism $\kappa : \Ooa \to \Oo$ of degree $-1$ defined as the composite
\[ \kappa : \Ooa \hookrightarrow {\cofree}(sE) \twoheadrightarrow sE \cong E \to \Oo. \]

\begin{lemma}
We have $\Theta + \frac{1}{2} \{ \kappa,\, \kappa \} = 0$ for $\Theta = \theta \cdot \varepsilon_{\Ooa}$. Hence $\kappa$ is a curved twisting morphism. 
It is in fact a truncated curved twisting morphism by the expression of $\eta(1)$ given in Equation \eqref{eq: eta}.
\end{lemma}

\begin{proof}
On $\eta(I)$, a direct computation using Equality (\ref{eq: eta}) shows that the equality $\Theta + \frac{1}{2} \{ \kappa,\, \kappa \} = 0$ is true. 
Then, on $\overline{\Ooa} \subset {\cofree}(sE)^{(\geq 1)} = sE \amalg \cofree(sE)^{(2)} \amalg \cdots$, using the fact that the decomposition map is the cofree one, we only have to consider elements whose projection on $\cofree(sE)^{(2)}$ is non zero. 
By means of Lemma \ref{lem: even odd} and using the definition of $\Ooa$, we get that these elements write
\begin{align}
\label{eq: R}
\sum s^2 r + \dots \in \cofree(sE)^{(2)} \amalg \cofree(sE)^{(\geq 4)},
\end{align}
with $\sum s^2 r \in s^2 R$. Since $\theta_{| \overline{\Ooa}} \equiv 0$, we obtain that $\kappa$ is a curved twisting morphism.
\end{proof}

\begin{defi}
A curved complete operad $(\Oo,\, \theta)$ is called \emph{Koszul} if $\Oo \cong (\widehat{\Gr \Oo},\, \bar\theta)$, where $\Gr \Oo$ is the graded dg operad associated with $\Oo$ and $\Gr \Oo$ admits a homogeneous quadratic presentation $\Gr \Oo \cong \free E/(R)$ in $\Filt(\Aa)$ for which it is a Koszul operad.
\end{defi}

\begin{thm}
\label{thm: qi for Koszul operad and PBW}
Let $\Oo$ be a curved complete Koszul operad. Then the map
\[ i : \Ooa \to \hat \B \Oo \]
is a quasi-isomorphism of cooperads. 
Moreover, there is a Poincar\'e--Birkhoff--Witt type isomorphism of $\sS$-modules
\[ \Ooa \cong \widehat{(\Gr \Oo)^\antishriek}.\]
\end{thm}

As we will see in examples, this last isomorphism provides a way to understand the cooperad $\Ooa$.

\begin{proof}
We denote by $F_p$ the filtration on $\hat \B \Oo$. It is induced by the filtration on $\Oo$ for which $\Oo$ is complete. 
We consider the (increasing) filtration $G_p := F_{-p}$ on $\hat \B \Oo$ (note that $G_p$ is not bounded below). 
We have \[ d_0 : G_p \hat \B \Oo \to G_{p-1} \hat \B \Oo \text{ and } d_2 : G_p \hat \B \Oo \to G_p \hat \B \Oo. \]
The filtration $G_p$ is complete since $\hat \B \Oo$ is complete for the filtration $F_p$ and it is exhaustive since $\hat \B \Oo = G_0 \hat \B \Oo$. We consider the spectral sequence $E_{p, q}^\bullet$ associated with this filtration. Using that the functor $\Gr$ is strong monoidal ($\ringK$ is a field), we have
\[ E_{p, q}^0 = G_p \hat \B \Oo_{p+q} / G_{p-1} \hat \B \Oo_{p+q} \cong (\B \Gr \Oo)_{p+q}^{(p)}\]
where the upper index stands for the degree of the gradation in $\B \Gr \Oo$ (the bar construction of the dg operad $\Gr \B \Oo$) and the last isomorphism comes from the fact that $\Gr$ is strong monoidal (see Lemma \ref{lemma:gr-flat is monoidal subcat}). The differential is given by $d^0 = d_{\B \Gr \Oo}$ (induced by $d_2$). Since $\Gr \Oo$ is a Koszul operad, we get
\[ E_{p, q}^1 = {{(\Gr \Oo)^\antishriek}}_{p+q}^{(p)}.\]

Now let us consider on $\hat \B \Oo$ and on $\B \Gr \Oo$ the syzygy degrees described in the beginning of Section \ref{subsection: bar construction}. We have the following compatibility condition: for $b \in F_p \hB \Oo \backslash F_{p+1} \hB \Oo$ with a homogeneous syzygy degree, the associated element in $(\B \Gr \Oo)^{(p)}$ has the same syzygy degree. 
The $E^1$ page is concentrated in syzygy degree $0$ (in $\B \Gr \Oo$) as the Koszul dual cooperad is, and by means of the compatibility of the syzygy degrees, we therefore obtain that the differentials $d^r$, $r \geq 1$, induced by $d_0$ and $d_2$ are $0$ (they both decrease the syzygy degree). This implies that the spectral sequence is regular (see Definition 5.2.10 in \cite{cW94}) and by the Complete Convergence Theorem 5.5.10 in \cite{cW94} that the spectral sequence converges to $\H_\bullet (\hat \B \Oo)$ (the spectral sequence is bounded above).

We have therefore
\[ {{(\Gr \Oo)^\antishriek}}^{(p)} = \amalg_q E_{p, q}^1 = \amalg_q E_{p, q}^\infty \cong \Gr_p \H_\bullet (\hat \B \Oo) \]
and using the syzygy degree, we get ${{(\Gr \Oo)^\antishriek}}^{(p)} \cong \Gr_p \H_0 (\hat \B \Oo) \cong \Gr_p \H_\bullet (\hat \B \Oo)$. By means of Proposition \ref{prop: Koszul dual cooperad and syzygy degree}, we obtain moreover that $\widehat{(\Gr \Oo)^\antishriek} \cong \Ooa$.

Finally, the morphism $\Ooa \to \hat \B \Oo$ is a quasi-isomorphism and we have an isomorphism of $\sS$-modules $\Ooa \cong \widehat{(\Gr \Oo)^\antishriek}$.
\end{proof}

As in Section \ref{sec: bar-cobar resolution}, we obtain a resolution for a Koszul operad $\Oo$, which is $\sS$-cofibrant when we assume that the filtration $F_\bullet$ on $\Oo$ comes from a gradation.

\begin{thm}
\label{thm: Koszul resolution}
Suppose that $(\Oo,\, \theta)$ is a Koszul complete curved operad such that $\Gr \Oo$ is connected bounded below weight graded. 
The natural map
\[
f_\kappa = (f_\kappa, 0) : \hOm \Oo^\antishriek\to \Oo
\]
is a resolution of curved complete operads.

Assuming moreover that the filtration $F_\bullet$ on $\Oo$ comes from a gradation and that the curvature $\theta$ is non-zero, this map is an $\sS$-cofibrant resolution in the model category structure on complete curved operads given in the Appendix \ref{appendix: MCS on gr-dg objects}.
\end{thm}

\begin{proof}
Since $\Gr$ is strong monoidal we have
\[ \Gr \hat \Omega \Ooa \cong \Omega \Gr (\Ooa) \cong \Omega (\Gr \Oo)^\antishriek. \]
Since $\Gr \Oo$ is a Koszul operad, Theorem 7.4.2 in \cite{LodayVallette} (adapted by means of the discussion before Theorem \ref{thm: bar-cobar resolution}) gives that $\hat \Omega \Ooa \to \Oo$ is a graded quasi-isomorphism which is easily seen to be a strict surjection.

As in Theorem \ref{thm: bar-cobar resolution}, when $\theta \neq 0$ and under a gradation assumption, the map $\free(\theta_\Oo) \to \hOm \Ooa$ is a cofibration in the model category structure on complete gr-dg $\sS$-modules given in Appendix \ref{appendix: MCS on gr-dg objects}.
\end{proof}

\begin{remark}
Let $E$ be an $\sS$-module that we put in weight degree 1 and $R$ be a weight degree 2 homogeneous $\sS$-submodule in $\free E$. Then if $\Pp = \free E/(R)$ is a Koszul operad for the filtration induced by the weight gradation on $E$, then assuming that $E$ is graded bounded below (for example graded and finitely generated), we obtain that the operad $\Gr \Oo$ is automatically connected bounded below weight graded.
\end{remark}

\section{The associative case}
\label{section: assocase}

In this section, we make the case of the operad encoding curved associative algebras explicit. It is a curved operad that we denote $\cAs$ and we prove that it encodes curved associative algebras. 

Again, we assume that $\Aa$ is the category of $\ringK$-modules, with $\ringK$ a field.

\subsection{The curved operad encoding curved associative algebras}

We recall that the operad $\As$ encoding associative algebras is the (trivially) filtered (complete) operad defined by
\[
\As \coloneqq \free \left(\as\right)/\left(\ass\right).
\]
Its representations in $\ringK$-modules are associative algebras.

\begin{prop}
We denote
\[
\cAs \coloneqq \left( \free \left(\curvA \amalg \as\right)/\left(\ass\right),\, 0,\, \theta \coloneqq \curvAs \right),
\]
where $\curvA$ is of degree $-2$, $\as$ is of degree 0 and the predifferential is zero.
We filter $\cAs$ by the number of $\curvA$, say
\[
F_p \cAs \coloneqq \{ \mu \in \cAs \textrm{ s.t. the number of } \curvA \textrm{ in } \mu \textrm{ is greater than or equal to } p\}.
\]
(This filtration comes in fact from a gradation.) 
It is a complete curved operad whose curvature belongs to $F_1 \cAs$. 
We can also write
\[ \cAs = \cfree \left(\curvA \amalg \as\right)/\left(\ass,\ \vartheta - \theta\right),
\]
where $\theta$ is defined above.
\end{prop}

For $n \geq 1$ and $S \subseteq \{1, \ldots, n\}$, we denote by $\mu_n^S$ the element in $\cAs$ obtained as the $(n-1)$-iterated composition of the generator $\as$ composed with the element $\curvA$ placed in positions indexed by the elements in $S$. (The choices involved in this composition do not matter because of the associativity relation.)

\begin{proof}
The only thing to prove is that the bracket with the curvature is always zero (since the predifferential of $\cAs$ is zero). 
Using the Koszul sign rule, we get that $[\theta,\, \curvA] = 0$. Assume that $S = \{s_1 < \cdots < s_k\}$ and $j \in \{1, \ldots, n-k\}$, we denote by $S_j$ the set $\{s_1 < \cdots < s_{l} < l+j < s_{l+1}+1 < \cdots < s_k+1\}$ such that $l$ is minimal with this property and $S_j'$ the set $\{s_1 < \cdots < s_{l-1} < l+j < s_{l}+1 < \cdots < s_k+1\}$ such that $l$ is maximal with this property. 
A quick calculation (using associativity) shows
\[ [\theta,\, \mu_n^S] = \thetamun - \sum_j \muntheta = \sum_j \munpucurvA - \sum_j \munpucurvAp = 0. \]
\end{proof}

\begin{lemma}
A $\cAs$-algebra on a complete gr-dg module $A$ is the same data as a curved complete associative algebra $(A,\, \mu,\, d_A,\, -\theta)$ with curvature $\theta \in F_1 A$.
\end{lemma}

\begin{proof}
A map of curved operad $\cAs \to \End_A$ is characterized by the image of the generators $\curvA$ and $\as$ which give respectively two maps $\theta : \ringK \to A$ and $\mu : A^{\otimes 2} \to A$. The relation defining $\cAs$ ensures that $\mu$ is associative. The fact that the curvature is sent to the curvature says that ${d_A}^2 = [-,\, \theta] = [-\theta,\, -]$.
\end{proof}

\begin{remark}
\begin{itemize}
\item
We have changed the convention ${d_A}^2 = [\theta,\, -]$ into ${d_A}^2 + [\theta,\, -] = 0$ in order to simplify the formula appearing in the definition of homotopy curved associative algebras.
\item The inclusion $\As \to \cAs$ is not a map of curved operads (since the curvature of $\As$ is zero) so it cannot be a quasi-isomorphism of curved complete operads.
\item The projection $\cAs \to \As$, sending the $0$-ary element to $0$, is a map of curved operads but the map $\Gr \cAs \to \Gr \As$ is not a quasi-isomorphism. 
Therefore, it is not a quasi-isomorphism of curved complete operads.
\end{itemize}
\end{remark}

\subsection{Homotopy curved associative algebras}

By forgetting the curvature of $\cAs$, we obtain a complete operad which is the coproduct $\As \ast \free(\curvA)$ of the complete operads $\free(\curvA)$ and $\As$. Let us explain how the coproduct $\ast$ is constructed.

\begin{defi}
Let $\Oo$ and $\Pp$ be two augmented (complete) operads. Then the \emph{coproduct} in the category of (complete) operads is defined to be
\[ \free\left(\overline{\Oo} \amalg \overline{\Pp}\right) / \left(R_\Oo \amalg R_\Pp \right), \]
where $R_\Oo$ and $R_\Pp$ are the relations in $\Oo$ and $\Pp$ respectively.
\end{defi}

\begin{prop}
If $\Oo$ and $\Pp$ are both quadratic augmented (complete) operads, then the coproduct $\Oo \ast \Pp$ is a quadratic augmented operad.
\end{prop}

\begin{proof}
For any two presented operads $\Oo = \free(E_1)/(R_1)$ and $\Pp = \free(E_2)/(R_2)$, the coproduct operad $\Oo \ast \Pp$ is naturally presented by $\free(E_1 \amalg E_2) / (R_1 \amalg R_2)$. If $(E_1,\, R_1)$ and $(E_2,\, R_2)$ are both quadratic presentations, then so is $(E_1 \amalg E_2,\, R_1 \amalg R_2)$.
\end{proof}

\begin{prop}
\label{prop: Koszul dual of a coproduct}
If $\Oo$ and $\Pp$ are both quadratic augmented (complete) operads, then the Koszul dual cooperad of the coproduct $\Oo \ast \Pp$ is the coproduct cooperad $\Ooa \oplus \Ppa$ (whose underlying $\sS$-module is $I \amalg \overline{\Ooa} \amalg \overline{\Ppa}$).
\end{prop}

\begin{proof}
The cooperad $\Ooa$ is the sub-cooperad of ${\cofree}(sE_1)$ which is universal among the sub-cooperads $\Cc$ of ${\cofree}(sE_1)$ such that the composite
\[ \Cc \rightarrowtail {\cofree}(sE_1) \twoheadrightarrow {\cofree}(sE_1)^{(2)} /s^2R_1 \]
is zero (see \cite[7.1.4]{LodayVallette}). The cooperad $\Ppa$ is defined similarly. Its follows that the map $\Ooa \oplus \Ppa \to {\cofree}(sE_1 \amalg sE_2)$ satisfies that the composite
\begin{multline*}
\Ooa \oplus \Ppa \rightarrowtail {\cofree}(sE_1 \amalg sE_2) \twoheadrightarrow {\cofree}(sE_1 \amalg sE_2)^{(2)} / (s^2R_1 \amalg s^2R_2) \cong\\
{\cofree}(sE_1)^{(2)} /s^2R_1 \amalg sE_1 \circ_{(1)} sE_2 \amalg sE_2 \circ_{(1)} sE_1 \amalg {\cofree}(sE_2)^{(2)} /s^2R_2
\end{multline*}
is zero. By the universal property of $(\Oo \ast \Pp)^{\antishriek}$, there exists a unique morphism of cooperads $\Ooa \oplus \Ppa \to (\Oo \ast \Pp)^{\antishriek}$ which makes the following diagram commutative:
\[
\begin{tikzcd}
\Ooa \oplus \Ppa
\dar\ar[rr, >->] && \cofree(sE_1 \amalg sE_2).
\\
(\Oo \ast \Pp)^{\antishriek} \ar[urr, >->] &&
\end{tikzcd} 
\]
Moreover, using the fact that $(\Oo \ast \Pp)^{\antishriek}$ is a sub-cooperad of the cofree cooperad ${\cofree}(sE_1 \amalg sE_2)$ and that its projection onto $sE_1 \circ_{(1)} sE_2 \amalg sE_2 \circ_{(1)} sE_1$ is zero, we get that the map $(\Oo \ast \Pp)^{\antishriek} \rightarrowtail {\cofree}(sE_1 \amalg sE_2)$ factors through ${\cofree}(sE_1) \oplus {\cofree}(sE2)$. The composite $(\Oo \ast \Pp)^{\antishriek} \rightarrowtail {\cofree}(sE_1) \oplus {\cofree}(sE2) \twoheadrightarrow {\cofree}(sE_1)$ is a cooperad morphism, hence its image $pr_1((\Oo \ast \Pp)^{\antishriek})$ is a sub-cooperad of ${\cofree}(sE_1)$. Using the map $\Ooa \oplus \Ppa \to (\Oo \ast \Pp)^{\antishriek}$, we obtain that it contains $\Ooa$. Moreover, the definition of $(\Oo \ast \Pp)^{\antishriek}$ ensures that the composite
\[ pr_1((\Oo \ast \Pp)^{\antishriek}) \rightarrowtail {\cofree}(sE_1) \twoheadrightarrow \cofree(sE_1)^{(2)} /s^2R_1 \]
is zero. The cooperad $pr_1((\Oo \ast \Pp)^{\antishriek})$ is therefore a sub-cooperad of $\Ooa$ (by the universal property satisfied by $\Ooa$). Eventually, $pr_1((\Oo \ast \Pp)^{\antishriek}) = \Ooa$. Similarly, $pr_2((\Oo \ast \Pp)^{\antishriek}) = \Ppa$. Using the fact that $\amalg$ is a biproduct and that $\Ooa$ and $\Ppa$ are counital, it follows that there exists a unique morphism $(\Oo \ast \Pp)^{\antishriek} \to \Ooa \oplus \Ppa$ which commutes with the projections. This map is injective since $(\Oo \ast \Pp)^{\antishriek} \rightarrowtail {\cofree}(sE_1) \oplus {\cofree}(sE2)$ is. The composite $\Ooa \oplus \Ppa \rightarrowtail (\Oo \ast \Pp)^{\antishriek} \rightarrowtail \Ooa \oplus \Ppa$ is the identity and this proves the proposition.
\end{proof}

We are now able to compute the Koszul dual cooperad of the curved operad $\cAs$. We denote the element $s\curvA$ by $\mu_0^c$ and the set of generators of $\Asa$ presented in \cite[9.1.5]{LodayVallette} by $\{ \mu_n^c\}_{n \geq 1}$.

\begin{thm}
The homogeneous quadratic curved complete operad $\cAs$ is Koszul. 
The Koszul dual cooperad $\cAs^{\antishriek}$ is equal to
\[ \cAs^{\antishriek} = \left(\left\{ \hat \mu_n^c \coloneqq \sum_{\substack{k \geq 0\\ S \subseteq [n+k], |S| = k}} (-1)^{s_1 + \dots + s_k-k(n+k)}\mu_{n+k}^{c, S} \right\}_{n \geq 0},\, \Delta,\, 0\right), \]
where $S = \{ s_j \}_{j =1}^k$, $\mu_{n+k}^{c, S}$ is the element $\mu_{n+k}^c$ on which we have grafted the element $s\curvA$ in the positions given by the set $S$, and $\Delta$ is the cofree decomposition. The terms $\hat \mu_n^c$ have degree $n-1$. 
Explicitly, for any $n \geq 0$, we have
\begin{align}
\label{eq: decomposition map cAs}
\Delta (\hat \mu_n^c) = \sum_{\substack{k \geq 0\\ i_1 + \cdots + i_k = n}} (-1)^{\sum (i_j-1)(k-j)} (\hat\mu_k^c ;\, \hat\mu_{i_1}^c ,\, \dots ,\, \hat\mu_{i_k}^c).
\end{align}
It is isomorphic to
\[ \cAs^{\antishriek} \cong (s\curvA \amalg {\As^{\antishriek}},\, \Delta^{\theta},\, 0), \]
where for any $n \geq 0$,
\[ \Delta^{\theta} (\mu_n^c) \coloneqq \sum_{\substack{k \geq 0\\ i_1 + \cdots + i_k = n}} (-1)^{\sum (i_j-1)(k-j)} (\mu_k^c ;\, \mu_{i_1}^c ,\, \cdots ,\, \mu_{i_k}^c). \]
Moreover, the natural map
\[ f_\kappa : \cAi \coloneqq \hOm \cAsa \to \cAs \]
is an $\sS$-cofibrant resolution (that is cofibrant in the model category structure on complete gr-dg $\sS$-modules given in Section \ref{section: model structure on complete curved operads}).
\end{thm}

\begin{proof}
To prove that the complete operad $\cAs$ is homogeneous quadratic in the sens of Definition \ref{defi: homogeneous quadratic curved operad}, the only non trivial condition is the last one, that is, noting $E = \curvA \amalg \as$ and $R = \ass$, the counit $\cofree \left(sE\right) \to I$ doesn't factor through the coideal quotient $(S)$, where $S \coloneqq \left(I \amalg \cofree(sE)^{(2)}\right) / \left(s^2 R \amalg (1 + s^2\tilde\theta(1))\right)$. 
It will follow from the following computations. 

We first prove Formula (\ref{eq: decomposition map cAs}). We recall the decomposition map on $\Asa$ from \cite[Lemma 9.1.2]{LodayVallette}
\[ \Delta(\mu_n^c) = \sum_{i_1 + \cdots + i_k = n} (-1)^{\sum (i_j-1)(k-j)} (\mu_k^c ;\, \mu_{i_1}^c ,\, \cdots ,\, \mu_{i_k}^c). \]
(Replacing $i_j+1$ as it is written in \emph{op. cit.} by $i_j-1$ does not change the sign but seems more natural with the computations to come.) 
Then, being careful with the Koszul sign rule, since $\mu_{n}^{c, S}$ is the element $\mu_n^c$ on which we have grafted the element $s\curvA$ in the position given by the set $S$, we get
\begin{multline}
\label{eq: coproduct of muS}
\Delta \left(\mu_{n}^{c, S}\right) =\\
\sum_{\substack{l_0 + i_1 + \dots + i_k = n\\ S = S_0 \sqcup S_1 \sqcup \dots \sqcup S_k}} (-1)^{\sum (i_j-1)(k+l_0-(j+n_j)) + L_0 + K} \left(\mu_{k+l_0}^{c, \tilde{S}_0} ;\, \mu_{i_1}^{c, \tilde{S}_1} ,\, \cdots ,\, \mu_{i_k}^{c, \tilde{S}_k} \right),
\end{multline}
where $\tilde{S}_0 = \tilde{S}_0^0 \sqcup \dots \sqcup \tilde{S}_0^m$ and for all $j$ between 1 and $m$, $\max \tilde{S}_0^{j-1} < \min \tilde{S}_0^j$ and
\begin{align*}
S & = \tilde{S}_0^0 \sqcup (\tilde{S}_1+|\tilde{S}_0^0|) \sqcup (\tilde{S}_0^1+i_1-1) \sqcup \dots \sqcup\\
& \qquad \left(\tilde{S}_m +|\tilde{S}_0^0| + \sum_{j =1}^{k-1}(|\tilde{S}_0^j| + i_j)\right) \sqcup
\left(\tilde{S}_0^k + \sum_{j=1}^k (i_j -1)\right)\\
& = S_0 \sqcup \dots \sqcup S_k
\end{align*}
(where $S_j$ corresponds to the elements of $\tilde{S}_j$ reindexed)
and
\[\left\{\begin{array}{ccl}
|S_j| & = & l_j,\\
n_j & = & |\{ s \in S_0;\, s < \max \{S_j\}\}|,\\
L_0 & = & \sum_{s \in S_0} \sum_{\{t;\, \max \{S_t\} < s\}} l_t, \text{ and}\\
K & = & \sum l_t \sum_{j > t} (i_j -1).
\end{array}\right.\]
Some terms are missing in the sum appearing to compute the sign in Formula (\ref{eq: coproduct of muS}): they correspond to the entries labelled by $s \in S_0$ and are equal to $(1-1)(k+l_0-*)$, where $*$ is the position of $s$ in $\mu_{k+l_0}^c$; therefore they have no effect. The sum $L_0$ is the Koszul sign due to the elements $s\curvA$ corresponding to the set $S_0$ passing through the elements $s\curvA$ corresponding to some set $S_j$. The sum $K$ is the Koszul sign due to the elements $s\curvA$ of the $S_j$ passing through some $\mu_{i_t}^c$. Moreover, we compute the sign in front of $\left(\mu_{k+l_0}^{c, S_0} ;\, \mu_{i_1}^{c, S_1} ,\, \cdots ,\, \mu_{i_k}^{c, S_k} \right)$ to the right of Formula (\ref{eq: decomposition map cAs}). There is a shift sending the entries of $\mu_{i_j}^c$ appearing in $\Delta(\mu_n^c)$ to $\{ 1,\, \dots,\, i_{j}\}$. We denote by $\tilde{S}_j = \{ \tilde{s}_j\}$ the set obtained by shifting this way the elements in $S_j$. We obtain:
\[ (-1)^{\sum (i_j-l_j-1)(k-j)} (-1)^{\sum_{\tilde{S}_0} \tilde{s}_0 - l_0(k+l_0) + \dots + \sum_{\tilde{S}_k} \tilde{s}_k - l_k i_k}. \]
We should therefore compare this sign with
\[ (-1)^{\sum_S s - |S| n} (-1)^{\sum (i_j-1)(k+l_0-j-n_j) + L_0 + \sum l_t \sum_{j > t} (i_j -1)}. \]
Since $|S| = \sum l_j$ and $n = l_0 + \sum i_j$, it is enough to compare
\[ - \sum l_j (k-j) + \sum_j \sum_{\tilde{S}_j} \tilde{s}_j - l_0(k+l_0) - \sum l_j i_j \]
with
\[ \sum_S s - \left(\sum l_j \right) \left(l_0 + \sum i_j \right) + \sum (i_j - 1)(l_0 - n_j) + L_0 + \sum l_t \sum_{j > t} (i_j -1). \]
We have
\begin{multline*}
-\sum l_t i_j + \sum l_t \sum_{j > t} (i_j -1) =\\
= -l_0 \sum i_j - \sum_{t \geq 1} l_t \sum_{j \neq t} i_j - \sum l_j i_j + l_0 \sum_{j > 0} (i_j -1) + \sum_{t \geq 1} l_t \sum_{j > t} i_j - \sum_{t \geq 1} l_t (k-t)\\
= - \sum_{t \geq 1} l_t \sum_{j < t} i_j - \sum l_j i_j - l_0 k - \sum_{t \geq 1} l_t (k-t).
\end{multline*}
It follows that we are interested in the difference
\[ D \coloneqq \sum_S s - \left(\sum_{j > 0} l_j \right) l_0 - \sum_{t \geq 1} l_t \sum_{j < t} i_j + \sum (i_j-1)(l_0-n_j) + L_0 - \sum_j \sum_{\tilde{S}_j} \tilde{s}_j. \]
A combinatorial check shows that
\begin{multline*}
\sum_S s - \sum_j \sum_{\tilde{S}_j} \tilde{s}_j = \\
\sum_{s \in S_0} \sum_{\{t;\, \max \{S_t\} < s\}} (i_t - 1) + \sum_{t \geq 1} l_t \left( \sum_{j < t} i_j + |\{ s \in S_0 ;\, s < \max \{S_t \}\}|\right),
\end{multline*}
where $\sum_{t \geq 1} l_t |\{ s \in S_0 ;\, s < \max \{S_t \}\}| = \sum_{s \in S_0} \sum_{\{t;\, s < \max \{S_t\}\}} l_t$. Moreover,
\[ \sum_{s \in S_0} \sum_{\{t;\, s < \max \{S_t\}\}} l_t + L_0 = \sum_{s \in S_0} \sum_{t > 0} l_t = l_0 \sum_{t > 0} l_t. \]
Then
\begin{align*}
D & = \sum_{s \in S_0} \sum_{\{t;\, \max \{S_t\} < s\}} (i_t - 1) + \sum (i_j-1)(l_0-n_j)\\
& = \sum_{s \in S_0} \sum_{\{t;\, \max \{S_t\} < s\}} (i_t - 1) + \sum (i_j-1)|\{ s \in S_0;\, s > \max \{S_j\}\}|\\
& = 2 \sum_{s \in S_0} \sum_{\{t;\, \max \{S_t\} < s\}} (i_t - 1).
\end{align*}
We conclude that Formula (\ref{eq: decomposition map cAs}) is true.

As an intermediate consequence, we get that $\hat \mu_1^c \in \cofree (sE)$ is sent to 0 in $(S)$ so the counit of $\cofree (sE)$ cannot factor through the coideal quotient $(S)$ (using that $s\curvA$ is in $F_1 sE$).

It follows from the above computations that the $\hat \mu_n^c$, $n \geq 0$, span a sub-cooperad $\Cc$ of ${\cofree}(s\curvA \amalg s\as)$. Moreover, an explicitation of the first terms of $\hat \mu_0^c$, $\hat\mu_1^c$, $\hat \mu_2^c$ and $\hat \mu_3^c$ shows that the composite
\[ \Cc \rightarrowtail {\cofree}(s\curvA \amalg s\as) \twoheadrightarrow \left(I\amalg \cofree(s\curvA \amalg s\as)^{(2)} \right) / \left( s^2 \left(\ass \right) \amalg (1 + s^2 \tilde \theta(1))\right) \]
is zero. By the universal property of the cooperad $\cAsa$, it follows that there exists a unique (injective) morphism $\Cc \rightarrowtail \cAsa$ such that the following diagram commutes:
\[
\begin{tikzcd}
\Cc
\dar\ar[r, >->] & {\cofree}(sE),
\\
\cAsa \ar[ur, >->] &
\end{tikzcd} 
\]
where $E = \curvA \amalg \as$. 

The operad $\cAs$ is the completion of the operad $\Gr \cAs$. The operad $\Gr \cAs$ admits a presentation in filtered modules similar to the one defining $\cAs$ and it is in fact the coproduct $\As \ast \free(\curvA)$. By Proposition \ref{prop: Koszul dual of a coproduct}, we get
\[ (\Gr \cAs)^\antishriek \cong \cofree(s\curvA) \oplus \Asa \cong s\curvA \amalg \Asa. \]
A proof similar to the proof of Proposition 6.1.7 in \cite{HirshMilles} shows that the quadratic operad $\Gr \cAs$ is a Koszul operad, so is $\cAs$. 
By the Poincaré--Birkhoff--Witt theorem of Theorem \ref{thm: qi for Koszul operad and PBW}, we have $\cAsa \cong \widehat{\left( \Gr \cAs\right)^\antishriek}$. It suffices now to check the dimension in each arity to get $\cAsa = \{ \hat \mu_n^c \}_{n \geq 0}$. Going through the Poincaré--Birkhoff--Witt isomorphism, we obtain the decomposition map $\Delta^\theta$. 

In order to have an $\sS$-cofibrant resolution, it is enough to remark that the curvature is non-zero, that for the weight filtration given by its generators, $\Gr \cAs$ is connected bounded below weight graded and that the filtration on $\cAs$ is induced by the gradation by the number of $\curvA$. We can therefore apply Theorem \ref{thm: Koszul resolution}.
\end{proof}

We can also compute the Koszul dual constant-quadratic operad (see Definition \ref{defi: Koszul dual operad}) associated with $\cAs$.

\begin{prop}
The Koszul dual constant-quadratic operad associated with $\cAs$ is given by
\[ \cAs^! \cong \free\left( \vcenter{\xymatrix@M=0pt@R=4pt@C=4pt{
           \circ \\
           \ar@{-}[u]}} 
           \amalg \as\right)/\left(\ass,\ \vcenter{ \xymatrix@M=0pt@R=4pt@C=4pt{
           & \circ & &\\
           & \ar@{-}[u] \ar@{-}[dr] & & \ar@{-}[dl] \\
           & & *{} \ar@{-}[d] &\\
           & & }} - |, \, \vcenter{
         \xymatrix@M=0pt@R=4pt@C=4pt{
           & & & \circ \\
           & \ar@{-}[dr] & & \ar@{-}[u] \ar@{-}[dl]\\
           & & *{} \ar@{-}[d] & \\
           & & }} - |\right), \]
where $\vcenter{\xymatrix@M=0pt@R=4pt@C=4pt{ \circ \\ \ar@{-}[u]}}$ is concentrated in homological degree $0$. 
Algebras over $\cAs^!$ are unital associative algebras.
\end{prop}

\begin{proof}
The computation of the Koszul dual operad $\As^!$ is given in Section 7.6.4 in \cite{LodayVallette}. The computation of the extra terms is similar.
\end{proof}

\begin{remark}
We therefore recover in an operadic context the curved Koszul duality theory of the operad encoding unital associative algebras presented in \cite{HirshMilles}. 
\end{remark}

\begin{prop}
A $\cAi$-algebra is equivalent to a complete graded vector space $(A,\, F_\bullet)$ equipped with an operation $m_0 : \ringK \to F_1A$ of degree $-2$ and for all $n \geq 1$ with $n$-ary filtered operations
\[ m_n : A^{\otimes n} \to A \text{ of degree } n-2,\]
which satisfies the following relations
\begin{equation}
\label{eq: relation cAi-algebras}
\sum_{p+q+r = n} (-1)^{p+qr}m_{n+1-q} \circ_{p+1} m_q = 0, \quad n\geq 0,
\end{equation}
where $k = p+1+r$.

This notion of algebras coincides, under different settings (not necessarily filtered), with the notions of $\Ai$-algebras given in \cite{GJ90}, of curved $\Ai$-algebras defined in \cite{CD01}, of weak $\Ai$-algebras given in \cite{bK06}, of filtered $\Ai$-algebras used in \cite{Fukaya}, of $[0,\, \infty[$-algebras studied in \cite{Nicolas} and of weakly curved $\Ai$-algebras studied in \cite{lP12} to give only one reference for each name.
\end{prop}

\begin{proof}
A morphism of curved operad $\gamma_A : \hOm \cAsa \to \End_{(A,\, d_A)}$ is characterized by a map of complete $\sS$-modules $\overline{\cAsa} \to \End_A$ and the predifferential $d_A$. Therefore it corresponds to a collection of filtered applications
\[ m_n : A^{\otimes n} \to A, \text{ of degree } n-2 \text{ for all } n,\]
such that $m_0 : \ringK \to F_1 A$ and $m_1 = d_A$. 
The fact that $\gamma_A$ commutes with the predifferentials and the fact that $\gamma_A$ sends the curvature to the curvature ensure that Equation (\ref{eq: relation cAi-algebras}) is satisfied for $n$.
\end{proof}

\begin{remark} 
\begin{enumerate}
\item
By the fact that complete curved associative algebras are examples of complete curved operads, Sections \ref{section: bar cobar constructions} and \ref{section: koszul duality for curved operads} apply to complete curved associative algebras and we therefore obtain functorial resolutions of complete curved associative algebras by the bar-cobar resolution and Koszul resolutions for Koszul complete curved associative algebras (cofibrant in the underlying category of complete gr-dg $\ringK$-modules). These results can also be seen as an example of an extension of the results in \cite{Milles2} to the curved / altipotent setting: the operadic Koszul morphism $\kappa : \cAsa \to \cAs$ allows us to define a bar-cobar adjunction on the level of curved associative algebras and altipotent coassociative coalgebras.
\item
The same yoga would work for curved Lie algebras since we can define the \emph{curved Lie} operad $\cLie$ as we did the curved associative operad $\cAs$. We obtain the curved operad $\cLi$ whose algebras coincide with curved $\Li$-algebras defined in \cite{bZ93} and used or studied for example in \cite{kC11b, mM12, RZ12, LS12}. It is however impossible to associate in the same way a curved operad with the operad $\mathrm{Com}$.
\end{enumerate}
\end{remark}

\subsection{Homotopy categories}

We now apply the results of Sections \ref{section: model cat for complete curved algfebras} and \ref{section: homotopy category of algebras over a curved operad} to the morphism
\[ f_\kappa : \cAi \to \cAs. \]
By Theorem \ref{thm: model cat struct for curved algebras}, the categories of algebras $\Alg(\cAs)$ and $\Alg(\cAi)$ each admit a cofibrantly generated model structure and by Theorem \ref{thm: inverse and direct image functor are adjoint derived functors}, the morphism $f_\kappa$ produces a Quillen adjunction between the model categories of curved algebras.

\begin{thm}
The functors ${f_\kappa}^*$ and ${f_\kappa}_*$
\[
\xymatrix{\lL {f_\kappa}^* : \Hoalg(\cAi) \ar@<.5ex>@^{->}[r] & \Hoalg(\cAs) : \rR {f_\kappa}_* = {f_\kappa}_* \ar@<.5ex>@^{->}[l]}
\]
are equivalences of the homotopy categories.
\end{thm}

\begin{proof}
To apply Theorem \ref{thm: equivalence of the homotopy categories}, it is enough to show that the complete curved operad $\cAs$ and $\cAi$ are $\sS$-split and that the morphism $f_\kappa$ is compatible with the chosen splittings. The two operads come from non symmetric operads. In this situation, we can consider the splitting
\[ \Oo(n) = \Oo_{ns}(n) \otimes \ringK[\sS_n] \to (\Oo_{ns}(n) \otimes \ringK[\sS_n])_{ns} \otimes \ringK[\sS_n] \]
induced by the map $ \Oo_{ns}(n) \to (\Oo_{ns}(n) \otimes \ringK[\sS_n])_{ns}$. The map $f_\kappa$ is compatible with these splittings.
\end{proof}

\begin{remark}
There are two possible notions of morphisms for $\cAi$-morphisms:
\begin{itemize}
\item
the classical morphisms of algebras of the curved operad $\cAi$, that is to say maps $f : A \to B$ compatible with the predifferential and the algebra structures;
\item
the notion, sometimes called $\infty$-$\cAi$-morphisms, of morphisms of coalgebras (cooperads) between an extension to homotopy curved algebras (or operads) of the bar constructions on $A$ and on $B$. (See for example Section 7.1 in \cite{lP12}.)
\end{itemize}
In the previous theorem, we consider the first notion of morphism. The study of $\infty$-morphisms requires extra constructions that might appear in other work.
\end{remark}

\appendix

\section{Categorical filtrations and completions}
\label{appendix: categorical stuff}

In this appendix, we will establish definitions of filtered and complete filtered objects for use in both our ground categories and in the categories of operads over them.

We want to establish the ability to work concretely with filtered and complete filtered objects, which we do by realizing them as reflective subcategories of diagram categories that are easier to manipulate. 
The strategy of studying a category by embedding it as a reflective subcategory of a better-behaved category is quite old~\cite{GabrielZisman:CFHT}. In fact the cases of filtered and complete filtered objects constitute two of the original motivating examples. Consequently, much of the content of this appendix may be well-known to experts.

The main technical point for us will be the transfer of a closed monoidal structure to a reflective subcategory in a coherent fashion, using a criterion of Day~\cite{DAY19721}.

The executive summary of (most of) this appendix is that the following diagram of subcategory inclusions is in fact a diagram of normal reflective embeddings.
This means that:
\begin{enumerate}
\item each of the solid arrows has a left adjoint reflector functor,
\item the counit of each adjunction is an isomorphism,
\item each category is closed symmetric monoidal,
\item the given arrows are all lax symmetric monoidal with the unit of the adjunction as the monoidal natural transformation, and
\item the left adjoints are all strong symmetric monoidal.
\end{enumerate}
The diagram is as follows.
\[
\begin{tikzcd}
\mathcal{A}^{\ob \mathbb{N}}\rar
&\mathcal{A}^{\mathbb{N}^{\op}}
&\Filt(\mathcal{A})\lar
&\Comp(\mathcal{A}).\lar
\end{tikzcd}
\]
Moreover, in the case where $\Aa$ is replaced by the category $\pgA$ of predifferential graded objects in $\Aa$, we can complete the diagram
\[
\begin{tikzcd}
\pgA^{\ob \mathbb{N}}\rar
&\pgA^{\mathbb{N}^{\op}}
&\Filt(\pgA)\lar
&\Comp(\pgA)\lar
& \compa(\Aa).\lar
\end{tikzcd}
\]

Throughout this appendix, we let $(\mathcal{A},\otimes,\mathbf{1},\ul{\mathcal{A}})$ be a closed symmetric monoidal Grothendieck category with initial object $\varnothing$.

We consider $\mathbb{N}$ as a poset category with initial object $0$. When desirable, we enrich $\mathbb{N}$ in $\mathcal{A}$ by saying $\ul{\mathbb{N}}(a,b)$ is $\varnothing$ if $a>b$ and $\mathbf{1}$ otherwise.

In general the underlying category of this enriched category is not equivalent to $\mathbb{N}$ itself (they are equivalent if and only if $\mathbf{1}$ admits no automorphisms and no maps to $\varnothing$).

Because of the diagrams of reflexive subcategories, the condition that $\Aa$ is a Grothendieck category ensures in fact that the categories $\Filt(\Aa)$, $\Comp(\Aa)$ and $\compa(\Aa)$ are quasi-abelian (see for example \cite[Lemma 3.3.5]{eR16} and \cite[Proposition 4.20]{FS10}). This property is useful for Section \ref{section: koszul duality for curved operads}.

Various generalizations are possible: for many of the results we can replace $\mathcal{A}$ with a regular cocomplete closed symmetric monoidal category and $\mathbb{N}$ with a nontrivial poset with an initial object.

\subsection{Categorical filtered objects}

\begin{defi}
Let $\mathcal{C}$ be a category. 
A \emph{$\mathbb{N}$-filtered object} in a category $\mathcal{C}$ is an $\mathbb{N}^{\op}$-indexed diagram in $\mathcal{C}$ where all maps are monomorphisms. 
The full subcategory spanned by $\mathbb{N}$-filtered objects is denoted $\Filt(\mathcal{C})$.
\end{defi}

\begin{lemma}
\label{lemma: cat of filtered objects is reflective}
The full subcategory spanned by $\mathbb{N}$-filtered objects in $\mathcal{A}$ inside $\mathcal{A}^{\mathbb{N}^{\op}}$ is a reflective subcategory.
Moreover, the reflector takes the functor $F:\mathbb{N}^{\op}\to\mathcal{A}$ to the functor $rF$ whose value at $e$ is the image of $F(e)\to F(0)$.
\end{lemma}

\begin{proof}
For a diagram $F:\mathbb{N}^{\op}\to\mathcal{A}$ there is a map $\psi_e:F(e)\to F(0)$ for each object $e$ of $\mathbb{N}$.
Define $rF(e)$ as $\image(\psi_e)$.
By the universal property of the image in a regular category, the morphisms in the diagram $rF$ are well-defined monomorphisms.

To see the adjunction, let $F$ be a diagram and $G$ a filtered object.
The data of a $\mathbb{N}$-filtered object map from $rF$ to $G$ consists of a map $f_e$ from the image of $\psi_e$ to $G(e)$ for each $e$.
Such data defines a map $\tilde{f}_e:F(e)\to \image(\psi_e)\to G(e)$.
On the other hand, given a map $\tilde{f}_e:F(e)\to G(e)$ compatible with $F(0)\to G(0)$, there is a unique lift through the image as in the following diagram:
\[
\begin{tikzcd}
F(e)
\dar\ar[rr]&& G(e)\dar
\\
rF(e)\ar[urr,dotted]\rar& F(0)\rar&G(0).
\end{tikzcd} 
 \] 
The coherence conditions for two different values of $e$ to be an $\mathbb{N}$-filtered object map or a map of diagrams coincide, in that the commutativity of the outside cell implies that the right square commutes if and only if the left square commutes. 
\[
\begin{tikzcd}
F(e)\ar[rr, bend left]\rar[swap]{\tilde{f}_e}\dar & G(e)\dar& rF(e)\dar\lar{f_e}\\
F(e')\ar[rr,bend right]\rar{\tilde{f}_{e'}}& G(e')& rF(e')\lar[swap]{f_{e'}}
\end{tikzcd}
\]
Then this objectwise natural bijection restricts to a natural bijection between filtered maps from $rF$ to $G$ and maps of diagrams from $F$ to $iG$.
\end{proof}

\subsection{Categorical complete filtered objects}
\label{subsection: complete condition}

In this section, we use the fact that $\mathcal{A}$ is an Abelian category, that it admits $\mathbb{N}^{\op}$-indexed limits and that its satisfies AB4. 
A \emph{filtered object} means a $\mathbb{N}$-filtered object in $\mathcal{A}$.
For a filtered object $V$ we write $i_e$ (or $i^V_e$ if $V$ is ambiguous) for the monomorphism $F_e V\to F_0 V$.

\begin{defi}
A filtered object is \emph{complete} if the natural map $F_0V\to \lim_a \coker i_a$ is an isomorphism.

The \emph{completion} of the filtered object $V$ is the sequence $\hat{V}$ (also denoted $V^\wedge$) defined as
\[F_e\hat{V} \coloneqq \ker (\lim_a \coker i_a \xrightarrow{\pi_e} \coker i_e)\]
with the natural maps induced by the identity on $\lim_a \coker(i_a)$.
\end{defi}

\begin{lemma}
\label{alternate completion stages lemma}
There is a natural isomorphism from $\lim_{a\ge e} \coker (F_aV\to F_eV)$ to $F_e\hat{V}$.
\end{lemma}

\begin{proof}
The natural map is constructed as follows.
The map $F_eV\to \coker( F_aV\to F_eV)$ is epic and the composition $F_eV\to \coker(F_aV\to F_eV)\to\coker i_e$ is equal to the composition $F_eV\to F_0V\to \coker i_e=0$.
So, passing to the limit, $F_eV\to \lim \coker(F_aV\to F_eV)$ factors through the kernel of $\lim_a \coker i_a\to \coker i_e$.

To see that this is an isomorphism, note that for $a\ge e$, the sequence
\[
0\to \coker (F_aV\to F_eV)\to \coker i_a \to \coker i_e \to 0
\]
is exact by the third isomorphism theorem which exists because the category satisfies AB4.
Then taking limits over $a$, we get an exact sequence
\[
0\to \lim (\coker F_aV\to F_eV)\to F_0\hat{V} \to \coker i_e
\]
since the limit is left exact.
But $F_0\hat{V}\to \coker i_e$ is epic since the epimorphism $F_0V\to \coker i_e$ factors through $F_0\hat{V}$, so the sequence is in fact short exact, which is the desired statement.
\end{proof}

\begin{lemma}
\label{lemma:stuff about completion}
Completion is a functor from filtered objects to the full subcategory of complete filtered objects. 
\end{lemma}

\begin{proof}
If completion is viewed as an assignment with codomain sequences, then functoriality is clear.
It should be verified that the codomain can be restricted, first to filtered objects and then to complete filtered objects.

So first we verify that the induced maps $F_{e'}\hat{V}\to F_e\hat{V}$ are monomorphisms for $e'>e$.

By inspection $F_0\hat{V} \cong \lim_a \coker i_a$ and so the map $F_e\hat{V}\to F_0\hat{V}$ is the inclusion of a kernel and thus a monomorphism.
Since the map $F_{e'}\hat{V}\to F_0\hat{V}$ can be factored $F_{e'}\hat{V}\to F_e\hat{V}\to F_0\hat{V}$, the left map $F_{e'}\hat{V}\to F_e\hat{V}$ is also a monomorphism.

To see that completions are complete, by the construction of $\hat{V}$ we have coherent isomorphisms $\coker i^{\hat{V}}_e \cong \coker i^V_e$ so the composition of the natural map with the two natural isomorphisms
\[F_0\hat{V}\to \lim_a \coker i^{\hat{V}}_a\cong \lim_a \coker F_aV \cong F_0\hat{V}\]
is the identity, which implies that the natural map is an isomorphism.
\end{proof}

\begin{lemma}
There is a natural filtered map $r$ (or $r_V$) from a filtered object $V$ to its completion $\hat{V}$.
This natural map is an isomorphism on complete objects.
\end{lemma}
\begin{proof}
We have already used a similar natural map (let us call it $r_0$) from $F_0V$ to $F_0\hat{V}\cong \lim_a\coker i_a$.
The composition
\[
F_eV\xrightarrow{i_e} F_0V\xrightarrow{r_0} F_0\hat{V}\xrightarrow{\pi_e} \coker i_e
\]
is zero so that there is a unique induced map $r_e$ from $F_eV$ to $F_e\hat{V}=\ker(\pi_e)$ as in the following diagram:

\[\begin{tikzcd}
F_eV\rar{i_e}\dar\ar[dd, dotted, bend right=40pt,swap,"r_e"] &
F_0V\dar\ar[dd,dashed, bend left=40pt,"r_0"]\\
0\rar\dar&\coker i_e\\
F_e\hat{V}=\ker \pi_e\uar\rar&F_0\hat{V}\uar{\pi_e}.
\end{tikzcd}\]
The collection $r_e$ is coherent so constitute the data of morphism of filtered objects.

If $V$ is complete then the natural map $F_0V\to F_0\hat{V}$ is an isomorphism, so that $F_e\hat{V}\to F_0\hat{V}\to F_0V$ is naturally isomorphic to $\ker \coker i_e$ which in turn is naturally isomorphic to the monomorphism $i_e$.
The identity isomorphism $F_eV\to F_eV$ satisfies the universal property of $r_e$ under the identification of this natural isomorphism, so $r_e$ itself must be an isomorphism.
\end{proof}

\begin{lemma}
\label{Lemma: unit of the completion is completion of the unit}
For any filtered object $V$ the two morphisms $r^{\hat{V}}$ and $\widehat{r^V}$ from $\hat{V}$ to $\hat{\hat{V}}$ coincide.
\end{lemma}

The proof is straightforward; we record the details in order to have a complete\footnote{In more ways than one.} account. 
Because $F_eW\to F_0W$ is a monomorphism for a filtered complex $W$ it suffices to check the maps $F_0\hat{V}\to F_0\hat{\hat{V}}$ coincide.
\begin{proof}
These are two maps
\[
F_0\hat{V} \to \lim_b \coker\ker(F_0\hat{V}\xrightarrow{\pi_b} \coker i_b).
\]
To check if these maps agree it suffices to check on all projections over the limit.
The $e$ term of the limit is the coimage of $\pi_e$, and so it suffices to check equality after postcomposing the monomorphism $\coker\ker(\pi_e)\to \coker i_b$.

The $e$ component of $r^{\hat{V}}$ is just the left half of the coimage factorization $F_0\hat{V}\to \coker\ker (F_0\hat{V}\xrightarrow{\pi_e} \coker i_e)$ so postcomposing the given monomorphism gives $\pi_e$.

On the other hand, the $e$ component of $\widehat{r^V}$ is given by the projection $\pi_e$ from $F_0\hat{V}$ to $\coker i_e$ followed by the map $\coker i_e\to \coker\ker\pi_e$ induced by $r^V_0$.
Postcomposition then gives
\[
F_0\hat{V}\xrightarrow{\pi_e} \coker i_e \xrightarrow{\text{induced by }r^V_0}\coker\ker \pi_e\xrightarrow{\text{right half of }\pi_e\text{ coimage factorization}}\coker i_e.
\]
Then to show this to be equal to $\pi_e$ it suffices to check that the composition of the right two maps in the composition is the identity of $\coker i_e$.  Then this equality can be checked after precomposition with the epimorphism $F_0V\to \coker i_e$.
We conclude by the commutativity of the following diagram, which shows these two to be equal (unlabeled morphisms are induced by cokernels and coimage factorizations).
\[
\begin{tikzcd}
F_0V\ar[dr,swap,"r^V_0"]\ar[rr]\ar[rrdd, bend right=50pt,""]
&&\coker i_e\dar{\text{induced by }r^V_0} \\
&F_0\hat{V}\rar\ar[dr,swap,"\pi_e"]&\coker\ker \pi_e\dar\\
&&\coker i_e
\end{tikzcd}
\] 
\end{proof}

\begin{cor}
\label{cor: cat of complete objects is reflective}
The inclusion of complete filtered objects inside filtered objects is reflective with reflector $V\xrightarrow{r} \hat{V}$.
\end{cor}

\begin{proof}
We have already defined a unit in Lemma~\ref{lemma:stuff about completion}.
For $W$ a complete filtered object, the $W$ component of the counit is supposed to be a map of filtered objects from $\widehat{W}$ to $W$.
Since $W$ is complete, we may take the inverse of the unit, which is an isomorphism on the complete filtered object $W$.
So $\epsilon_W=r_{W}^{-1}$.

Then $W\xrightarrow{r^{W}} \hat{W}\xrightarrow{\epsilon_W} W$ is the identity on $W$ by definition and the composition $\hat{V}\xrightarrow{\widehat{r^V}} \hat{\hat{V}} \xrightarrow{\epsilon_{\hat{V}}} \hat{V}$ is the identity by Lemma~\ref{Lemma: unit of the completion is completion of the unit}.
\end{proof}

\subsection{Coproducts and products in complete filtered objects}

Next we spend a little time showing that in complete objects, the canonical map from the coproduct to the product is a monomorphism. This allows us to detect whether two maps into a coproduct are equal by means of the product projections. We now use that $\mathcal{A}$ moreover has all products and coproducts and is AB5.

\begin{remark}
\label{remark: biproducts}
Biproducts are inherited by a reflective subcategory. 
To see this, recall that a coproduct in the subcategory is calculated as the reflector applied to the coproduct of the same objects in the ambient category.
On the other hand, the product of the same objects in the reflective subcategory is calculated in the ambient category, where it coincides with a coproduct in the ambient category. 
This implies that the given coproduct already (essentially) lies in the reflective subcategory, so the reflector is an isomorphism on this coproduct.
\end{remark}

We recall the definition of strict morphism from \cite[1.1.5]{pD71}.
\begin{defi}
\label{defi: strict}
In a category with finite limits and colimits, we say that a morphism $f:Y\to X$ is \emph{strict} if the natural morphism $\coim(f) \to \im(f)$ is an isomorphism.

In the context of (possibly complete) filtered objects, a morphism $f:Y\to X$ is \emph{strict} if and only if it is strictly compatible with filtration, that is to say, for each $n$, the following is a pullback square:
\[
\begin{tikzcd}
f(F_nY)\rar\dar & F_nX\dar\\
f(F_0Y)\rar & F_0X.
\end{tikzcd}
\]

\end{defi}

\begin{lemma}
\label{lemma: strict monic induces quotient monic}
Suppose $Y\to X$ is a strict, monic morphism of filtered objects. 
Then the induced map $F_0Y/F_nY \to F_0X/F_nX$ is a monomorphism for all $n$.
\end{lemma}

\begin{proof}
Using the fact that $Y\to X$ is strict, a standard diagram chase in the (Abelian) ground category shows that the map of cokernels of the Cartesian square
\[F_nX/f(F_nY) \to F_0X/f(F_0Y)\]
is monic.

Now consider the snake lemma in the (Abelian) ground category for the diagram
\[
\begin{tikzcd}
& F_nY\dar\rar & F_0Y\dar\rar & F_0Y/F_nY\dar\rar&0
\\
0\rar& F_nX\rar & F_0X\rar & F_0X/F_nX.
\end{tikzcd}
\]
Since $Y\to X$ is monic, the kernel of $F_0Y\to F_0X$ is zero so the kernel of $F_0Y/F_nY\to F_0X/F_nX$ is in an exact sequence following a zero term and preceding a monomorphism.
\end{proof}

\begin{lemma}
\label{lemma: comparison in diagrams}
The natural map in diagrams from a coproduct to the product over the same indexing set is monic.
\end{lemma}

\begin{proof}
The map from each finite subcoproduct to the product admits a retraction so is monic.
By AB5, the map in the colimit from the total coproduct to the product is also monic.
\footnote{We learned this short proof from Zhen Lin Low.}
\end{proof}

\begin{lemma}
\label{lemma: strictness in filtered}
The natural map in filtered objects from a coproduct to the product over the same indexing set is strict.
\end{lemma}

\begin{proof}
Let $X^i$ be a family of filtered objects, and let $Y$ map to $\coprod F_0X^i$ and $\prod F_nX^i$ over $\prod F_0X^i$.
We construct a map from $Y$ to $\coprod F_nX^i$ over both of these.
Since the map from $\coprod F_nX^i$ to $\coprod F_0X^i$ this will show that this lift is unique, witnessing the desired universal property.

Since the ground category satisfies AB4, $\coprod F_nX^i$ is the kernel of the map $\coprod F_0X^i\xrightarrow{q} \coprod (F_0X^i/F_nX^i)$, so it suffices to show that the map from $Y$ to $\coprod F_nX^i$ vanishes under $q$. 
By Lemma~\ref{lemma: comparison in diagrams}, it suffices to show that this vanishing after composing with the monomorphism from $\coprod (F_0X^i/F_nX^i)$ to $\prod (F_0X^i/F_nX^i)$. 
This in turn can be checked by projection to each $F_0X^j/F_nX^j$ factor.
But these composites factor as 
\[Y\to \coprod F_0X^i\to \prod F_0X^i\to F_0X^j\to F_0X^j/F_nX^j\]
and because the map from $Y$ to $\prod F_0X^i$ factors through $\prod F_nX^i$, the entire composite is zero, as desired.
\end{proof}

\begin{lemma}
\label{lemma: comparison in complete}
The natural map in complete objects from a coproduct to the product over the same indexing set is monic.
\end{lemma}

\begin{proof}
Lemma~\ref{lemma: comparison in diagrams} shows this in diagrams. 
We will upgrade this first to the filtered and then to the complete context.

Both products and coproducts in the filtered context coincide with those in diagrams, the former without any hypotheses and the latter because the ground category satisfies AB4. 
Monomorphisms in filtered objects are created in diagrams (or by the terminal entry), so we have the same result there.

Now passing to the complete context, we need to verify that this monomorphism passes to the completion. 
By Lemma~\ref{lemma: strictness in filtered}, the monomorphism is also strict, so by Lemma~\ref{lemma: strict monic induces quotient monic}, the successive quotients 
\[
\left.\left(\coprod_i F_0X^i\right) \middle/ \left(\coprod_i F_nX^i\right)\right. \to \left.\left(\prod_i F_0X^i\right) \middle/ \left(\prod_i F_nX^i\right)\right.
\]
are monic. 
The limit functor is a right adjoint so left exact, so we get a monomorphism in filtration degree zero at the level of completions.
But this is enough to detect monomorphisms in complete objects.
\end{proof}

\subsection{Tensor products in the filtered and complete settings}\label{filtComplSetting}

Now we give the criterion of Day for transfer of a closed symmetric monoidal structure.
\begin{defi}
A class of objects $G$ in a category $\mathcal{C}$ is \emph{strongly generating} if a morphism $f:x\to y$ in $\mathcal{C}$ is an isomorphism whenever the induced maps of sets $\mathcal{C}(g,x)\to \mathcal{C}(g,y)$ is an isomorphism for all $g\in G$.
\end{defi}

\begin{thm}[Day]
\label{thm: Day reflective thm}
Let $\mathcal{D}$ be a closed symmetric monoidal category and let $i:\mathcal{C}\to \mathcal{D}$ be a reflective subcategory, with reflector $r$.
Let $G_{\mathcal{D}}$ be a strongly generating class of objects of $\mathcal{D}$.
The following are equivalent:
\begin{enumerate}
\item There exists:
\begin{enumerate}
\item A closed symmetric monoidal structure on $\mathcal{C}$,
\item A symmetric monoidal enrichment of the inclusion $i$ which commutes with the underlying set functor, and
\item A strong symmetric monoidal enrichment of the reflector $r$.
\end{enumerate}
\item (Day's condition (1)) For $d$ in ${\mathcal{D}}$ and $c$ in ${\mathcal{C}}$, the component of the unit of the reflection adjunction $\id_\mathcal{D}\to ir$ is an isomorphism for the hom object $\ul{\Dd}(d,ic)$.
\item (Day's condition (2), simplified) For $d$ in $G_{\mathcal{D}}$ and $c$ in $\mathcal{C}$, the component of the unit of the reflection adjunction $\id_\mathcal{D}\to ir$ is an isomorphism for the hom object $\ul{\Dd}(d,ic)$.
\end{enumerate}
\end{thm}

In the case that Day's conditions are satisfied, since $ri$ is naturally isomorphic to $\id_\mathcal{C}$, we have the following natural isomorphisms for $c$ and $c'$ objects in $\mathcal{C}$:
\begin{align*}
c\otimes_{\mathcal{C}} c' 
\cong 
(ric)\otimes_{\mathcal{C}} (ric')
\cong
r(ic\otimes_{\mathcal{D}}ic').
\end{align*}
Then for $c$ and $c'$ and $c''$ objects in $\mathcal{C}$, we have
\begin{align*}
\mathcal{C}(r(ic\otimes_{\mathcal{D}}ic'), c'')
&\cong
\mathcal{D}(ic\otimes_{\mathcal{D}}ic', ic'') && \text{by adjunction}
\\&\cong
\mathcal{D}(ic,\ul{\mathcal{D}}(ic',ic'')) && \text{by adjunction}
\\&\cong
\mathcal{D}(ic,ir\ul{\mathcal{D}}(ic',ic'')) &&\text{by Day's condition (1)}
\\&\cong
\mathcal{C}(c,r\ul{\mathcal{D}}(ic',ic'')) && \text{by full faithfulness of }i
\end{align*}
which shows that the internal hom $\ul{\mathcal{C}}(c',c'')$ is naturally isomorphic to $r\ul{\mathcal{D}}(ic',ic'')$.
In particular this implies that 
\[i\ul{\mathcal{C}}(c',c'')\cong ir\ul{\mathcal{D}}(ic',ic'')\cong \ul{\mathcal{D}}(ic',ic'')\] 
again by full faithfulness and Day's condition (1).

Thus the motto is ``internal homs are computed in the big category $\mathcal{D}$ but the monoidal product needs to be reflected''.

\subsection{Strong generation of sequences}

We will repeatedly use the same set of strong generators in what follows. 

\begin{defi}
Let $m$ be an object of $\mathcal{A}$ and $x$ an object of $\mathbb{N}$.
We define a $\mathcal{A}$-valued $\mathbb{N}$-presheaf $m_x$ as follows:
\[
F_e(m_x)
\coloneqq
\begin{cases}
m&e \le x\\
\varnothing &\text{otherwise}
\end{cases}
\]
with structure maps either the initial map or the identity.
\end{defi}

\begin{lemma}
\label{lemma: generators}
The class of presheaves $\{m_x\}$ as $m$ and $x$ vary over the objects of $\mathcal{A}$ and $\mathbb{N}$ is strongly generating.
\end{lemma}

\begin{proof}
A map of presheaves from $m_x$ to an arbitrary presheaf $X$ is determined uniquely by the level $x$ map from $m$ to $F_xX$.
Every such map in $\mathcal{A}$ actually determines a map of presheaves, using the initial map for indices greater than $x$ and using the commutativity of the following diagram for indices $e$ less than $x$:
\[
\begin{tikzcd}
m\ar[r]\ar[d,"\id"]
& F_xX\dar\\
m\ar[r,dashed]& F_eX.
\end{tikzcd}
\]
A map of presheaves is an isomorphism just when it is so objectwise, and this argument shows that the given set tests at every object of $\mathbb{N}$ against every object of $\mathcal{A}$.
\end{proof}

\begin{lemma}
\label{lemma: generators pass to reflective subcategories}
Let $\mathcal{C}$ be a reflective subcategory of $\mathcal{D}$ with reflector $r$ and let $G$ be a strongly generating class for $\mathcal{D}$.
Then $rG$ is strongly generating for $\mathcal{C}$.
\end{lemma}

\begin{proof}
Let $f:x\to y$ in $\mathcal{C}$ and suppose
$f_*:\mathcal{C}(rg,x)\to \mathcal{C}(rg,y)$ is an isomorphism for all $g$ in $G$.
By adjunction, $(i(f))_*:\mathcal{D}(g,ix)\to \mathcal{D}(g,iy)$ is an isomorphism for all $g$.
Then $i(f)$ is an isomorphism since $G$ is generating for $\mathcal{D}$.
Then $f$ is an isomorphism since $i$ is fully faithful.
\end{proof}

\begin{cor}
\label{cor:generators}
The class of presheaves $\{m_x\}$ as $m$ and $x$ vary over the objects of $\mathcal{A}$ and $\mathbb{N}$ is strongly generating for the categories of filtered and complete objects.
\end{cor}
\begin{proof}
The initial map and isomorphisms are monic so $m_x$ is already filtered, so it is fixed by the reflector to filtered objects. 
It's easy to see that $m_x$ is also complete because the limit involved in the completion of $m_x$ stabilizes.
This means that $m_x$ is also fixed by the reflector to complete objects.
The previous lemma completes the proof.
\end{proof}

\subsection{Monoidal structure on filtered objects}

In order to transfer a monoidal structure from $\mathcal{A}$-valued presheaves on $\mathbb{N}$ to $\mathbb{N}$-filtered objects of $\mathcal{A}$, we should have a monoidal structure on diagrams.

\begin{fact}
Day convolution gives the diagram category $\mathcal{A}^{\mathbb{N}^{\op}}$ a closed symmetric monoidal structure with product
\[
F_e(V\otimes W) \coloneqq \colim_{a+b\ge e} F_aV\otimes F_bW
\]
unit
\[
F_e(\mathbf{1}_{\mathcal{A}^{\mathbb{N}^{\op}}})
\coloneqq
\begin{cases}
\mathbf{1}_\mathcal{A}&e=0\\
\varnothing &\text{otherwise}
\end{cases}
\]
and internal hom object
\[
F_e(\ul{\mathcal{A}^{\mathbb{N}^{\op}}}(V,W))\coloneqq \lim_{e+a\ge b} \ul{\mathcal{A}}(F_aV,F_bW).
\]
\end{fact}

\begin{lemma}
\label{lemma: description of internal hom from generator}
We have the following description of the internal hom object from a generator:
\[
F_e(\ul{\mathcal{A}^{\mathbb{N}^{\op}}}(m_x,V))\cong\ul{\mathcal{A}}(m,F_{x+e}V),
\]
with morphisms induced by those of $V$.
\end{lemma}

\begin{proof}
There are maps from $\ul{\mathcal{A}}(m,F_{x+e}V)$ to $\ul{\mathcal{A}}(F_a(m_x),F_bV)$ for $e+a\ge b$ defined as follows:
\begin{itemize}
  \item if $a > x$, then $F_a(m_x)$ is initial so $\ul{\mathcal{A}}(F_a(m_x),F_bV)$ is terminal and the map requires no data.
  \item if $a\le x$, then $F_a(m_x)\cong m$.
Since $e+a\ge b$, we have $b\le x+e$, so there is a map $\phi:F_{x+e}V\to F_bV$.
Then $\ul{\mathcal{A}}(m,F_{x+e}V)\to \ul{\mathcal{A}}(m,F_bV)$ is induced by $\phi$.
\end{itemize}
It is immediate that these maps are compatible with the structure morphisms of the subcategory of $\mathbb{N}^{\op}\times\mathbb{N}$ over which we are taking the limit and realize $\ul{\mathcal{A}}(m,F_{x+e}V)$ as the limit.
\end{proof}

\begin{lemma}
\label{lemma: Day's condition for filtered objects}
The inclusion of the reflective subcategory of $\mathbb{N}$-filtered objects into $\mathbb{N}$-presheaves satisfies Day's simplified condition (2) with respect to the Day convolution symmetric monoidal structure and the generators of Corollary~\ref{cor:generators}.
\end{lemma}

\begin{proof}
Let $V$ be an $\mathbb{N}$-filtered object and $m_x$ a $\mathbb{N}$-presheaf generator.
We must show that 
\[\ul{\mathcal{A}^{\mathbb{N}^{\op}}}(m_x,V)
\to ir\ul{\mathcal{A}^{\mathbb{N}^{\op}}}(m_x,V)
\]
is an isomorphism. 
It is equivalent to show that $\ul{\mathcal{A}^{\mathbb{N}^{\op}}}(m_x,V)$ is filtered. 
In turn this is equivalent to showing that the structure map from index $e$ to index $0$ is a monomorphism for all $e$ in $\mathbb{N}$.
By the description of Lemma~\ref{lemma: description of internal hom from generator}, it is enough to prove that
\[
\ul{\mathcal{A}}(m,F_{x+e}V) \to
\ul{\mathcal{A}}(m,F_xV)
\]
is a monomorphism for arbitrary $m$ and $x$.
But $F_{x+e}V\to F_xV$ is a monomorphism by assumption and the functor $\ul{\mathcal{A}}(m,-)$ is a right adjoint and thus preserves monomorphisms.
\end{proof}

\begin{cor}
\label{cor: product on filtered objects}
The category of $\mathbb{N}$-filtered objects in $\mathcal{A}$ is closed symmetric monoidal with product 
\[
F_e(V\bar\otimes W) \coloneqq \image((\colim_{a+b\ge e} F_aV\otimes F_bW)\to V\otimes W),
\]
unit 
\[
F_e(\mathbf{1}_{\Filt(\mathcal{A})})
\coloneqq
\begin{cases}
\mathbf{1}_\mathcal{A}&e=0\\
\varnothing &\text{otherwise}
\end{cases}
\]
and internal hom object
\[
F_e(\ul{\Filt(\mathcal{A})}(V,W))\coloneqq \lim_{e+a\ge b} \ul{\mathcal{A}}(F_aV,F_bW).
\]
\end{cor}

\subsection{Monoidal structure on complete filtered objects}

Finally we use Day's theorem again to transfer the monoidal structure on filtered objects to one on complete filtered objects.
Again, we make use of the assumptions of section~\ref{subsection: complete condition}.

\begin{lemma}
\label{lemma: hom from a ground object to a complete object is complete}
Let $W$ be a complete $\mathbb{N}$-filtered object and $m$ an object of $\mathcal{A}$.
Then the sequence $\ul{\mathcal{A}}(m,W)$ is a complete filtered object.
\end{lemma}

\begin{proof}
We saw that $\ul{\mathcal{A}}(m,W)$ was filtered in Lemma~\ref{lemma: Day's condition for filtered objects}.
To see that it is complete, we should verify that the canonical map 
\[\ul{\mathcal{A}}(m,F_0W)\xrightarrow{r_0} \lim_e \coker\left(\ul{\mathcal{A}}(m,F_eW)\xrightarrow{(i_e)_*} \ul{\mathcal{A}}(m,F_0W)\right)\]
is an isomorphism. 

Since $W$ is complete, we have 
\[
\ul{\mathcal{A}}(m,F_0W)\cong
\ul{\mathcal{A}}(m,\lim_e \coker i_e)\cong \lim_e\ul{\mathcal{A}}(m,\coker i_e)
\]
and under this chain of isomorphisms, the natural map 
\[
\begin{tikzcd}
\lim_e \coker\left(\ul{\mathcal{A}}(m,F_eW)\to \ul{\mathcal{A}}(m,F_0W)\right)
\dar{s_0}\\
\lim_e \left(\ul{\mathcal{A}}(m,\coker(F_eW\to F_0W))\right)
\end{tikzcd}
\]
is isomorphic to a right inverse to $r_0$.
Since $\mathcal{A}$ is Abelian, it then suffices to show that $s_0$ has no kernel.
But the kernel of $s_0$ passes inside the limit and for each index $e$ is null by left exactness of $\mathcal{A}(m,-)$. 
\end{proof}

\begin{lemma}
\label{lemma: shifts of complete objects are complete}
Let $W$ be a complete $\mathbb{N}$-filtered object and $x$ an element of $\mathbb{N}$.
Then the shifted complex $W[x]$ with $F_e(W[x]) \coloneqq F_{x+e}W$ is a complete filtered object.
\end{lemma}

\begin{proof}
For any $x$ and $e$ the map $\coker(F_{x+e}W\to F_xW)\to \coker (F_{x+e}W\to F_0W)$ is a monomorphism since $F_xW\to F_0W$ is.
Therefore the limit 
\[\lim_e \coker(F_{x+e}W\to F_xW) \xrightarrow{\phi} \lim_e \coker(F_{x+e}W\to F_0W)\cong \lim_e \coker(F_eW\to F_0W)
\]
is a monomorphism.
The left hand side is $F_0\widehat{W[x]}$ and the right hand side is $F_0\widehat{W}$.
Moreover, the further projection $\pi_x$ to $\coker(F_xW\to F_0W)$ factors through the projection to $\coker(F_xW\to F_xW)\cong 0$ so $\phi$ lifts to $F_0\widehat{W[x]}\xrightarrow{\tilde{\phi}}F_x\widehat{W}=\ker \pi_x$:
\[
\begin{tikzcd}
\lim_e \coker(F_{x+e}W\to F_xW)\rar{\phi}\dar& \lim_e \coker(F_eW\to F_0W)\dar{\pi_x}\\
\coker(F_xW\to F_xW)\rar& \coker(F_xW\to F_0W).
\end{tikzcd}
\]
The map $\tilde{\phi}$ is a monomorphism because it lifts the monomorphism $\phi$.

On the other hand the composition $F_xW\cong F_0W[x]\xrightarrow{r^{W[x]}_0} F_0\widehat{W[x]}\xrightarrow{\tilde{\phi}} F_x\widehat{W}$ agrees with the canonical map $F_xW\to F_x\widehat{W}$ which is an isomorphism because $W$ is complete.
This shows that $\tilde{\phi}$ is also an epimorphism.
Therefore $\tilde{\phi}$ is an isomorphism and by two out of three, so is $r^{W[x]}_0$.
\end{proof}

\begin{cor}\label{cor:reflCompl}
The reflective subcategory of complete $\mathbb{N}$-filtered objects inside $\mathbb{N}$-filtered objects satisfies Day's simplified condition (2) with respect to the symmetric monoidal structure of Corollary~\ref{cor: product on filtered objects} and the generators of Corollary~\ref{cor:generators}.
\end{cor}

\begin{proof}
We are checking that for any generator $m_x$ and any complete filtered object $W$, the object 
\[
 \ul{\Filt(\mathcal{A})}(m_x,W)
 \] 
is complete.
We already have a description of this filtered object as a presheaf from Lemma~\ref{lemma: description of internal hom from generator} (and the proof of Lemma~\ref{lemma: Day's condition for filtered objects}): in index $e$ this presheaf has $\ul{\mathcal{A}}(m,F_{x+e}W)$. 
Then we are done by Lemmas~\ref{lemma: hom from a ground object to a complete object is complete} and~\ref{lemma: shifts of complete objects are complete}.
\end{proof}

\begin{cor}
\label{cor: product on complete filtered objects}
The category of complete $\mathbb{N}$-filtered objects in $\mathcal{A}$ is closed symmetric monoidal with product 
\[
V\hat\otimes W \coloneqq (V\bar{\otimes} W)^\wedge,
\]
and unit and internal hom calculated as in filtered objects.
\end{cor}

\subsection{Consequences for \texorpdfstring{$\sS$}{S}-modules}
The tensor product of $\sS$-modules can be described in terms of colimits and the monoidal product in the ground category.
Then because completion is a left adjoint and by Theorem~\ref{thm: Day reflective thm} and Corollary~\ref{cor:reflCompl}, completion commutes with the tensor product of $\sS$ modules.
That is, for $\sS$-modules $A$ and $B$ we have the following natural isomorphism:
\[
(A\bar{\otimes} B)^\wedge \cong \widehat{A}\hat\otimes\widehat{B}.
\]

Similarly, since completion is a left adjoint, it commutes with the colimit defining the composition product of $\sS$-modules so that for any filtered $\sS$-modules $A$ and $B$, we have the following sequence of natural isomorphisms:
\begin{align*}
(A\circ B)^\wedge 
&\cong
\left(\colim_k A(k)\bar\otimes_{\sS_k} (B)^{\bar\otimes k}\right)^\wedge
\\&\cong
\colim_k \widehat{A}(k)\hat\otimes_{\sS_k} \left(B^{\bar\otimes k}\right)^\wedge
\\&\cong
\colim_k \widehat{A}(k)\hat\otimes_{\sS_k} (\widehat{B})^{\hat\otimes k}
\\&\cong
\widehat{A}\hat\circ \widehat{B}.
\end{align*}
We may collect this into the following corollary.

\begin{cor}
\label{cor:completion is monoidal for circle product}
Completion is strong monoidal with respect to the induced composition products $\circ$ and $\hat{\circ}$ on $\mathbb{N}$-filtered $\sS$-modules and complete $\mathbb{N}$-filtered $\sS$-modules.
\end{cor}

\subsection{Filtered gr-dg objects}

Recall that $\fapgA$ is the category of filtered gr-dg objects, that is, filtered predifferential graded objects such that the induced predifferential structure on the associated graded object is actually a differential.
Likewise $\capgA$ consists of complete predifferential graded objects such that the induced predifferential structure on the associated graded object is actually a differential.

\begin{lemma}
\label{lemma: filtered formal coreflective}
The category of filtered gr-dg objects $\fapgA$ is a reflective subcategory of the one of filtered pg objects.
\end{lemma}

\begin{proof}
Let $(V,\, F,\, d)$ be a filtered predifferential graded object. 

The reflector takes this object to a filtered gr-dg object $(V,\fdg{F},d)$ constructed as follows.
Define $\fdg{F}^p(V)$ to be the following subobject of $F_0V$:
\[
\fdg{F}^p(V)=\sum_{i=0}^p d^{2i}F_{p-i}V.
\]
The connecting map $\fdg{F}^p V\to V$ is monic by construction, which forces all the connecting maps to be monic.
By construction $d$ respects this filtration.
This is clearly functorial.

For adjunction, consider a map $f$ from a filtered pg object $X$ into a filtered gr-dg object $Y$.
Then $d^{2i}F_{p-i}X$ must be taken by $f$ to $d^{2i}F_{p-i}Y\subset F_pY$ so that $f$ respects the more restrictive filtration $\fdg{F}$.
On the other hand given a map $f$ from $\fdg{X}$ to $Y$, the same formula for $f$ certainly respects the less restrictive filtration on $X$.
Thus the identity on the underlying map $F_0X\to F_0Y$ induces the adjunction.
\end{proof}

\begin{remark}
In fact, filtered gr-dg objects is also a coreflective subcategory. 
The coreflector takes the object $(V,\, F,\, d)$ to a filtered gr-dg object $(\check{V},\check{F},\check{d})$ constructed as follows.
We define $F_p\check{V}$ to be the subset of elements $v$ of $F_pV$ such that there exist indices $p_i$ for all natural numbers $i$ such that 
\begin{enumerate}
  \item The sequence of $p_i$ begins with $p$ and is strictly increasing: \
  \[p=p_0 < p_1< \cdots\]
  \item The element $d^{2i}v$ is equivalent to $0\pmod{F_{>p_{i}}V}$.
\end{enumerate}
This subset is a subobject (choosing the maximum of the $p_i$ sequences for a sum) and closed by construction under $d$, which induces $\check{d}$.

An equivalent construction proceeds as follows.
Consider the functor $T$ so that $F_p(TV)$ is the pullback
\[
\begin{tikzcd}
F_p(TV)\rar[dashed]\dar[dashed,tail]
&
\colim_{q>p}F_qV\dar[tail]\\
F_pV\rar{d^2}& F_pV
\end{tikzcd}
\]
Then $F_p\check{V} = \lim_r F_p(T^{\cdot r}V)$.
\end{remark}

\begin{lemma}
\label{lemma: completion preserves formal dg}
Completion preserves the gr-dg property.
\end{lemma}

\begin{proof}
We use the characterization of Lemma~\ref{alternate completion stages lemma} (which uses the fact that the ground category satisfies AB4).
Let $Y$ be gr-dg. 
Consider $F_j\hat{Y}\xrightarrow{d^{2}}F_j\hat{Y}$. 
This is an endomorphism of $\lim_a \coker (F_aY\to F_jY)$ which admits an epimorphism from $F_jY$.
Then since $Y$ is gr-dg, the map $F_jY\xrightarrow{d^{2}}F_jY$ factors through $F_{j+1}Y$.
This means that we get a similar factorization
\[
\lim_a \coker (F_aY\to F_jY) \xrightarrow{d^{2}} \lim_a \coker (F_aY\to F_{j+1}Y).
\]
\end{proof}

\begin{cor}\label{cor: complete gr-dg reflective}
Complete gr-dg objects $\compa(\Aa)$ are a reflective subcategory of both complete pg objects and filtered gr-dg objects.
\end{cor}

\begin{proof}
By Lemma~\ref{lemma: completion preserves formal dg}, in both cases the functor
$X\mapsto \widehat{\fdg{X}}$ can serve as reflector (on filtered gr-dg objects this is the same as plain completion). 
That is, at least it is a functor to the right codomain category.

Then the proof in both cases is a chain of canonical isomorphisms using the formal dg reflector $X\mapsto \fdg{X}$ and completion $X\mapsto \hat{X}$. 
Each one comes either from an adjunction, or full faithfulness of one of the right adjoints.

We check that, for $Y$ complete and gr-dg and $X$ filtered, we have:
\begin{align*}
\Hom_{\compa\ }(\widehat{\fdg{X}},Y)&\cong \Hom_{\Comp}(\widehat{\fdg{X}},Y)
\\&\cong \Hom_{\Filt}(\fdg{X},Y)
\\&\cong \Hom_{\fapg}(\fdg{X},Y)
\\&\cong \Hom_{\Filt}(X,Y).
\end{align*}
Now if $X$ is itself either complete or gr-dg, then this is finally isomorphic by the full faithfulness of the subcategories to the subcategory morphisms.
\end{proof}

\subsection{Monoidal structure on (complete) filtered gr-dg objects}

We use Day's theorem again to transfer the monoidal structure on (complete) filtered objects to one on (complete) gr-dg filtered objects.
Again, we uses the assumptions of section~\ref{subsection: complete condition} .

\begin{lemma}
\label{lemma: hom from a ground object to an fdg object is fdg}
Let $W$ be an object in $\fapgA$ and $m$ an object of $\Filt(\pgA)$.
Then the sequence $\ul{\mathcal{A}}(\fdg{m},W)$ is gr-dg.
\end{lemma}

\begin{proof}
On the hom object, the square of the differential is the commutator with $d^2$, interpreted in $\fdg{m}$ and in $W$.
The $p$th filtered component of $\ul{\mathcal{A}}(\fdg{m},W)$ consists of morphisms which increase filtration degree by $p$.
Since $d^2$ increases filtration degree by $1$ on both sides, the commutator with $d^2$ of such a morphism increases filtration degree by $p+1$. 
\end{proof}

\begin{lemma}
\label{lemma: shifts of fdg objects are fdg}
Let $W$ be a filtered gr-dg object and $x$ an element of $\mathbb{N}$.
Then the shifted complex $W[x]$ with $F_eW[x] \coloneqq F_{x+e}W$ is gr-dg.
\end{lemma}

\begin{proof}
The filtered gr-dg condition for $F_{n+x}W$ yields the condition for $F_nW$.
\end{proof}

\begin{cor}\label{cor:reflCompldg}
The reflective subcategory of (complete) gr-dg objects inside (complete) $\mathbb{N}$-filtered objects satisfies Day's simplified condition (2) with respect to the symmetric monoidal structure of Corollary~\ref{cor: product on filtered objects} (Corollary~\ref{cor: product on complete filtered objects}) and the generators of Corollary~\ref{cor:generators} reflected into (complete) gr-dg objects via Lemma~\ref{lemma: generators pass to reflective subcategories}.
\end{cor}

\begin{proof}
By Lemma~\ref{lemma: shifts of fdg objects are fdg}, and because $m_x$ is complete and filtered, we need only check that for any generator $m_x$ and any gr-dg object $W$, that the object 
\[
 \ul{\Filt(\mathcal{A})}(m_x,W)
 \] 
is gr-dg.
We already have a description of this filtered object as a presheaf from Lemma~\ref{lemma: description of internal hom from generator}: in index $e$ this presheaf has $\ul{\mathcal{A}}(m,F_{x+e}W)$. 
Then we are done by Lemmas~\ref{lemma: hom from a ground object to an fdg object is fdg} and~\ref{lemma: shifts of fdg objects are fdg}.
\end{proof}

\begin{cor}
\label{cor: product on fdg objects}
The category of (complete) gr-dg objects in $\mathcal{A}$ is closed symmetric monoidal with product 
\[
\fdg{(V\bar{\otimes} W)},\qquad (\text{resp. } \fdg{(V\hat{\otimes} W)},)
\]
and unit and internal hom calculated as in filtered objects.

By abuse of notation, we still denote by $\bar{\otimes}$ (resp. $\hat{\otimes}$) the tensor product in the category of (complete) gr-dg objects.
\end{cor}

\section{Gr-flat objects and the associated graded}
\label{appendix: gr-flat}

This section builds up the necessary material to prove Proposition~\ref{prop: graded of gr-flat coop is coop}.

\begin{lemma}
\label{lemma: pushout products and the filt reflector}
Let the ground category $\Aa$ be a Grothendieck category. 
Let $\mathcal{C}$ be the class of monomorphisms with flat cokernel.
Then $\mathcal{C}$ is closed under the pushout product.
\end{lemma}

\begin{proof}
This proof follows~\cite[Theorem 7.2]{Hovey:CPMCSRT} word for word, although both the hypotheses and conclusion here are weaker.
Let $A_1\to X_1$ and $A_2\to X_2$ be monomorphisms with flat cokernel $C_1$ and $C_2$ respectively.
Then $0\to A_1\to X_1\to C_1\to 0$ is pure exact so remains exact after tensoring with $A_2$ or $X_2$. Then by the $3\times 3$ lemma, the middle row of the following diagram is exact:
\[
\begin{tikzcd}
0\rar&A_1\otimes A_2\dar\rar&X_1\otimes A_2 \dar\rar& C_1\otimes A_2\dar\rar&0
\\
0\rar&A_1\otimes X_2\dar\rar&(X_1\otimes A_2)\amalg_{(A_1\otimes A_2)}(A_1\otimes X_2) \dar\rar& C_1\otimes A_2\dar\rar&0
\\
0\rar&A_1\otimes X_2\rar&X_1\otimes X_2 \rar& C_1\otimes X_2\rar&0.
\end{tikzcd}
\]
Since $C_1$ is flat, the map $C_1\otimes A_2\to C_1\otimes X_2$ is a monomorphism with cokernel $C_1\otimes C_2$. 
Then the snake lemma applied to the bottom two rows shows that 
\[
(X_1\otimes A_2)\amalg_{(A_1\otimes A_2)}(A_1\otimes X_2)\to X_1\otimes X_2
\]
has zero kernel and cokernel isomorphic to $C_1\otimes C_2$, which is thus flat.
\end{proof}

\begin{lemma}
\label{pushouts preserve flat cok}
Let $A\gets B\to C$ and $A'\gets B\to C'$ be diagrams in $\Cc$, and suppose given a map of diagrams from the former to the latter which is the identity on $B$ and a monomorphism with flat cokernel on the other two entries.
Then the induced morphism between the pushouts is a monomorphism with flat cokernel.
\end{lemma}

\begin{proof}
First, we argue that the induced morphism is a monomorphism. 
Let $A\amalg_B C\to I$ be a monomorphism from the pushout to an injective object (Grothendieck categories have enough injectives).
Then since $A\to A'$ and $C\to C'$ are monomorphisms, there are extensions $A'\to I$ of $A\to I$ and $C'\to I$ of $C\to I$. 
Because these morphisms are extensions, they remain compatible with $B\to I$, and so pass to an extension $A'\amalg_B C'\to I$ of $A\amalg_B C\to I$. 
Then $A\amalg_B C\to A'\amalg_B C'$ is the first morphism in a monic composition, hence is monic.

To see the statement about the cokernel, note that cokernel (in the pushout diagram category in $\Cc$ and in $\Cc$) commutes with the colimit functor, which is a left adjoint. 
Then the cokernel of the morphism of pushouts is the sum of the cokernels, and flat objects are closed under sum.
\end{proof}

\begin{lemma}
\label{lemma: diagram product of gr-flats is gr-flat}
The $\mathbb{N}$-indexed diagram product of two gr-flat filtered objects is a gr-flat filtered object.
\end{lemma}

\begin{proof}
Let $X$ and $Y$ be gr-flat filtered objects. 
We would like to show that $F_n(X\otimes Y)\to F_{n-1}(X\otimes Y)$ is a monomorphism. 
The colimit defining $F_n(X\otimes Y)$ is of the following diagram:
\[
\begin{tikzcd}[column sep=tiny]
F_nX \otimes F_0Y
&&
F_{n-1}X\otimes F_1Y
&&
\cdots
\\
&
F_nX\otimes F_1Y\ar[ur]\ar[ul]
&&
F_{n-1}X\otimes F_2Y\ar[ur]\ar[ul]
\end{tikzcd}
\]
We can rewrite this as the following colimit:
\[
\begin{tikzcd}[column sep=-3.5em]
\displaystyle
(F_nX\otimes F_0Y)\coprod_{F_nX\otimes F_1Y}(F_{n-1}X\otimes F_1Y)
&&
\displaystyle
(F_{n-1}X\otimes F_1Y)\coprod_{F_{n-1}X\otimes F_2Y}(F_{n-2}X\otimes F_2Y)
\\
&F_{n-1}X\otimes F_1Y\ar[ur]\ar[ul]
&&\cdots\ar[ul]
\end{tikzcd}
\]
The map from $F_n(X\otimes Y)$ to $F_{n-1}(X\otimes Y)$ is realized under colimit by a map of diagrams from this latter diagram to
the following diagram 
\[
\begin{tikzcd}[column sep=tiny]
F_{n-1}X\otimes F_0Y
&&
F_{n-2}X\otimes F_1Y
\\
&F_{n-1}X\otimes F_1Y\ar[ur]\ar[ul]
&&\cdots\ar[ul]
\end{tikzcd}
\]
where the maps involved are 
\begin{enumerate}
\item monomorphisms with flat cokernel on the upper row by Lemma~\ref{lemma: pushout products and the filt reflector}, and
\item identities on the lower row.
\end{enumerate}
The map of colimits can thus be constructed inductively from maps of pushout diagrams where the morphism at the corner is the identity and all other morphisms are monomorphisms with flat cokernel. 
The inductive step is supplied by Lemma~\ref{pushouts preserve flat cok}.
\end{proof}

\begin{lemma}
\label{lemma:gr-flat is monoidal subcat}
The full subcategory of gr-flat (complete) filtered objects in a Gro\-then\-dieck category is a monoidal subcategory of (complete) filtered objects and the restriction of $\Gr$ to this subcategory is strong symmetric monoidal.
\end{lemma}

\begin{proof}
We'll do the complete filtered case, but the filtered case is more or less a sub-case of this.
Let $X$ and $Y$ be gr-flat complete filtered objects.
The monoidal product of $X$ and $Y$ in complete filtered objects is
\[
(\Filt(X\otimes Y))^\wedge
\]
But by Lemma~\ref{lemma: diagram product of gr-flats is gr-flat}, the regular diagram product $X\otimes Y$ is already filtered so $\Filt$ is an isomorphism on it. 
Then $\Gr(\widehat{X\otimes Y})\cong \Gr(X\otimes Y)$ since completion preserves associated graded. 
Again by Lemma~\ref{lemma: diagram product of gr-flats is gr-flat}, $X\otimes Y$ is gr-flat.
This suffices to establish that gr-flat constitute a monoidal subcategory.

We have already essentially showed that the restriction of $\Gr$ to this subcategory is strong. Specifically, we've already argued that the second and third step of the composition
 \[
\Gr(X)\otimes_{\Gr}\Gr(Y)\to \Gr(X\otimes Y)\xrightarrow{\Gr(\Filt(\ ))}\Gr(X\bar\otimes Y)\xrightarrow{\Gr(\widehat{\ })}\Gr(X\hat\otimes Y),
\]
where $\otimes$ stands for $\mathbb{N}^{\op}$-indexed diagrams, are isomorphisms.
The first step is also an isomorphism because $\Gr$ is the reflector of a closed normal embedding, so the adjoint inclusion is strong monoidal.
\end{proof}

\begin{remark}
The gr-flat condition is more or less essential; if $F_iX/F_{i+1}X$ is not flat then tensoring with a module with trivial filtration can yield a diagram which is not filtered.
\end{remark}

Now we are ready to prove Proposition~\ref{prop: graded of gr-flat coop is coop}.

\begin{proof}
[Proof of Proposition~\ref{prop: graded of gr-flat coop is coop}]
The dual of~\cite[Lemma 3.1.1]{bF17} says that a unit preserving colax symmetric monoidal functor takes cooperads to cooperads.
The original lemma does not require cocompleteness of the ground category or the preservation of colimits in each variable by the monoidal product.
Therefore the dual does not require completeness or (unlikely) limit-preservation properties.
The definition of an operad in that reference uses only the monoidal product and morphisms in the ambient category, so cooperads in a monoidal subcategory of $\Cc$ coincide with cooperads in $\Cc$ whose underlying objects are in the monoidal subcategory.\footnote{We cannot use the dual of~\cite[Theorem 12.11(1)]{YauJohnson:FPAM} here without modification because this last statement fails---the monadic construction of operads there uses colimits intimately.
This implies that the category of Yau--Johnson cooperads in a monoidal subcategory of $\Cc$ may not even be well-defined (if the subcategory is not complete), and that if they are well-defined the isomorphism with cooperads whose underlying objects lie in the monoidal subcategory may fail.}
Now Lemma~\ref{lemma:gr-flat is monoidal subcat} says that $\Gr$ is a strong monoidal functor from the subcategory of gr-flat complete filtered objects in $\Cc$ to graded objects in $\Cc$, and that its image lies in the subcategory of degreewise flat objects.
\end{proof}

\section{Model category structures}
\label{appendix: MCS on gr-dg objects}

In this appendix, we endow the category $\compa(\ringK\text{-}\Mod)$ of complete gr-dg $\ringK$-modules with a cofibrantly generated model structure. We propose to consider the graded quasi-isomorphisms as a class of weak equivalences and a set of generating cofibrations similar to the one described in \cite{CESLW19}. We describe the fibrations and show that the model category structure on complete gr-dg $\ringK$-modules is combinatorial and is a monoidal model category structure. Then we transfer the cofibrantly generated model category structure to the category of complete curved operads via the free curved operad functor and to the category of complete algebras over a curved operad by means of the corresponding free functor and we study the cofibrant objects. We finally show a base change result for the categories of algebras over some curved operads.

We work over unbounded chain complexes over a ring $\ringK$. In this appendix, we assume that $\ringK$ is a field of characteristic $0$.

\subsection{Cofibrantly generated model structure}

We recall some definitions and a result from \cite[2.1]{mH91}.

\begin{defi}
Let $I$ be a class of maps in a category $\Mm$.
\begin{enumerate}
\item
A map is \emph{$I$-injective} if it has the right lifting property with respect to every map in $I$. The class of $I$-injective maps is denoted $I$-inj.
\item
A map is \emph{$I$-projective} if it has the left lifting property with respect to every map in $I$. The class of $I$-projective maps is denoted $I$-proj.
\item
A map is an \emph{$I$-cofibration} if it has the left lifting property with respect to every $I$-injective map. The class of $I$-cofibrations is the class ($I$-inj)-proj and is denoted $I$-cof.
\item
A map is an \emph{$I$-fibration} if it has the right lifting property with respect to every $I$-projective map. The class of $I$-fibrations is the class ($I$-proj)-inj and is denoted $I$-fib.
\item
We assume that $\Mm$  is cocomplete. A map is a \emph{relative $I$-cell complex} if it is a transfinite composition of pushouts of elements of $I$. 
Such a morphism $f : A \to B$ is the composition of a $\lambda$-sequence $X : \lambda \to \Mm$, for an ordinal $\lambda$ such that, for each $\beta$ such that $\beta + 1 < \lambda$, there is a pushout square
\[
\begin{tikzcd}
C_\beta \ar{d}{g_\beta} \ar{r} & X_\beta \ar{d}\\
D_\beta \ar{r} & X_{\beta + 1} \arrow[lu, phantom, "\ulcorner", very near start]
\end{tikzcd}
\]
such that $g_\beta \in I$. 
We denote the collection of relative $I$-cell complexes by \emph{$I$-cell}.
\end{enumerate}
\end{defi}

\begin{defi}
Let $\Mm$ be a cocomplete category and $\Nn$ be a collection of morphisms of $\Mm$. An object $A \in \Mm$ is \emph{$\alpha$-small relative to $\Nn$}, for a cardinal $\alpha$, if for all $\alpha$-filtered ordinals $\lambda$ and all $\lambda$-sequences
\[ X_0 \to X_1 \to \dots \to X_\beta \to \cdots \]
such that $X_\beta \to X_{\beta +1}$ is in $\Nn$ for $\beta + 1 < \lambda$, the map of sets
\[ \colim_{\beta < \lambda} \Hom_{\Mm}(A,\, X_\beta) \to \Hom_{\Mm}(A,\, \colim_{\beta < \lambda} X_\beta) \]
is an isomorphism. We say that $A$ is \emph{small relative to $\Nn$} if it is $\alpha$-small relative to $\Nn$ for some $\alpha$. We say that $A$ is \emph{small} if it is small relative to $\Mm$.
\end{defi}

\begin{defi}
A model category $\Cc$ is said to be \emph{cofibrantly generated} if there are sets $I$ and $J$ of maps such that:
\begin{enumerate}
\item
the domains of the mas in $I$ are small relative to $I$-cell,
\item
the domains of the maps in $J$ are small relative to $J$-cell,
\item
the class of fibrations is $J$-inj,
\item
the class of trivial fibrations is $I$-inj.
\end{enumerate}
The set $I$ is called the set of \emph{generating cofibrations} and the set $J$ is called the set of \emph{generating trivial cofibrations}.
\end{defi}

We now recall the following theorem from \cite[Theorem 2.1.19]{mH91}. It shows how to construct cofibrantly generated model categories.

\begin{thm}
\label{thm: cofibrantly generated model structure}
Suppose $\Mm$ is a complete and cocomplete category. Suppose $\Ww$ is a subcategory of $\Mm$, and $I$ and $J$ are sets of maps of $\Mm$. Then there is a cofibrantly generated model structure on $\Mm$ with $I$ as the set of generating cofibrations, $J$ as the set of generating trivial cofibrations, and $\Ww$ as the subcategory of weak equivalences if and only if the following conditions are satisfied:
\begin{enumerate}
\item
the subcategory $\Ww$ has the two-out-of-three property and is closed under retracts,
\item
the domains of $I$ are small relative to $I$-cell,
\item
the domains of $J$ are small relative to $J$-cell,
\item
$I$-inj $= \Ww \cap J$-inj,
\item
$J$-cof $\subseteq \Ww \cap I$-cof.
\end{enumerate}
\end{thm}

\subsection{The category of complete gr-dg $\ringK$-modules}
\label{section: model structure on complete gr-dg modules}

We apply Theorem \ref{thm: cofibrantly generated model structure} to endow the category of complete gr-dg $\ringK$-modules with a proper model structure. In this section, the category $\Aa$ denotes the category of $\ringK$-modules.\\

We recall that a map $p : (X,\, F) \to (Y,\, F')$ is strict when it satisfies $p(F_qX) = p(X) \cap F'_qY$ for all $q$. When $p$ is a surjection, this means that $p(F_qX) = F'_qY$ for all $q$. 
We denote by $\ringK^{(q)}$ the (complete) filtered $\ringK$-module given by
\[ \ringK^{(q)} = F_0 \ringK^{(q)} = F_q \ringK^{(q)} \supseteq F_{q+1} \ringK^{(q)} = 0. \]
(The filtration is induced by a gradation concentrated in degree $q$.) 
The notation $\ringK_n^{(q)}$ means that we consider it in degree $n$ within a (complete) filtered complex. 

Taking notation close to that used in \cite{CESLW19}, we define, for all $n\in \zZ$ and $q \in \nN$, the complete gr-dg modules
\begin{align*}
\hat \Zz^{0, \infty}_{q, n} \coloneqq & \left( \ringK^{(q)}_n \xrightarrow{1} \ringK^{(q)}_{n-1} \xrightarrow{1} \ringK^{(q+1)}_{n-2} \xrightarrow{1} \ringK^{(q+1)}_{n-3} \xrightarrow{1} \ringK^{(q+2)}_{n-4} \to \dots \right)^{\wedge}\\
= & \amalg_{k \in \nN} \ringK_{n - k}^{\left(q + \lfloor \frac{k}{2} \rfloor \right)},
\end{align*}
where $\amalg$ is the coproduct in complete gr-dg modules, and 
\begin{align*}
\hat \Zz^{1, \infty}_{q, n} \coloneqq & \left( \ringK^{(q)}_n \xrightarrow{1} \ringK^{(q+1)}_{n-1} \xrightarrow{1} \ringK^{(q+1)}_{n-2} \xrightarrow{1} \ringK^{(q+2)}_{n-3} \xrightarrow{1} \ringK^{(q+2)}_{n-4} \to \dots \right)^{\wedge}\\
= & \amalg_{k \in \nN} \ringK_{n - k}^{\left(q + \lceil \frac{k}{2} \rceil \right)}.
\end{align*}
We also define the complete gr-dg $\ringK$-module
\[
\hat \Bb^{1, \infty}_{q, n} \coloneqq \hat \Zz^{0, \infty}_{q, n+1} \amalg \hat \Zz^{0, \infty}_{q+1, n}.
\]
We denote by $\varphi^\infty_{q, n} : \hat \Zz^{1, \infty}_{q, n} \to \hat \Bb^{1, \infty}_{q, n}$ the morphism of complete gr-dg modules defined by the following diagram
\[ \xymatrix@C=16pt{
& \ringK_n^{(q)} \ar[r] \ar[d]^{{\tiny \begin{pmatrix} 1\\ 1 \end{pmatrix}}} & \ringK_{n-1}^{(q+1)} \ar[r] \ar[d]^{{\tiny \begin{pmatrix} 1\\ 1 \end{pmatrix}}} & \ringK_{n-2}^{(q+1)} \ar[d]^{{\tiny \begin{pmatrix} 1\\ 1 \end{pmatrix}}} \ar[r] & \cdots\\
\ringK_{n+1}^{(q)} \ar[r] & \ringK_n^{(q)} \amalg \ringK_n^{(q+1)} \ar[r] & \ringK_{n-1}^{(q+1)} \amalg \ringK_{n-1}^{(q+1)} \ar[r] & \ringK_{n-2}^{(q+1)} \amalg \ringK_{n-2}^{(q+2)} \ar[r] & \cdots.
} \]

In order to apply Theorem \ref{thm: cofibrantly generated model structure}, we consider the subcategory of $\compa(\Aa)$
\[ \Ww \coloneqq \left\{ f : (M,\, F,\, d_M) \to (N,\, F',\, d_N) \ |\  f \text{ is a graded quasi-isomorphism} \right\}\]
of weak equivalences, and the sets
\begin{align*}
I_0^\infty & \coloneqq \{ \varphi^\infty_{q, n} : \hat \Zz^{1, \infty}_{q, n} \to \hat \Bb^{1, \infty}_{q, n} \}_{n \in \zZ,\, q \in \nN}\\
J_0^\infty & \coloneqq \{ 0 \to \hat \Zz^{0, \infty}_{q, n} \}_{n \in \zZ,\, q \in \nN}
\end{align*}
of generating cofibrations and generating acyclic cofibrations.

In the following proposition, we characterize the morphisms having the right lifting property with respect to the morphisms in $J_0^\infty$. These morphisms will be the fibrations.

\begin{prop}
\label{prop: fibrations}
A map $p : (Y,\, F,\, d_Y) \to (X,\, F',\, d_X)$ has the right lifting property with respect to all the morphisms in $J_0^\infty$ if and only if the map $p : (Y,\, F,\, d_Y) \to (X,\, F',\, d_X)$ is a strict surjection.
\end{prop}

\begin{proof}
For every $q$, a diagram of the form
\[ \begin{tikzcd}
0 \ar{d} \ar{r} & (Y,\, F) \ar{d}{p}\\
\hat \Zz^{0,\, \infty}_{q, n} \ar{r} & (X,\, F')
\end{tikzcd} \]
is characterized by an element $x^q \in F'_qX_n$. A lift in this diagram is equivalent to an element $y^q$ in $F_q Y_n$ such that $p(y^q) = x^q$.
\end{proof}

We provide an equivalent description by means of the functor $\Gr$.

\begin{lemma}
A map $p : (Y,\, F) \to (X,\, F')$ satisfies $p_n : Y_n \to X_n$ is a strict surjection for all $n \in \zZ$ if and only if the map $(\Gr_q p)_n : \Gr_q Y_n \to \Gr_q X_n$ is surjective for all $q \in \nN$ and for all $n \in \zZ$.
\end{lemma}

\begin{proof}
The direct implication is immediate. The reverse implication is given by induction (we need $Y$ to be complete for the filtration $F$ in order to prove this fact). Indeed, assume that $\Gr_q p$ is surjective for all $q$ and let $x \in X$ and $q = 0$. If $\Gr_q x = 0$, then we fix $y_q = 0$. Otherwise, there exists by assumption $y_q \in F_q Y$ such that $\Gr_q p(\Gr_q y_q) = \Gr_q (x)$. The element $x-p(y_q)$ is in $F_{q+1} X$. This reasoning works for any $q$ provided that $x \in F_q X$. By induction and since $Y$ is complete, we get that $p(\sum_q y_q) = x$. This proof works for $F_q p$ instead of $p$ so we get that $p$ is a strict surjection.
\end{proof}

We therefore sometimes write \emph{gr-surjection} for a strict surjection.

In the following lemma, we describe the pushouts in the category $\compa(\Aa)$ and some properties related to them.

\begin{lemma}
\label{lemma: pushout}
Let $f : X \to Y$ and $g : X \to Y'$ be morphisms in $\fapgA$. Then
\begin{enumerate}
\item
the pushout
\[
\begin{tikzcd}
X \ar{r}{f} \ar{d}{g} & Y \ar{d}{g'}\\
Y' \ar{r}{f'} & Y' \amalg_X Y \arrow[lu, phantom, "\ulcorner", very near start]
\end{tikzcd}
\]
in $\fapgA$ is given by
\[ Y' \amalg_X Y \coloneqq \left((Y' \amalg Y)/(g(x)-f(x);\ x \in X),\, F,\, d\right).\]
The filtration $F$ is given by
\[ F_p (Y' \amalg_X Y) \coloneqq \im \left( F_p Y' \amalg F_p Y \xrightarrow{f'-g'} (Y' \amalg Y)/(g(x)-f(x);\ x \in X)\right), \]
and the predifferential $d$ is induced by the predifferentials on $Y'$ and $Y$;
\item
if $f$ is a monomorphism, so is $f$;
\item
if $f$ is a strict morphism, so is $f'$.
\end{enumerate}
Using the completion map (reflector), the morphisms $f$ and $g$ induces morphisms $\hat f : \hat X \to \hat Y$ and $\hat g : \hat X \to \hat Y'$ in $\compa(\Aa)$. Then
\begin{enumerate}
\item[(a)]
the pushout
\[
\begin{tikzcd}
\hat X \ar{r}{\hat f} \ar{d}{\hat g} & \hat Y \ar{d}{\hat g'}\\
\hat Y' \ar{r}{\hat f'} & \hat Y' \amalg_{\hat X} \hat Y \arrow[lu, phantom, "\ulcorner", very near start]
\end{tikzcd}
\]
in $\compa(\Aa)$ is given by the completion of the pushout $Y' \amalg_X Y$ in $\fapgA$;
\item[(b)]
if $f$ is a monomorphism (resp. strict morphism), so is $f$;
\item[(c)]
when $f$ is a strict monomorphism, the pushout $\hat Y' \amalg_{\hat X} \hat Y$ computed in $\fapgA$ is already complete. That is to say the inclusion map, adjoint to the completion map, preserves this pushout.
\end{enumerate}
\end{lemma}

\begin{proof}
\begin{enumerate}
\item
The pushout is obtained as the cokernel of the map $X \xrightarrow{g-f} Y' \amalg Y$. It is computed as the pushout in the category of diagrams $\Aa^{\nN^{\text{op}}}$ to which we apply the reflector $\Aa^{\nN^{\text{op}}} \to \Filt(\Aa)$. We recall that $\Aa = \ringK$-modules. This extends to gr-dg modules.
\item
This is direct from the description of $Y' \amalg_X Y$.
\item
The statement follows from the fact that
\[ (f')^{-1}(F_p(Y' \amalg_X Y)) = g(f^{-1}(F_pB))+F_pC.\]
\end{enumerate}
\begin{enumerate}
\item[(a)]
Colimits in $\compa(\Aa)$ are computed in this way.
\item[(b)]
This follows from $(2)$ and $(3)$ above.
\item[(c)]
When $\hat f$ is a strict monomorphism, so is $\hat X \xrightarrow{\hat g- \hat f} \hat Y' \amalg \hat Y$. Thus the filtration on the pushout
\[ \hat Y' \amalg_{\hat X} \hat Y = \widehat{Y' \amalg Y} /  (\widehat{g-f})(X) \]
is given by
\begin{align*}
F_p (\hat Y' \amalg_{\hat X} \hat Y) & = \left(F_p \hat Y' \amalg F_p \hat Y\right) / F_p \left((\hat g-\hat f)(X) \right)\\
& = F_p \left(\widehat{Y' \amalg Y}\right) / F_p \left((\widehat{g-f})(X) \right).
\end{align*}
It is the quotient filtration induced by the filtration on $\widehat{Y' \amalg Y}$ and the pushout is already complete.
\end{enumerate}
\end{proof}

\begin{lemma}
\label{lemma: pushout sphere}
We have the following pushout diagram
\[
\begin{tikzcd}
\hat \Zz^{1, \infty}_{q, n} \ar{d}{\varphi^\infty_{q, n}} \ar{r} & 0 \ar{d}\\
\hat \Bb^{1, \infty}_{q, n} \ar{r} & \hat \Zz^{1, \infty}_{q, n} \arrow[lu, phantom, "\ulcorner", very near start].
\end{tikzcd}
\]
Moreover, if the map $p$ has the right lifting property with respect to the maps $\{0 \to \hat \Zz^{1, \infty}_{q, n}\}$, then the map $\Gr p$ is surjective on cycles. 
Under the assumption that $p$ is a strict surjection, the reverse is true.
\end{lemma}

\begin{proof}
The first part of the lemma follows from Lemma \ref{lemma: pushout}. 

Then, the set of diagrams
\begin{equation}
\label{diagram: sphere}
\begin{tikzcd}
0 \ar{d} \ar{r} & (Y,\, F) \ar{d}{p}\\
\hat \Zz^{1, \infty}_{q, n} \ar{r} & (X,\, F'),
\end{tikzcd}
\end{equation}
corresponds (bijectively) to the set
\[
A_{q, n} \coloneqq \left\{ x \in F'_q X_{n};\ d_X(x) \in F'_{q+1} X \right\}.
\]
The set of lifts in such diagrams corresponds (bijectively) to the set
\[
A_{q, n} \coloneqq \left\{ (x,\, y) \in F'_q X_{n} \amalg F_q Y_n;\ d_Y(y) \in F_{q+1} Y \right\}.
\]
It is clear that if $p$ has the right lifting property with respect to the maps $\{0 \to \hat \Zz^{1, \infty}_{q, n}\}$, then $\Gr p$ is surjective on cycles. 
Let assume that $\Gr p$ is surjective on cycles and $x \in F'_q X_{n}$ such that $d_X(x) \in F'_{q+1} X$, that is the data of a diagram (\ref{diagram: sphere}). Using the fact that $\Gr p$ is surjective on cycles, we get that there exists $y \in F_q Y_n$ such that $d_Y(y) \in F_{q+1}Y$ and $p(y) = x + x^{q+1}$, with $x^{q+1} \in F'_{q+1} X_n$. Under the assumption that $p$ is a strict surjection, we get that there exists $y^{q+1} \in F_{q+1} Y_n$ such that $p(y^{q+1}) = x^{q+1}$. Finally, $y-y^{q+1}$ provides the requested lift.
\end{proof}

We can now characterize what will be the trivial fibrations. We recall that
\[ I_0^\infty = \{ \varphi^\infty_{q, n} : \hat \Zz^{1, \infty}_{q, n} \to \hat \Bb^{1, \infty}_{q, n} \}_{n \in \zZ,\, q \in \nN}.\]

\begin{prop}
\label{prop: trivial fibrations}
A map $p : (Y,\, F,\, d_Y) \to (X,\, F',\, d_X)$ has the right lifting property with respect to all the maps in $I_0^\infty$ if and only if the map $p$ it is a strict surjection and a graded quasi-isomorphism. 
In particular, $I_0^\infty$-inj $= J_0^\infty$-inj $\cap \Ww$.
\end{prop}

\begin{proof}
The map $0 \to \hat \Bb^{1, \infty}_{q, n} = \hat \Zz^{0, \infty}_{q, n+1} \amalg \hat \Zz^{0, \infty}_{q+1, n}$ is the composition of the maps $0 \to \hat \Zz^{1, \infty}_{q, n}$ and $\hat \Zz^{1, \infty}_{q, n} \to \hat \Bb^{1, \infty}_{q, n}$. By the pushout diagram presented in Lemma \ref{lemma: pushout sphere}, we obtain $I_0^\infty$-inj $= I_0^\infty \cup J_0^\infty$-inj (since $0 \to \hat \Zz^{0, \infty}_{q, n+1}$ is a retract of $0 \to \hat \Bb^{1, \infty}_{q, n}$). It follows that maps in $I_0^\infty$-inj are in particular strict surjections (by Proposition \ref{prop: fibrations}).

We now characterize the diagrams
\[
\begin{tikzcd}
\hat \Zz^{1, \infty}_{q, n} \ar{d}{\varphi^\infty_{q, n}} \ar{r} & (Y,\, F) \ar{d}{p}\\
\hat \Bb^{1, \infty}_{q, n} \ar{r} & (X,\, F'),
\end{tikzcd}
\]
in $\compa(\ringK$-$\Mod)$ admitting a lifting. The set of such diagrams corresponds (bijectively) to the set
\[
B_{q, n} \coloneqq \left\{ \begin{gathered} (t,\, x,\, y) \in F'_q X_{n+1} \amalg F'_{q+1}X_{n} \amalg F_q Y_n; \\ p(y)=d_X(t)+x \text{ and } d_Y(y) \in F_{q+1} Y \end{gathered} \right\}.
\]
The set of such diagrams admitting a lifting is in bijection with 
\[
B'_{q, n} \coloneqq \left\{ \begin{gathered} (t,\, x,\, y,\, z) \in F'_q X_{n+1} \amalg F'_{q+1}X_{n} \amalg F_q Y_n \amalg F_q Y_{n+1}; \\ p(z) = t,\, p(y)=d_X(t)+x \text{ and } y-d_Y(z) \in F_{q+1} Y \end{gathered} \right\},
\]
since the different conditions satisfied by the tuples in $B'_{q, n}$ already implies that $d_Y(y) \in F_{q+1}Y$.

First, we suppose that $p$ has the right lifting property with respect to maps in $I_0^\infty$. 
It remains to prove that $p$ is a graded quasi-isomorphism. By Lemma \ref{lemma: pushout sphere}, we have that $\Gr p$ is surjective on cycles, so that $\H_\bullet(\Gr p)$ is surjective as well. 
We then prove that $\H_\bullet(\Gr p)$ is injective. Let $\bar{\bar{y}} \in \ker \H_n(\Gr_q p)$, that is there exists a lift $y \in F_q Y_n$ such that $d_Y(y) \in F_{q+1} Y_{n-1}$ and $p(y) \in \im\, {d_X}_{|F_q X} + F_{q+1} X$. We fix $p(y) = d_X(x^q) + x^{q+1}$ for some $x^q \in F_qX$ and $x^{q+1} \in F_{q+1}X$. This is precisely the data of an element in $B_{q, n}$. By the lifting property, we obtain $z \in F_{q}Y_{n+1}$ such that $y-d_Y(z) \in F_{q+1}Y$. It follows that the class of $y = d_Y(z) + (y-d_Y(z))$ is $0$ in $\H_\bullet \Gr Y$ and $\H_\bullet(\Gr p)$ is injective.

Conversely, assume that $p$ is a strict surjection and a graded quasi-isomorphism. It remains to show that $p$ has the right lifting property with respect to the maps $\{ \hat \Zz^{1, \infty}_{q, n} \to \hat \Bb^{1, \infty}_{q, n} \}_{q, n}$. Let $(t,\, x,\, y) \in B_{q, n}$. Since $p$ is a strict surjection, we can find $z \in F_q Y_{n+1}$ such that $p(z) = t$. If $y - d_Y(z) \in F_{q+1} Y$, we have found the requested lift. Otherwise $y - d_Y(z)$ provides a non zero element in $\Gr_q Y$. The fact that $p$ is a strict surjection implies that $\Gr_q p$ is surjective. We therefore get a short exact sequence
\[ 0 \to K \to \Gr_q Y \to \Gr_q X \to 0 \]
in dg-modules where $K$ is the kernel of $\Gr_q p$. By the associated long exact sequence and the fact that $\Gr_q p$ is a quasi-isomorphism, we obtain that $K$ is acyclic. We have $\bar y - d_{\Gr Y} (\bar z) \in Z_{n-1} (K)$ since $d_Y(y-d_Y(z)) \in F_{q+1}Y$ and $p(y-d_Y(z)) = x \in F'_{q+1} X$. So there exist $u^q \in F_q Y_{n+1}$ such that $p(u^q) \in F_{q+1} X$ and $y^{q+1} \in F_{q+1}Y_{n}$ with the property that $y - d_Y (z) = d_Y (u^q) + y^{q+1}$. Using the fact that $p$ is a strict surjection, we obtain that there exists $u^{q+1} \in F_{q+1} Y_{n+1}$ such that $p(u^{q+1}) = p(u^q)$. Finally, we have
\[ p(z + u^q - u^{q+1}) = t \text{ and } y-d_Y(z + u^q - u^{q+1}) = y^{q+1} - u^{q+1} \in F_{q+1} Y. \]
This provides the requested lift.
\end{proof}

\begin{defi}
Let $f : (Y,\, F,\, d_Y) \to (X,\, F',\, d_X)$ be a map of gr-dg $\ringK$-modules. We denote by $C(f)$ the mapping cone of the map $f$ defined by
\[ F_q C(f)_n \coloneqq F_q Y_{n-1} \amalg F_q X_n, \]
with the predifferential $D(y,\, x) = (-d_Y(y),\, f(y) + d_X(y))$.
\end{defi}

\begin{lemma}
\label{lemma: acyclicity cone}
The cone $C(\id_Y)$ of the identity map is gr-acyclic (that is to say its gr-homology is $0$).
\end{lemma}

\begin{proof}
We have $\Gr (D)(\bar y,\, \bar z) = (-\bar d \bar y,\, \bar y + \bar d \bar z)$ so $\Gr (D)(\bar y,\, \bar z) = 0$ if and only if $\bar d \bar y = 0$ and $\bar y = -\bar d \bar z$. This is equivalent to the fact that $(\bar y,\, \bar z) = \Gr (D)(\bar z,\, 0)$. Hence $C(\id_Y)$ is gr-acyclic.
\end{proof}

\begin{prop}
\label{prop: J-cof are WE}
We have $J_0^\infty$-cof  $\subseteq \Ww$.
\end{prop}

\begin{proof}
We follow the proof given in \cite{CESLW19}. Let $f : A \to B$ be a $J_0^\infty$-cofibration. By Proposition \ref{prop: fibrations}, this means that $f$ has the left lifting property with respect to the maps $p$ which are strict surjections. We consider the diagram
\[
\begin{tikzcd}
A \ar{d}{f} \ar{r} & A \amalg C(\id_B)[-1] \ar{d}{(f,\, \pi_1)}\\
B \ar{r}{=} & B,
\end{tikzcd}
\]
where $\pi_1 : C(\id_B)[-1] = B_\bullet \amalg B_{\bullet+1} \to B_\bullet$ is the projection on the first factor. The map $\pi_1$ is a strict surjection. It follows that the diagram admits a lift $h : B \to A \amalg C(\id_B)[-1]$. Applying the functor $\Gr$ to the diagram and using the fact that the cone is gr-acyclic (by Lemma \ref{lemma: acyclicity cone}), the two commutative triangles give that $f$ is a graded quasi-isomorphism.
\end{proof}

\begin{thm}
\label{thm: cofibrantly generated model structure on complete gr-dg modules}
The category $\compa(\ringK$-$\Mod)$ of gr-dg $\ringK$-modules admits a proper cofibrantly generated model category structure, where:
\begin{enumerate}
\item
weak equivalences are graded quasi-isomorphisms,
\item
fibrations are strict surjections, and
\item
$I_0^\infty$ and $J_0^\infty$ are the sets of generating cofibrations and generating acyclic cofibrations respectively.
\end{enumerate}
\end{thm}

\begin{proof}
We apply Theorem \ref{thm: cofibrantly generated model structure}. The graded quasi-isomorphisms satisfy the two-out-of-three property. It follows that the subcategory $\Ww$ satisfies the two-out-of-three property. This category is also closed under retract since graded quasi-isomorphisms are closed under retract. 
The non trivial domains of the maps in $I_0^\infty$ and in $J_0^\infty$ are $\hat \Zz^{1, \infty}_{q, n}$. The complete gr-dg modules $\hat \Zz^{1, \infty}_{q, n}$ are $\aleph_1$-small since completion commutes with $\aleph_1$-filtered colimits. 
Finally, Propositions \ref{prop: trivial fibrations} and \ref{prop: J-cof are WE} ensure that we can apply Theorem \ref{thm: cofibrantly generated model structure} and we obtain the desired cofibrantly generated model structure. 
Using Proposition \ref{prop: fibrations}, we can see that every object is fibrant. So by \cite[Corollary 13.1.3]{pH03} the model category structure is right proper. 
It remains to show that the model structure is left proper, that is to say that weak equivalences are preserved by pushout along cofibrations. It is enough to show that for any diagram
\[
\begin{tikzcd}
X \ar{d} \ar{r}{f} & Y \ar{d}\\
X' \ar{d}{g} \ar{r} & Y' \ar{d}{g'} \arrow[lu, phantom, "\ulcorner", very near start]\\
X'' \ar{r} & Y'' \arrow[lu, phantom, "\ulcorner", very near start],
\end{tikzcd}
\]
in which both squares are cocartesian, $f$ belongs to $I_0^{\infty}$, and $g$ belongs to $\Ww$, the map $g'$ also belongs to $\Ww$. 
We first show that it is true in the category of chain complexes of $\ringK$-modules. We show that if $f$ is injective and $g$ is a quasi-isomorphism, we obtain that the map $g'$ is a quasi-isomorphism. This follows first from the fact that in the category of chain complexes pushout along an injection $f : X \to Y$ provides an injection $X' \to Y'$. Then in this situation, the second pushout gives a short exact sequence
\[ 0 \to X' \to Y' \amalg X'' \to Y'' \to 0\]
which provides a long exact sequence in homology. Finally, the fact that $g$ is a quasi-isomorphism implies that so is $g'$.

We consider the functor
\[ \begin{array}{lccl}
\Gr : & \compa(\Aa) & \to & (\dg\Aa)^{\ob \nN},\\
& (V,\, F,\, d_V) & \mapsto & (V,\, F,\, d_V)^{\text{gr}} = \left(\Gr V,\, d_{\Gr V} \right).
\end{array} \]
Given a map $f = \varphi^\infty_{q, n}$ in $I_0^\infty$, we show that the two cocartesian squares in $\compa(\Aa)$ give two cocartesian squares under $\Gr$. The functor $\Gr$ is the composition of the inclusion functor $i_1: \compa(\Aa) \to \Filt(\pg\Aa)$, the inclusion functor $i_2 : \Filt(\pg\Aa) \to (\pg\Aa)^{\nN^{\op}}$ and the quotient functor $q_1 : (\pg\Aa)^{\nN^{\op}} \to (\dg\Aa)^{\ob \nN}$. All the maps in $I_0^\infty$ are strict monomorphisms. 
From Lemma \ref{lemma: pushout}, we get that the inclusion $i_1$ preserves pushouts in which one map is a strict monomorphism (the inclusion $\fapgA \to \Filt(\pg\Aa)$ preserves the pushouts). Moreover, when $f$ is a strict monomorphism, so is the map $X' \to Y'$. Therefore the inclusion $i_1$ sends the two cocartesian squares to two cocartesian squares. Provided that the maps $f$ and $X' \to Y'$ are strict morphisms, Proposition 1.1.11 in \cite{pD71} shows that the functor $q_1 \cdot i_2$ also sends the two cocartesian squares to two cocartesian squares (the predifferentials don't affect this property). Finally, the functor $\Gr$ sends maps in $I_0^\infty$ to cofibrations in $(\dg\Aa)^{\ob \nN}$ (objectwise injections with free, hence projective, cokernel) and maps in $\Ww$ to quasi-isomorphisms in $(\dg\Aa)^{\ob \nN}$ so this shows that $\Gr g'$ is a quasi-isomorphism, that is $g'$ belongs to $\Ww$.
\end{proof}

We can provide a description of the cofibrations in this model category as follows.

\begin{prop}
\label{prop: cofibration between complete gr-dg modules}
A map of complete gr-dg $\ringK$-modules $i : (A,\, d_A) \to (A',\, d_{A'})$ is a cofibration in the model category structure described in Theorem \ref{thm: cofibrantly generated model structure on complete gr-dg modules} if and only if it is a retract of a complete gr-dg $\ringK$-module map $(B,\, d_B) \to (B \amalg S,\, \delta)$, where $S$ is endowed with an exhaustive filtration
\[ S_0 = \{ 0\} \subset S_1 \subset S_2 \subset \dots \subset \colim_i S_i = S \]
of complete gr-dg $\ringK$-modules such that $S_{i-1} \rightarrowtail S_i$ are split monomorphisms of complete $\ringK$-modules with cokernels isomorphic to a sum of complete $\ringK$-modules
\[ S_i /S_{i-1} \cong \coprod_\alpha \left(\xi^\alpha \cdot \ringK \amalg \hat \Zz^{0, \infty}_{q_\alpha+1, n_\alpha}\right) \]
where $\xi^\alpha$ is in homological degree $n_\alpha +1$ and filtration degree $q_\alpha$. The predifferential $\delta$ is the one of $\hat \Zz^{0, \infty}_{q_\alpha+1, n_\alpha}$ on $\hat \Zz^{0, \infty}_{q_\alpha+1, n_\alpha}$ and 
\[
\delta(\xi^\alpha) + \zeta^\alpha \in B \amalg S_{i-1},
\]
with $\zeta^\alpha$ is a generator of the gr-dg $\ringK$-module $\hat \Zz^{0, \infty}_{q_\alpha+1, n_\alpha}$. 
\end{prop}

\begin{proof}
We recall that $\ringK$ is a field of characteristic 0. 
By Proposition 2.1.18 in \cite{mH91}, cofibrations are retracts of relative $I_0^\infty$-cell complexes. We therefore study the pushouts of elements of $I_0^\infty$ of the form
\[
\begin{tikzcd}
\amalg_\alpha \hat \Zz^{1, \infty}_{q_\alpha, n_\alpha} \ar{d}{\amalg_\alpha i_\alpha} \ar{r}{f} & (B \amalg S_{i-1},\, \delta) \ar{d}\\
\amalg_\alpha \hat \Bb^{1, \infty}_{q_\alpha, n_\alpha} \ar{r} & (B',\, \delta').
\end{tikzcd}
\]
We denote by $z^\alpha$ the image under $f$ of the generating (as a gr-dg $\ringK$-modules) element of $\hat \Zz^{1, \infty}_{q_\alpha, n_\alpha}$. If we denote by $\xi^\alpha$ and by $\zeta^\alpha$ the generating (as a gr-dg $\ringK$-modules) elements of $\hat \Bb^{1, \infty}_{q_\alpha, n_\alpha} = \hat \Zz^{0, \infty}_{q_\alpha, n_\alpha+1} \amalg \hat \Zz^{0, \infty}_{q_\alpha+1, n_\alpha}$, and by $d^k \xi^\alpha$, resp. by $d^k \zeta^\alpha$, their successives predifferentials, the pushout $B'$ is equal to
\[ (B \amalg S_{i-1}) \coprod \amalg_\alpha \left( (\xi^\alpha \cdot \ringK) \amalg \hat \Zz^{0, \infty}_{q_\alpha+1, n_\alpha}\right), \]
where the predifferential on $\hat \Zz^{0, \infty}_{q_\alpha+1, n_\alpha}$ is the one of $\hat \Zz^{0, \infty}_{q_\alpha+1, n_\alpha}$ and $\delta \xi^\alpha = z^\alpha - \zeta^\alpha$. 
By induction, we get the result.
\end{proof}

\subsection{Properties of the model structure on gr-dg $\ringK$-modules}

We prove that the model category structure on gr-dg $\ringK$-modules is combinatorial and is a monoidal model structure.

\begin{defi}
Let $\Mm$ be a category endowed with a model structure. 
We say that $\Mm$ is \emph{combinatorial} if it is
\begin{enumerate}
\item
locally presentable as a category, and
\item
cofibrantly generated as a model category.
\end{enumerate}
\end{defi}

\begin{lemma}
\label{lemma: the cat gr-dg modules is presentable}
The category of complete gr-dg $\ringK$-modules is a (locally) presentable category. 
As a consequence, the category $\compa(\Aa)$ is combinatorial.
\end{lemma}

\begin{proof}
 The category of $\Aa = \ringK$-modules is locally presentable (see for instance Example $5.2.2.a$ in \cite{fB94}). Applying \cite[Corollary 1.54]{AR94}, the category $\Aa^{\nN^{\text{op}}}$ of $\nN^{\text{op}}$-indexed diagrams is locally presentable. By Lemma \ref{lemma: cat of filtered objects is reflective} and Corollary \ref{cor: cat of complete objects is reflective}, the category $\Comp(\Aa)$ of complete filtered objects in $\Aa$ is a reflective subcategory of the category of $\nN^{\text{op}}$-indexed diagrams. The monad $T$ associated with the adjunction sends an $\nN^{\text{op}}$-indexed diagrams $\{X(p)\}_p$ to the completion $\hat X$ (seen in $\Aa^{\nN^{\text{op}}}$) of its associated filtered object $\tilde{X}$ defined by $F_p \tilde{X} = \im(X(p) \to X(0))$. Let us show that the monad $T$ commutes with filtered colimits. First, the image of a map can be written as the equalizer of a cokernel pair. It therefore commutes with filtered colimits since colimits commute colimits and in the category $\Aa$, finite limits commute with filtered colimits (Corollary $3.4.3$ in \cite{fB94}). Then, the completion does not commute with filtered colimits but it commutes with $\kappa$-filtered colimits, for $\kappa$ a regular cardinal such that $\kappa > \aleph_0$ since the completion is defined in the locally presentable category $\Aa$ by a diagram obtained by means of two cokernels, a kernel and a limit indexed by $\aleph_0$ (see for example Corollary $5.2.8$ in \cite{fB94}). 
It follows that $T$ commutes with $\aleph_1$-filtered colimits and therefore by Theorem $5.5.9$ in \cite{fB94}, the category $\Comp(\Aa)$ of complete filtered objects in $\Aa$ is locally presentable. 
Again by Corollary 1.54 in \cite{AR94}, the category $\Comp(\Aa)^{\ob \zZ}$ of graded complete objects in $\Aa$ is locally presentable. 
Let $A = \ringK[[d]]$ be the $\ringK$-algebra of formal power series generated by $d$ of degree $-1$ and endowed with the filtration $F_p A = d^{2p} \ringK[[d]]$ for all $p \geq 0$. It is complete for the filtration and therefore provides an algebra in $\Comp(\Aa)^{\zZ^{\text{disc}}}$. 
The category of complete gr-dg $\ringK$-modules is the category of $A$-modules in $\Comp(\Aa)^{\zZ^{\text{disc}}}$. Again by Example $5.2.2.a$ in \cite{fB94}, we obtain that $\compa(\Aa)$ is locally presentable.
\end{proof}

We recall from \cite[Definition 4.2.6]{mH91} the notion of symmetric monoidal model category and we prove that the category $\compa(\Aa)$ is a symmetric monoidal model category.

\begin{defi}
\label{defi: symmetric monoidal model category}
We say that a symmetric monoidal category $(\Mm,\, \otimes,\, \mathbbm{1})$ endowed with a model structure is a \emph{symmetric monoidal model category} when
\begin{enumerate}
\item
the natural morphism
\[
i_1 \square i_2 : (X_1\otimes A_2)\amalg_{(A_1\otimes A_2)}(A_1\otimes X_2)\to X_1\otimes X_2
\]
induced by cofibrations $i_1 : A_1 \to X_1$ and $i_2 : A_2 \to X_2$ forms a cofibration, respectively an acyclic cofibration if $i_1$ or $i_2$ is also acyclic.
\item
Let $Q\mathbbm{1} \xrightarrow{\pi} \mathbbm{1}$ be the cofibrant replacement for the unit obtained by using the functorial factorization to factor $0 \to \mathbbm{1}$ into a cofibration followed by a trivial fibration. Then the natural map $Q \mathbbm{1} \otimes X \xrightarrow{\pi \otimes \id} \mathbbm{1} \otimes X$ is a weak equivalence for all cofibrant $X$. Similary, the natural map $X \otimes Q \mathbbm{1} \xrightarrow{\id \otimes \pi} X \otimes \mathbbm{1}$ is a weak equivalence for all cofibrant $X$. 
\end{enumerate}
\end{defi}

\begin{prop}
\label{prop: gr-dg modules symmetric monoidal model cat}
The category $\compa(\Aa)$ is a symmetric monoidal model category.
\end{prop}

\begin{proof}
To prove the first condition of Definition \ref{defi: symmetric monoidal model category}, by \cite[Corollary 4.2.5]{mH91}, it is enough to prove the claim in the case of generating (acyclic) cofibrations.
\begin{enumerate}
\item
We first consider the cofibrations $i_1 = \varphi_{q, n}^{\infty} : \hat \Zz_{q, n}^{1, \infty} \to \hat \Bb_{q, n}^{1, \infty}$ and $i_2 = \varphi_{p, m}^{\infty}$. The generators, as an $\ringK$-module, of $X_1 \otimes X_2$ are
\[ \left\{d^k 1_{n+1}^q \otimes d^{l} 1_{m+1}^p,\ d^k 1_n^{q+1} \otimes d^{l} 1_{m+1}^p,\ d^k 1_{n+1}^q \otimes d^{l} 1_m^{p+1},\ d^k 1_n^{q+1} \otimes d^{l} 1_m^{p+1}\right\}_{k,\, l \in \nN}, \]
whereas the ones of $(X_1\otimes A_2)\amalg_{(A_1\otimes A_2)}(A_1\otimes X_2)$ are
\begin{multline*}
Z \coloneqq \left\{d^k (d1_{n+1}^q+1_n^{q+1}) \otimes d^{l} 1_{m+1}^p,\ d^k (d1_{n+1}^q+1_n^{q+1}) \otimes d^{l} 1_{m}^{p+1}, \right.\\
\left. d^k 1_{n+1}^q \otimes d^l (d1_{m+1}^p+1_m^{p+1}),\ d^k 1_n^{q+1} \otimes d^{l} (d1_{m+1}^p+1_m^{p+1})\right\}_{k,\, l \in \nN}.
\end{multline*}
We prove by induction on the filtration degree that the map
\[
i_1 \square i_2 : (X_1\otimes A_2)\amalg_{(A_1\otimes A_2)}(A_1\otimes X_2) \to X_1\otimes X_2
\]
can be written as the coproduct of some cofibrations $\varphi_{r_s, o_s}^{\infty}$ and $\varphi_{r_s, o_s-1}^{\infty}$. We first describe $\varphi_{r_0, o_0}^{\infty}$ and $\varphi_{r_0, o_0-1}^{\infty}$.
We have
\begin{multline*}
(d1_{n+1}^q+1_n^{q+1}) \otimes 1_{m+1}^p + (-1)^{n+1} 1_{n+1}^q \otimes (d1_{m+1}^p+1_m^{p+1}) =\\
d\left( 1_{n+1}^q \otimes 1_{m+1}^p\right) + \left(1_n^{q+1} \otimes 1_{m+1}^p + (-1)^{n+1} 1_{n+1}^q \otimes 1_m^{p+1} \right).
\end{multline*}
This gives a cofibration $\varphi_{r_0, o_0}^{\infty}$ where $r_0 = q+p$ and $o_0 = n+m+1$, and
\begin{itemize}
\item
$(d1_{n+1}^q+1_n^{q+1}) \otimes 1_{m+1}^p + (-1)^{n+1} 1_{n+1}^q \otimes (d1_{m+1}^p+1_m^{p+1})$ is the generator of $\hat \Zz_{r_0, o_0}^{1, \infty}$,
\item
 $1_{n+1}^q \otimes 1_{m+1}^p$ is the generator of $\hat \Zz_{r_0, o_0+1}^{0, \infty}$, and
\item
$1_n^{q+1} \otimes 1_{m+1}^p - (-1)^{n} 1_{n+1}^q \otimes 1_m^{p+1}$ is the generator of $\hat \Zz_{r_0+1, o_0}^{0, \infty}$.
\end{itemize}
Similarly,
\begin{itemize}
\item the element in $\hat \Zz_{r_0, o_0-1}^{1, \infty}$
\begin{multline*}
z_{n+m}^{p+q} = d\left(1_{n+1}^{q} \otimes (d1_{m+1}^p+1_m^{p+1})\right) + (-1)^n d\left((d1_{n+1}^q+1_n^{q+1}) \otimes 1_{m+1}^{p} \right)\\
- \left((d1_{n+1}^q+1_n^{q+1}) \otimes 1_{m}^{p+1} + 1_{n}^{q+1} \otimes (d1_{m+1}^p+1_m^{p+1}) \right),
\end{multline*}
\item the element $b_{n+m+1}^{p+q} = 1_{n+1}^q \otimes d1_{m+1}^{p} + (-1)^n d1_{n+1}^{q} \otimes 1_{m+1}^p$ in $\hat \Zz_{r_0, o_0}^{0, \infty}$, and
\item the element $b_{n+m}^{p+q+1} = (-1)^{n+1} \left( 1_{n+1}^q \otimes d 1_m^{p+1} - d 1_n^{q+1} \otimes 1_{m+1}^p\right) - 2 \times 1_{n}^{q+1} \otimes 1_m^{p+1}$ in $\hat \Zz_{r_0+1, o_0-1}^{0, \infty}$
\end{itemize}
fit into the equality $z_{n+m+1}^{p+q} = db_{n+m+2}^{p+q} + b_{n+m+1}^{p+q+1}$ and therefore give a cofibration $\varphi_{r_0, o_0-1}^{\infty}$. 
Let us show that we have the equality
\begin{equation}
\label{equality: intersection with im}
\langle d^k1_{n+1}^{q} \otimes d^l 1_{m+1}^p,\, \rangle_{k, l\in \nN} \cap \im \left( i_1 \square i_2\right) = \{ 0 \}.
\end{equation}
The term $d^k1_{n+1}^{q} \otimes d^l 1_{m+1}^p$ appears only in two terms of $Z$ which are $d^{k-1} (d1_{n+1}^q+1_n^{q+1}) \otimes d^{l} 1_{m+1}^p$ and $d^k 1_{n+1}^q \otimes d^{l-1} (d1_{m+1}^p+1_m^{p+1})$. The terms $d^{k-1}1_n^{q+1} \otimes d^{l} 1_{m+1}^p$ and $d^k 1_{n+1}^q \otimes d^{l-1}1_m^{p+1}$ appear only in one other term each. We get the two equalities:
\begin{multline*}
d^{k-1} (d1_{n+1}^q+1_n^{q+1}) \otimes d^{l} 1_{m+1}^p - d^{k-1} 1_n^{q+1} \otimes d^{l-1} (d1_{m+1}^p+1_m^{p+1}) = \\
d^k1_{n+1}^{q} \otimes d^l 1_{m+1}^p - d^{k-1}1_{n}^{q+1} \otimes d^{l-1} 1_{m}^{p+1}
\end{multline*}
and
\begin{multline*}
d^k 1_{n+1}^q \otimes d^{l-1} (d1_{m+1}^p+1_m^{p+1}) - d^{k-1} (d1_{n+1}^q+1_n^{q+1}) \otimes d^{l-1} 1_{m}^{p+1} =\\
d^k1_{n+1}^{q} \otimes d^l 1_{m+1}^p - d^{k-1}1_{n}^{q+1} \otimes d^{l-1} 1_{m}^{p+1}.
\end{multline*}
The term $d^{k-1}1_{n}^{q+1} \otimes d^{l-1} 1_{m}^{p+1}$ doesn't appear in an other term in $Z$. This proves Equality \eqref{equality: intersection with im}. 
This implies that the element $1_{n+1}^{q} \otimes 1_{m+1}^p$ and its derivatives don't belong to $\im (i_1 \square i_2)$. This also gives that $b_{n+m+1}^{p+q}$ and its derivatives don't belong to $\im (i_1 \square i_2) \amalg \im \varphi_{r_0, o_0}^{\infty}$. 
Let $S_0$ be an $\ringK$-linear complement of the domains of $\varphi_{r_0, o_0}^{\infty}$ and $\varphi_{r_0, o_0-1}^{\infty}$ in $(X_1\otimes A_2)\amalg_{(A_1\otimes A_2)}(A_1\otimes X_2)$. It is a tedious but direct computation to show that modulo $F_{q+p+2} (X_1 \otimes X_2)$, the map $i_1 \square i_2$ is given by $\id_{S_0} \amalg \varphi_{r_0, o_0}^{\infty} \amalg \varphi_{r_0, o_0-1}^{\infty}$ (that is to say the codomain of $\id_{S_0} \amalg \varphi_{r_0, o_0}^{\infty} \amalg \varphi_{r_0, o_0-1}^{\infty}$ is equal to $X_1 \otimes X_2$ modulo $F_{q+p+2} (X_1 \otimes X_2)$). This is the base case of the induction. 

We define similarly the maps $\varphi_{r_s, o_s}^{\infty}$ and $\varphi_{r_s, o_s-1}^{\infty}$. 
We have
\begin{multline*}
d^{2s}(d1_{n+1}^q+1_n^{q+1}) \otimes d^{2s}1_{m+1}^p + (-1)^{n+1} d^{2s}1_{n+1}^q \otimes d^{2s}(d1_{m+1}^p+1_m^{p+1})\\
= d\left( d^{2s}1_{n+1}^q \otimes d^{2s}1_{m+1}^p\right)\\
+ d^{2s}1_n^{q+1} \otimes d^{2s}1_{m+1}^p + (-1)^{n+1} d^{2s}1_{n+1}^q \otimes d^{2s}1_m^{p+1}.
\end{multline*}
This gives a cofibration $\varphi_{r_s, o_s}^{\infty}$ where $r_s = q+p+2s$ and $o_s = n+m+1-4s$, and
\begin{itemize}
\item
$d^{2s}(d1_{n+1}^q+1_n^{q+1}) \otimes d^{2s}1_{m+1}^p + (-1)^{n+1} d^{2s}1_{n+1}^q \otimes d^{2s}(d1_{m+1}^p+1_m^{p+1})$ is the generator of $\hat \Zz_{r_s, o_s}^{1, \infty}$,
\item
$d^{2s}1_{n+1}^q \otimes d^{2s}1_{m+1}^p$ is the generator of $\hat \Zz_{r_s, o_s+1}^{0, \infty}$, and
\item
$d^{2s}1_n^{q+1} \otimes d^{2s}1_{m+1}^p - (-1)^{n} d^{2s}1_{n+1}^q \otimes d^{2s}1_m^{p+1}$ is the generator of $\hat \Zz_{r_s+1, o_s}^{0, \infty}$.
\end{itemize}
Then,
\begin{itemize}
\item the element in $\hat \Zz_{r_s, o_s-1}^{1, \infty}$
\begin{multline*}
z_{n+m-4s}^{p+q+2s} = d\left(d^{2s}1_{n+1}^{q} \otimes d^{2s}(d1_{m+1}^p+1_m^{p+1}) \right) + (-1)^n \left(d^{2s}(d1_{n+1}^q+1_n^{q+1}) \otimes d^{2s}1_{m+1}^{p} \right)\\
- \left(d^{2s}(d1_{n+1}^q+1_n^{q+1}) \otimes d^{2s}1_{m}^{p+1} + d^{2s}1_{n}^{q+1} \otimes d^{2s}(d1_{m+1}^p+1_m^{p+1}) \right),
\end{multline*}
\item the element $b_{n+m+1-4s}^{p+q+2s} = d^{2s}1_{n+1}^q \otimes d^{2s+1}1_{m+1}^{p} + (-1)^n d^{2s+1}1_{n+1}^{q} \otimes d^{2s}1_{m+1}^p$ in $\hat \Zz_{r_s, o_s}^{0, \infty}$, and
\item the element $b_{n+m-4s}^{p+q+1+2s} =$
\[
 (-1)^{n+1} \left( d^{2s}1_{n+1}^q \otimes d^{2s+1} 1_m^{p+1} - d^{2s+1} 1_n^{q+1} \otimes d^{2s}1_{m+1}^p\right) - 2 \times d^{2s}1_{n}^{q+1} \otimes d^{2s}1_m^{p+1}
\]
in $\hat \Zz_{r_s+1, o_s-1}^{0, \infty}$
\end{itemize}
fit into the equality $z_{n+m+1-4s}^{p+q+2s} = db_{n+m+2-4s}^{p+q+2s} + b_{n+m+1-4s}^{p+q+1+2s}$ and give a cofibration $\varphi_{r_s, o_s-1}^{\infty}$. 
We denote by $\textrm{Codom}_{s}$ the codomain of the map $\amalg_{0 \leq k \leq s} \left( \varphi_{r_k, o_k}^{\infty} \amalg \varphi_{r_k, o_k-1}^{\infty}\right)$. 
We want to prove the following statement:
\begin{center}
for $k\geq 0$, we have the equality
\begin{equation}
\tag{$E_s$}
\label{equality: intersection with im and codomain}
\langle d^k1_{n+1}^{q} \otimes d^l 1_{m+1}^p,\, \rangle_{k, l\geq 2s} \cap \left(\im \left( i_1 \square i_2\right) + \textrm{Codom}_{s-1}\right) = \{ 0 \}.
\end{equation}
and there exists an $\ringK$-linear subspace $S_s$ of $X_1 \otimes X_2$ such that modulo $F_{q+p+2(s+1)} \left(X_1 \otimes X_2 \right)$, the map $i_1 \square i_2$ is given by $\id_{S_s} \coprod \amalg_{0 \leq k \leq s} \left( \varphi_{r_k, o_k}^{\infty} \amalg \varphi_{r_k, o_k-1}^{\infty}\right)$.
\end{center}
Let assume Equality \eqref{equality: intersection with im and codomain} for $s-1$ and the existence of $S_{s-1}$ such that modulo $F_{q+p+2s} \left(X_1 \otimes X_2 \right)$, the map $i_1 \square i_2$ is given by
\[ \id_{S_{s-1}} \coprod \amalg_{0 \leq k \leq s-1} \left( \varphi_{r_k, o_k}^{\infty} \amalg \varphi_{r_k, o_k-1}^{\infty}\right). \]
First of all, we obtain $\im \left( i_1 \square i_2\right) + \textrm{Codom}_{s-1}$ from $\im \left( i_1 \square i_2\right) + \textrm{Codom}_{s-2}$ by adding the terms $d^{2s}1_{n+1}^q \otimes d^{2s}1_{m+1}^p$ and $b_{n+m+1-4s}^{p+q+2s}$ and their boundaries. It is easy to see that all these terms are sums of terms of the form $d^{k}1_{n+1}^q \otimes d^{l}1_{m+1}^p$ with at least one term with $k \geq 2s$ or $l \geq 2s$. By Equalities \eqref{equality: intersection with im} and \eqref{equality: intersection with im and codomain} for $s-1$, this gives Equality \eqref{equality: intersection with im and codomain}. 
It remains to show that there exists $S_s$ such that modulo $F_{q+p+2(s+1)} \left(X_1 \otimes X_2 \right)$, the map $i_1 \square i_2$ is given by $\id_{S_s} \coprod \amalg_{0 \leq k \leq s} \left( \varphi_{r_k, o_k}^{\infty} \amalg \varphi_{r_k, o_k-1}^{\infty}\right)$, that is to say the codomain of this last map is equal to $X_1 \otimes X_2$ modulo $F_{q+p+2(s+1)} \left(X_1 \otimes X_2 \right)$. By the induction hypothesis, it is enough to show that all the elements in $F_{q+p+2(s+1)} \left(X_1 \otimes X_2 \right) / F_{q+p+2(s+1)} \left(X_1 \otimes X_2 \right)$ are in the codomain modulo $F_{q+p+2(s+1)} \left(X_1 \otimes X_2 \right)$. A basis of the module
\[
F_{q+p+2s}\left(X_1 \otimes X_2\right)/F_{q+p+2(s+1)} \left(X_1 \otimes X_2 \right) 
\]
is given by the classes of the elements in the set
\begin{multline*}
C_s \coloneqq \left\{d^{2k+t} 1_{n+1}^q \otimes d^{2(2s-k+u)+l} 1_{m+1}^p\right\}_{t,\, u,\, l \in \llbracket 0,\, 1\rrbracket,\, k \in \llbracket 0,\, 2s+u\rrbracket} \bigcup\\
\left\{d^{2k+t} 1_n^{q+1} \otimes d^{2(2s-1-k+u)+l} 1_{m+1}^p,\right.\\
\left. d^{2k+t} 1_{n+1}^q \otimes d^{2(2s-1-k+u)+l} 1_m^{p+1}\right\}_{t,\, u,\, l \in \llbracket 0,\, 1\rrbracket,\, k \in \llbracket 0,\, 2s-1+u\rrbracket}\\
\bigcup \left\{ d^{2k+t} 1_n^{q+1} \otimes d^{2(2(s-1)-k+u)+l} 1_m^{p+1}\right\}_{t,\, u,\, l \in \llbracket 0,\, 1\rrbracket,\, k \in \llbracket 0,\, 2(s-1)+u\rrbracket}.
\end{multline*}
First of all, using the terms $d^k1_{n+1}^{q} \otimes d^l 1_{m+1}^p$ and the terms in $Z$ such that $k,\, l \geq 2s$, the same tedious computation as for the base case shows that the terms $d^k 1_{n+1}^q \otimes d^{l} 1_{m+1}^p$, $d^k 1_n^{q+1} \otimes d^{l} 1_{m+1}^p$, $d^k 1_{n+1}^q \otimes d^{l} 1_m^{p+1}$, $d^k 1_n^{q+1} \otimes d^{l} 1_m^{p+1}$, for $k,\, l \geq 2s$, are in the codomain modulo $F_{q+p+2(s+1)} \left(X_1 \otimes X_2 \right)$. Finally, it is enough to say that we get the terms $d^k 1_{n+1}^q \otimes d^{l} 1_{m+1}^p$, $d^k 1_n^{q+1} \otimes d^{l} 1_{m+1}^p$, $d^k 1_{n+1}^q \otimes d^{l} 1_m^{p+1}$, $d^k 1_n^{q+1} \otimes d^{l} 1_m^{p+1}$, for $k,\, l$ such that $k$ or $l$ is $< 2s$ using the previous terms and the boundaries of some lift of the class modulo $F_{q+p+2s} \left(X_1 \otimes X_2 \right)$ of the terms in $C_{s-1}$. The picture of this proof is nested cones.
This gives the inductive step and therefore proves that the map $i_1 \square i_2$ is a cofibration since we are working in a complete setting.

 When the two maps are acyclic cofibrations, say $i_1 : 0 \to \hat \Zz_{q, n}^{0, \infty}$ and $i_2 : 0 \to \hat \Zz_{p, m}^{0, \infty}$, the map $i_1\square i_2$ is given by $0 \to \Zz_{q, n+1}^{0, \infty} \otimes \hat \Zz_{p, m}^{0, \infty}$ which is a coproduct of acyclic cofibrations. When we assume that $i_1$ or $i_2$ is an acyclic cofibration, say $i_1 = \varphi_{q, n}^{1, \infty}$ and $i_2 : 0 \to \hat \Zz_{p, m}^{0, \infty}$, we can see that the map $i_1\square i_2$ is isomorphic to the inclusion $\hat \Zz_{q, n}^{1, \infty} \otimes \hat \Zz_{p, m}^{0, \infty} \to \hat \Zz_{q, n}^{1, \infty} \otimes \hat \Zz_{p, m}^{0, \infty} \amalg \Zz_{q, n+1}^{0, \infty} \otimes \hat \Zz_{p, m}^{0, \infty}$, which is an acyclic cofibration.
\item
By \cite[Remark 4.2.3]{mH91} and the fact that a left Quillen functor preserves weak equivalences between cofibrant objects, we can consider any cofibrant replacement of the unit. The map $\hat \Zz_{0, 0}^{1, \infty} \to \ringK$ is one. We have to show that for all cofibrant $X$, the natural map $\hat \Zz_{0, 0}^{1, \infty} \otimes X \to X$ (resp. $X \otimes \hat \Zz_{0, 0}^{1, \infty} \to X$) is a weak equivalence. By \cite[Lemma 4.2.7]{mH91}, it is enough to show that the map $X \to \Hom(\hat \Zz_{0, 0}^{1, \infty},\, X)$ (where the internal $\Hom$ is computed in gr-dg filtered $\ringK$-modules) is a weak equivalence for all $X$ (since all objects are fibrant). Since this map is an equality, we are done.
\end{enumerate}
\end{proof}

\subsection{Model structure on complete curved operads}
\label{section: model structure on complete curved operads}

In this section, we recall from \cite[Theorem 11.3.2]{pH03} the theorem of transfer of cofibrantly generated model structure, essentially due to Quillen \cite[Section II.4]{Quillen2} and we apply it to endow the category of complete curved operads with a cofibrantly generated model structures. The model structure is obtained from the adjunction between the free complete curved operad and the forgetful functor from complete curved operads to complete gr-dg $\sS$-modules.\\

Let $A$ be an associative $\ringK$-algebra. Applying Theorem \ref{thm: transfer theorem} to the free-forgetful adjunction between complete gr-dg $\ringK$-modules and complete gr-dg $A$-modules, we obtain, for the $\ringK$-algebra $A = \ringK [\sS_m]$, $m \in \nN$, a proper cofibrantly generated model category structure on the category $\Mm_m$ of complete gr-dg $\ringK[\sS_m]$-modules. 
Considering this collection $(\Mm_m,\, W_m,\, I_m,\, J_m)_{m \in \nN}$ of proper cofibrantly generated model category structures, we have that the product
\[ \capgSA \coloneqq \left(\prod_{m \in \nN} \Mm_m,\, \prod_{m \in \nN} W_m,\, \prod_{m \in \nN} I_m \right) \]
(for $\Aa = \ringK$-$\Mod$) is also a proper cofibrantly generated model category structure (see for example \cite[11.6]{pH03}). A morphism $f : M \to N$ in $\Mm$ is a weak equivalence (resp. fibration) if the underlying collection of morphisms $\{ M(m) \to N(m)\}_{m \in \nN}$ consists of weak equivalences (resp. fibrations). Moreover the set $\texttt{I}$ (resp. $\texttt{J}$) of generating cofibrations (resp. acyclic cofibrations) can be described explicitly as follows:
\[ \texttt{I} = \{ \hat \Zz^{1, \infty}_{q, n}(m) \to \hat \Bb^{1, \infty}_{q, n}(m)\} \text{ and } \texttt{J} = \{ 0 \to \hat \Zz^{0, \infty}_{q, n}(m)\}, \]
where $\hat \Zz^{k, \infty}_{q, n}(m)$ (resp. $\hat \Bb^{1, \infty}_{q, n}(m)$) is the complete gr-dg $\sS$-module obtained by the complete free gr-dg $\ringK[\sS_m]$-module $\hat \Zz^{k, \infty}_{q, n} \otimes \ringK[\sS_m]$ (resp. $\hat \Bb^{1, \infty}_{q, n} \otimes \ringK[\sS_m]$) put in arity $m$. 
We denote by $\Ww$ the subcategory of weak equivalences. 
Notice that the domains of elements of $\texttt{I}$ or $\texttt{J}$ are small ($\aleph_1$-small) in the category $\capgSA$.

We make a description of the cofibrations in $\capgSA$ as in Proposition \ref{prop: cofibration between complete gr-dg modules} in order to be able to prove that the bar-cobar resolutions and Koszul resolutions are $\sS$-cofibrant resolutions.

\begin{prop}
\label{prop: cofibration between complete gr-dg S-modules}
A map of complete gr-dg $\sS$-modules $i : (M,\, d_M) \to (M',\, d_{M'})$ is a cofibration in the model category structure described above if and only if it is a retract of a complete gr-dg $\sS$-module map $(N,\, d_N) \to (N \amalg S,\, \delta)$, where $S$ is endowed with an exhaustive filtration
\[ S_0 = \{ 0\} \subset S_1 \subset S_2 \subset \dots \subset \colim_i S_i = S \]
of complete gr-dg free $\sS$-modules such that $S_{i-1} \rightarrowtail S_i$ are split monomorphisms of complete $\sS$-modules with cokernels isomorphic to a sum of complete $\sS$-modules
\[ S_i /S_{i-1} \cong \coprod_\alpha \left(\xi^\alpha \cdot \ringK[\sS_{m_\alpha}] \amalg \hat \Zz^{0, \infty}_{q_\alpha+1, n_\alpha}(m_\alpha)\right) \]
where $\xi^\alpha$ is in homological degree $n_\alpha +1$ and filtration degree $q_\alpha$. The predifferential $\delta$ is the one of $\hat \Zz^{0, \infty}_{q_\alpha+1, n_\alpha}(m_\alpha)$ on $\hat \Zz^{0, \infty}_{q_\alpha+1, n_\alpha}(m_\alpha)$ and 
\[
\delta(\xi^\alpha) + \zeta^\alpha \in N \amalg S_{i-1},
\]
with $\zeta^\alpha$ is a generator of the gr-dg $\ringK$-module $\hat \Zz^{0, \infty}_{q_\alpha+1, n_\alpha}(m_\alpha)$. 
\end{prop}

\begin{proof}
The proof is similar to the proof of Proposition \ref{prop: cofibration between complete gr-dg modules}. 
\end{proof}

\begin{prop}
\label{prop: cofibration are retract of good cofibration}
Any map $(\tilde M,\, d_{\tilde M}) \to (\tilde N,\, d_{\tilde N})$ of complete gr-dg $\sS$-modules is a retract of a map $(M,\, d_M) \to (N,\, d_N)$ of complete gr-dg free $\sS$-modules. Moreover, assume that $\tilde N = \tilde M \amalg \tilde S$ with $\tilde S$ endowed with an exhaustive filtration
\[ \tilde S_0 = \{ 0\} \subset \tilde S_1 \subset \tilde S_2 \subset \dots \subset \colim_i \tilde S_i = \tilde S \]
such that $\tilde S_{i-1} \rightarrowtail \tilde S_i$ are split monomorphisms of complete $\sS$-modules with cokernels isomorphic to a sum of complete $\sS$-modules
\[ \tilde S_i /\tilde S_{i-1} \cong \coprod_\alpha \left(\ringK\{\xi^\alpha\} \amalg \hat \Zz^{0, \infty}_{q_\alpha+1, n_\alpha}\{m_\alpha\}\right), \]
where $\ringK\{\xi^\alpha\}$ is an $\sS_{m_\alpha}$-module of rank 1 generated by $\xi^\alpha$ in homological degree $n_\alpha +1$ and filtration degree $q_\alpha$ and $\hat \Zz^{0, \infty}_{q_\alpha+1, n_\alpha}\{m_\alpha\}$ is the tensor product of $\hat \Zz^{0, \infty}_{q_\alpha+1, n_\alpha}$ with an $\sS_{m_\alpha}$-module of rank 1 and is generated by $\zeta^\alpha$ in homological degree $n_\alpha$ and filtration degree $q_\alpha + 1$. 
Assume moreover that the predifferential $d_N$ is internal on $M$, corresponds to the one of $\hat \Zz^{0, \infty}_{q_\alpha+1, n_\alpha}$ on $\hat \Zz^{0, \infty}_{q_\alpha+1, n_\alpha}\{m_\alpha\}$ and satisfies $d_N(\xi^\alpha) + \zeta^\alpha \in M \amalg \tilde S_{i-1}$. 

Then $N$ can be chosen to be equal to $M \amalg S$, where $S$ has the same properties as $\tilde S$ and such that the cokernels of the $S_{i-1} \rightarrowtail S_i$ are free $\sS$-modules (that is $S$ is has in Proposition \ref{prop: cofibration between complete gr-dg S-modules}).

Under these hypotheses, the map $(\tilde M,\, d_{\tilde M}) \to (\tilde N,\, d_{\tilde N})$ is a cofibration as a retract of a cofibration.
\end{prop}

\begin{proof}
We copy the proof of Lemma 39 in \cite{MerkulovVallette2}, adapting the setting. Let $\overline{M}(m)$ denote the set of equivalence classes of elements in $\tilde M(m)$ under the action of $\ringK[\sS_m]^\times$. 
For simplicity, we use the generic notation $\overline{M}$. 
We choose a set of representatives $\{ \bar x_i\}_{i \in \Ii}$. Let $M$ be the free $\sS$-module generated by the $\{ \bar x_i\}_{i \in \Ii}$. The generator associated with $\bar x_i$ will be denoted by $x_i$. For any $\tilde x \in \tilde M(m)$, we consider the subgroup $\sS_{\tilde x} \coloneqq \{ \sigma \in \sS_m \ |\ \tilde x\cdot \sigma = \chi(\sigma) \tilde x,\ \chi(\sigma) \in \ringK \}$. In this case, $\chi$ is a character of $\sS_{\tilde x}$ (that is a group homomorphism $\sS_{\tilde x} \to \ringK^\times$). We define the following element of $M$:
\[ \texttt{N}(\bar x_i) \coloneqq \frac{1}{|\sS_{\bar x_i}|} \sum \chi(\sigma^{-1}) \cdot x_i \sigma, \]
where the sum runs over $\sigma \in \sS_{\bar x_i}$ ($\texttt{N}$ preserves the filtration). The image under the boundary map $d_{\tilde M}$ of an $\bar x_i$ is a (convergent) sum $\sum \lambda_{j} \bar x_{j} \sigma_j$. We define the boundary map $d_M$ on $M$ by
\[ d_M (x_i) \coloneqq \sum_j \lambda_j \left(\frac{1}{|\sS_{\bar x_i}|} \sum \chi(\sigma^{-1}) \cdot \texttt{N}(\bar x_j)\sigma\right) \sigma_j, \]
where the second sum runs over $\sigma \in \sS_{\bar x_i}$ and the sum lies in the complete gr-dg module since $\texttt{N}$ preserves the filtration. Finally, we define the maps of complete gr-dg $\sS$-modules $(M,\, d_M) \to (\tilde M,\, d_{\tilde M})$ by $x_i \mapsto \bar x_i$ and $(\tilde M,\, d_{\tilde M}) \to (M,\, d_M)$ by $\bar x_i \mapsto \texttt{N}(\bar x_i)$. They form a deformation retract, which preserves the filtrations and the different properties on the cokernels also hold.
\end{proof}

\begin{remark}
\label{rem: example of cofibrations}
It is already implicitly written before but we recall the following fact in order to make references.
\begin{enumerate}
\item
Considering the pushout of complete gr-dg $\sS_m$-modules
\[
\begin{tikzcd}
\hat \Zz^{1, \infty}_{q, n}\{m\} \ar{d}{i} \ar{r}{0} & 0\\
\hat \Bb^{1, \infty}_{q, n}\{m\} & 
\end{tikzcd}
\]
we get that the map $0 \to \hat \Zz^{1, \infty}_{q, n+1}\{m\}$ is a cofibration.
\item
Considering the composition of cofibrations $0 \to \hat \Zz^{1, \infty}_{q, n}\{m\} \to \hat \Bb^{1, \infty}_{q, n} \{m\} = \hat \Zz^{0, \infty}_{q, n+1}\{m\} \amalg \hat \Zz^{0, \infty}_{q+1, n}\{m\}$, we obtain that the retracts $0 \to \hat \Zz^{0, \infty}_{q, n+1}\{m\}$ and $0 \to \hat \Zz^{0, \infty}_{q+1, n}\{m\}$ are cofibrations.
\end{enumerate}
\end{remark}

\begin{thm}[Theorem 3.3 in \cite{sC95}, Theorem 11.3.2 in \cite{pH03}]
\label{thm: transfer theorem}
Let $\Cc$ is a cofibrantly generated model category, with $\texttt{I}$ as the set of generating cofibrations, $\texttt{J}$ as the set of generating trivial cofibrations and $W$ as class of weak equivalences. Let $\Dd$ be a complete and cocomplete category and let $F : \Cc \rightleftharpoons \Dd : G$ be a pair of adjoint functors. Suppose further that:
\begin{enumerate}
\item
$G$ preserves filtered $\aleph_1$-colimits,
\item
$G$ maps relative $F \texttt{J}$-cell complexes to weak equivalences.
\end{enumerate}
Then the category $\Dd$ is endowed with a cofibrantly generated model structure in which $F \texttt{I}$ is the set of generating cofibrations, $F \texttt{J}$ is the set of generating trivial cofibrations, and the weak equivalences are the maps that $G$ takes into a weak equivalence in $\Cc$. We have that fibrations are precisely those maps that $G$ takes into a fibration in $\Cc$. Moreover $(F,\, G)$ is a Quillen pair with respect to these model structures.
\end{thm}

In order to apply this theorem, we prove several results. The first one concerns the adjunction between the free curved operad $\cfree$ and the corresponding forgetful functor.

\begin{prop}
\label{prop: monad and curved operads}
The adjunction between the free curved operad and the forgetful functor (see Theorem \ref{thm: free curved operad})
\[
\xymatrix{\cfree : \capgSA \ar@<.5ex>@^{->}[r] & \mathsf{Curved\ operads} : U \ar@<.5ex>@^{->}[l]}
\]
provides a monad $U \cdot \cfree$ whose category of algebras is naturally isomorphic to the category of curved operads.
\end{prop}

\begin{proof}
We apply the crude monadicity theorem that we recall from \cite[Section 3.5]{BW05}. The functor $U : \mathsf{Curved\ operads} \to \capgSA$ is monadic if it satisfies the hypotheses:
\begin{itemize}
\item
the functor $U$ has a left adjoint,
\item
the functor $U$ reflects isomorphisms,
\item
the category $\mathsf{Curved\ operads}$ has coequalizers of those reflexive pairs $(f,\, g)$ for which $(U f,\, U g)$ is a coequalizer and $U$ preserves those coequalizers.
\end{itemize}
The free functor is left adjoint to the forgetful functor $U$. The forgetful functor clearly reflects isomorphisms. The coequalizer of a pair $f,\, g : \Oo \to \Pp$ in the category of curved operads is given by the quotient map
\[
(\Pp,\, d_\Pp,\, \theta_\Pp) \to (\Pp/\left(\im(f-g) \right),\, \bar d_\Pp,\, \bar \theta_\Pp)
\]
where $\left(\im(f-g)\right)$ is the (operadic) ideal generated by $\im(f-g)$, the map $\bar d_\Pp$ induced by $d_\Pp$ is well-defined since $\left(\im(f-g)\right)$ is stable under $d_\Pp$ and the element $\bar \theta_\Pp$ is the class of $\theta_\Pp$ in $\Pp/\left(\im(f-g) \right)$. The element $\bar \theta_\Pp$ is a curvature since $\theta_\Pp$ is. When $(f,\, g)$ is a reflexive pair, the ideal generated by $\im(f-g)$ is equal to $\im(f-g)$. It follows that the third condition is satisfied since the coequalizers in $\capgSA$ are given by $\Pp/\im(f-g)$. 
Hence the forgetful functor $U$ is monadic and the result follows from the crude monadicity theorem.
\end{proof}

\begin{prop}
\label{prop: limits and colimits for curved operads}
The category of curved operads has all limits and small colimits.
\end{prop}

\begin{proof}
By Proposition 4.3.1 in \cite{fB94} and Proposition \ref{prop: monad and curved operads}, the category of curved operads admits the same type of limits as the category $\capgSA$ which is complete (and they are preserved by $U$).

By Proposition 4.3.2 in \cite{fB94} and Proposition \ref{prop: monad and curved operads}, if some colimits present in $\capgSA$ are preserved by the monad $U \cdot \free$, the category of curved operads admits the same type of colimits and they are preserved by $U$. Using the fact that $\circ$ preserves filtered colimits in each variable, since $\otimes$ does, and that colimits commutes with colimits, we get by the construction of the underlying $\sS$-module of the free complete operad given in Section \ref{section: free operad} that the monad $U' \cdot \free$, where $U'$ is the forgetful functor from complete gr-dg operad to complete gr-dg $\sS$-modules, commutes with filtered colimits. Similarly, the same is true when we replace $\free$ by $\free_+$ and the forgetful functor by the corresponding forgetful functor. Then, the free curved operad $\cfree (M,\, d_M)$ described in Section \ref{section: free curved operad} is given by a quotient of $\free_+ (M,\, d_M)$ which is a coequalizer (hence a colimit). We finally get that the monad $U \cdot \cfree$ commutes with filtered colimits. 
We therefore obtain that the category of curved operads admits filtered colimits. It is therefore enough to check that it admits pushouts (by Theorem 1 in Chapter IX of \cite{sML71}, a category with filtered colimits and pushouts has all small colimits). The pushouts can be computed explicitly. The construction is well-known in the category of operads in $\Aa$ (see for example the proof of Theorem 1.13 in \cite{GetzlerJones}). Given two maps of filtered operads $\mathbf{c} \to \mathbf{a}$ and $\mathbf{c} \to \mathbf{b}$ and the pushout in operads $\mathbf{F_0 a \vee_{F_0 c} F_0 b}$. We endow the pushout with the filtration given by $\im\left( F_p \free(\mathbf{a} \amalg \mathbf{b}) \to \mathbf{F_0 a \vee_{F_0 c} F_0 b} \right)$ which is the initial one such that there are filtered maps $\mathbf{a} \to \mathbf{a \vee_{c} b}$ and $\mathbf{b} \to \mathbf{a \vee_{c} b}$ and $\mathbf{a \vee_{c} b}$ is a filtered operad. It is easy to check that this defines the pushout in filtered operads. Taking completion gives the pushout in complete operads and this extends in a straightforward manner to the gr-dg setting. In order to obtain the pushout of two maps $(\mathbf{c},\, d_\mathbf{c},\, \theta_{\mathbf{c}}) \to (\mathbf{a},\, d_\mathbf{a},\, \theta_{\mathbf{a}})$ and $(\mathbf{c},\, d_\mathbf{c},\, \theta_{\mathbf{c}}) \to (\mathbf{b},\, d_\mathbf{b},\, \theta_{\mathbf{b}})$ in the curved setting, the only difference is to consider the quotient $\mathbf{a \vee_{c} b}/(\theta_{\mathbf{a}} - \theta_{\mathbf{b}})$ so that
\[ \mathbf{(\mathbf{a},\, d_\mathbf{a},\, \theta_{\mathbf{a}}) \vee_{(\mathbf{c},\, d_\mathbf{c},\, \theta_{\mathbf{c}})} (\mathbf{b},\, d_\mathbf{b},\, \theta_{\mathbf{b}})} \coloneqq (\mathbf{a \vee_{c} b}/(\theta_{\mathbf{a}} - \theta_{\mathbf{b}}),\, d_{\mathbf{a \vee_{c} b}/(\theta_{\mathbf{a}} - \theta_{\mathbf{b}})},\, \bar \theta_{\mathbf{a}} = \bar \theta_{\mathbf{b}}). \]
\end{proof}

We now apply the transfer Theorem \ref{thm: transfer theorem} to the adjunction
\[
\xymatrix{\cfree : \capgSA \ar@<.5ex>@^{->}[r] & \mathsf{Curved\ operads} : U. \ar@<.5ex>@^{->}[l]}
\]
We have seen in the beginning of Section \ref{section: model structure on complete curved operads} that the category on the left-hand side is a cofibrantly generated model category. In order to apply Theorem \ref{thm: transfer theorem}, we have to understand $\cfree \texttt{J}$-cell complexes, with
\[ \cfree \texttt{J} = \{ \cfree(0) \to \cfree(\hat \Zz^{0, \infty}_{q, n}(m)) \}_{n \in \zZ,\, q \in \nN,\, m \in \nN}. \]
Then, we study $\cfree \texttt{I}$-cell complexes, with
\[ \cfree \texttt{I} = \{ \cfree(\hat \Zz^{1, \infty}_{q, n}(m)) \to \cfree(\hat \Bb^{1, \infty}_{q, n}(m)) \}_{n \in \zZ,\, q \in \nN,\, m \in \nN} \]
in order to describe cofibrant objects. 
The classes $\cfree \texttt{I}$, resp. $\cfree \texttt{J}$, will be the generating cofibrations, resp. generating acyclic cofibrations.

\begin{lemma}
\label{lemma: relative cfreeJ-cell complex}
A morphism of curved operads is a relative $\cfree \texttt{J}$-cell complex if and only if it is a map $\Oo \to \Oo \vee \cfree(Z,\, d_Z)$, where $(Z,\, d_Z)$ is a complete gr-dg $\sS$-module equal to a direct sum of complete gr-dg $\sS$-modules $\hat \Zz^{0, \infty}_{q, n}(m)$. In particular, $(Z,\, d_Z)$ is a free $\sS$-module and it is gr-acyclic (that is to say its gr-homology is $0$) and it satisfies $\ker(d_Z) = \{ 0\}$. 
Explicitly, we can write
\[ \cfree (Z,\, d_Z) \cong \left(\free (\vartheta I \amalg (Z/ \im {d_Z}^2)),\, d,\, \vartheta\right), \]
with $d \vartheta = 0$ and $d^2 = [\vartheta,\, -]$.
\end{lemma}

\begin{proof}
We have $\cfree(0) = \free(\vartheta I)/ (\im(0^2 - [\vartheta,\, -])) \cong \free(\vartheta I)$. Pushouts of elements of $\cfree J$ are therefore as follows
\[
\begin{tikzcd}
\free(\vartheta I) \ar{d}{\vee_\alpha \cfree(j_\alpha)} \ar{r} & \Oo \ar{d}\\
\bigvee_\alpha \cfree(\hat \Zz^{0, \infty}_\alpha) \ar{r} & \Oo \vee \left( \bigvee_\alpha \cfree(\hat \Zz^{0, \infty}_\alpha) \right),
\end{tikzcd}
\]
with each $\hat \Zz^{0, \infty}_\alpha$ equal to a $\hat \Zz^{0, \infty}_{q, n}(m)$. Since the coproduct of free curved operads is the free curved operad on the sum of their generating modules (see the proof of Proposition \ref{prop: limits and colimits for curved operads}), the composite of two such pushouts is equal to $\Oo \to \Oo \vee \cfree \left( \amalg_\alpha \hat \Zz^{0, \infty}_\alpha \coprod \amalg_\beta \hat \Zz^{0, \infty}_\beta \right)$. Hence a transfinite composition of such pushouts has the form $\Oo \to \Oo \vee \cfree \left( Z \right)$, with $Z$ a gr-acyclic gr-dg $\sS$-module whose components are free complete gr-dg $\sS$-modules $\hat \Zz^{0, \infty}_{q, n}(m)$. 
To prove the isomorphism given an explicit description of $\cfree (Z,\, d_Z)$, we define a map $\vartheta I \amalg Z \to \free(\vartheta I \amalg Z/ \im {d_Z}^2)$ by sending $\vartheta$ to $\vartheta$ and an element $d^k x$, for a generator $x$ of a free gr-dg $\sS$-module $\hat \Zz^{0, \infty}_{q, n}(m)$, to
\[ d^k x \mapsto \underbrace{[\vartheta,\, \dots [\vartheta}_{\lfloor \frac{k}{2} \rfloor \text{ times}},\, d^{k- 2\lfloor \frac{k}{2} \rfloor} x] \dots ].\]
By the universal property of the free complete gr-dg operad, we obtain a (surjective) map
\[ \free(\vartheta I \amalg Z) \to \free(\vartheta I \amalg Z/ \im {d_Z}^2)\]
and a direct computation shows that the ideal $(\im({d_Z}^2-[\vartheta,\, -]))$ is sent to zero. It is easy to describe an inverse to this map in complete $\sS$-modules.
\end{proof}

\begin{thm}
\label{thm: model structure on curved operads}
The category of complete curved operads is endowed with a cofibrantly generated model category structure where the generating (acyclic) cofibrations are the images under the free functor of the generated (acyclic) cofibrations. A map $f : \Oo \to \Pp$ is a
\begin{itemize}
\item
weak equivalence if and only if, in every arity, it is a graded quasi-iso\-mor\-phism of gr-dg $\sS$-modules,
\item
fibration if and only if, in every arity, it is a gr-surjection,
\item
cofibrations if and only if it has the left lifting property with respect to acyclic fibrations.
\end{itemize}
Moreover $(\cfree,\, U)$ is a Quillen pair with respect to the cofibrantly model structures. The generating cofibrations are the maps $\cfree \texttt{I}$ and the generating acyclic cofibrations are the maps $\cfree \texttt{J}$.
\end{thm}

\begin{proof}
By Proposition \ref{prop: limits and colimits for curved operads}, in order to apply Theorem \ref{thm: transfer theorem}, it is enough to show that $U$ preserves filtered $\aleph_1$-colimits and that it maps relative $\cfree \texttt{J}$-cell complexes to weak equivalences. 
We have already seen in the proof of Proposition \ref{prop: limits and colimits for curved operads} that $U$ preserves filtered colimits; in particular, it preserves filtered $\aleph_1$-colimits. 
By Lemma \ref{lemma: relative cfreeJ-cell complex}, a relative $\cfree \texttt{J}$-cell complex can be written as a map $j : \Oo \to \Oo \vee \left( \bigvee_\alpha \cfree(\hat \Zz^{0, \infty}_{\alpha})\right) = \Oo \vee \cfree\left(\amalg_\alpha \hat \Zz^{0, \infty}_{\alpha} \right)$. 
Using the description of the pushout of curved operads in the proof of Proposition \ref{prop: limits and colimits for curved operads}, we can compute
\begin{align*}
\Oo \vee \cfree\left(\amalg_\alpha \hat \Zz^{0, \infty}_{\alpha} \right) & \cong \left(\Oo \vee_{gr\text{-}dg\ op} \free_+\left(\amalg_\alpha \hat \Zz^{0, \infty}_{\alpha}/\im {d_\alpha}^2 \right)\right)/(\theta_\Oo - \vartheta)\\
& \cong \Oo \vee_{gr\text{-}dg\ op} \free \left(\amalg_\alpha \hat \Zz^{0, \infty}_{\alpha}/\im {d_\alpha}^2 \right).
\end{align*}

The functor
\[ \Gr : \compa(\Aa) \to (\dg\Aa)^{\ob \nN},\ (V,\, F,\, d_V) \mapsto (\Gr_\bullet V,\, \Gr d_V) \]
commutes with direct sums (see for example \cite[Proposition 7.3.8]{bF17}) and satisfies, for gr-flat complete modules, $\Gr M \otimes_{\Gr} \Gr N \cong \Gr (M \hat \otimes N)$ (see Lemma \ref{lemma:gr-flat is monoidal subcat}). By Maschke's Theorem ($\ringK$ is a field of characteristic $0$), any $\ringK[\sS_\ast]$-module is flat. It follows that when $M$ and $N$ are two gr-dg $\sS$-modules, we have $\Gr M \circ \Gr N \cong \Gr (M \circ N)$. Moreover the functor $\Gr$ preserves filtered colimits. Indeed, it is the composition
\[ \compa(\Aa) \xrightarrow{i_1} \Filt(\pg\Aa) \xrightarrow{i_2} (\pg\Aa)^{\nN^{\op}} \xrightarrow{q_1} (\dg\Aa)^{\ob \nN}, \]
where $i_1$ and $i_2$ are inclusions and $q_1$ sends a directed sequence to the ($\nN$-indexed) product of the cokernel of the maps appearing in the directed sequence. By Proposition $1.62$ in \cite{AR94}, the functor $i_2$ preserves filtered colimits. 
Since cokernels are coequalizers (therefore colimits) in $(\pg\Aa)^{\nN^{\op}}$ and $(\dg\Aa)^{\ob \nN}$, and colimits commute with colimits, the map $q_2$ preserves all colimits. Then, because of the presence of the quotient map $q_1$, when $\colim_i M_i$ is a filtered colimit, we have $\Gr \colim_i M_i = (q_1 \cdot i_2)( \colim_i i_1(M_i))$ (even if $i_1$ does not preserve ($\aleph_0$-)filtered colimits). 
The free operad functor is obtained as a filtered colimit of finite coproducts of $\circ$ monoidal products (see Section \ref{section: free operad}), so $\Gr$ commutes with the free curved operad functor. 

The coproduct of two operads can be computed explicitly and is a quotient of the free operad on the direct sum of the two underlying $\sS$-modules of the operads (see the proof of Theorem 1.13 in \cite{GetzlerJones}). The graded functor $\Gr$ therefore preserves the coproduct of two operads. We get
\[
\Gr \left(\Oo \vee_{gr\text{-}dg\ op} \free \left(\amalg_\alpha \hat \Zz^{0, \infty}_{\alpha}/\im {d_\alpha}^2 \right)\right) \cong \Gr (\Oo) \vee_{dg\ op} \free \left(\amalg_\alpha \Gr \left(\hat \Zz^{0, \infty}_{\alpha}/\im {d_\alpha}^2\right) \right).
\]
We get, using the fact that $\ringK$ is a field of characteristic $0$ and the arguments given in \cite[Theorem 3.2]{Hinich03}, that the map $\Gr(U(j))$ is a quasi-isomorphism.
\end{proof}

We conclude this section by computing the cofibrant objects in the model category of complete curved operads. 
We say until the end of this section that a curved operad $\Oo$ is \emph{cofibrant} when the unique morphism of curved operads $\free (\vartheta I) \to \Oo$ is a cofibration. 
We recall that the notion of quasi-free complete curved operad is given in Definition \ref{defi: quasi-free}.

\begin{prop}
\label{prop: cofibrant complete curved operad}
A complete curved operad is cofibrant if and only if it is a retract of a quasi-free complete curved operad $\left(\free_+(S),\, d\right)/\left(\im\left( d^2-[\vartheta,\, -]\right)\right)$, where $S$ is a complete $\sS$-module endowed with an exhaustive filtration
\[ S_0 = \{ 0\} \subset S_1 \subset S_2 \subset \dots \subset \colim_i S_i = S \]
of free $\sS$-modules such that $S_{i-1} \rightarrowtail S_i$ are split monomorphisms of complete $\sS$-modules with cokernels isomorphic to a sum of complete $\sS$-modules
\[ S_i /S_{i-1} \cong \coprod_\alpha \left(\xi^\alpha \cdot \ringK[\sS_{m_\alpha}] \amalg \hat \Zz^{0, \infty}_{q_\alpha+1, n_\alpha}(m_\alpha)\right) \]
where $\xi^\alpha$ is in homological degree $n_\alpha +1$ and filtration degree $q_\alpha$. The predifferential $d$ is the one of $\hat \Zz^{0, \infty}_{q_\alpha+1, n_\alpha}(m_\alpha)$ on $\hat \Zz^{0, \infty}_{q_\alpha+1, n_\alpha}(m_\alpha)$ and 
\[
d(\xi^\alpha) + \zeta^\alpha \in \left(\free_+ (S_{i-1}),\, d\right)/\left( \im\left(d^2-[\vartheta,\, -]\right)\right),
\]
with $\zeta^\alpha$ is a generator of the gr-dg $\sS$-module $\hat \Zz^{0, \infty}_{q_\alpha+1, n_\alpha}(m_\alpha)$. 
\end{prop}

\begin{proof}
By Proposition 2.1.18 in \cite{mH91}, cofibrations are retracts of relative $\cfree \texttt{I}$-cell complexes. We therefore study the pushouts of elements of $\cfree \texttt{I}$ of the form
\[
\begin{tikzcd}
\bigvee_\alpha \cfree \left(\hat \Zz^{1, \infty}_{\alpha}\right) \ar{d}{\vee_\alpha \cfree(i_\alpha)} \ar{r}{f} & \left(\free_+ (S_{i-1}), d_{i-1}\right)/\left( \im\left({d_{i-1}}^2-[\vartheta,\, -]\right) \right) \ar{d}\\
\bigvee_\alpha \cfree(\hat \Bb^{1, \infty}_\alpha) \ar{r} & \Pp,
\end{tikzcd}
\]
with each $\hat \Zz^{1, \infty}_{\alpha}$ equal to a $\hat \Zz^{1, \infty}_{q_\alpha, n_\alpha}(m_\alpha)$ and $\hat \Bb^{1, \infty}_{\alpha}$ equal to a $\hat \Bb^{1, \infty}_{q_\alpha, n_\alpha}(m_\alpha)$. We denote by $z^\alpha$ the image under $f$ of the generating (as a gr-dg $\sS$-modules) element of $\hat \Zz^{1, \infty}_{q_\alpha, n_\alpha}(m_\alpha)$. If we denote by $\xi^\alpha$ and by $\zeta^\alpha$ the generating (as a gr-dg $\sS$-modules) elements of $\hat \Bb^{1, \infty}_{q_\alpha, n_\alpha}(m_\alpha) = \hat \Zz^{0, \infty}_{q_\alpha, n_\alpha+1}(m_\alpha) \amalg \hat \Zz^{0, \infty}_{q_\alpha+1, n_\alpha}(m_\alpha)$, and by $d^k \xi^\alpha$, resp. by $d^k \zeta^\alpha$, their successives predifferentials, the pushout $\Pp$ is equal to
\[ \left(\free_+ (S_{i-1}) \vee_{compl.\, op.} \free\left( \xi^\alpha \cdot \ringK[\sS_{m_\alpha}] \coprod \hat \Zz^{0, \infty}_{q_\alpha+1, n_\alpha}(m_\alpha)\right),\ d_i \right)/\left( \im\left({d_{i}}^2-[\vartheta,\, -]\right)\right), \]
where $d_i$ is the derivation defined on $S_{i-1}$ by $d_{i-1}$, the differential on $\hat \Zz^{0, \infty}_{q_\alpha+1, n_\alpha}(m_\alpha)$ is the one of $\hat \Zz^{0, \infty}_{q_\alpha+1, n_\alpha}(m_\alpha)$ and $d_i \xi^\alpha = z^\alpha - \zeta^\alpha$. 
By induction, we get the result.
\end{proof}

\begin{cor}
Equivalently, a complete curved operad is cofibrant if and only if it is a retract of a quasi-free complete curved operad $\left(\free_+(\tilde S),\, d\right)$, where $\tilde S$ is a complete $\sS$-module endowed with an exhaustive filtration
\[ \tilde S_0 = \{ 0\} \subset \tilde S_1 \subset \tilde S_2 \subset \dots \subset \colim_i \tilde S_i = \tilde S \]
of free $\sS$-modules such that $\tilde S_{i-1} \rightarrowtail \tilde S_i$ are split monomorphisms of complete $\sS$-modules with cokernels isomorphic to a sum of complete $\sS$-modules
\[ \tilde S_i /\tilde S_{i-1} \cong \coprod_\alpha \left(\xi^\alpha \cdot \ringK[\sS_{m_\alpha}] \amalg \zeta^\alpha \cdot \ringK[\sS_{m_\alpha}]\right) \]
where $\xi^\alpha$ is in homological degree $n_\alpha +1$ and filtration degree $q_\alpha$ and $\zeta^\alpha$ is in homological degree $n_\alpha $ and filtration degree $q_\alpha+1$. The predifferential $d$ is such that $d(\xi^\alpha) + \zeta^\alpha \in \left(\free_+ (\tilde S_{i-1}),\, d_{i-1}\right)$, and $d(\zeta^\alpha)$ is obtained by the fact that $d^2(\xi^\alpha) = [\vartheta,\, \xi^\alpha]$. 
\end{cor}

\begin{prop}
\label{prop: quasi-free are retract of good quasi-free}
Any quasi-free complete curved operad $(\free_+(X),\, \partial)$ is a retract of a quasi-free complete curved operad $(\free_+(\tilde S),\, \partial')$, where the components of $\tilde S$ are free $\sS$-modules. Moreover, assume that $X$ is endowed with an exhaustive filtration
\[ X_0 = \{ 0\} \subset X_1 \subset X_2 \subset \dots \subset \colim_i X_i = X \]
such that $X_{i-1} \rightarrowtail X_i$ are split monomorphisms of complete $\sS$-modules with cokernels isomorphic to a sum of complete $\sS$-modules
\[ X_i /X_{i-1} \cong \coprod_\alpha \left(\ringK\{\xi^\alpha\} \amalg \ringK\{\zeta^\alpha\}\right), \]
where $\ringK\{\xi^\alpha\}$ is an $\sS_{m_\alpha}$-module of a rank 1 generated by $\xi^\alpha$ in homological degree $n_\alpha +1$ and filtration degree $q_\alpha$ and $\ringK\{\zeta^\alpha\}$ is an $\sS_{m_\alpha}$-module of a rank 1 generated by $\zeta^\alpha$ in homological degree $n_\alpha $ and filtration degree $q_\alpha+1$. 
Assume moreover that the predifferential $\partial$ is such that $\partial(\xi^\alpha) + \zeta^\alpha \in \left(\free_+ (X_{i-1}), \partial\right)$, and $\partial(\zeta^\alpha)$ is obtained by the fact that $\partial^2(\xi^\alpha) = [\vartheta,\, \xi^\alpha]$. 
Then $\tilde S$ can be chosen with the same properties and such that the cokernels of the $\tilde S_{i-1} \rightarrowtail \tilde S_i$ are free $\sS$-modules.

Under these hypotheses, the complete curved operad $(\free_+(X),\, \partial)$ is cofibrant, as a retract of a cofibrant complete curved operad.
\end{prop}

\begin{proof}
The proof is similar to the proof of Proposition \ref{prop: cofibration are retract of good cofibration}. 
Let $\overline{X}(m)$ denote the set of equivalence classes under the action of $\ringK[\sS_m]^\times$. 
We choose a set of representatives $\{ \bar x_i\}_{i \in \Ii}$. Let $\tilde S$ be the free $\sS$-module generated by the $\{ \bar x_i\}_{i \in \Ii}$. The generator associated with $\bar x_i$ will be denoted by $s_i$. For any $x \in X(m)$, we consider the subgroup $\sS_x \coloneqq \{ \sigma \in \sS_m \ |\ x\cdot \sigma = \chi(\sigma) x,\ \chi(\sigma) \in \ringK \}$. In this case, $\chi$ is a character of $\sS_x$. We define the following element of $\tilde S$:
\[ \texttt{N}(\bar x_i) \coloneqq \frac{1}{|\sS_{\bar x_i}|} \sum \chi(\sigma^{-1}) \cdot s_i \sigma, \]
where the sum runs over $\sigma \in \sS_{\bar x_i}$ ($\texttt{N}$ preserves the filtration). The image under the boundary map $\partial$ of an $\bar x_i$ is a (potentially infinite) sum of trees $\sum T(\bar x_{i_1},\, \ldots,\, \bar x_{i_k})$. We define the boundary map $\partial'$ on $\free_+(\tilde S)$ by
\[ \partial' (s_i) \coloneqq \sum \frac{1}{|\sS_{\bar x_i}|} \sum \chi(\sigma^{-1}) \cdot T\left(\texttt{N}(\bar x_i),\, \ldots,\, \texttt{N}(\bar x_{i_k}) \right) \sigma, \]
where the second sum runs over $\sigma \in \sS_{\bar x_i}$ and the sum lies in the complete gr-dg module since $\texttt{N}$ preserves the filtration. Finally, we define the maps of curved operads $\free_+(\tilde S) \to \free_+(X)$ by $s_i \mapsto \bar x_i$ and $\free_+(X) \to \free_+(\tilde S)$ by $\bar x_i \mapsto N(\bar x_i)$. They form a deformation retract, which preserves the filtration on $X$ and the different properties on the cokernels also hold.
\end{proof}

\subsection{Model structure on categories of complete curved algebras}
\label{section: model cat for complete curved algfebras}

Following Hinich \cite{Hinich97}, we endow the category $\Alg(\Oo)$ of algebras over an $\sS$-split complete curved operad $(\Oo,\, d,\, \theta)$ with a model category structure. We apply Theorem \ref{thm: transfer theorem} to the free-forgetful functor adjunction between the categories of gr-dg $\ringK$-modules and $\Alg(\Oo)$.

We first describe the free-forgetful adjunction.

\begin{prop}
\label{prop: free-forgetful adjunction for curved algebras}
The forgetful functor $\# : \Alg(\Oo) \to \mathsf{compl.\, gr}\textsf{-}\mathsf{dg}\, \ringK\textsf{-}\mathsf{Mod}$ admits a left adjoint \emph{free $\Oo$-algebra functor} $F_{\Oo} : \mathsf{compl.\, gr}\textsf{-}\mathsf{dg}\, \ringK\textsf{-}\mathsf{Mod} \to \Alg(\Oo)$ given by
\[ (V,\, d_V) \mapsto F_{\Oo}(V,\, d_V) \coloneqq \left(\Oo(V)/\left(\im\left(\eta \otimes ({d_{V}}^2) - \theta \otimes \id_V\right) \right),\, d_{\overline{\Oo(V)}}\right). \]
\end{prop}

\begin{proof}
The proof is similar to the proof of Theorem \ref{thm: free curved operad}. Let $(V,\, d_V)$ be a complete gr-dg $\ringK$-module and $(A,\, d_A)$ be an $(\Oo,\, d,\, \theta)$-algebra. We denote by $U \Oo = U (\Oo,\, d,\, \theta)$ the complete gr-dg operad underlying $(\Oo,\, d,\, \theta)$. The above construction $F_{\Oo}(V,\, d_V)$ is a $U \Oo$-algebra as a quotient by the ideal $\left(\im\left( \eta \otimes ({d_{V}}^2) - \theta \otimes \id_V \right)\right)$ of the free $U\Oo$-algebra
\[ U \Oo(V,\, d_V) = \left(\amalg_{n\geq 0}\Oo(n) \otimes_{\sS_n} V^{\otimes n},\, d_{\Oo(V)}\right) \]
where $d_{\Oo(V)}$ is the gr-dg predifferential induced by the predifferentials on $\Oo$ and on $V$. It is a $(\Oo,\, d,\, \theta)$-algebra since the condition that $\theta$ is sent to ${d_{\overline{\Oo(V)}}}^2$ in $\End_{F_{\Oo}(V,\, d_V)}$ follows from the fact that we have considered the quotient by the ideal $\left( \im\left(\eta \otimes {d_{V}}^2 - \theta \otimes \id_V\right) \right)$. Indeed, in $F_{\Oo}(V,\, d_V)$
\begin{align*}
{d_{\overline{\Oo(V)}}}^2 & = {d_\Oo}^2 \otimes \id_{V^{\otimes \bullet}} + \sum_j \id_\Oo \otimes {\id_V}^{\otimes j} \otimes {d_V}^2 \otimes \id^{\otimes (\bullet -j)}\\
& = [\theta,\, -] \otimes \id_{V^{\otimes \bullet}} + \sum_j \id_\Oo \otimes {\id_V}^{\otimes j} \otimes (\theta \otimes \id_V) \otimes \id^{\otimes (\bullet -j)}\\
& = \gamma_{\overline{\Oo(V)}}\left(\theta \circ \id_{\overline{\Oo(V)}}\right).
\end{align*}
We have the adjunction
\[ \Hom_{\mathsf{gr}\textsf{-}\mathsf{dg}\, \ringK\textsf{-}\mathsf{Mod}}\left((V,\, d_V),\, (A,\, d_A)^\#\right) \cong \Hom_{U \Oo\textsf{-}\mathsf{alg.}}\left( U \Oo(V,\, d_V),\, U(A,\, d_A)\right),\]
where $U(A,\, d_A)$ is the $U \Oo$-algebra underlying $(A,\, d_A)$. Since $(A,\, d_A)$ is a $(\Oo,\, d,\, \theta)$-algebra, a morphism of $U \Oo$-algebras $U \Oo(V,\, d_V)\to U(A,\, d_A)$ automatically sends the ideal $\left(\im\left(\eta \otimes {d_{V}}^2 - \theta \otimes \id_V \right)\right)$ to $0$ and coincides (bijectively) with a morphism of $(\Oo,\, d,\, \theta)$-algebras $F_{\Oo}(V,\, d_V) \to (A,\, d_A)$. We get the bijection
\[ \Hom_{U \Oo\textsf{-}\mathsf{alg.}}\left( U \Oo(V,\, d_V),\, U(A,\, d_A)\right) \cong \Hom_{\mathsf{gr}\textsf{-}\mathsf{dg}\, \ringK\textsf{-}\mathsf{Mod}}\left(F_{\Oo}(V,\, d_V),\, (A,\, d_A)\right)\]
which gives the result.
\end{proof}

In order to apply Theorem \ref{thm: transfer theorem}, we prove several results. The first one concerns the adjunction between the free $(\Oo,\, d,\, \theta)$-algebra functor $F_\Oo$ and the forgetful functor $\#$.

\begin{prop}
\label{prop: monad and curved algebras}
The adjunction between the free $(\Oo,\, d,\, \theta)$-algebra functor and the forgetful functor (see Proposition \ref{prop: free-forgetful adjunction for curved algebras})
\[
\xymatrix{F_\Oo : \mathsf{compl.\, gr}\textsf{-}\mathsf{dg}\, \ringK\textsf{-}\mathsf{Mod} \ar@<.5ex>@^{->}[r] & \Alg(\Oo) : \# \ar@<.5ex>@^{->}[l]}
\]
provides a monad $\# \cdot F_\Oo$ whose category of algebras is naturally isomorphic to the category of $(\Oo,\, d,\, \theta)$-algebras.
\end{prop}

\begin{proof}
As in the proof of Proposition \ref{prop: monad and curved operads}, we apply the crude monadicity theorem given in \cite[Section 3.5]{BW05}. The result follows from the fact that the functor $\# : \Alg(\Oo) \to \mathsf{compl.\, gr}\textsf{-}\mathsf{dg}\, \ringK\textsf{-}\mathsf{Mod}$ is monadic. 
The free functor is left adjoint to the forgetful functor $\#$. The forgetful functor clearly reflects isomorphisms. It remains to show that the category $\Alg(\Oo)$ has coequalizers of those reflexive pairs $(f,\, g)$ for which $(\# f,\, \# g)$ is a coequalizer and $\#$ preserves those coequalizers. 
The coequalizer of a pair $f,\, g : (A,\, d_A) \to (B,\, d_B)$ in the category of $(\Oo,\, d,\, \theta)$-algebras is given by the quotient map $(B,\, d_B) \to (B/\left(\im(f-g) \right),\, \bar d_B)$ where $\left(\im(f-g)\right)$ is the ($\Oo$-)ideal generated by $\im(f-g)$ and the map $\bar d_B$ induced by $d_B$ is well-defined since $\left(\im(f-g)\right)$ is stable under $d_B$. The quotient algebra is still a $(\Oo,\, d,\, \theta)$-algebra. When $(f,\, g)$ is a reflexive pair, the ideal generated by $\im(f-g)$ is equal to $\im(f-g)$. It follows that the remaining condition is satisfied since the coequalizers in $\mathsf{compl.\, gr}\textsf{-}\mathsf{dg}\, \ringK\textsf{-}\mathsf{Mod}$ are given by $B/\im(f-g)$.
\end{proof}

\begin{prop}
\label{prop: limits and colimits for curved algebras}
The category of $(\Oo,\, d,\, \theta)$-algebras has all limits and small colimits.
\end{prop}

\begin{proof}
By Proposition 4.3.1 in \cite{fB94} and Proposition \ref{prop: monad and curved algebras}, the category of $(\Oo,\, d,\, \theta)$-algebras admits the same type of limits as the category $\mathsf{compl.\, gr}\textsf{-}\mathsf{dg}\, \ringK\textsf{-}\mathsf{Mod}$ which is complete (and they are preserved by $\#$).

Proposition 4.3.2 in \cite{fB94} and Proposition \ref{prop: monad and curved algebras} implies that if some colimits in $\mathsf{compl.\, gr}\textsf{-}\mathsf{dg}\, \ringK\textsf{-}\mathsf{Mod}$ are preserved by the monad $\# \cdot F_\Oo$, the category of $(\Oo,\, d,\, \theta)$-algebras admits the same type of colimits and they are preserved by $\#$.

Using the fact that $\otimes$ preserves filtered colimits in each variable and that colimits commutes with colimits, we get that the monad $\# \cdot F_\Oo$ preserves filtered colimits. We therefore obtain that the category of $(\Oo,\, d,\, \theta)$-algebras admits filtered colimits. It is therefore enough to check that it admits pushouts (by Theorem 1 in Chapter IX of \cite{sML71}, a category with filtered colimits and pushouts has all small colimits). The pushouts can be computed explicitly as follows. Let $I_A$ (resp. $I_B$) be the kernel of the map $\gamma_A : \Oo(A) \to A$ (resp. $\gamma_B$) and $\gamma_A^0 : \Oo(0) \to A$ (resp. $\gamma_B^0 : \Oo(0) \to B$) the algebra structure given by $0$-ary elements in $\Oo$. The coproduct of two $(\Oo,\, d,\, \theta)$-algebras $(A,\, d_A)$ and $(B,\, d_B)$ is given as usual by the quotient
\[ A \vee B \coloneqq \Oo(A^\# \amalg B^\#) / \left( I_A \amalg I_B \amalg \im\left(\gamma_A^0 - \gamma_B^0\right)\right). \]
The filtration (resp. predifferential) is induced by the filtrations (resp. predifferentials) on $A$ and $B$. It gives a $(\Oo,\, d,\, \theta)$-algebra since $A$ and $B$ are and since $\Oo$ is curved. 
Given two maps of complete $(\Oo,\, d,\, \theta)$-algebras $\phi : (C,\, d_C) \to (A,\, d_A)$ and $\psi : (C,\, d_C) \to (B,\, d_B)$, we obtain the corresponding pushout $A \vee_C B$ as the quotient of the coproduct $A\vee B$ by the ideal generated by the image of the map $\phi - \psi$.
\end{proof}

We apply the transfer Theorem \ref{thm: transfer theorem} to the adjunction
\[
\xymatrix{F_\Oo : \mathsf{compl.\, gr}\textsf{-}\mathsf{dg}\, \ringK\textsf{-}\mathsf{Mod} \ar@<.5ex>@^{->}[r] & \Alg(\Oo) : \#. \ar@<.5ex>@^{->}[l]}
\]
We have seen in the beginning of Section \ref{section: model structure on complete curved operads} that the category $ \mathsf{compl.\, gr}\textsf{-}\mathsf{dg}\, \ringK\textsf{-}\mathsf{Mod}$ is a cofibrantly generated model category with generating cofibrations $I_0^\infty$ and generating acyclic cofibrations $J_0^\infty$. In order to apply Theorem \ref{thm: transfer theorem}, we have to understand $F_\Oo J_0^\infty$-cell complexes, with
\[ F_\Oo J_0^\infty = \{ \Oo(0) \to F_\Oo(\hat \Zz^{0, \infty}_{q, n}) \}_{n \in \zZ,\, q \in \nN}. \]
Then, we study $\cfree I$-cell complexes, with
\[ \cfree I_0^\infty = \{ F_\Oo(\hat \Zz^{1, \infty}_{q, n}) \to F_\Oo(\hat \Bb^{1, \infty}_{q, n}) \}_{n \in \zZ,\, q \in \nN} \]
in order to describe cofibrant objects. 
The classes $F_\Oo I_0^\infty$, resp. $F_\Oo J_0^\infty$, will be the generating cofibrations, resp. generating acyclic cofibrations.

\begin{lemma}
\label{lemma: relative FOJ0i-cell complex}
A morphism of $(\Oo,\, d,\, \theta)$-algebras is a relative $F_\Oo J_0^\infty$-cell complex if and only if it is a map $A \to A \vee F_\Oo(Z,\, d_Z)$, where $(Z,\, d_Z)$ is a complete gr-dg module equal to a direct sum of complete gr-dg modules $\hat \Zz^{0, \infty}_{q, n}$. In particular, $(Z,\, d_Z)$ is a free module and it is gr-acyclic (that is to say its gr-homology is $0$) and it satisfies $\ker(d_Z) = \{ 0\}$. 
Explicitly, we can write
\[ F_\Oo (Z,\, d_Z) \cong \left(\Oo (Z/ \im {d_Z}^2),\, d_{F_\Oo(Z)}\right), \]
with ${d_{F_\Oo(Z)}}^2 = \theta \otimes \id_{Z/\im {d_Z}^2}$.
\end{lemma}

\begin{proof}
We have $F_\Oo(0) = \Oo(0)$. Pushouts of elements of $F_\Oo J_0^\infty$ are therefore as follows
\[
\begin{tikzcd}
\Oo(0) \ar{d}{\vee_\alpha F_\Oo(j_\alpha)} \ar{r} & A \ar{d}\\
\bigvee_\alpha F_\Oo(\hat \Zz^{0, \infty}_\alpha) \ar{r} & A \vee \left( \bigvee_\alpha F_\Oo(\hat \Zz^{0, \infty}_\alpha) \right),
\end{tikzcd}
\]
with each $\hat \Zz^{0, \infty}_\alpha$ equal to a $\hat \Zz^{0, \infty}_{q, n}$. Since the coproduct of free $(\Oo,\, d,\, \theta)$-algebras is the free $(\Oo,\, d,\, \theta)$-algebra on the sum of their generating modules (see the proof of Proposition \ref{prop: limits and colimits for curved algebras}), the composite of two such pushouts is equal to $A \to A \vee F_\Oo \left( \amalg_\alpha \hat \Zz^{0, \infty}_\alpha \coprod \amalg_\beta \hat \Zz^{0, \infty}_\beta \right)$. Hence a transfinite composition of such pushouts has the form $A \to A \vee F_\Oo \left( Z \right)$, with $Z$ a gr-acyclic gr-dg module whose components are free gr-dg modules $\hat \Zz^{0, \infty}_{q, n}$. 
The last identification is direct.
\end{proof}

We call \emph{$\sS$-split operad} what is defined to be $\Sigma$-split in \cite{Hinich97}. This definition extends without modification to complete curved operads.

\begin{thm}
\label{thm: model cat struct for curved algebras}
Let $(\Oo,\, d,\, \theta)$ be a complete curved operad which is $\sS$-split. The category $\Alg(\Oo)$ of $(\Oo,\, d,\, \theta)$-algebras is a cofibrantly generated model category with generating cofibrations $F_\Oo (I_0^\infty)$ and generating acyclic cofibrations $F_\Oo(J_0^\infty)$. The weak equivalences (resp. fibrations) are the maps that are graded quasi-isomorphisms (resp. strict surjections). 
\end{thm}

\begin{proof}
We apply Theorem \ref{thm: transfer theorem} to the free-forgetful adjunction given in Proposition \ref{prop: free-forgetful adjunction for curved algebras}. By Proposition \ref{prop: limits and colimits for curved algebras}, the category $\Alg(\Oo)$ is complete and cocomplete. We have seen in the proof of Proposition \ref{prop: limits and colimits for curved algebras} that the forgetful functor $\#$ preserves filtered colimit; in particular it preserves filtered $\aleph_1$-colimits. 
We finally have to show that $\#$ maps relative $F_\Oo(J_0^\infty)$-cell complexes to weak equivalences.  By Lemma \ref{lemma: relative FOJ0i-cell complex}, a relative $F_\Oo J_0^\infty$-cell complex can be written as a map $j : A \to A \vee F_\Oo\left(\amalg_\alpha \hat \Zz_\alpha^{0, \infty}\right)$. By means of the description of the pushouts of $(\Oo,\, d,\, \theta)$-algebras given in the proof of Proposition \ref{prop: limits and colimits for curved algebras}, we can compute
\begin{align*}
A \vee F_\Oo\left(\amalg_\alpha \hat \Zz^{0, \infty}_{\alpha} \right) & \cong A \vee \Oo\left(\amalg_\alpha \hat \Zz^{0, \infty}_{\alpha}/\im {d_\alpha}^2 \right)\\
& \cong \Oo\left(A^\# \amalg \left(\amalg_\alpha \hat \Zz^{0, \infty}_{\alpha}/\im {d_\alpha}^2 \right)\right)/\left(I_A \amalg \left(\gamma_A^0 - \id_{\Oo(0)}\right)\right).
\end{align*}
The functor
\[ \Gr : \compa(\Aa) \to (\dg\Aa)^{\ob \nN},\ (V,\, F,\, d_V) \mapsto (\Gr_\bullet V,\, \Gr d_V) \]
commutes with direct sums (see for example \cite[Proposition 7.3.8]{bF17}) and satisfies, for gr-flat complete modules, $\Gr M \otimes_{\Gr} \Gr N \cong \Gr (M \hat \otimes N)$ (see Lemma \ref{lemma:gr-flat is monoidal subcat}). Here $\ringK$ is a field of characteristic $0$, so any module is flat. It follows that the functor $\Gr$ commutes with the free $U \Oo$-algebra functor and we get
\[
\Gr \left(A \vee F_\Oo\left(\amalg_\alpha \hat \Zz^{0, \infty}_{\alpha} \right)\right) \cong \Gr (A) \vee_{dg\ \Oo\text{-}alg.} \Gr \Oo \left(\amalg_\alpha \Gr \left(\hat \Zz^{0, \infty}_{\alpha}/\im {d_\alpha}^2\right) \right).
\]
By the proof of Theorem 4.1.1 in \cite{Hinich97}, using the fact that $\Oo$ is $\sS$-split, we obtain that the map $\Gr(j^\#)$ is a quasi-isomorphism.
\end{proof}

We conclude this section by computing the cofibrant objects in the model category of complete $(\Oo,\, d,\, \theta)$-algebras. 
In the following results, a complete $(\Oo,\, d,\, \theta)$-algebra $A$ is called \emph{cofibrant} when the morphism $\Oo(0) \to A$ is a cofibration. 
We use the term \emph{quasi-free} for a complete $(\Oo,\, d,\, \theta)$-algebras whose underlying module is free when we forget the predifferential.

\begin{prop}
\label{prop: cofibrant complete curved algebra}
A complete $(\Oo,\, d,\, \theta)$-algebra is cofibrant if and only if it is a retract of a quasi-free complete $(\Oo,\, d,\, \theta)$-algebra $\left(\Oo(S),\, d\right)/\left(\im\left(\eta \otimes d^2-\theta \otimes \id\right)\right)$, where $S$ is a complete module endowed with an exhaustive filtration
\[ S_0 = \{ 0\} \subset S_1 \subset S_2 \subset \dots \subset \colim_i S_i = S \]
of free modules such that $S_{i-1} \rightarrowtail S_i$ are split monomorphisms of complete modules with cokernels isomorphic to a sum of complete modules
\[ S_i /S_{i-1} \cong \coprod_\alpha \left(\xi^\alpha \cdot \ringK \amalg \hat \Zz^{0, \infty}_{q_\alpha+1, n_\alpha}\right) \]
where $\xi^\alpha$ is in homological degree $n_\alpha +1$ and filtration degree $q_\alpha$. The predifferential $d$ is the one of $\hat \Zz^{0, \infty}_{q_\alpha+1, n_\alpha}$ on $\hat \Zz^{0, \infty}_{q_\alpha+1, n_\alpha}$ and
\[
d(\xi^\alpha) + \zeta^\alpha \in \left(\Oo (S_{i-1}),\, d\right)/\left( \im\left(\eta \otimes d^2-\theta\otimes \id \right)\right), 
\]
with $\zeta^\alpha$ is a generator of the gr-dg module $\hat \Zz^{0, \infty}_{q_\alpha+1, n_\alpha}$. 
\end{prop}

\begin{proof}
The proof is similar to the proof of Proposition \ref{prop: cofibrant complete curved operad}.
\end{proof}

\begin{cor}
Equivalently, a complete $(\Oo,\, d,\, \theta)$-algebra is cofibrant if and only if it is a retract of a quasi-free complete $(\Oo,\, d,\, \theta)$-algebra $\left(\Oo(\tilde S),\, d\right)$, where $\tilde S$ is a complete module endowed with an exhaustive filtration
\[ \tilde S_0 = \{ 0\} \subset \tilde S_1 \subset \tilde S_2 \subset \dots \subset \colim_i \tilde S_i = \tilde S \]
of free modules such that $\tilde S_{i-1} \rightarrowtail \tilde S_i$ are split monomorphisms of complete modules with cokernels isomorphic to a sum of complete modules
\[ \tilde S_i /\tilde S_{i-1} \cong \coprod_\alpha \left(\xi^\alpha \cdot \ringK \amalg \zeta^\alpha \cdot \ringK\right) \]
where $\xi^\alpha$ is in homological degree $n_\alpha +1$ and filtration degree $q_\alpha$ and $\zeta^\alpha$ is in homological degree $n_\alpha $ and filtration degree $q_\alpha+1$. The predifferential $d$ is such that $d(\xi^\alpha) + \zeta^\alpha \in \left(\Oo (\tilde S_{i-1}),\, d_{i-1}\right)$, and $d(\zeta^\alpha)$ is obtained by the fact that $d^2(\xi^\alpha) = \theta\otimes \xi^\alpha$. 
\end{cor}

\subsection{Homotopy category of algebras over a curved operad}
\label{section: homotopy category of algebras over a curved operad}

We now come to the study of the homotopy category. We denote by $\Hoalg(\Oo)$ the homotopy category of $\Alg(\Oo)$. We show how a morphism of complete curved operads can provide a Quillen adjunction or a Quillen equivalence between the model category structures.

\begin{defi}
Let $\alpha : (\Oo,\, d,\, \theta) \to (\Oo',\, d',\, \theta')$ be a morphism of complete curved operads.
\begin{enumerate}
\item
We denote by $\alpha_* : \Alg(\Oo') \to \Alg(\Oo)$ the \emph{direct image functor} given by precomposition
\[ (\Oo,\, d,\, \theta) \xrightarrow{\alpha} (\Oo',\, d',\, \theta') \to \End_A. \]
This functor is exact since only the algebra structure changes.
\item
We denote by $\alpha^* : \Alg(\Oo) \to \Alg(\Oo')$ the \emph{inverse image functor}, left adjoint to $\alpha_*$, given by the following definition: for $(A,\, d_A)$ in $\Alg(\Oo)$,
\[ \alpha^*(A,\, d_A) \coloneqq \Oo'(A^\#)/\left(\alpha(\id_{A^\#})(I_A) \right). \]
We can check that $\alpha^*(A,\, d_A)$ satisfies ${d_{\alpha^*(A,\, d_A)}}^2 = \theta' \otimes \id_{\alpha^*(A,\, d_A)}$. By a computation similar to the one in the proof of Proposition \ref{prop: free-forgetful adjunction for curved algebras}, this follows from the fact that $\im(\eta \otimes {d_A}^2 - \theta \otimes \id_A) \subset I_A$ by the fact that $(A,\, d_A)$ is a $(\Oo,\, d,\, \theta)$-algebra and since $\alpha(\theta) = \theta'$.
\end{enumerate}
\end{defi}

This pair of adjoint functors form a Quillen pair and therefore provides an adjunction on the level of the homotopy categories.

\begin{thm}
\label{thm: inverse and direct image functor are adjoint derived functors}
Inverse and direct image functors form a Quillen pair, that is we have the adjunction
\[
\xymatrix{\lL \alpha^* : \Hoalg(\Oo) \ar@<.5ex>@^{->}[r] & \Hoalg(\Oo') : \rR \alpha_* = \alpha_*. \ar@<.5ex>@^{->}[l]}
\]
\end{thm}

\begin{proof}
From Proposition 8.5.3 in \cite{pH03}, it is enough to prove that the left adjoint functor $\alpha^*$ preserves cofibrations and acyclic cofibrations. By means of the fact that (acyclic) cofibrations are retract of relative $I_0^\infty$-cell (resp. $J_0^\infty$-cell) complexes (see Proposition 11.2.1 in \cite{pH03}), it is enough to prove the result on generating (acyclic) cofibrations. We have $\alpha^*(\Oo(0),\, 0) \cong \Oo'(0)$ since $I_{\Oo(0)} = \left( \mu \otimes (\nu_1 \otimes \cdots \otimes \nu_n) - \gamma_\Oo(\mu \otimes \nu_1 \otimes \cdots \otimes \nu_n) \otimes 1,\ \mu \in \Oo(n),\ \nu_i \in \Oo(0)\right)$. Similarly, we show that $\alpha^*(F_\Oo(\hat \Zz_{q, n}^{0, \infty})) \cong F_\Oo'(\hat \Zz_{q, n}^{0, \infty})$ by means of the fact that the ideal $\im(\eta \otimes {d_{\hat \Zz_{q, n}^{0, \infty}}}^2 - \theta \otimes \id_{\hat \Zz_{q, n}^{0, \infty}}) \subset I_{F_\Oo(\hat \Zz_{q, n}^{0, \infty})}$ is sent to the ideal $\im(\eta \otimes {d_{\hat \Zz_{q, n}^{0, \infty}}}^2 - \alpha(\theta) \otimes \id_{\hat \Zz_{q, n}^{0, \infty}})$ with $\alpha(\theta) = \theta'$. This shows that $\alpha^*$ preserves acyclic cofibrations. The same reasoning shows that $\alpha^*$ sends (generating) cofibrations to (generating) cofibrations.
\end{proof}

Finally, we compare the categories $\Hoalg(\Oo)$ and $\Hoalg(\Oo')$ when the morphism $\alpha : (\Oo,\, d,\, \theta) \to (\Oo',\, d',\, \theta')$ is a weak equivalence, that is a graded quasi-isomorphism, of $\sS$-split complete curved operads compatible with the splitting (that is $\alpha$ sends the splitting to the splitting).

\begin{thm}
\label{thm: equivalence of the homotopy categories}
Let $\alpha : (\Oo,\, d,\, \theta) \to (\Oo',\, d',\, \theta')$ be a graded quasi-isomorphism of $\sS$-split complete curved operads compatible with the splittings. Then the functors $\alpha^*$ and $\alpha_*$ form a pair of Quillen equivalences, that is the functors
\[
\xymatrix{\lL \alpha^* : \Hoalg(\Oo) \ar@<.5ex>@^{->}[r] & \Hoalg(\Oo') : \rR \alpha_* = \alpha_* \ar@<.5ex>@^{->}[l]}
\]
are equivalences of the homotopy categories.
\end{thm}

\begin{proof}
The functors $\alpha_*$ reflects weak equivalences so by Corollary 1.3.16 in \cite{mH91}, it is enough to show that the unit of the adjunction $A \to \alpha_*(\alpha^*(A))$ is a weak equivalence, that is a graded quasi-isomorphism, for every cofibrant $(\Oo,\, d,\, \theta)$-algebra $(A,\, d_A)$. Since we are working over a field $\ringK$ of characteristic $0$, we have already seen that the functor $\Gr$ commutes with direct sums and preserves the tensor products. It follows that
\[ \Gr \alpha^*(A,\, d_A) \cong (\Gr \Oo')(\Gr A^\#)/\left((\Gr \alpha)(\id_{\Gr A^\#})(I_{\Gr A}) \right), \]
since $\Gr I_A \cong I_{\Gr A}$. The result now follows from the proof given in Section 4.7 in \cite{Hinich97}.
\end{proof}

\bibliographystyle{alpha}
\bibliography{bib}

\end{document}